\documentclass[10pt,aps,prb, preprint,reqno]{amsart}
\usepackage{amsmath}
\usepackage{amssymb}
\usepackage{yfonts}
\usepackage{hyperref}
\usepackage{mathrsfs}
\usepackage {latexsym}
\usepackage{graphics}
\usepackage{txfonts}
\usepackage{breqn}
\usepackage{cite}
\usepackage{color}
\usepackage[shortlabels]{enumitem}

\usepackage{wasysym}
\usepackage{tikz}
\usetikzlibrary{snakes}
\usetikzlibrary{arrows.meta}
\usetikzlibrary{decorations.pathmorphing}
\usetikzlibrary{er, positioning} 

\newcommand{\circlesign}[1]{ 
\mathbin{
\mathchoice
{\buildcirclesign{\displaystyle}{#1}}
{\buildcirclesign{\textstyle}{#1}}
{\buildcirclesign{\scriptstyle}{#1}}
{\buildcirclesign{\scriptscriptstyle}{#1}}
} 
}

\newcommand\buildcirclesign[2]{%
\begin{tikzpicture}[baseline=(X.base), inner sep=0, outer sep=0]
\node[draw,circle] (X)  {\ensuremath{#1 #2}};
\end{tikzpicture}%
}

\newtheorem{theorem}{Theorem}[section]
\newtheorem{remark}{Remark}[section]
\newtheorem{lemma}[theorem]{Lemma}

\newtheorem{proposition}[theorem]{Proposition}
\newtheorem{corollary}[theorem]{Corollary}
\newtheorem{define}{Definition}[section]

\newcommand{\joinR}{\hspace{-.1em}}
\newcommand{\RomanI}{I}
\newcommand{\RomanII}{\mbox{\RomanI\joinR\RomanI}}
\newcommand{\RomanIII}{\mbox{\RomanI\joinR\RomanII}}

\DeclareMathOperator*{\Id}{Id}

\DeclareMathOperator*{\supp}{supp}
\DeclareMathOperator*{\divergence}{div}
\DeclareMathSymbol{:}{\mathord}{operators}{"3A}

\allowdisplaybreaks

\begin{document}
\title[2D MHD system forced by space-time white noise]{Remarks on the two-dimensional magnetohydrodynamics system forced by space-time white noise}

 
\author[Kazuo Yamazaki]{Kazuo Yamazaki}  
\address{University of Nebraska, Lincoln, 203 Avery Hall, PO Box, 880130, Lincoln, NE 68588-0130, U.S.A.; Phone: 402-473-3731; Fax: 402-472-8466}
\email{kyamazaki2@unl.edu}
\date{}

\begin{abstract}
We study the two-dimensional magnetohydrodynamics system forced by space-time white noise. Due to a lack of an explicit invariant measure, the approach of Da Prato and Debussche (2002, J. Funct. Anal., \textbf{196}, pp. 180--210) on the Navier-Stokes equations does not seem to fit. We follow instead the approach of Hairer and Rosati (2023, arXiv:2301.11059 [math.PR]), take advantage of the structure of Maxwell's equation, such as anti-symmetry, to find an appropriate paracontrolled ansatz and many crucial cancellations, and prove the global-in-time existence and uniqueness of its solution.
\end{abstract}

\maketitle

\section{Introduction}\label{Section 1}
\subsection{Motivation from physics and mathematics}\label{Motivation from physics and mathematics}
Led by the pioneers such as Alfv$\acute{\mathrm{e}}$n \cite{A42a}, Batchelor \cite{B50}, and Chandrasekhar \cite{C51}, the study of magnetohydrodynamics (MHD) concerning the properties of electrically conducting fluids has continuously attracted much attention from scientific community over the past 60 years. For example, while fluid turbulence is often investigated through Navier-Stokes equations, MHD turbulence occurs in laboratory settings such as fusion confinement devices (e.g. reversed field pinch), as well as astrophysical systems (e.g. solar corona) and the conventional system of equations for such study is that of the MHD. 

The idea of investigating hydrodynamic fluctuations via forcing a system of partial differential equations (PDEs) by stochastic force, especially space-time white noise (STWN) (see \eqref{STWN u}, \eqref{STWN b}), can be traced back as far as \cite{LL57} by Landau and Lifshitz in 1957 (see also \cite{N65}) followed by many others: ferromagnetics (e.g. \cite{MM75}); Kardar-Parisi-Zhang (KPZ) equation \cite{KPZ86}, Navier-Stokes equations (e.g. \cite{FNS77, YO86}); MHD system (e.g. \cite{CT92}), $\Phi^{4}$ model from quantum field theory (e.g. \cite{PW81}), Rayleigh-B$\acute{\mathrm{e}}$nard convection (e.g. \cite{ACHS81, GP75, HS92, SH77, ZS71}).  Theories of fluid turbulence in the two-dimensional (2D) case are generally richer due to the well-known advantage that the vorticity is transported by the velocity (see \eqref{est 10}), and motivated due to possible applications and ease in simulations. For example, we quote 
\begin{quotation}
[the vorticity constraint has profound effects on inertial energy transfer ... A principal reason for exploring 2D turbulence has been the possible application to intermediate-scale meteorological flows. Another motivation is that 2D flows are more easily simulated] ... $\sim$ Kraichnan \cite{K67}.
\end{quotation}
The 2D case of the MHD system forced by STWN \eqref{MHD} has also caught special attention from physicists; e.g. Fournier, Sulem, and Pouquet \cite{FSP82} applied renormalization group approach and found a critical spatial dimension threshold of about 2.8 such that if its dimension is higher, then the system displays two non-trivial regimes: kinetic and magnetic; we also refer to \cite[Section 34]{V09} for more discussions concerning renormalization group approach applied on the MHD system forced by STWN. 

Due to an explicit knowledge of a Gaussian invariant measure of the 2D Navier-Stokes equations forced by STWN, which is precisely due to the fact that the vorticity is transported by velocity, Da Prato and Debussche \cite{DD02} in 2002 proved that for almost every initial data with respect to such a measure, there exists a unique solution globally in time. Due to the coupling with Maxwell's equation, such knowledge of a Gaussian invariant measure and therefore an extension of \cite{DD02} to the 2D MHD system forced by STWN has remained absent in the literature, to the best of the author's knowledge. In the hope to gain a better understanding of the MHD turbulence, we follow the approach of Hairer and Rosati \cite{HR23} to address this issue.

\subsection{Main equations}\label{Main equations} 
We set up a minimum amount of notations here before introducing our main equations. Components of any vector are represented by sub-index. For any vector $k= (k_{1}, k_{2}) \in \mathbb{Z}^{2}$, we denote $k^{\bot} \triangleq (k_{2}, -k_{1})$. We work with spatial variable $x \in \mathbb{T}^{d} = (\mathbb{R}/\mathbb{Z})^{d}$ for $d \in \{2,3\}$ with primary focus on $d=2$. Let us define $\partial_{t} \triangleq \frac{\partial}{\partial t}, \partial_{i} \triangleq \frac{\partial}{\partial x_{i}}$ for $i \in \{1,\hdots, d\}$, and $\mathbb{P}_{\neq 0} f \triangleq f - \fint_{\mathbb{T}^{d}} f(x) dx$. We write $A \lesssim B$ whenever there exists a constant $C \geq 0$ that is independent of important parameters such that $A \leq CB$ while $A \approx B$ whenever $A \lesssim B$ and $A \gtrsim B$. Moreover, we write $A \overset{(\cdot)}{\lesssim} B$ whenever $A \lesssim B$ due to $(\cdot)$. Let us denote the Lebesgue, homogeneous and inhomogeneous Sobolev spaces by $L^{p}, \dot{H}^{s}$ and $H^{s}$ for $p \in [1,\infty], s \in \mathbb{R}$ with corresponding norms $\lVert \cdot \rVert_{L^{p}}, \lVert \cdot \rVert_{\dot{H}^{s}}$, and $\lVert \cdot \rVert_{H^{s}}$, respectively. We also denote the Schwartz space and its dual respectively by $\mathcal{S}$ and $\mathcal{S}'$ and Fourier transform over $\mathbb{T}^{d}$ by $\hat{f} \triangleq \mathcal{F} f$. Lastly, we denote the Leray projection operator by $\mathbb{P}_{L}$, which particularly in the 2D case can be written as 
\begin{equation}\label{est 200}
\mathbb{P}_{L} f(x) \triangleq \sum_{k \in \mathbb{Z}^{2} \setminus \{0\}} e^{i2\pi  k \cdot x} \left( \hat{f} (k) \cdot \frac{k^{\bot}}{\lvert k \rvert^{2}} \right) k^{\bot} \hspace{3mm} \forall \hspace{1mm} f \in \mathcal{S}' (\mathbb{T}^{2}) \text{ such that } \hat{f}(0) = 0.
\end{equation} 

Now we fix a probability space $(\Omega, \mathcal{F}, \mathbb{P})$ and denote the velocity and pressure fields and viscous diffusivity respectively by $u: \hspace{1mm} \mathbb{R}_{\geq 0} \times \mathbb{T}^{d} \mapsto \mathbb{R}^{d}$, $\pi: \hspace{1mm} \mathbb{R}_{\geq 0} \times \mathbb{T}^{d} \mapsto \mathbb{R}$, and $\nu_{u} \geq 0$. We introduce forcing by a certain perturbation $\zeta_{u}$ (see \eqref{est 6}), and STWN $\xi_{u}$, a distribution-valued Gaussian field with a correlation of 
\begin{equation}\label{STWN u}
\mathbb{E} [ \xi_{u,i}(t,x) \xi_{u,j} (s,y)] = 1_{ \{i=j\}}  \delta(t-s) \prod_{l=1}^{d} \delta(x_{l} - y_{l}), 
\end{equation} 
where $\mathbb{E}$ denotes a mathematical expectation; i.e., for all $i, j \in \{1, \hdots, d \}$,
\begin{align}\label{est 234}
\mathbb{E} [\xi_{u,i} (\phi) \xi_{u,j} (\psi) ] = 1_{ \{i=j\}} \int_{\mathbb{R} \times \mathbb{T}^{d}} \phi(t,x) \psi(t,x) dx dt \hspace{3mm} \forall \hspace{1mm} \phi, \psi \in \mathcal{S} (\mathbb{R} \times \mathbb{T}^{d}). 
\end{align} 
Then the stochastic Navier-Stokes equations is written as  
\begin{equation}\label{NS} 
\partial_{t} u + (u\cdot\nabla) u + \nabla \pi = \nu_{u} \Delta u + \mathbb{P}_{\neq 0} (\xi_{u} + \zeta_{u}), \hspace{5mm} \nabla\cdot u = 0,
\end{equation} 
given any initial data $u^{\text{in}}$ that is divergence-free. The equation \eqref{NS} in case $\nu_{u} = 0$ reduces to the stochastic Euler equations. Additionally, we denote the magnetic field and magnetic diffusivity by $b: \hspace{1mm} \mathbb{R}_{\geq 0} \times \mathbb{T}^{d} \mapsto \mathbb{R}^{d}$ and $\nu_{b} \geq 0$, respectively. We force the Navier-Stokes equations \eqref{NS} by Lorentz force $(b\cdot\nabla) b$ and Maxwell's equation by another perturbation $\zeta_{b}$ (see \eqref{est 6}) and its own STWN $\xi_{b}$ such that 
\begin{subequations}\label{STWN b}
\begin{align}
&\mathbb{E} [ \xi_{b,i}(t,x) \xi_{b,j} (s,y) ] = 1_{ \{i=j\}} \delta(t-s) \prod_{l=1}^{d} \delta(x_{l} - y_{l}), \label{STWN b1}\\
& \mathbb{E} [ \xi_{u,i} (t,x) \xi_{b,j} (s,y) ] = 0 \hspace{3mm} \forall \hspace{1mm} i, j \in \{1, \hdots, d\}, x, y \in \mathbb{T}^{d}, s, t \in [0,\infty); \label{STWN b2}
\end{align} 
\end{subequations} 
i.e., for all $i, j \in \{1,\hdots, d \}$, 
\begin{equation}\label{est 233}
\mathbb{E} [\xi_{b,i} (\phi) \xi_{b,j} (\psi) ] = 1_{ \{i=j\}} \int_{\mathbb{R} \times \mathbb{T}^{d}} \phi(t,x) \psi(t,x) dx dt, \hspace{3mm}  \mathbb{E}[ \xi_{u,i} (\phi) \xi_{b,j} (\psi) ] = 0. 
\end{equation} 
Then the stochastic MHD system, for simplicity after applying $\mathbb{P}_{L}$ on both equations already, is written as 
\begin{subequations}\label{MHD} 
\begin{align}
& \partial_{t}  u + \mathbb{P}_{L} (u\cdot\nabla) u = \nu_{u} \Delta u - \mathbb{P}_{L} (b\cdot\nabla) b + \mathbb{P}_{L} \mathbb{P}_{\neq 0} (\xi_{u} + \zeta_{u}),  \label{MHD a}\\
& \partial_{t} b +  \nabla\times (b\times u) = \nu_{b} \Delta b + \mathbb{P}_{L} \mathbb{P}_{\neq 0} (\xi_{b} + \zeta_{b}), \label{MHD b}
\end{align}
\end{subequations} 
given any initial data $(u^{\text{in}}, b^{\text{in}})$, both of which are divergence-free and mean-zero. We point out that \eqref{STWN b2} is a standard assumption upon forcing a system of equations with distinguished structures by STWN (see e.g. \cite[Equation (9)]{ZS71}, \cite[Equation (3)]{GP75}). Hereafter, the stochastic equation with zero stochastic force will be referred to as the deterministic case; e.g. ``the Navier-Stokes equations'' refers to \eqref{NS} with $\xi_{u} = \zeta_{u} = 0$. 

\subsection{Previous works}\label{Previous works} 
In this subsection, we review previous works on the MHD system. As we will see, extending results on the Navier-Stokes equations to the MHD system, which reduces to the Navier-Stokes equations when $b\equiv 0$, always relies on the strategic coupling of nonlinear terms and exploiting cancellations. 
\begin{remark}\label{Remark 1.1}
As a fundamental example, let us recall the derivation of the energy identities for \eqref{NS} and \eqref{MHD}. To explain in the deterministic case for for simplicity, we first write an equivalent formulation of  \eqref{MHD b} without random force 
\begin{equation}\label{est 5}
\partial_{t} b  + (u\cdot\nabla) b -  (b\cdot\nabla) u   = \nu_{b} \Delta b.
\end{equation} 
Then, taking $L^{2}(\mathbb{T}^{d})$-inner products of the first equation in \eqref{MHD a} with $u$ and \eqref{est 5} with $b$ produces four nonlinear terms, of which we compute the first one as 
\begin{equation}\label{est 8} 
\int_{\mathbb{T}^{d}} (u\cdot\nabla) u \cdot u dx = \sum_{i,j=1}^{d} \int_{\mathbb{T}^{d}} \frac{1}{2} u_{i} \partial_{i} \lvert u_{j} \rvert^{2} dx = -\sum_{i,j=1}^{d} \int_{\mathbb{T}^{d}} (\partial_{i} u_{i}) \frac{1}{2} \lvert u_{j} \rvert^{2} dx = 0 
\end{equation}
due to $\nabla\cdot u =0$ from \eqref{NS}. Analogous computations show $\int_{\mathbb{T}^{d}} (u\cdot\nabla) b \cdot b dx = 0$ while $\int_{\mathbb{T}^{d}} (b\cdot\nabla) b \cdot u dx \neq 0$ and $\int_{\mathbb{T}^{d}} (b\cdot\nabla) u \cdot b dx \neq 0$; nevertheless,
\begin{equation}\label{est 11} 
\int_{\mathbb{T}^{d}} ( b\cdot\nabla) b \cdot u + (b\cdot\nabla) u \cdot b dx = 0, 
\end{equation} 
leading to the energy conservation and dissipation when  $\nu_{u} = \nu_{b} = 0$ and $\nu_{u}, \nu_{b} > 0$, respectively. Let us emphasize that we just saw that the first and third nonlinear terms $(u\cdot\nabla) u$ and $(u\cdot\nabla) b$ canceled individually, while the second and fourth nonlinear terms $-(b\cdot\nabla) b$ and $-(b\cdot\nabla) u$, did not cancel individually but they did in sum, respectively. 
\end{remark} 
Based on the energy bound, mathematical analysis of the MHD system was pioneered by Duvaut and Lions \cite{DL72} and fundamental results such as the global existence of a Leray-Hopf weak solution in both cases $d \in \{2,3\}$ and its uniqueness in case $d= 2$ can be found in \cite[Theorem 3.1]{ST83}. In the stochastic case, various results were obtained by many researchers in case noise is white only in time: the existence of a global-in-time weak solution in cases of additive and multiplicative noise in the three-dimensional (3D) case, along with path-wise uniqueness in the 2D case as long as the noise is Lipschitz \cite{M22, S10, SS99}; ergodicity in case of an additive noise in the 2D case \cite{BD07}; large deviation principle in the 2D case \cite{CM10}; Markov selection, irreducibility, and strong Feller property in the 3D case \cite{ Y19a, Y20b}; tamed stochastic MHD system \cite{S21}. 

The main difficulty that arises due to STWN is its roughness; for simplicity, we explain the case of the Navier-Stokes equations and assume hereafter common diffusivity coefficients $\nu \triangleq \nu_{u} = \nu_{b}$. This is more general than the unit viscosity in \cite{HR23}; however, we had trouble considering the case $\nu_{u} \neq \nu_{b}$ (see Remark \ref{Remark 4.2}). To be precise, we recall the H$\ddot{\mathrm{o}}$lder-Besov space $\mathcal{C}^{\alpha} \triangleq B_{\infty,\infty}^{\alpha}$ for any $\alpha \in \mathbb{R}$ that coincides with the classical H$\ddot{\mathrm{o}}$lder space $C^{\alpha}$ whenever $\alpha \in (0,\infty) \setminus \mathbb{N}$, although $C^{k} \subsetneq \mathcal{C}^{k}$ for $k \in \mathbb{N}$ (see \cite[p. 99]{BCD11}); we defer detailed definitions of Besov space to the Section \ref{Preliminaries}. If we denote scaling by 
\begin{equation}\label{est 9} 
\mathfrak{s} = (2, \underbrace{1, \hdots, 1}_{d\text{-many}}) \text{ so that } \lvert \mathfrak{s} \rvert = d + 2, 
\end{equation}
we know e.g. by \cite[Lemma 10.2]{H14} that the STWN $\xi_{u}$ and $\xi_{b}$ $\mathbb{P}$-almost surely ($\mathbb{P}$-a.s.) belong to $\mathcal{C}_{x}^{\alpha} \cap C_{t}^{\frac{\alpha}{2}}$ for every $\alpha < - \frac{\lvert \mathfrak{s} \rvert}{2}$. In the 2D case, it follows that $\xi_{u} \in \mathcal{C}_{x}^{\alpha}$ for $\alpha < -2$. The general assumption on the perturbations $\zeta_{u}$ and $\zeta_{b}$ are such that for $\kappa > 0$ to be taken sufficiently small, 
\begin{equation}\label{est 6} 
(\zeta_{u}, \zeta_{b}) \in (\mathcal{C}_{x}^{-2+ 3 \kappa} \cap \mathcal{C}_{t}^{\frac{-2 + 3 \kappa}{2}}; \mathbb{R}^{2}) \times (\mathcal{C}_{x}^{-2+ 3 \kappa} \cap \mathcal{C}_{t}^{\frac{-2 + 3 \kappa}{2}}; \mathbb{R}^{2}).   
\end{equation} 
Therefore, the $\xi_{u}$ and $\xi_{b}$ are rougher than $\zeta_{u}$ and $\zeta_{b}$. 

Da Prato and Debussche \cite{DD02} considered \eqref{NS} in case $\zeta_{u} \equiv 0$, decomposed its solution $u$ to $u = v_{u} + X_{u}$ where   
\begin{equation}\label{est 14} 
\partial_{t} X_{u} =  \nu \Delta X_{u} + \mathbb{P}_{L} \mathbb{P}_{\neq 0} \xi_{u},  \hspace{3mm}  X_{u}(0,x) = 0,  
\end{equation} 
and 
\begin{equation*}
\partial_{t} v_{u} + \mathbb{P}_{L} \divergence (v_{u} + X_{u} )^{\otimes 2} =  \nu \Delta v_{u}, \hspace{3mm} v_{u}(0,x) = u^{\text{in}}(x), 
\end{equation*} 
in which we denoted $W^{\otimes 2} \triangleq W \otimes W$ (see also \cite{DD03b}). Because $\xi_{u} \in \mathcal{C}_{x}^{\alpha}$ for $\alpha < -2$, we have $X_{u} \in \mathcal{C}_{x}^{\alpha}$ for $\alpha < 0$ and therefore the product $X_{u} \otimes X_{u}$ is ill-defined according to Bony's estimates that informally states that a product $fg$ is well-defined if and only if $f \in \mathcal{C}_{x}^{\alpha_{1}}, g \in \mathcal{C}_{x}^{\alpha_{2}}$ for $\sum_{j=1}^{2}\alpha_{j} > 0$ (see Lemma \ref{Lemma 3.1}). Nevertheless, Da Prato and Debussche were able to define it through Wick products and prove existence of path-wise unique solution locally in time (see \cite[Proposition 5.1]{DD02}). Moreover, the solution $u$ to the 2D Navier-Stokes equations on $\mathbb{T}^{2}$ satisfies 
\begin{equation}\label{est 10} 
\int_{\mathbb{T}^{2}} (u\cdot\nabla) u \cdot \Delta u dx = 0,
\end{equation} 
which can be readily seen by writing $\Delta u = - \nabla\times \nabla \times u$ and integrating by parts to deduce $\int_{\mathbb{T}^{2}} (u\cdot\nabla) u \cdot \Delta u dx = \int_{\mathbb{T}^{2}} (u\cdot\nabla) \omega \cdot \omega dx = 0$ similarly to \eqref{est 8} where $\omega \triangleq \nabla \times u$ represents vorticity (see ``$(b(x), Ax) = 0$'' on \cite[p. 185]{DD02}). This leads to an explicit knowledge of an invariant measure ``$\mu = \times_{k\in\mathbb{Z}_{0}^{2}} \mathcal{N} (0, 1/2\lvert k \rvert^{2})$'' on \cite[p. 185]{DD02}, of which its Gaussianity can be exploited to deduce the existence of path-wise unique solution globally in time starting from $\mu$-almost every initial data $u^{\text{in}}$ (see \cite[Theorem 5.1]{DD02}). 

By definition of $\lvert \mathfrak{s} \rvert$ from \eqref{est 9}, it follows that the roughness issue only becomes worse in higher dimensions and the trick by Da Prato and Debussche no longer applies, e.g. for the 3D Navier-Stokes equations. This is the content of the research direction of singular stochastic PDEs (e.g. \cite{BCJ94, BG97} on Burgers' and KPZ equations forced by STWN). In particular, Hairer \cite{H11} was the first to realize that the rough path theory invented by Lyons \cite{L98} can be used to prove the solution theory for the Burgers' equation forced by STWN; subsequently, he also went on to solve the KPZ equation in \cite{H13}. Soon after, systematic general approaches to singular stochastic PDEs were invented: the theory of regularity structures by Hairer \cite{H14} and the theory of paracontrolled distributions by Gubinelli, Imkeller, and Perkowski \cite{GIP15}. We refer to \cite{CC18} on the 3D $\Phi^{4}$ model, \cite{HM18a} for the strong Feller property of singular stochastic PDEs, \cite{ZZ15, ZZ17a} on the 3D Navier-Stokes equations forced by STWN, and \cite{Y20a, Y21a, Y23a} on the 3D MHD system forced by STWN. 

We briefly mention the recent progress on the convex integration technique applied to stochastic PDEs to prove non-uniqueness. In particular, Hofmanov$\acute{\mathrm{a}}$, Zhu, and Zhu \cite{HZZ19} applied this technique, which was initially developed in the deterministic case (see e.g. \cite{BV19b}), to the 3D Navier-Stokes equations forced by the noise that white only in time. In relevance to the current manuscript, we mention that the author extended it to the 3D MHD system in \cite{Y21d} with noise that is white only in time. On the other hand, we refer to \cite{HZZ21b, LZ23} that employed the convex integration on the 3D and 2D Navier-Stokes equations forced by STWN, respectively. Finally, interestingly Hofmanov$\acute{\mathrm{a}}$, Zhu, and Zhu \cite{HZZ22} were able to apply the convex integration technique to the 2D surface quasi-geostrophic equations forced by noise that is white only in space, even in locally critical and locally supercritical cases when the theory of regularity structures seems inapplicable because its nonlinear term is rougher than its noise. 

Despite various successful examples of extensions of the results on the Navier-Stokes equations to the MHD system, proving the existence and uniqueness of the solution to the 2D MHD system forced by STWN globally in time has remained open, to the best of the author's knowledge. The approach of \cite{DD02} on the Navier-Stokes equations comes to a full stop upon realizing that an analogue of \eqref{est 10} does exist for the MHD system because 
\begin{equation}\label{est 25}
\int_{\mathbb{T}^{2}} (u\cdot\nabla) u \cdot \Delta u dx - \int_{\mathbb{T}^{2}} (b\cdot\nabla) b \cdot \Delta u dx \overset{\eqref{est 10}}{=} - \int_{\mathbb{T}^{2}} (b\cdot\nabla) b \cdot \Delta u dx\neq 0. 
\end{equation} 
Hope for further cancellations upon coupling as we did in \eqref{est 11} does not work because
\begin{align}
& \int_{\mathbb{T}^{2}} (u\cdot\nabla) u \cdot \Delta u - (b\cdot\nabla) b \cdot \Delta u dx + \int_{\mathbb{T}^{2}} (u\cdot\nabla) b \cdot \Delta b - (b\cdot\nabla) u \cdot \Delta b dx \nonumber  \\
=& -2 \int_{\mathbb{T}^{2}} [\partial_{1} b_{1} (\partial_{1} u_{2} + \partial_{2} u_{1}) - \partial_{1} u_{1} (\partial_{1} b_{2} + \partial_{2} b_{1} ) ] (\partial_{1} b_{2} - \partial_{2} b_{1}) dx \neq 0. \label{est 12}
\end{align}
In fact, this issue \eqref{est 12} is the heart of the matter of why although Yudovich \cite{Y63} was able to prove the global regularity of the solution to the 2D Euler equations starting from bounded vorticity, its extension to the 2D MHD system with zero viscous diffusion remains completely open despite much effort by many mathematicians (see \cite{CWY14, FMMNZ14, JZ14a, Y18a} for recent progress). Although global well-posedness was shown via another approach for some singular stochastic PDEs (e.g. \cite{MW17} on the 2D $\Phi^{4}$ model), their approaches do not seem applicable to the MHD system \eqref{MHD}.  Therefore, the problems of extending \cite{DD02} to the 2D MHD system forced by STWN and \cite{Y63} to the MHD system with zero viscous diffusion remained open. In this manuscript, we follow the more recent approach of Hairer and Rosati \cite{HR23} to be described next and address the first problem.

\section{Statement of main results and new ideas}\label{Section 2}
Let us introduce a notation of $\mathbb{M}^{d} \triangleq \{ A = (a_{i,j})_{1\leq i,j \leq d}\}$. Via this notation we can rewrite \eqref{MHD} using \eqref{est 5} as  
\begin{subequations}\label{est 16}
\begin{align} 
&\partial_{t} u + \mathbb{P}_{L} \divergence (u\otimes u - b\otimes b) = \nu \Delta u + \mathbb{P}_{L} \mathbb{P}_{\neq 0} (\xi_{u} + \zeta_{u}), \label{est 16a} \\
&\partial_{t} b + \mathbb{P}_{L}  \divergence (b\otimes u - u\otimes b) =  \nu \Delta b + \mathbb{P}_{L} \mathbb{P}_{\neq 0} (\xi_{b} + \zeta_{b}).\label{est 16b} 
\end{align} 
\end{subequations}  
As we discussed at \eqref{est 9}, $u, b \in \mathcal{C}_{x}^{\alpha}$ for $\alpha < 0$ and thus the products within the nonlinear terms are ill-defined. In addition to \eqref{est 14}, we consider 
\begin{equation}\label{est 15} 
\partial_{t} X_{b} =  \nu\Delta X_{b} + \mathbb{P}_{L} \mathbb{P}_{\neq 0} \xi_{b}, \hspace{3mm} X_{b}(0,x) = 0. 
\end{equation} 
It follows that there exist unique solutions $X_{u}, X_{b} \in C([0,\infty); \mathcal{C}^{-\kappa}(\mathbb{T}^{2}))$ for all $\kappa > 0$ to \eqref{est 14} and \eqref{est 15}, respectively. We can split the solution of \eqref{est 16} to $(u,b) = (v_{u} + X_{u}, v_{b} + X_{b})$ where $(v_{u}, v_{b})$ satisfies 
\begin{subequations}\label{est 17}
\begin{align}
& \partial_{t} v_{u} +  \mathbb{P}_{L} \divergence (( v_{u}+ X_{u})^{\otimes 2} - (v_{b} + X_{b})^{\otimes 2}) = \nu\Delta v_{u} + \mathbb{P}_{L} \mathbb{P}_{\neq 0} \zeta_{u},   \label{est 17a} \\
& \partial_{t} v_{b}  + \mathbb{P}_{L} \divergence ((v_{b} +X_{b}) \otimes (v_{u} + X_{u}) - (v_{u} + X_{u}) \otimes (v_{b} + X_{b} ))= \nu\Delta v_{b} + \mathbb{P}_{L} \mathbb{P}_{\neq 0} \zeta_{b},  \label{est 17b}\\
&v_{u}(0,x) = u^{\text{in}}(x),v_{b} (0,x) = b^{\text{in}}(x),
\end{align}
\end{subequations}
for $(u^{\text{in}}, b^{\text{in}})$ that are both mean-zero and divergence-free. The following result is standard and can be proven via the approach of \cite{DD02} just like \cite[Theorem 2.3]{HR23}. 
\begin{proposition}\label{Proposition 2.1}
\rm{(Cf. \cite[Theorem 2.3]{HR23})} There exists a null set $\mathcal{N} \subset \Omega$ such that for any $\omega \in \Omega \setminus \mathcal{N}$ and $\kappa > 0$, the following holds. For any $(u^{\text{in}}, b^{\text{in}}) \in\mathcal{C}^{-1+ \kappa} \times \mathcal{C}^{-1+ \kappa}$ that are both divergence-free and mean-zero, there exist $T^{\max} (\omega, u^{\text{in}}, b^{\text{in}}) \in (0,\infty]$ and a unique maximal mild solution $(v_{u}, v_{b})(\omega)$ to \eqref{est 17} on $[0, T^{\max} (\omega, u^{\text{in}}, b^{\text{in}}))$ such that $(v_{u}, v_{b})(\omega, 0, x) = (u^{\text{in}}, b^{\text{in}})(x)$. 
\end{proposition} 
We now state our main results. 
\begin{theorem}\label{Theorem 2.2} 
There exists a null set $\mathcal{N}' \subset \Omega$ such that for any $\omega \in \Omega \setminus \mathcal{N}'$, the following holds. For any $\kappa > 0, (u^{\text{in}}, b^{\text{in}}) \in \mathcal{C}^{-1+ \kappa} \times \mathcal{C}^{-1+ \kappa}$ that are both divergence-free and mean-zero, there exist $T^{\max} (\omega, u^{\text{in}}, b^{\text{in}}) \in (0,\infty]$ and a unique maximal mild solution $(v_{u}, v_{b})(\omega)$ to \eqref{est 17} on $[0, T^{\max} (\omega, u^{\text{in}}, b^{\text{in}}))$ such that $(v_{u}, v_{b})(\omega, 0, x) = (u^{\text{in}}, b^{\text{in}})(x)$ and 
\begin{equation*}
T^{\max}(\omega, u^{\text{in}}, b^{\text{in}}) = \infty \hspace{3mm} \forall \hspace{1mm} \omega \in \Omega \setminus \mathcal{N}'. 
\end{equation*} 
\end{theorem} 
To describe the next result, we define 
\begin{equation*}
L_{\sigma}^{2} \triangleq \left\{f: \hspace{1mm} \mathbb{T}^{2} \mapsto \mathbb{R}^{2}, f \in L^{2} (\mathbb{T}^{2}), \nabla\cdot f = 0, \fint_{\mathbb{T}^{2}} f(x) dx = 0 \right\} 
\end{equation*}
but defer the definition of a high-low (HL) weak solution to \eqref{est 17} until Definition \ref{Definition 5.1}, which has slightly better regularity than \cite[Definition 6.1]{HR23}. 

\begin{theorem}\label{Theorem 2.3} 
Consider the same null set $\mathcal{N}' \subset \Omega$ from Theorem \ref{Theorem 2.2}. For every $\omega \in \Omega \setminus \mathcal{N}'$, $\kappa > 0$, and $(u^{\text{in}}, b^{\text{in}}) \in L_{\sigma}^{2} \times L_{\sigma}^{2}$, there exists a unique HL weak solution to  \eqref{est 17} on $[0,\infty)$.  
\end{theorem} 

\begin{remark}\label{Remark 2.1}
As pointed out on \cite[p. 6]{HR23}, starting from $(u^{\text{in}}, b^{\text{in}}) \in L_{\sigma}^{2} \times L_{\sigma}^{2}$, neither Proposition \ref{Proposition 2.1} or Theorem \ref{Theorem 2.2} guarantees even the local existence of a solution due to $L^{2} \not\subseteq \mathcal{C}^{-1+ \kappa}$ for any $\kappa > 0$. 
\end{remark} 

\begin{remark}\label{Remark 2.2}
The estimates on the lower frequency energy will be of the form 
\begin{align}\label{new 1}
\partial_{t} \lVert f(t) \rVert_{L^{2}}^{2} \lesssim  \ln ( e + \lVert f(t) \rVert_{L^{2}}^{2}) \lVert f(t) \rVert_{L^{2}}^{2}
\end{align}
(see e.g. \eqref{est 96} and \eqref{est 232}).  Although in the $H^{k}(\mathbb{R}^{d})$-norm for $k > 0$ sufficiently large rather than $L^{2}(\mathbb{T}^{2})$, we point out that this is precisely the inequality that led to the phenomenon of global regularity of logarithmically supercritical PDEs that was initiated by Tao on wave equation \cite{T07}, extended by Tao himself \cite{T09} to the Navier-Stokes equations (see also \cite{BMR14}) and then to many other PDEs including the MHD system \cite{W11, Y14a, Y18}. E.g., we refer to 
\begin{align}\label{new 2}
\partial_{t} \lVert u(t) \rVert_{H^{k} (\mathbb{R}^{d})}^{2} \leq C a(t) \lVert u(t) \rVert_{H^{k}(\mathbb{R}^{d})}^{2} \log (2+ \lVert u(t) \rVert_{H^{k}(\mathbb{R}^{d})}^{2} 
\end{align}
on \cite[p. 362]{T09}. The reason for this coincidence of \eqref{new 1}-\eqref{new 2} can also be explained. Informally, Hairer and Rosati splits Fourier frequencies between the low and the high parts with the threshold given by $\lambda$ (see Definition \ref{Definition 3.1}) and subsequently chooses $\lambda_{t}$ as essentially the $L^{2}$-norms of the solution (see \eqref{est 37}). In fact, this is essentially the same approach as that of Tao; we quote ``$P_{\leq N}$ and $P_{> N}$ are the Fourier projections to the regions $\{ \xi: \hspace{1mm} \lvert \xi \rvert \leq N\}$ and $\{\xi: \hspace{1mm} \lvert \xi \rvert > N \}$'' and ``optimize in $N$, setting $N \hspace{1mm} := 1+ E_{k}$'' on \cite[p. 365]{T09}, where $E_{k}(t)$ is the $\dot{H}^{k}$-norm of the solution at time $t$ according to \cite[Equation (7)]{T09}.
\end{remark} 

\begin{remark}\label{Remark 2.3}  
We highlight some of the new ideas and novelties of this manuscript. 
\begin{enumerate}
\item Throughout \cite{HR23}, the symmetry of the Navier-Stokes equations is exploited through symmetric tensor product $\otimes_{s}$ (see \eqref{est 18}), symmetric gradient matrix $\nabla_{\text{symm}}$ (see \eqref{est 20}), and then such symmetric tensor products within Bony's paraproducts, specifically $\circlesign{\prec}_{s}, \circlesign{\succ}_{s}$, and $\circlesign{\circ}_{s}$ (see \eqref{est 23}). Maxwell's equation of \eqref{est 16b}   has no symmetry; in fact, it is anti-symmetric (skew-symmetric). Additionally, while the two nonlinear terms of the Navier-Stokes part, e.g. $( v_{u}+ X_{u})^{\otimes 2}$ and $- (v_{b} + X_{b})^{\otimes 2}$ of \eqref{est 17a}, are individually symmetric, the two nonlinear terms, e.g. $(v_{b} +X_{b}) \otimes (v_{u} + X_{u})$ and $-(v_{u} + X_{u}) \otimes (v_{b} + X_{b})$ of \eqref{est 17b}, are neither symmetric or anti-symmetric individually but anti-symmetric when summed, and this requires extra care in handling them (see Remark \ref{Remark 4.1}). Thus, as counterparts to symmetry, we introduce anti-symmetric tensor products $\otimes_{a}$ in \eqref{est 19}, anti-symmetric gradient matrix $\nabla_{\text{anti}}$ \eqref{est 21}, and Bony's paraproducts with such anti-symmetric tensor products $\circlesign{\prec}_{a}, \circlesign{\succ}_{a}$, and $\circlesign{\circ}_{a}$ in \eqref{est 24}. This approach made the crucial step of finding the suitable paracontrolled ansatz $(w_{u}^{\sharp}, w_{b}^{\sharp})$ in \eqref{est 31} significantly easier (see Remark \ref{Remark 4.3} for some of the difficulties). Such anti-symmetric nature of Maxwell's equation has played important roles in previous studies of the MHD system, e.g. the construction of the magnetic Reynolds stress in the convex integration scheme of \cite{BBV21}. The notations of $\otimes_{a}, \nabla_{\text{anti}}, \circlesign{\prec}_{a}, \circlesign{\succ}_{a},$ and $\circlesign{\circ}_{a}$ seem natural and may prove to be convenient for research on the MHD system in general. We also refer to $\nabla_{\text{spec}}$ in \eqref{est 64}.   

\item Some of the definitions and identities, even when the magnetic field vanishes and $\nu_{u} = 1$, differ from those of \cite{HR23} (e.g. \eqref{est 17}, \eqref{est 31}, and \eqref{est 214}). The regularity of the HL weak solution in Definition \ref{Definition 5.1} is slightly higher than that of \cite[Definition 6.1]{HR23}. Moreover, many of our computations are also different from \cite{HR23} (e.g. \eqref{est 33}, \eqref{est 235}, \eqref{est 178}). In particular, while \cite{HR23} mostly relied on $B_{p,\infty}^{s}$-estimates (see \cite[Lemma A.1]{HR23}), because most of the estimates start from $H^{s}$-norms for $s \in \mathbb{R}$, it seems that relying on Lemma \ref{Lemma 3.1} or sometimes directly computing from the definition of Besov space (e.g. \eqref{est 173}) can give more straight-forward estimates instead of bounding $l^{2}$-norm by $l^{\infty}$-norm in the expense of giving up $\epsilon$-much regularity every time, although $\epsilon > 0$ can be arbitrarily small. We also use the well-known product estimate in $\dot{H}^{s}$-norms (see Lemma \ref{Lemma 3.2}) when it is convenient. 

\item As expected from our discussions in \eqref{est 8}, \eqref{est 11}, \eqref{est 25}-\eqref{est 12}, our proof will consist of multiple discoveries of non-trivial cancellations, all of which are crucial (e.g. \eqref{est 45}-\eqref{est 48}, \eqref{est 57}-\eqref{est 236}, \eqref{est 237}-\eqref{est 75}, \eqref{est 238}, and \eqref{est 252}).

\item We chose to allow $\nu \neq 1$ because the explicit dependence of the renormalization constant $r_{\lambda}$ in \eqref{est 140b} on parameters such as $\nu > 0$ is of interest from the physics point of view.  
\end{enumerate} 
\end{remark} 

The rest of this manuscript is organized as follows. In Section \ref{Preliminaries} we give preliminaries needed for proofs, in Sections \ref{Section 4}-\ref{Section 5} we prove respectively  Theorems \ref{Theorem 2.2} and \ref{Theorem 2.3}. In the Appendix, we include some computations for completeness.  

\section{Preliminaries}\label{Preliminaries}
\subsection{Notations and assumptions}\label{Subsection 3.1}
We define $\mathbb{N} \triangleq \{1, 2, \hdots, \}$ while $\mathbb{N}_{0} \triangleq \mathbb{N} \cup \{0\}$. For convenience, we write $\lVert (u,b) \rVert_{X}^{2} = \lVert u \rVert_{X}^{2} + \lVert b \rVert_{X}^{2}$ for various $X$-norms. We also write $C_{t}$ to denote $\sup_{s\in [0,t]}$; e.g. $\lVert f \rVert_{C_{t}C_{x}} = \sup_{s\in[0,t]} \sup_{x\in \mathbb{T}^{2}} \lvert f(t,x) \rvert$. The heat kernel will be denoted by $P_{t} \triangleq e^{ \nu \Delta t}$. We also define a fractional Laplacian $(-\Delta)^{\alpha}$ as a Fourier operator with a Fourier symbol of $\lvert k \rvert^{2\alpha}$ for any $\alpha \in \mathbb{R}$. Following \cite{HR23} we set symmetric tensor product 
\begin{equation}\label{est 18} 
u\otimes_{s} v \triangleq \frac{1}{2} ( u \otimes v + v \otimes u), 
\end{equation} 
and additionally introduce anti-symmetric tensor product 
\begin{equation}\label{est 19} 
u\otimes_{a} v \triangleq \frac{1}{2} (u\otimes v - v \otimes u);  
\end{equation} 
we emphasize that $u\otimes_{s} v = v \otimes_{s} u$ while $u \otimes_{a} v = - v \otimes_{a} u$. Following \cite{HR23}, for any $\phi \in C^{1}(\mathbb{T}^{2}; \mathbb{R}^{2})$, we set for $i,j\in \{1,2\}$, $(\nabla \phi)_{i,j} \triangleq \partial_{i} \phi_{j}$ and 
\begin{equation}\label{est 20}
(\nabla_{\text{symm}} \phi)_{i,j} = \frac{1}{2} (\partial_{i} \phi_{j} + \partial_{j} \phi_{i}), 
\end{equation} 
and additionally 
\begin{equation}\label{est 21} 
(\nabla_{\text{anti}} \phi)_{i,j} \triangleq \frac{1}{2} ( \partial_{i} \phi_{j} - \partial_{j} \phi_{i}). 
\end{equation} 

\subsection{Besov spaces and Bony's paraproducts}\label{Subsection 3.2} 
We let $\chi$ and $\rho$ be smooth functions with compact support on $\mathbb{R}^{2}$ that are non-negative, and radial such that the support of $\chi$ is contained in a ball while that of $\rho$ in an annulus and 
\begin{align*}
& \chi(\xi) + \sum_{j\geq 0} \rho(2^{-j} \xi) = 1 \hspace{3mm} \forall \hspace{1mm} \xi \in \mathbb{R}^{2}, \\
&\supp (\chi) \cap  \supp (\rho(2^{-j} \cdot )) = \emptyset \hspace{1mm} \forall \hspace{1mm} j \in\mathbb{N}, \hspace{2mm} \supp (\rho(2^{-i}\cdot)) \cap \supp (\rho (2^{-j} \cdot)) = \emptyset \hspace{1mm} \text{ if } \lvert i-j \rvert > 1.
\end{align*}
Denoting by $\rho_{j}(\cdot) \triangleq \rho(2^{-j} \cdot)$, we define the Littlewood-Paley operators $\Delta_{j}$ for $j \in \mathbb{N}_{0} \cup \{-1\}$ by 
\begin{equation}\label{est 208}
\Delta_{j} f \triangleq 
\begin{cases}
\mathcal{F}^{-1} (\chi) \ast f & \text{ if } j = -1, \\
\mathcal{F}^{-1} (\rho_{j}) \ast f & \text{ if } j \in \mathbb{N}_{0},
\end{cases}
\end{equation}
and inhomogeneous Besov spaces $B_{p,q}^{s} \triangleq f \in \mathcal{S}': \lVert f \rVert_{B_{p,q}^{s}} < \infty \}$ where 
\begin{equation*}
\lVert f \rVert_{B_{p,q}^{s}} \triangleq  \lVert 2^{sm} \lVert \Delta_{m} f \rVert_{L_{x}^{p}} \rVert_{l_{m}^{q}} \hspace{3mm} \forall \hspace{1mm} p, q \in [1,\infty], s \in \mathbb{R}. 
\end{equation*} 
We define the low-frequency cut-off operator $S_{i} f \triangleq \sum_{-1  \leq j \leq i-1} \Delta_{j} f$ and Bony's paraproducts and remainder respectively as 
\begin{equation*}  
f \prec g  \triangleq \sum_{i\geq -1} S_{i-1} f \Delta_{i} g \hspace{1mm} \text{ and } \hspace{1mm}  f  \circ g  \triangleq \sum_{i\geq -1} \sum_{j: \lvert j\rvert \leq 1} \Delta_{i} f  \Delta_{i+j} g  
\end{equation*}
so that $f g  = f  \prec g  + f  \succ g  + f  \circ g$, where $f  \succ g  = g  \prec f $ (see \cite[Sections 2.6.1 and 2.8.1]{BCD11}). We first extend such definitions by 
\begin{align*}
&f  \circlesign{\prec} g \triangleq \sum_{i \geq -1} S_{i-1} f  \otimes \Delta_{i} g, \\
&f  \circlesign{\succ} g \triangleq \sum_{i \geq -1} \Delta_{i} f  \otimes S_{i-1} g,\\
& f  \circlesign{\circ} g \triangleq \sum_{i \geq -1} \sum_{j: \lvert j \rvert \leq 1}\Delta_{i} f  \otimes \Delta_{i+j} g, 
\end{align*}
so that $f  \otimes g = f  \circlesign{\prec} g + f  \circlesign{\succ} g + f  \circlesign{\circ} g$. Following \cite{HR23} we extend such definitions via 
\begin{equation}\label{est 23} 
f  \circlesign{\prec}_{s} g \triangleq \sum_{i \geq -1} S_{i-1} f  \otimes_{s} \Delta_{i} g \hspace{1mm} \text{ and } \hspace{1mm} f  \circlesign{\circ}_{s} g \triangleq \sum_{i \geq -1} \sum_{j: \lvert j \rvert \leq 1} \Delta_{i} f  \otimes_{s} \Delta_{i+j} g 
\end{equation} 
so that $f  \otimes_{s} g = f  \circlesign{\prec}_{s} g + f  \circlesign{\succ}_{s} g + f  \circlesign{\circ}_{s} g$ where $f  \circlesign{\succ}_{s} g = g \circlesign{\prec}_{s} f $. Finally, we define 
\begin{subequations}\label{est 24} 
\begin{align}
&f  \circlesign{\prec}_{a} g \triangleq \sum_{i\geq -1} S_{i-1} f  \otimes_{a} \Delta_{i} g, \\
&f  \circlesign{\succ}_{a} g \triangleq \sum_{i \geq -1} \Delta_{i} f  \otimes_{a} S_{i-1} g, \\
&f  \circlesign{\circ}_{a} g \triangleq \sum_{i\geq -1} \sum_{j: \lvert j \rvert \leq 1} \Delta_{i} f  \otimes_{a} \Delta_{i+j} g, 
\end{align}
\end{subequations} 
so that $f  \otimes_{a} g = f  \circlesign{\prec}_{a} g + f  \circlesign{\succ}_{a} g + f  \circlesign{\circ}_{a} g$. Analogous definitions follow in cases with one of $f$ and $g$ is or both of $f$ and $g$ are $\mathbb{M}^{2}$-valued (see \cite[p. 37]{HR23}). We recall from \cite{BCD11} that there exist $N_{1}, N_{2} \in \mathbb{N}$ such that 
\begin{equation}\label{est 152} 
\Delta_{m} (f \prec g) = \sum_{j: \lvert j-m \rvert \leq N_{1}} (S_{j-1} f) \Delta_{j} g \hspace{1mm} \text{ and } \hspace{1mm} \Delta_{m} ( f \circ g) = \Delta_{m}\sum_{i\geq m-N_{2}}  \sum_{\lvert j \rvert \leq 1} \Delta_{i} f \Delta_{i+j} g. 
\end{equation} 
For convenience we record some special cases of Bony's estimates: 
\begin{lemma}\label{Lemma 3.1} 
\rm{(\cite[Proposition 3.1]{AC15}; cf. \cite[Lemma A.1]{HR23} for $B_{p,\infty}^{s}$-estimates for $p\in[1,\infty]$ and $s \in \mathbb{R}$)} Let $\alpha, \beta \in \mathbb{R}$. Then 
\begin{subequations}\label{est 40} 
\begin{align}
& \lVert f \prec g \rVert_{H^{\beta -\alpha}} \lesssim_{\alpha, \beta} \lVert f \rVert_{L^{2}} \lVert g \rVert_{\mathcal{C}^{\beta}}  \hspace{8mm}  \forall \hspace{1mm} f \in L^{2}, g \in \mathcal{C}^{\beta} \text{ if } \alpha > 0, \label{est 40a} \\
& \lVert f \succ g \rVert_{H^{\alpha}} \lesssim_{\alpha} \lVert f \rVert_{H^{\alpha}} \lVert g \rVert_{L^{\infty}}  \hspace{10mm}  \forall \hspace{1mm} f \in H^{\alpha}, g \in L^{\infty},  \label{est 40b} \\
& \lVert f \prec g \rVert_{H^{\alpha + \beta}} \lesssim_{\alpha, \beta} \lVert f \rVert_{H^{\alpha}} \lVert g \rVert_{\mathcal{C}^{\beta}} \hspace{7mm}  \forall \hspace{1mm} f \in H^{\alpha}, g \in \mathcal{C}^{\beta}  \text{ if } \alpha < 0,  \label{est 40c}  \\
& \lVert f \succ g \rVert_{H^{\alpha + \beta}} \lesssim_{\alpha,\beta} \lVert f \rVert_{H^{\alpha}}  \lVert g \rVert_{\mathcal{C}^{\beta}}  \hspace{7mm}  \forall \hspace{1mm} f \in H^{\alpha}, g \in \mathcal{C}^{\beta} \text{ if } \beta < 0, \label{est 40d} \\
& \lVert f \circ g \rVert_{H^{\alpha+ \beta}} \lesssim_{\alpha, \beta} \lVert f \rVert_{H^{\alpha}} \lVert g \rVert_{\mathcal{C}^{\beta}}  \hspace{8mm} \forall \hspace{1mm} f \in H^{\alpha}, g \in \mathcal{C}^{\beta}  \text{ if } \alpha + \beta > 0. \label{est 40e} 
\end{align} 
\end{subequations} 
\end{lemma} 
The following well-known inequality is convenient for our estimates: 
\begin{lemma}\label{Lemma 3.2}
\rm{(Cf. \cite[Lemma 2.5]{Y14b})} Let $\sigma_{1}, \sigma_{2} < 1$ such that $\sigma_{1} + \sigma_{2} > 0$. Then 
\begin{equation}\label{est 74}
\lVert f g \rVert_{\dot{H}^{\sigma_{1} + \sigma_{2} - 1}} \lesssim_{\sigma_{1}, \sigma_{2}} \lVert f \rVert_{\dot{H}^{\sigma_{1}}} \lVert g \rVert_{\dot{H}^{\sigma_{2}}} \hspace{3mm} \forall \hspace{1mm} f \in \dot{H}^{\sigma_{1}}(\mathbb{T}^{2}), g \in \dot{H}^{\sigma_{2}}(\mathbb{T}^{2}). 
\end{equation}     
 \end{lemma} 
 
 \begin{define}\label{Definition 3.1} 
\rm{(Cf. \cite[Definition 4.1]{HR23})} Let $\mathfrak{h}:  \hspace{1mm} [0,\infty) \mapsto [0,\infty)$ be a smooth function such that 
\begin{equation*}
\mathfrak{h}(r) \triangleq  
\begin{cases}
1 & \text{ if } r \geq 1, \\
0 & \text{ if } r \leq \frac{1}{2}, 
\end{cases} 
\hspace{5mm} \mathfrak{l} \triangleq 1- \mathfrak{h}.
\end{equation*} 
Then, we define for any $\lambda > 0$
\begin{equation*}
\check{\mathfrak{h}}_{\lambda}(x) \triangleq \mathcal{F}^{-1} \left( \mathfrak{h} \left( \frac{ \lvert \cdot \rvert}{\lambda} \right) \right) (x), \hspace{5mm} \check{\mathfrak{l}}_{\lambda}(x) \triangleq \mathcal{F}^{-1} \left( \mathfrak{l} \left(\frac{ \lvert \cdot \rvert}{\lambda} \right) \right)(x), 
\end{equation*} 
and then the projections onto higher and lower frequencies respectively by 
\begin{equation}\label{est 68}
\mathcal{H}_{\lambda}: \mathcal{S}' \mapsto \mathcal{S}' \text{ by } \mathcal{H}_{\lambda} f \triangleq \check{\mathfrak{h}}_{\lambda} \ast f \text{ and } \mathcal{L}_{\lambda}: \mathcal{S}' \mapsto \mathcal{S} \text{ by } \mathcal{L}_{\lambda} f \triangleq f - \mathcal{H}_{\lambda} f = \check{\mathfrak{l}}_{\lambda} \ast f.  
\end{equation} 
\end{define} 
The following is a straight-forward generalization of \cite[Lemmas 4.2-4.3]{HR23}:  
\begin{lemma}\label{Lemma 3.3} 
\rm{(Cf. \cite[Lemmas 4.1-4.2]{HR23})} For any $p, q \in [1,\infty]$, $\alpha, \beta \in \mathbb{R}$ such that $\beta \geq \alpha$, 
\begin{equation}\label{est 33}
\lVert \mathcal{L}_{\lambda} f \rVert_{B_{p,q}^{\beta}} \lesssim \lambda^{\beta - \alpha} \lVert f \rVert_{B_{p,q}^{\alpha}} \hspace{2mm} \forall \hspace{1mm} f \in B_{p,q}^{\alpha} \hspace{1mm} \text{ and } \hspace{1mm} \lVert \mathcal{H}_{\lambda} f \rVert_{B_{p,q}^{\alpha}} \lesssim \lambda^{\alpha - \beta} \lVert f \rVert_{B_{p,q}^{\beta}} \hspace{2mm} \forall \hspace{1mm} f \in B_{p,q}^{\beta}.
\end{equation} 
\end{lemma} 
 
\section{Proof of Theorem \ref{Theorem 2.2}}\label{Section 4} 
Returning to \eqref{est 17} solved by $(v_{u}, v_{b})$, because $\zeta_{u}, \zeta_{b} \in \mathcal{C}_{x}^{-2+ 3 \kappa}$ due to \eqref{est 6} are the roughest terms in there (because $X_{u}, X_{b} \in \mathcal{C}_{x}^{\alpha}$ for any $\alpha < 0$), we introduce the equations of $Y_{u}, Y_{b}$ that are respectively forced by $\zeta_{u}, \zeta_{b}$:
\begin{subequations}\label{est 26} 
\begin{align}
&\partial_{t} Y_{u} = \nu\Delta Y_{u} - \mathbb{P}_{L} \divergence (2X_{u} \otimes_{s} Y_{u} + X_{u}^{\otimes 2} - 2  X_{b} \otimes_{s} Y_{b} - X_{b}^{\otimes 2}) + \mathbb{P}_{L} \mathbb{P}_{\neq 0} \zeta_{u}, \label{est 26a}\\
&\partial_{t} Y_{b} = \nu\Delta Y_{b} - \mathbb{P}_{L} \divergence (X_{b} \otimes Y_{u} + Y_{b} \otimes X_{u} + X_{b} \otimes X_{u} \nonumber \\
& \hspace{27mm} - X_{u} \otimes Y_{b} - Y_{u} \otimes X_{b} - X_{u} \otimes X_{b})  + \mathbb{P}_{L} \mathbb{P}_{\neq 0}\zeta_{b}, \label{est 26b} \\
& Y_{u}(0,x) = 0, \hspace{1mm} Y_{b}(0,x) = 0. \label{est 26c}
\end{align}
\end{subequations} 

\begin{remark}\label{Remark 4.1}  
In contrast to \eqref{est 26a}, we cannot write \eqref{est 26b} in a symmetric form. Yet, $X_{b} \otimes Y_{u} + Y_{b} \otimes X_{u} + X_{b} \otimes X_{u}$ that we selected from the nonlinear term $(v_{b} + X_{b}) \otimes (v_{u} +X_{u})$ of \eqref{est 17b} and $ - X_{u} \otimes Y_{b} - Y_{u} \otimes X_{b} - X_{u} \otimes X_{b}$ that we selected from $- (v_{u} + X_{u}) \otimes (v_{b} + X_{b} )$ of \eqref{est 17b} together form anti-symmetry, allowing us to rewrite \eqref{est 26} via notations from \eqref{est 19} as follows:  
\begin{subequations}\label{est 27} 
\begin{align}
&\partial_{t} Y_{u} = \nu\Delta Y_{u} - \mathbb{P}_{L} \divergence (2X_{u} \otimes_{s} Y_{u} + X_{u}^{\otimes 2} - 2  X_{b} \otimes_{s} Y_{b} - X_{b}^{\otimes 2}) 
+ \mathbb{P}_{L} \mathbb{P}_{\neq 0} \zeta_{u},  \label{est 27a}\\
&\partial_{t} Y_{b} = \nu\Delta Y_{b} - \mathbb{P}_{L} \divergence (2X_{b} \otimes_{a} Y_{u} + 2Y_{b} \otimes_{a} X_{u} + 2X_{b} \otimes_{a} X_{u})  + \mathbb{P}_{L} \mathbb{P}_{\neq 0}\zeta_{b}, \label{est 27b} \\
& Y_{u}(0,x) = 0,  \hspace{1mm} Y_{b}(0,x) = 0. 
\end{align}
\end{subequations}
\end{remark} 
Now if we define 
\begin{equation}\label{es 139} 
D_{u} \triangleq 2(X_{u} + Y_{u}) \text{ and } D_{b} \triangleq 2(X_{b} + Y_{b}), 
\end{equation} 
then it follows from \eqref{est 27} that 
\begin{equation}\label{est 28}
(w_{u}, w_{b}) \triangleq (v_{u} - Y_{u}, v_{b} - Y_{b}) 
\end{equation} 
satisfies 
\begin{subequations}\label{est 29}
\begin{align}
&\partial_{t} w_{u} + \mathbb{P}_{L} \divergence (w_{u}^{\otimes 2} + D_{u} \otimes_{s} w_{u} + Y_{u}^{\otimes 2} - w_{b}^{\otimes 2} - D_{b} \otimes_{s} w_{b} - Y_{b}^{\otimes 2}) = \nu\Delta w_{u},  \label{est 29a}\\
&\partial_{t} w_{b} +  \mathbb{P}_{L} \divergence (w_{b} \otimes w_{u} + w_{b}\otimes_{a} D_{u} + Y_{b} \otimes Y_{u}  \nonumber \\
& \hspace{16mm} - w_{u} \otimes w_{b} - w_{u} \otimes_{a} D_{b} - Y_{u} \otimes Y_{b})  = \nu\Delta w_{b}, \label{est 29b}\\
& w_{u}(0,x) = u^{\text{in}}(x), w_{b}(0,x) = b^{\text{in}}(x). \label{est 29c}
\end{align}
\end{subequations} 
Because $\zeta_{u}, \zeta_{b} \in \mathcal{C}_{x}^{-2+ 3 \kappa}$ due to \eqref{est 6}, we can expect $Y_{u}, Y_{b} \in \mathcal{C}_{x}^{2\kappa}$. Thus, in $L^{2}(\mathbb{T}^{2})$-estimates of $(w_{u}, w_{b})$, the ill-defined terms are 
\begin{align*}
&\langle w_{u}, \nu\Delta w_{u} + \mathbb{P}_{L} \divergence (2X_{u} \otimes_{s} w_{u} - 2X_{b} \otimes_{s} w_{b}) \rangle, \\
&\langle w_{b}, \nu\Delta w_{b} + \mathbb{P}_{L} \divergence (2w_{b} \otimes_{a} X_{u} - 2w_{u} \otimes_{a} X_{b}) \rangle. 
\end{align*}

\begin{define}\label{Definition 4.1}
Recall $P_{t} = e^{\nu \Delta t}$ from Section \ref{Subsection 3.1}. For any $\gamma > 0, T > 0$, and $\beta \in \mathbb{R}$, we define 
\begin{equation*}
\mathcal{M}_{T}^{\gamma} \mathcal{C}_{x}^{\beta} \triangleq \{ f: \hspace{1mm} t \mapsto t^{\gamma} \lVert f(t) \rVert_{\mathcal{C}_{x}^{\beta}} \text{ is continuous over } [0,T], \lVert f \rVert_{\mathcal{M}_{T}^{\gamma} \mathcal{C}_{x}^{\beta}}  < \infty \}.
\end{equation*} 
where $\lVert f \rVert_{\mathcal{M}_{T}^{\gamma} \mathcal{C}_{x}^{\beta}} \triangleq \lVert t^{\gamma} \lVert f(t) \rVert_{\mathcal{C}_{x}^{\beta}} \rVert_{C_{T}}$. Then a pair $(w_{u}, w_{b}) \in \mathcal{M}_{T}^{\gamma} \mathcal{C}_{x}^{\beta} \times \mathcal{M}_{T}^{\gamma} \mathcal{C}_{x}^{\beta}$ is a mild solution to \eqref{est 29} over $[0,T]$ if 
\begin{align*}
w_{u}(t) = P_{t} u^{\text{in}} - \int_{0}^{t} P_{t-s} \mathbb{P}_{L} \divergence &(w_{u}^{\otimes 2} + D_{u} \otimes_{s} w_{u} + Y_{u}^{\otimes 2} \nonumber \\
&- w_{b}^{\otimes 2} - D_{b} \otimes_{s} w_{b} - Y_{b}^{\otimes 2}) (s) ds, \\
w_{b}(t) = P_{t} b^{\text{in}} - \int_{0}^{t} P_{t-s} \mathbb{P}_{L} \divergence &(w_{b} \otimes w_{u} + w_{b} \otimes_{a} D_{u} + Y_{b} \otimes Y_{u} \nonumber \\
&- w_{u} \otimes w_{b} - w_{u} \otimes_{a} D_{b} - Y_{u} \otimes Y_{b}) (s) ds. 
\end{align*}
\end{define} 
Considering the special structure of the MHD system \eqref{MHD}, recalling the definition of $\nabla_{\text{symm}}\phi, \nabla_{\text{anti}}\phi \in \mathbb{M}^{2}$ for any $\mathbb{R}^{2}$-valued function $\phi$ from \eqref{est 20}-\eqref{est 21} and $\mathcal{L}_{\lambda}$ from Definition \ref{Definition 3.1}, we define 
\begin{equation}\label{est 64}
\nabla_{\text{spec}} (X_{u}, X_{b})  \in \mathbb{M}^{4} \text{ by } \nabla_{\text{spec}} (X_{u}, X_{b}) \triangleq 
\begin{pmatrix}
\nabla_{\text{symm}} X_{u} & \nabla_{\text{anti}} X_{b} \\
-\nabla_{\text{anti}} X_{b} & - \nabla_{\text{symm}} X_{u}
\end{pmatrix}.  
\end{equation} 
The reason for this definition will become clear in \eqref{est 236}.  Next, for any given $\lambda \geq 1$ and $t \in [0,\infty)$, we define the \emph{enhanced noise} by 
\begin{equation}\label{est 165} 
t \mapsto ( \nabla_{\text{spec}} \mathcal{L}_{\lambda} (X_{u}, X_{b} ) (t), \nabla_{\text{spec}} \mathcal{L}_{\lambda}(X_{u}, X_{b})(t) \circlesign{\circ} P^{\lambda}(t) - r_{\lambda} \Id) 
\end{equation} 
where $\Id$ is an identity matrix in $\mathbb{M}^{4}$ 
\begin{subequations}\label{est 140} 
\begin{align}
&P^{\lambda}(t,x) \triangleq \left(-\frac{\nu \Delta}{2} + 1 \right)^{-1} \nabla_{\text{spec}} \mathcal{L}_{\lambda}(X_{u}, X_{b}) (t,x), \label{est 140a} \\ 
&r_{\lambda}(t) \triangleq \sum_{k\in\mathbb{Z}^{2} \setminus \{0\}} \frac{1}{4}  \mathfrak{l} \left( \frac{ \lvert k \rvert}{\lambda}\right)^{2}\left(\frac{1- e^{-2\nu \lvert k \rvert^{2} t}}{\nu} \right)  \left(\frac{ \nu \lvert k \rvert^{2}}{2} + 1\right)^{-1}.  \label{est 140b} 
\end{align}
\end{subequations}    
With $\{\lambda^{i} \}_{i \in \mathbb{N}}$ to be additionally defined in Definition \ref{Definition 4.2}, we now define for any $t \in [0,\infty)$ and $\kappa > 0$  
\begin{subequations}\label{est 39} 
\begin{align}
& L_{t}^{\kappa} \triangleq 1+  \sum_{j\in \{u,b\}} [ \lVert X_{j} \rVert_{C_{t}\mathcal{C}_{x}^{-\kappa}} + \lVert Y_{j} \rVert_{C_{t}\mathcal{C}_{x}^{2\kappa}}],\label{est 39a} \\
& N_{t}^{\kappa} \triangleq L_{t}^{\kappa} +  \sup_{i\in\mathbb{N}}  \lVert ( \nabla_{\text{spec}} \mathcal{L}_{\lambda^{i}} (X_{u}, X_{b})  \circlesign{\circ} P^{\lambda^{i}} - r_{\lambda^{i}}  \Id \rVert_{C_{t}\mathcal{C}_{x}^{-\kappa}}. \label{est 39b} 
\end{align}
\end{subequations} 
The following classical result can be obtained just like \cite[Proposition 3.2]{HR23}.
\begin{proposition}\label{Proposition 4.1} 
\rm{(Cf. \cite[Proposition 3.2]{HR23})} Fix any $\kappa \in (0, \frac{1}{2})$ and then $\gamma \triangleq 1 - \frac{\kappa}{2}$. Suppose that $D_{u}, D_{b} \in C([0,\infty); \mathcal{C}^{-\kappa})$ and $Y_{u}, Y_{b} \in C([0,\infty); \mathcal{C}^{\kappa})$ so that all of $Y_{u}^{\otimes 2}, Y_{b}^{\otimes 2}, Y_{u} \otimes Y_{b}$, and $Y_{b} \otimes Y_{u}$ belong to $C([0,\infty); \mathcal{C}^{2\kappa})$. Then, for all $(u^{\text{in}}, b^{\text{in}}) \in \mathcal{C}^{-1+ 2 \kappa} \times \mathcal{C}^{-1+ 2 \kappa}$ that are both divergence-free and mean-zero, \eqref{est 29} has a unique mild solution $(w_{u}, w_{b}) \in \mathcal{M}_{T^{\max}}^{\frac{\gamma}{2}} \mathcal{C}^{\frac{3\kappa}{2}} \times \mathcal{M}_{T^{\max}}^{\frac{\gamma}{2}} \mathcal{C}^{\frac{3\kappa}{2}}$ over $[0, T^{\max} (L_{t}^{\kappa}, u^{\text{in}}, b^{\text{in}}))$ where $T^{\max}(L_{t}^{\kappa}, u^{\text{in}}, b^{\text{in}}) \in (0,\infty]$. 
\end{proposition} 

The following result is an immediate consequence of the convergence result in Proposition \ref{Proposition 5.4}.  
\begin{proposition}\label{Proposition 4.2} 
\rm{(Cf. \cite[Lemma 3.1]{HR23})} Let $(\Omega, \mathcal{F}, \mathbb{P})$ be a probability space on which the STWN $\xi_{u}, \xi_{b}$ satisfy \eqref{est 233} and \eqref{est 234}. Then there exists a null set $\mathcal{N}'' \subset \Omega$ such that 
\begin{equation*}
N_{t}^{\kappa} (\omega) < \infty \hspace{3mm} \forall \hspace{1mm} \omega \in \Omega \setminus \mathcal{N}'' \hspace{1mm} \forall \hspace{1mm} t \geq 0, \hspace{1mm} \forall \hspace{1mm} \kappa  > 0. 
\end{equation*} 
\end{proposition} 
Given any $(u^{\text{in}}, b^{\text{in}}) \in \mathcal{C}^{-1+ 2 \kappa} \times \mathcal{C}^{-1+ 2 \kappa}$ that are both divergence-free and mean-zero, we can take the null sets $\mathcal{N}$ from Proposition \ref{Proposition 2.1} and $\mathcal{N}''$ from Proposition \ref{Proposition 4.2} and consider a null set $\mathcal{N} \cup \mathcal{N}''$ so that for all $\omega \in \Omega \setminus (\mathcal{N} \cup \mathcal{N}'')$, we have $N_{t}^{\kappa}(\omega) < \infty$ for all $t \geq 0$ and $\kappa > 0$, and consequently $D_{u}, D_{b} \in C([0,\infty); \mathcal{C}^{-\kappa})$ and $Y_{u}, Y_{b} \in C([0,\infty); \mathcal{C}^{\kappa})$. Thus, the hypothesis of Proposition \ref{Proposition 4.1} is valid for any $\omega \in \Omega \setminus (\mathcal{N} \cup \mathcal{N}'')$ and thus its implication holds with the same $T^{\max}$ from Proposition \ref{Proposition 2.1}. . 

Next, we introduce two commutators:
\begin{subequations}\label{est 30}
\begin{align}
C^{\circlesign{\prec}_{s}} (f,g) \triangleq& (\partial_{t} - \nu\Delta) (f \circlesign{\prec}_{s} g) - f \circlesign{\prec}_{s} (\partial_{t} - \nu\Delta) g \nonumber \\
=& ( ( \partial_{t} - \nu\Delta ) f) \circlesign{\prec}_{s} g - 2\nu \sum_{k=1}^{2} \partial_{k} f \circlesign{\prec}_{s} \partial_{k} g, \label{est 30a}\\
C^{\circlesign{\prec}_{a}} (f,g) \triangleq& (\partial_{t} - \nu\Delta) (f \circlesign{\prec}_{a} g) - f \circlesign{\prec}_{a} (\partial_{t} - \nu\Delta) g \nonumber \\
=& ( ( \partial_{t} - \nu\Delta ) f) \circlesign{\prec}_{a} g - 2\nu \sum_{k=1}^{2} \partial_{k} f \circlesign{\prec}_{a} \partial_{k} g,  \label{est 30b}
\end{align}
\end{subequations} 
where the second equalities in both \eqref{est 30a}-\eqref{est 30b} can be verified via \eqref{est 23} and \eqref{est 18}. Furthermore, we define $Q_{u}$ and $Q_{b}$ to solve 
\begin{equation}\label{est 32}
(\partial_{t} - \nu\Delta) Q_{u} = 2 X_{u}, \hspace{1mm} Q_{u} (0) = 0, \hspace{2mm} \text{ and } \hspace{2mm} (\partial_{t} - \nu\Delta) Q_{b} = 2X_{b}, \hspace{1mm} Q_{b}(0) = 0 
\end{equation} 
so that $Q_{u}, Q_{b} \in \mathcal{C}^{\alpha}$ for any $\alpha < 2 - \kappa$ assuming $X_{u}, X_{b} \in \mathcal{C}_{x}^{-\kappa}$ (cf. ``$K_{i}^{u}, K_{i}^{b}$'' in \cite[Equation (2.12)]{Y23a}). Now we define  $w_{u}^{\sharp}$ and $w_{b}^{\sharp}$ by 
\begin{subequations}\label{est 31}
\begin{align}
& w_{u} = - \mathbb{P}_{L} \divergence (w_{u} \circlesign{\prec}_{s} Q_{u} - w_{b} \circlesign{\prec}_{s}Q_{b}) + w_{u}^{\sharp}, \label{est 31a}\\
& w_{b} = -\mathbb{P}_{L} \divergence (w_{b} \circlesign{\prec}_{a} Q_{u} - w_{u} \circlesign{\prec}_{a} Q_{b}) + w_{b}^{\sharp} \label{est 31b}
\end{align}
\end{subequations}  
(cf. ``$u_{i}^{\sharp}, b_{i}^{\sharp}$'' in \cite[Equations (2.15)-(2.17)]{Y23a}). According to \eqref{est 31}, \eqref{est 32}, and \eqref{est 29}, the system of equations satisfied by $w_{u}^{\sharp}$ and $w_{b}^{\sharp}$ defined via \eqref{est 31} is 
\begin{subequations}\label{est 255} 
\begin{align}
& \partial_{t} w_{u}^{\sharp} + \mathbb{P}_{L} \divergence (w_{u}^{\otimes 2} + D_{u} \otimes_{s} w_{u} - 2 X_{u} \circlesign{\succ}_{s} w_{u} - C^{\circlesign{\prec}_{s}} (w_{u}, Q_{u}) + Y_{u}^{\otimes 2} \nonumber \\
& \hspace{15mm} - w_{b}^{\otimes 2} - D_{b} \otimes_{s} w_{b} + 2 X_{b} \circlesign{\succ}_{s} w_{b} + C^{\circlesign{\prec}_{s}} (w_{b}, Q_{b}) - Y_{b}^{\otimes 2})  = \nu\Delta w_{u}^{\sharp}, \label{est 255a}  \\
&\partial_{t} w_{b}^{\sharp} + \mathbb{P}_{L} \divergence (w_{b} \otimes w_{u} + w_{b} \otimes_{a} D_{u} - 2 w_{b} \circlesign{\prec}_{a} X_{u} - C^{\circlesign{\prec}_{a}} (w_{b}, Q_{u}) + Y_{b} \otimes Y_{u} \nonumber \\
& \hspace{15mm} - w_{u} \otimes w_{b} - w_{u} \otimes_{a} D_{b} + 2w_{u} \circlesign{\prec}_{a}   X_{b} + C^{\circlesign{\prec}_{a}}(w_{u}, Q_{b}) - Y_{u} \otimes Y_{b})  = \nu\Delta w_{b}^{\sharp}.\label{est 255b} 
\end{align}
\end{subequations} 

\begin{remark}\label{Remark 4.2} 
 If we allow $\nu_{u} \neq \nu_{b}$ and then naturally  $(\partial_{t} - \nu_{u} \Delta)Q_{u} = 2X_{u}$ and $(\partial_{t} - \nu_{b} \Delta)Q_{b} = 2X_{b}$ in \eqref{est 32}, then e.g. for the equation of $\partial_{t} w_{u}^{\sharp}$ in \eqref{est 255a}, we would have $\nu_{b} w_{b} \circlesign{\prec}_{s} \Delta Q_{b}$ from $w_{b} \circlesign{\prec} (\partial_{t} Q_{b} - 2 X_{b})$ while $\nu_{u} \mathbb{P}_{L} \divergence \Delta (w_{b} \circlesign{\prec}_{s} Q_{b})$ to create diffusion $\nu_{u} \Delta w_{u}^{\sharp}$. Consequently, we cannot form the commutator $C^{\circlesign{\prec}_{s}} (w_{b},Q_{b})$ unless $\nu_{u} = \nu_{b}$ no matter how one redefines $C^{\circlesign{\prec}_{s}}$ in \eqref{est 30}. A similar issue occurs in the formulation of \eqref{est 42}.  
\end{remark}
 
\begin{define}\label{Definition 4.2} 
\rm{(Cf. \cite[Definition 4.4]{HR23})} Fix any $\mathfrak{a} > 0$ and initial data $(u^{\text{in}}, b^{\text{in}}) \in (L_{\sigma}^{2} \cap \mathcal{C}^{-1+\kappa}) \times (L_{\sigma}^{2} \cap \mathcal{C}^{-1+\kappa})$ for some $\kappa > 0$. Define a family of stopping times $\{T_{i}\}_{i\in\mathbb{N}_{0}}$ by 
\begin{subequations}\label{est 36} 
\begin{align}
T_{0} \triangleq& \hspace{1mm} 0, \\
T_{i+1} (\omega, u^{\text{in}}, b^{\text{in}}) \triangleq& \inf\{t \geq T_{i}: \hspace{1mm}  \lVert w_{u}(t) \rVert_{L^{2}} + \lVert w_{b}(t) \rVert_{L^{2}} \geq i + 1 \} \wedge T^{\max} (\omega, u^{\text{in}}, b^{\text{in}})
\end{align}
\end{subequations} 
with $T^{\max} (\omega, u^{\text{in}}, b^{\text{in}})$ from Proposition \ref{Proposition 2.1}. We set 
\begin{equation*}
i_{0} (u^{\text{in}}, b^{\text{in}}) \triangleq \max\{i \in \mathbb{N}: \hspace{1mm} i \leq \lVert u^{\text{in}} \rVert_{L^{2}} + \lVert b^{\text{in}} \rVert_{L^{2}} \} 
\end{equation*}  
so that $T_{i} = 0$ if and only if $i \leq i_{0}(u^{\text{in}}, b^{\text{in}})$. Set $\lambda^{i} \triangleq (i+1)^{\mathfrak{a}}$, 
\begin{equation}\label{est 37}
\lambda_{t} \triangleq 
\begin{cases}
(1+ \lceil \lVert u^{\text{in}} \rVert_{L^{2}} + \lVert b^{\text{in}} \rVert_{L^{2}} \rceil )^{\mathfrak{a}} & \text{ if } t = 0, \\
(1+ \lVert w_{u} (T_{i} ) \rVert_{L^{2}} + \lVert w_{b}(T_{i}) \rVert_{L^{2}} )^{\mathfrak{a}} & \text{ if } t > 0 \text{ and } t \in [T_{i}, T_{i+1}).  
\end{cases} 
\end{equation} 
As $(u^{\text{in}}, b^{\text{in}}) \in L_{\sigma}^{2} \times L_{\sigma}^{2}$, we have $i_{0} (u^{\text{in}}, b^{\text{in}}) < \infty$. Finally, $\lambda_{t} = \lambda^{i}$ for all $t \in [T_{i}, T_{i+1})$ such that $i \geq i_{0} (u^{\text{in}}, b^{\text{in}})$. 
\end{define}
 
For $Q_{u}$ and $Q_{b}$ from \eqref{est 32}, $\mathcal{H}_{\lambda}$ from \eqref{est 68}, $\lambda_{t}$ from \eqref{est 37}, and $w_{u}$ and $w_{b}$ from \eqref{est 29}, we define their high-frequency components as 
\begin{subequations}\label{est 38}
\begin{align}
& Q_{u}^{\mathcal{H}} (t) \triangleq \mathcal{H}_{\lambda_{t}} Q_{u}(t), \hspace{3mm} Q_{b}^{\mathcal{H}}(t) \triangleq \mathcal{H}_{\lambda_{t}} Q_{b}(t), \label{est 38a}\\
& w_{u}^{\mathcal{H}} \triangleq - \mathbb{P}_{L} \divergence (w_{u} \circlesign{\prec}_{s} Q_{u}^{\mathcal{H}} - w_{b} \circlesign{\prec}_{s} Q_{b}^{\mathcal{H}}), \hspace{3mm} w_{u}^{\mathcal{L}} \triangleq w_{u} - w_{u}^{\mathcal{H}}, \label{est 38b}\\
& w_{b}^{\mathcal{H}} \triangleq - \mathbb{P}_{L} \divergence (w_{b} \circlesign{\prec}_{a} Q_{u}^{\mathcal{H}} - w_{u} \circlesign{\prec}_{a} Q_{b}^{\mathcal{H}}), \hspace{3mm} w_{b}^{\mathcal{L}} \triangleq w_{b} - w_{b}^{\mathcal{H}}. \label{est 38c}
\end{align}
\end{subequations}  

\begin{proposition}\label{Proposition 4.3} 
\rm{(Cf. \cite[Lemma 4.6]{HR23})} Fix any $\mathfrak{a} > 0$ from Definition \ref{Definition 4.2}, $\mathcal{N}$ from Proposition \ref{Proposition 2.1} and $\mathcal{N}''$ from Proposition \ref{Proposition 4.2}, and define $N_{t}^{\kappa}$ by \eqref{est 39}. Then, for any $\delta > 0$ and $\omega \in \Omega \setminus (\mathcal{N} \cup \mathcal{N}'')$, there exists a constant $C(\delta) > 0$ such that $w_{u}^{\mathcal{H}}$ and $w_{b}^{\mathcal{H}}$ defined by \eqref{est 38b}-\eqref{est 38c} satisfy 
\begin{align}
&(\lVert w_{u}^{\mathcal{H}}\rVert_{H^{1- 2\kappa - \delta}} + \lVert w_{b}^{\mathcal{H}}  \rVert_{H^{1- 2 \kappa - \delta}} )(\omega, t) \nonumber \\
\leq& C(\delta) (1+ \lVert w_{u}\rVert_{L^{2}} + \lVert w_{b} \rVert_{L^{2}})^{1- \mathfrak{a} \delta}(\omega, t)  N_{t}^{\kappa}(\omega) \hspace{3mm} \forall \hspace{1mm} t \in [0, T^{\max} (\omega, u^{\text{in}}, b^{\text{in}}).  \label{est 67}
\end{align}
\end{proposition}  

\begin{proof}[Proof of Proposition \ref{Proposition 4.3}]
From \eqref{est 38b}, we can estimate via Plancherel theorem, 
\begin{align*}
\lVert w_{u}^{\mathcal{H}} (t) \rVert_{H^{1- 2 \kappa - \delta}}  
\overset{\eqref{est 40a}}{\lesssim}& \lVert w_{u} (t) \rVert_{L^{2}} \lVert Q_{u}^{\mathcal{H}} (t) \rVert_{\mathcal{C}^{2- \frac{3\kappa}{2} -\delta}} + \lVert w_{b}(t) \rVert_{L^{2}} \lVert Q_{b}^{\mathcal{H}} (t) \rVert_{\mathcal{C}^{2- \frac{3\kappa}{2} - \delta}}  \nonumber \\
\overset{\eqref{est 38a}\eqref{est 33}}{\lesssim}& \lVert w_{u}(t) \rVert_{L^{2}} \lambda_{t}^{-\delta} \lVert Q_{u} (t) \rVert_{\mathcal{C}^{2- \frac{3\kappa}{2}}} + \lVert w_{b}(t) \rVert_{L^{2}} \lambda_{t}^{-\delta} \lVert Q_{b}(t) \rVert_{\mathcal{C}^{2- \frac{3\kappa}{2}}} \nonumber \\
\overset{\eqref{est 32} \eqref{est 39}}{\leq}&  C(\delta) (1+ \lVert w_{u}  \rVert_{L^{2}} + \lVert w_{b} \rVert_{L^{2}})^{1- \mathfrak{a} \delta}(t) N_{t}^{\kappa}(\omega).  
\end{align*}
Similar computations starting from \eqref{est 38c} gives the same upper bound for $ \lVert w_{b}^{\mathcal{H}} (t) \rVert_{H^{1- 2 \kappa - \delta}}$. This completes the proof of Proposition \ref{Proposition 4.3}. 
\end{proof}

We now fix $i \in \mathbb{N}$, $i > i_{0}(u^{\text{in}}, b^{\text{in}})$, and $t \in [T_{i}, T_{i+1})$ for some $T_{i} < T_{i+1}$ and compute using \eqref{est 38}, \eqref{est 32}, and \eqref{est 29}, 
\begin{subequations}\label{est 42} 
\begin{align}
& \partial_{t} w_{u}^{\mathcal{L}} - \nu \Delta w_{u}^{\mathcal{L}}    \nonumber \\
=& - \mathbb{P}_{L} \divergence (w_{u}^{\otimes 2} + D_{u} \otimes_{s} w_{u} - 2 (\mathcal{H}_{\lambda_{t}} X_{u}) \circlesign{\succ}_{s} w_{u} - w_{b}^{\otimes 2} - D_{b} \otimes_{s} w_{b} + 2(\mathcal{H}_{\lambda_{t}} X_{b}) \circlesign{\succ}_{s} w_{b})  \nonumber \\
& + \mathbb{P}_{L} \divergence \left( C^{\circlesign{\prec}_{s}} (w_{u}, Q_{u}^{\mathcal{H}}) - Y_{u}^{\otimes 2} - C^{\circlesign{\prec}_{s}} (w_{b}, Q_{b}^{\mathcal{H}}) + Y_{b}^{\otimes 2} \right),  \\
&  \partial_{t} w_{b}^{\mathcal{L}} - \nu \Delta w_{b}^{\mathcal{L}}  \nonumber \\
=& - \mathbb{P}_{L} \divergence ( w_{b} \otimes w_{u} + w_{b} \otimes_{a} D_{u} - 2 w_{b} \circlesign{\prec}_{a} \mathcal{H}_{\lambda_{t}} X_{u} - w_{u} \otimes w_{b} - w_{u} \otimes_{a} D_{b} + 2w_{u} \circlesign{\prec}_{a} \mathcal{H}_{\lambda_{t}} X_{b}) \nonumber \\
&+ \mathbb{P}_{L} \divergence ( C^{\circlesign{\prec}_{a}} (w_{b}, Q_{u}^{\mathcal{H}}) - Y_{b} \otimes Y_{u} - C^{\circlesign{\prec}_{a}} (w_{u}, Q_{b}^{\mathcal{H}}) + Y_{u} \otimes Y_{b}), 
\end{align}
\end{subequations} 
which crucially rely on our choices of $w_{u}^{\mathcal{H}}$ and $w_{b}^{\mathcal{H}}$, in \eqref{est 38b}, especially their signs. Taking $L^{2}(\mathbb{T}^{2})$-inner products on \eqref{est 42} with $(2w_{u}^{\mathcal{L}}, 2w_{b}^{\mathcal{L}})$ we obtain 
\begin{equation}\label{est 94} 
\partial_{t}  \lVert (w_{u}^{\mathcal{L}}, w_{b}^{\mathcal{L}})(t) \rVert_{L^{2}}^{2}  = \sum_{k=1}^{4} \RomanI_{k} 
\end{equation} 
where 
\begin{subequations}\label{est 43}
\begin{align}
\RomanI_{1} \triangleq& 2 \langle w_{u}^{\mathcal{L}}, \nu \Delta w_{u}^{\mathcal{L}} - \divergence ( 2 (\mathcal{L}_{\lambda_{t}} X_{u}) \otimes_{s} w_{u}^{\mathcal{L}} - 2( \mathcal{L}_{\lambda_{t}} X_{b}) \otimes_{s} w_{b}^{\mathcal{L}} ) \rangle(t) \nonumber  \\
+& 2 \langle w_{b}^{\mathcal{L}}, \nu \Delta w_{b}^{\mathcal{L}} - \divergence (2w_{b}^{\mathcal{L}} \otimes_{a} ( \mathcal{L}_{\lambda_{t}} X_{u}) - 2w_{u}^{\mathcal{L}} \otimes_{a} (\mathcal{L}_{\lambda_{t}} X_{b} ) \rangle(t),\label{est 43a}\\
\RomanI_{2} \triangleq& -2 \langle w_{u}^{\mathcal{L}}, \divergence ( 2 ( \mathcal{H}_{\lambda_{t}} X_{u}) \otimes_{s} w_{u}^{\mathcal{L}} - 2 (\mathcal{H}_{\lambda_{t}} X_{u}) \circlesign{\succ}_{s} w_{u}^{\mathcal{L}} \nonumber \\
& \hspace{14mm} - 2 ( \mathcal{H}_{\lambda_{t}} X_{b}) \otimes_{s} w_{b}^{\mathcal{L}} + 2(\mathcal{H}_{\lambda_{t}} X_{b}) \circlesign{\succ}_{s} w_{b}^{\mathcal{L}} ) \rangle(t) \nonumber \\
& - 2 \langle w_{b}^{\mathcal{L}}, \divergence (2w_{b}^{\mathcal{L}} \otimes_{a} (\mathcal{H}_{\lambda_{t}} X_{u}) - 2 w_{b}^{\mathcal{L}} \circlesign{\prec}_{a} (\mathcal{H}_{\lambda_{t}} X_{u}) \nonumber  \\
& \hspace{14mm} -2 w_{u}^{\mathcal{L}} \otimes_{a} (\mathcal{H}_{\lambda_{t}} X_{b}) + 2 w_{u}^{\mathcal{L}} \circlesign{\prec}_{a} (\mathcal{H}_{\lambda_{t}} X_{b}) \rangle(t),  \label{est 43b}\\
\RomanI_{3} \triangleq& - 2 \langle w_{u}^{\mathcal{L}}, \divergence (2X_{u} \otimes_{s} w_{u}^{\mathcal{H}} - 2 (\mathcal{H}_{\lambda_{t}} X_{u}) \circlesign{\succ}_{s} w_{u}^{\mathcal{H}}  \nonumber \\
& \hspace{14mm} - 2 X_{b} \otimes_{s} w_{b}^{\mathcal{H}} + 2(\mathcal{H}_{\lambda_{t}} X_{b}) \circlesign{\succ}_{s} w_{b}^{\mathcal{H}}) \rangle(t) \nonumber \\
&-2 \langle w_{b}^{\mathcal{L}}, \divergence (2w_{b}^{\mathcal{H}} \otimes_{a} X_{u} - 2 w_{b}^{\mathcal{H}} \circlesign{\prec}_{a} (\mathcal{H}_{\lambda_{t}} X_{u})  \nonumber \\
& \hspace{14mm} - 2 w_{u}^{\mathcal{H}} \otimes_{a} X_{b} + 2w_{u}^{\mathcal{H}} \circlesign{\prec}_{a} (\mathcal{H}_{\lambda_{t}} X_{b} ) \rangle(t), \label{est 43c}\\
\RomanI_{4} \triangleq& -2 \langle w_{u}^{\mathcal{L}}, \divergence (w_{u}^{\otimes 2} + 2Y_{u} \otimes_{s} w_{u} - w_{b}^{\otimes 2} - 2Y_{b} \otimes_{s} w_{b} \nonumber  \\
& \hspace{14mm} - C^{\circlesign{\prec}_{s}} (w_{u}, Q_{u}^{\mathcal{H}}) + Y_{u}^{\otimes 2} + C^{\circlesign{\prec}_{s}} (w_{b}, Q_{b}^{\mathcal{H}}) - Y_{b}^{\otimes 2} ) \rangle(t) \nonumber \\
& -2 \langle   w_{b}^{\mathcal{L}}, \divergence (w_{b} \otimes w_{u} + 2w_{b} \otimes_{a} Y_{u} - w_{u} \otimes w_{b} - 2w_{u} \otimes_{a} Y_{b} \nonumber\\
& \hspace{14mm} - C^{\circlesign{\prec}_{a}} (w_{b}, Q_{u}^{\mathcal{H}}) + Y_{b} \otimes Y_{u} + C^{\circlesign{\prec}_{a}} (w_{u}, Q_{b}^{\mathcal{H}})- Y_{u} \otimes Y_{b}) \rangle(t). \label{est 43d}
\end{align}
\end{subequations} 
First, we rewrite \eqref{est 43a} as follows: 
\begin{align}
\RomanI_{1} = - \nu \lVert ( w_{u}^{\mathcal{L}},  w_{b}^{\mathcal{L}})(t)& \rVert_{\dot{H}^{1}}^{2}   + 2 \left\langle w_{u}^{\mathcal{L}}, \frac{\nu}{2} \Delta w_{u}^{\mathcal{L}} - \divergence \left( 2 (\mathcal{L}_{\lambda_{t}} X_{u}) \otimes_{s} w_{u}^{\mathcal{L}} - 2( \mathcal{L}_{\lambda_{t}} X_{b}) \otimes_{s} w_{b}^{\mathcal{L}} \right) \right\rangle(t) \nonumber \\
&+ 2 \left\langle w_{b}^{\mathcal{L}}, \frac{\nu }{2} \Delta w_{b}^{\mathcal{L}} - \divergence \left( 2w_{b}^{\mathcal{L}} \otimes_{a} ( \mathcal{L}_{\lambda_{t}} X_{u}) - 2w_{u}^{\mathcal{L}} \otimes_{a} (\mathcal{L}_{\lambda_{t}} X_{b}) \right) \right\rangle(t). \label{est 44}
\end{align}
We need to discover multiple crucial cancellations within $\RomanI_{1}$ for our estimates to go through. We investigate the first nonlinear term within \eqref{est 44}:
\begin{align}
& \divergence ( 2( \mathcal{L}_{\lambda_{t}} X_{u}) \otimes_{s} w_{u}^{\mathcal{L}}) \nonumber \\
\overset{\eqref{est 18}}{=}&
\begin{pmatrix}
\partial_{1} (\mathcal{L}_{\lambda_{t}} X_{u})_{1} w_{u,1}^{\mathcal{L}} + (\mathcal{L}_{\lambda_{t}} X_{u})_{1} \partial_{1} w_{u,1}^{\mathcal{L}} + \partial_{2} (\mathcal{L}_{\lambda_{t}} X_{u})_{1} w_{u,2}^{\mathcal{L}} + (\mathcal{L}_{\lambda_{t}} X_{u})_{1} \partial_{2} w_{u,2}^{\mathcal{L}} \\
\partial_{1} (\mathcal{L}_{\lambda_{t}} X_{u})_{2} w_{u,1}^{\mathcal{L}} + (\mathcal{L}_{\lambda_{t}} X_{u})_{2} \partial_{1} w_{u,1}^{\mathcal{L}} + \partial_{2} (\mathcal{L}_{\lambda_{t}} X_{u})_{2} w_{u,2}^{\mathcal{L}} + (\mathcal{L}_{\lambda_{t}} X_{u})_{2} \partial_{2} w_{u,2}^{\mathcal{L}} 
\end{pmatrix} 
\nonumber \\
&+ 
\begin{pmatrix}
\partial_{1} w_{u,1}^{\mathcal{L}} (\mathcal{L}_{\lambda_{t}} X_{u})_{1} + w_{u,1}^{\mathcal{L}} \partial_{1} (\mathcal{L}_{\lambda_{t}} X_{u})_{1} + \partial_{2} w_{u,1}^{\mathcal{L}} (\mathcal{L}_{\lambda_{t}} X_{u})_{2} + w_{u,1}^{\mathcal{L}} \partial_{2} (\mathcal{L}_{\lambda_{t}} X_{u})_{2} \\
\partial_{1} w_{u,2}^{\mathcal{L}} (\mathcal{L}_{\lambda_{t}} X_{u})_{1} + w_{u,2}^{\mathcal{L}} \partial_{1} (\mathcal{L}_{\lambda_{t}} X_{u})_{1} + \partial_{2} w_{u,2}^{\mathcal{L}} (\mathcal{L}_{\lambda_{t}} X_{u})_{2} + (w_{u,2}^{\mathcal{L}}) \partial_{2} (\mathcal{L}_{\lambda_{t}} X_{u})_{2} 
\end{pmatrix} \nonumber \\
=& 
\begin{pmatrix}
\partial_{1} (\mathcal{L}_{\lambda_{t}} X_{u})_{1} w_{u,1}^{\mathcal{L}} + \partial_{2} (\mathcal{L}_{\lambda_{t}} X_{u})_{1} w_{u,2}^{\mathcal{L}} + \partial_{1} w_{u,1}^{\mathcal{L}} (\mathcal{L}_{\lambda_{t}} X_{u})_{1} + \partial_{2} w_{u,1}^{\mathcal{L}} (\mathcal{L}_{\lambda_{t}} X_{u})_{2} \\
\partial_{1} (\mathcal{L}_{\lambda_{t}} X_{u})_{2} w_{u,1}^{\mathcal{L}} + \partial_{2} (\mathcal{L}_{\lambda_{t}} X_{u})_{2} w_{u,2}^{\mathcal{L}} + \partial_{1} w_{u,2}^{\mathcal{L}} (\mathcal{L}_{\lambda_{t}} X_{u})_{1} + \partial_{2} w_{u,2}^{\mathcal{L}} (\mathcal{L}_{\lambda_{t}} X_{u})_{2} 
\end{pmatrix} \label{est 45} 
\end{align}
where the second equality made use of eight cancellations due to $\nabla\cdot \mathcal{L}_{\lambda_{t}} X_{u} = 0$ and $\nabla\cdot w_{u}^{\mathcal{L}} = 0$. By identical computations, finding eight cancellations due to $\nabla\cdot \mathcal{L}_{\lambda_{t}} X_{b} = 0$ and $\nabla\cdot w_{b}^{\mathcal{L}} = 0$, we can rewrite the second nonlinear term within \eqref{est 44} as  
\begin{align}
& \divergence ( 2( \mathcal{L}_{\lambda_{t}} X_{b}) \otimes_{s} w_{b}^{\mathcal{L}} \nonumber \\
\overset{\eqref{est 18}}{=}&  
\begin{pmatrix}
\partial_{1} (\mathcal{L}_{\lambda_{t}} X_{b})_{1} w_{b,1}^{\mathcal{L}} + (\mathcal{L}_{\lambda_{t}} X_{b})_{1} \partial_{t} w_{b,1}^{\mathcal{L}} + \partial_{2} (\mathcal{L}_{\lambda_{t}} X_{b})_{1} w_{b,2}^{\mathcal{L}} + (\mathcal{L}_{\lambda_{t}} X_{b})_{1} \partial_{2} w_{b,2}^{\mathcal{L}} \\
\partial_{1} (\mathcal{L}_{\lambda_{t}} X_{b})_{2} w_{b,1}^{\mathcal{L}} + (\mathcal{L}_{\lambda_{t}} X_{b})_{2} \partial_{1} w_{b,1}^{\mathcal{L}} + \partial_{2} (\mathcal{L}_{\lambda_{t}} X_{b})_{2} w_{b,2}^{\mathcal{L}} + (\mathcal{L}_{\lambda_{t}} X_{b})_{2} \partial_{2} w_{b,2}^{\mathcal{L}} 
\end{pmatrix} 
 \nonumber \\
&+ 
\begin{pmatrix}
\partial_{1} w_{b,1}^{\mathcal{L}} (\mathcal{L}_{\lambda_{t}} X_{b})_{1} + w_{b,1}^{\mathcal{L}} \partial_{1} (\mathcal{L}_{\lambda_{t}} X_{b})_{1} + \partial_{2} w_{b,1}^{\mathcal{L}} (\mathcal{L}_{\lambda_{t}} X_{b})_{2} + w_{b,1}^{\mathcal{L}} \partial_{2} (\mathcal{L}_{\lambda_{t}} X_{b})_{2} \\
\partial_{1} w_{b,2}^{\mathcal{L}} (\mathcal{L}_{\lambda_{t}} X_{b})_{1} + w_{b,2}^{\mathcal{L}} \partial_{1} (\mathcal{L}_{\lambda_{t}} X_{b})_{1} + \partial_{2} w_{b,2}^{\mathcal{L}} (\mathcal{L}_{\lambda_{t}} X_{b})_{2} + (w_{b,2}^{\mathcal{L}}) \partial_{2} (\mathcal{L}_{\lambda_{t}} X_{b})_{2} 
\end{pmatrix} \nonumber \\
=& 
\begin{pmatrix}
\partial_{1} (\mathcal{L}_{\lambda_{t}} X_{b})_{1} w_{b,1}^{\mathcal{L}} + \partial_{2} (\mathcal{L}_{\lambda_{t}} X_{b})_{1} w_{b,2}^{\mathcal{L}} + \partial_{1} w_{b,1}^{\mathcal{L}} (\mathcal{L}_{\lambda_{t}} X_{b})_{1} + \partial_{2} w_{b,1}^{\mathcal{L}} (\mathcal{L}_{\lambda_{t}} X_{b})_{2} \\
\partial_{1} (\mathcal{L}_{\lambda_{t}} X_{b})_{2} w_{b,1}^{\mathcal{L}} + \partial_{2} (\mathcal{L}_{\lambda_{t}} X_{b})_{2} w_{b,2}^{\mathcal{L}} + \partial_{1} w_{b,2}^{\mathcal{L}} (\mathcal{L}_{\lambda_{t}} X_{b})_{1} + \partial_{2} w_{b,2}^{\mathcal{L}} (\mathcal{L}_{\lambda_{t}} X_{b})_{2} 
\end{pmatrix}. \label{est 46}
\end{align}
It is obvious that this same approach would not work for either the third nonlinear term $2w_{b}^{\mathcal{L}} \otimes_{a} ( \mathcal{L}_{\lambda_{t}} X_{u})$  or the fourth nonlinear term $- 2w_{u}^{\mathcal{L}} \otimes_{a} (\mathcal{L}_{\lambda_{t}} X_{b})$ within \eqref{est 44} due to the opposite signs we get within each anti-symmetric tensor product (recall the definitions of $\otimes_{a}$ from \eqref{est 19}). At first sight, this seems to raise a concern, because those with experience on the MHD system know that cancellations as a sum within the MHD system typically occur only for second and fourth nonlinear terms 
\begin{equation*}
- 2( \mathcal{L}_{\lambda_{t}} X_{b}) \otimes_{s} w_{b}^{\mathcal{L}}- 2w_{u}^{\mathcal{L}} \otimes_{a} (\mathcal{L}_{\lambda_{t}} X_{b}), 
\end{equation*}
not the third and fourth 
\begin{equation*}
2w_{b}^{\mathcal{L}} \otimes_{a} ( \mathcal{L}_{\lambda_{t}} X_{u}) - 2w_{u}^{\mathcal{L}} \otimes_{a} (\mathcal{L}_{\lambda_{t}} X_{b}); 
\end{equation*}
recall Remark \ref{Remark 1.1}. However, the first and second nonlinear terms already canceled individually in \eqref{est 45}-\eqref{est 46}, and we are faced with the third and fourth nonlinear terms that do not cancel individually. Nevertheless, if we expand them together, we can obtain 
\begin{align*}
& \langle w_{b}^{\mathcal{L}},  \divergence ( 2w_{b}^{\mathcal{L}} \otimes_{a} ( \mathcal{L}_{\lambda_{t}} X_{u}) - 2w_{u}^{\mathcal{L}} \otimes_{a} (\mathcal{L}_{\lambda_{t}} X_{b}) ) \rangle \nonumber \\
\overset{\eqref{est 19}}{=}&  \langle w_{b}^{\mathcal{L}},  \divergence( w_{b}^{\mathcal{L}} \otimes \mathcal{L}_{\lambda_{t}} X_{u} - \mathcal{L}_{\lambda_{t}} X_{u} \otimes w_{b}^{\mathcal{L}} - w_{u}^{\mathcal{L}} \otimes \mathcal{L}_{\lambda_{t}} X_{b} + \mathcal{L}_{\lambda_{t}} X_{b} \otimes w_{u}^{\mathcal{L}} ) \rangle, 
\end{align*}
and it turns out that we can combine first and fourth terms in a pair, as well as second and third terms in another pair to obtain the necessary cancellations as follows: 
\begin{align}
&\divergence ( w_{b}^{\mathcal{L}} \otimes \mathcal{L}_{\lambda_{t}} X_{u} + \mathcal{L}_{\lambda_{t}} X_{b} \otimes w_{u}^{\mathcal{L}}) \nonumber \\
=& 
\begin{pmatrix}
\partial_{1} w_{b,1}^{\mathcal{L}} (\mathcal{L}_{\lambda_{t}} X_{u})_{1} + w_{b,1}^{\mathcal{L}} \partial_{1} (\mathcal{L}_{\lambda_{t}} X_{u})_{1} + \partial_{2} w_{b,1}^{\mathcal{L}} (\mathcal{L}_{\lambda_{t}} X_{u})_{2} + w_{b,1}^{\mathcal{L}} \partial_{2} (\mathcal{L}_{\lambda_{t}} X_{u})_{2} \\
\partial_{1} w_{b,2}^{\mathcal{L}} (\mathcal{L}_{\lambda_{t}} X_{u})_{1} + w_{b,2}^{\mathcal{L}} \partial_{1} (\mathcal{L}_{\lambda_{t}} X_{u})_{1} + \partial_{2} w_{b,2}^{\mathcal{L}} (\mathcal{L}_{\lambda_{t}} X_{u})_{2} + w_{b,2}^{\mathcal{L}} \partial_{2} (\mathcal{L}_{\lambda_{t}} X_{u})_{2} 
\end{pmatrix}  \nonumber \\
&+ 
\begin{pmatrix}
\partial_{1} (\mathcal{L}_{\lambda_{t}} X_{b})_{1} w_{u,1}^{\mathcal{L}} + (\mathcal{L}_{\lambda_{t}} X_{b})_{1} \partial_{1} w_{u,1}^{\mathcal{L}} + \partial_{2} (\mathcal{L}_{\lambda_{t}}X_{b})_{1} w_{u,2}^{\mathcal{L}} + (\mathcal{L}_{\lambda_{t}} X_{b})_{1} \partial_{2} w_{u,2}^{\mathcal{L}} \\
\partial_{1} (\mathcal{L}_{\lambda_{t}} X_{b})_{2} w_{u,1}^{\mathcal{L}} + (\mathcal{L}_{\lambda_{t}} X_{b})_{2} \partial_{1} w_{u,1}^{\mathcal{L}} + \partial_{2} (\mathcal{L}_{\lambda_{t}} X_{b})_{2} w_{u,2}^{\mathcal{L}} + (\mathcal{L}_{\lambda_{t}} X_{b})_{2} \partial_{2} w_{u,2}^{\mathcal{L}} 
\end{pmatrix} \nonumber \\
=& 
\begin{pmatrix}
\partial_{1} w_{b,1}^{\mathcal{L}} (\mathcal{L}_{\lambda_{t}} X_{u})_{1} + \partial_{2} w_{b,1}^{\mathcal{L}} (\mathcal{L}_{\lambda_{t}} X_{u})_{2} + \partial_{1} (\mathcal{L}_{\lambda_{t}} X_{b})_{1} w_{u,1}^{\mathcal{L}} + \partial_{2} (\mathcal{L}_{\lambda_{t}} X_{b})_{1} w_{u,2}^{\mathcal{L}} \\
\partial_{1} w_{b,2}^{\mathcal{L}} ( \mathcal{L}_{\lambda_{t}} X_{u})_{1} + \partial_{2} w_{b,2}^{\mathcal{L}} (\mathcal{L}_{\lambda_{t}} X_{u})_{2} + \partial_{1} (\mathcal{L}_{\lambda_{t}} X_{b})_{2} w_{u,1}^{\mathcal{L}} + \partial_{2} (\mathcal{L}_{\lambda_{t}} X_{b})_{2} w_{u,2}^{\mathcal{L}} 
\end{pmatrix}, \label{est 47} 
\end{align} 
where we canceled eight terms due to $\nabla\cdot \mathcal{L}_{\lambda_{t}} X_{u} = 0$ and $\nabla\cdot w_{u}^{\mathcal{L}} = 0$, and 
\begin{align}
& \divergence ( w_{u}^{\mathcal{L}} \otimes \mathcal{L}_{\lambda_{t}} X_{b} + \mathcal{L}_{\lambda_{t}} X_{u} \otimes w_{b}^{\mathcal{L}}) \nonumber \\
=& 
\begin{pmatrix}
\partial_{1} w_{u,1}^{\mathcal{L}} (\mathcal{L}_{\lambda_{t}} X_{b})_{1} + w_{u,1}^{\mathcal{L}} \partial_{1} (\mathcal{L}_{\lambda_{t}} X_{b})_{1} + \partial_{2} w_{u,1}^{\mathcal{L}} (\mathcal{L}_{\lambda_{t}} X_{b})_{2} + w_{u,1}^{\mathcal{L}} \partial_{2} (\mathcal{L}_{\lambda_{t}} X_{b})_{2} \\
\partial_{1} w_{u,2}^{\mathcal{L}} (\mathcal{L}_{\lambda_{t}} X_{b})_{1} + w_{u,2}^{\mathcal{L}} \partial_{1} (\mathcal{L}_{\lambda_{t}} X_{b})_{1} + \partial_{2} w_{u,2}^{\mathcal{L}} (\mathcal{L}_{\lambda_{t}} X_{b})_{2} + w_{u,2}^{\mathcal{L}} \partial_{2} (\mathcal{L}_{\lambda_{t}} X_{b})_{2} 
\end{pmatrix} \nonumber \\
&+  
\begin{pmatrix}
\partial_{1} (\mathcal{L}_{\lambda_{t}} X_{u})_{1} w_{b,1}^{\mathcal{L}} + (\mathcal{L}_{\lambda_{t}} X_{u})_{1} \partial_{1} w_{b,1}^{\mathcal{L}} + \partial_{2} (\mathcal{L}_{\lambda_{t}} X_{u})_{1} w_{b,2}^{\mathcal{L}} + (\mathcal{L}_{\lambda_{t}} X_{u})_{1} \partial_{2} w_{b,2}^{\mathcal{L}} \\
\partial_{1} (\mathcal{L}_{\lambda_{t}} X_{u})_{2} w_{b,1}^{\mathcal{L}} + (\mathcal{L}_{\lambda_{t}} X_{u})_{2} \partial_{1} w_{b,1}^{\mathcal{L}} + \partial_{2} (\mathcal{L}_{\lambda_{t}} X_{u})_{2} w_{b,2}^{\mathcal{L}} + (\mathcal{L}_{\lambda_{t}} X_{u})_{2} \partial_{2} w_{b,2}^{\mathcal{L}} 
\end{pmatrix} \nonumber \\
=& 
\begin{pmatrix}
\partial_{1} w_{u,1}^{\mathcal{L}} (\mathcal{L}_{\lambda_{t}} X_{b})_{1} + \partial_{2} w_{u,1}^{\mathcal{L}} (\mathcal{L}_{\lambda_{t}} X_{b})_{2} + \partial_{1} (\mathcal{L}_{\lambda_{t}} X_{u})_{1} w_{b,1}^{\mathcal{L}} + \partial_{2} (\mathcal{L}_{\lambda_{t}} X_{u})_{1} w_{b,2}^{\mathcal{L}} \\
\partial_{1}w_{u,2}^{\mathcal{L}} (\mathcal{L}_{\lambda_{t}} X_{b})_{1} + \partial_{2} w_{u,2}^{\mathcal{L}} (\mathcal{L}_{\lambda_{t}} X_{b})_{2} + \partial_{1} (\mathcal{L}_{\lambda_{t}} X_{u})_{2} w_{b,1}^{\mathcal{L}} + \partial_{2} (\mathcal{L}_{\lambda_{t}} X_{u})_{2} w_{b,2}^{\mathcal{L}} 
\end{pmatrix}, \label{est 48}
\end{align}
where we canceled eight terms due to $\nabla\cdot \mathcal{L}_{\lambda_{t}} X_{b} = 0$ and $\nabla\cdot w_{b}^{\mathcal{L}} = 0$. Therefore, applying \eqref{est 45}, \eqref{est 46}, \eqref{est 47}, and \eqref{est 48} to \eqref{est 44} gives us within $\RomanI_{1}$, 
\begin{align}
& \langle w_{u}^{\mathcal{L}}, \divergence \left(2( \mathcal{L}_{\lambda_{t}} X_{u}) \otimes_{s} w_{u}^{\mathcal{L}} - 2( \mathcal{L}_{\lambda_{t}} X_{b}) \otimes_{s} w_{b}^{\mathcal{L}} \right) \rangle(t) \nonumber \\
&+ \langle w_{b}^{\mathcal{L}}, \divergence \left(2 w_{b}^{\mathcal{L}} \otimes_{a} (\mathcal{L}_{\lambda_{t}}X_{u})-2w_{u}^{\mathcal{L}} \otimes_{a} (\mathcal{L}_{\lambda_{t}}X_{b})  \right) \rangle(t) = \sum_{k=1}^{4} \RomanI_{1,k},  \label{est 55} 
\end{align}
where  
\begin{subequations}\label{est 49} 
\begin{align}
\RomanI_{1,1} \triangleq &  \int_{\mathbb{T}^{2}}  \begin{pmatrix}
\partial_{1} (\mathcal{L}_{\lambda_{t}} X_{u})_{1} w_{u,1}^{\mathcal{L}} + \partial_{2} (\mathcal{L}_{\lambda_{t}} X_{u})_{1} w_{u,2}^{\mathcal{L}} + \partial_{1} w_{u,1}^{\mathcal{L}} (\mathcal{L}_{\lambda_{t}} X_{u})_{1} + \partial_{2} w_{u,1}^{\mathcal{L}} (\mathcal{L}_{\lambda_{t}} X_{u})_{2} \\
\partial_{1} (\mathcal{L}_{\lambda_{t}} X_{u})_{2} w_{u,1}^{\mathcal{L}} + \partial_{2} (\mathcal{L}_{\lambda_{t}} X_{u})_{2} w_{u,2}^{\mathcal{L}} + \partial_{1} w_{u,2}^{\mathcal{L}} (\mathcal{L}_{\lambda_{t}} X_{u})_{1} + \partial_{2} w_{u,2}^{\mathcal{L}} (\mathcal{L}_{\lambda_{t}} X_{u})_{2} 
\end{pmatrix}  \nonumber \\
& \hspace{80mm} \cdot 
\begin{pmatrix}
w_{u,1}^{\mathcal{L}}\\
w_{u,2}^{\mathcal{L}} 
\end{pmatrix} (t)dx, \label{est 49a} \\
\RomanI_{1,2} \triangleq& - \int_{\mathbb{T}^{2}}  \begin{pmatrix}
\partial_{1} (\mathcal{L}_{\lambda_{t}} X_{b})_{1} w_{b,1}^{\mathcal{L}} + \partial_{2} (\mathcal{L}_{\lambda_{t}} X_{b})_{1} w_{b,2}^{\mathcal{L}} + \partial_{1} w_{b,1}^{\mathcal{L}} (\mathcal{L}_{\lambda_{t}} X_{b})_{1} + \partial_{2} w_{b,1}^{\mathcal{L}} (\mathcal{L}_{\lambda_{t}} X_{b})_{2} \\
\partial_{1} (\mathcal{L}_{\lambda_{t}} X_{b})_{2} w_{b,1}^{\mathcal{L}} + \partial_{2} (\mathcal{L}_{\lambda_{t}} X_{b})_{2} w_{b,2}^{\mathcal{L}} + \partial_{1} w_{b,2}^{\mathcal{L}} (\mathcal{L}_{\lambda_{t}} X_{b})_{1} + \partial_{2} w_{b,2}^{\mathcal{L}} (\mathcal{L}_{\lambda_{t}} X_{b})_{2} 
\end{pmatrix}  \nonumber \\
& \hspace{80mm} \cdot 
\begin{pmatrix}
w_{u,1}^{\mathcal{L}} \\
w_{u,2}^{\mathcal{L}} 
\end{pmatrix} (t)dx, \label{est 49b} \\
\RomanI_{1,3} \triangleq& \int_{\mathbb{T}^{2}}   \begin{pmatrix}
\partial_{1} w_{b,1}^{\mathcal{L}} (\mathcal{L}_{\lambda_{t}} X_{u})_{1} + \partial_{2} w_{b,1}^{\mathcal{L}} (\mathcal{L}_{\lambda_{t}} X_{u})_{2} + \partial_{1} (\mathcal{L}_{\lambda_{t}} X_{b})_{1} w_{u,1}^{\mathcal{L}} + \partial_{2} (\mathcal{L}_{\lambda_{t}} X_{b})_{1} w_{u,2}^{\mathcal{L}} \\
\partial_{1} w_{b,2}^{\mathcal{L}} ( \mathcal{L}_{\lambda_{t}} X_{u})_{1} + \partial_{2} w_{b,2}^{\mathcal{L}} (\mathcal{L}_{\lambda_{t}} X_{u})_{2} + \partial_{1} (\mathcal{L}_{\lambda_{t}} X_{b})_{2} w_{u,1}^{\mathcal{L}} + \partial_{2} (\mathcal{L}_{\lambda_{t}} X_{b})_{2} w_{u,2}^{\mathcal{L}} 
\end{pmatrix}  \nonumber \\
& \hspace{80mm} \cdot 
\begin{pmatrix}
w_{b,1}^{\mathcal{L}} \\
w_{b,2}^{\mathcal{L}} 
\end{pmatrix} (t)dx, \label{est 49c}  \\
\RomanI_{1,4} \triangleq& - \int_{\mathbb{T}^{2}}  \begin{pmatrix}
\partial_{1} w_{u,1}^{\mathcal{L}} (\mathcal{L}_{\lambda_{t}} X_{b})_{1} + \partial_{2} w_{u,1}^{\mathcal{L}} (\mathcal{L}_{\lambda_{t}} X_{b})_{2} + \partial_{1} (\mathcal{L}_{\lambda_{t}} X_{u})_{1} w_{b,1}^{\mathcal{L}} + \partial_{2} (\mathcal{L}_{\lambda_{t}} X_{u})_{1} w_{b,2}^{\mathcal{L}} \\
\partial_{1}w_{u,2}^{\mathcal{L}} (\mathcal{L}_{\lambda_{t}} X_{b})_{1} + \partial_{2} w_{u,2}^{\mathcal{L}} (\mathcal{L}_{\lambda_{t}} X_{b})_{2} + \partial_{1} (\mathcal{L}_{\lambda_{t}} X_{u})_{2} w_{b,1}^{\mathcal{L}} + \partial_{2} (\mathcal{L}_{\lambda_{t}} X_{u})_{2} w_{b,2}^{\mathcal{L}} 
\end{pmatrix}  \nonumber \\
& \hspace{80mm} \cdot 
\begin{pmatrix}
w_{b,1}^{\mathcal{L}} \\
w_{b,2}^{\mathcal{L}} 
\end{pmatrix} (t)dx. \label{est 49d} 
\end{align}
\end{subequations} 
We need to find more cancellations. First, within $\RomanI_{1,1}$, we compute 
\begin{align}
\RomanI_{1,1} \overset{\eqref{est 49a}}{=}& \int_{\mathbb{T}^{2}} [ \partial_{1} (\mathcal{L}_{\lambda_{t}} X_{u})_{1} \lvert w_{u,1}^{\mathcal{L}} \rvert^{2} + \partial_{2} (\mathcal{L}_{\lambda_{t}} X_{u})_{1} w_{u,2}^{\mathcal{L}} w_{u,1}^{\mathcal{L}}  \nonumber\\
& \hspace{5mm} + (\mathcal{L}_{\lambda_{t}} X_{u})_{1} \frac{1}{2}\partial_{1} \lvert w_{u,1}^{\mathcal{L}} \rvert^{2} + (\mathcal{L}_{\lambda_{t}} X_{u})_{2} \frac{1}{2} \partial_{2} \lvert w_{u,1}^{\mathcal{L}} \rvert^{2} \nonumber\\
& \hspace{5mm} + \partial_{1} (\mathcal{L}_{\lambda_{t}} X_{u})_{2} w_{u,1}^{\mathcal{L}} w_{u,2}^{\mathcal{L}} + \partial_{2} (\mathcal{L}_{\lambda_{t}} X_{u})_{2} \lvert w_{u,2}^{\mathcal{L}} \rvert^{2} \nonumber\\
& \hspace{5mm} + (\mathcal{L}_{\lambda_{t}}X_{u})_{1} \frac{1}{2} \partial_{1} \lvert w_{u,2}^{\mathcal{L}} \rvert^{2} + (\mathcal{L}_{\lambda_{t}} X_{u})_{2} \frac{1}{2} \partial_{2} \lvert w_{u,2}^{\mathcal{L}} \rvert^{2} ](t) dx\nonumber\\
=& \int_{\mathbb{T}^{2}} [ \partial_{1} (\mathcal{L}_{\lambda_{t}} X_{u})_{1} \lvert w_{u,1}^{\mathcal{L}} \rvert^{2} + \partial_{2} (\mathcal{L}_{\lambda_{t}} X_{u})_{1} w_{u,2}^{\mathcal{L}} w_{u,1}^{\mathcal{L}}   \nonumber \\
& \hspace{5mm} + \partial_{1} (\mathcal{L}_{\lambda_{t}} X_{u})_{2} w_{u,1}^{\mathcal{L}} w_{u,2}^{\mathcal{L}} + \partial_{2} (\mathcal{L}_{\lambda_{t}} X_{u})_{2} \lvert w_{u,2}^{\mathcal{L}} \rvert^{2} ] (t) dx  \label{est 57} 
\end{align}
where four terms were cancelled due to $\nabla\cdot \mathcal{L}_{\lambda_{t}} X_{u} = 0$. Similar cancellations can be found within $\RomanI_{1,3}$ as follows: 
\begin{align}
\RomanI_{1,3} \overset{\eqref{est 49c}}{=}&  \int_{\mathbb{T}^{2}} [ (\mathcal{L}_{\lambda_{t}} X_{u})_{1} \frac{1}{2} \partial_{1} \lvert w_{b,1}^{\mathcal{L}} \rvert^{2}+ (\mathcal{L}_{\lambda_{t}} X_{u})_{2} \frac{1}{2} \partial_{2} \lvert w_{b,1}^{\mathcal{L}} \rvert^{2}  \nonumber \\
& \hspace{3mm} +  \partial_{1} (\mathcal{L}_{\lambda_{t}} X_{b})_{1} w_{u,1}^{\mathcal{L}} w_{b,1}^{\mathcal{L}} + \partial_{2} (\mathcal{L}_{\lambda_{t}} X_{b})_{1} w_{u,2}^{\mathcal{L}} w_{b,1}^{\mathcal{L}} \nonumber \\
& \hspace{3mm} + (\mathcal{L}_{\lambda_{t}} X_{u})_{1} \frac{1}{2} \partial_{1} \lvert w_{b,2}^{\mathcal{L}} \rvert^{2} + (\mathcal{L}_{\lambda_{t}} X_{u})_{2} \frac{1}{2} \partial_{2} \lvert w_{b,2}^{\mathcal{L}} \rvert^{2}  \nonumber \\
& \hspace{3mm} + \partial_{1} (\mathcal{L}_{\lambda_{t}} X_{b})_{2} w_{u,1}^{\mathcal{L}} w_{b,2}^{\mathcal{L}} + \partial_{2} (\mathcal{L}_{\lambda_{t}} X_{b})_{2} w_{u,2}^{\mathcal{L}} w_{b,2}^{\mathcal{L}}](t) dx \nonumber \\
=& \int_{\mathbb{T}^{2}} [\partial_{1} (\mathcal{L}_{\lambda_{t}} X_{b})_{1} w_{u,1}^{\mathcal{L}} w_{b,1}^{\mathcal{L}} + \partial_{2} (\mathcal{L}_{\lambda_{t}} X_{b})_{1} w_{u,2}^{\mathcal{L}} w_{b,1}^{\mathcal{L}} \nonumber \\
& \hspace{3mm} + \partial_{1} (\mathcal{L}_{\lambda_{t}} X_{b})_{2} w_{u,1}^{\mathcal{L}} w_{b,2}^{\mathcal{L}} + \partial_{2} (\mathcal{L}_{\lambda_{t}} X_{b})_{2} w_{u,2}^{\mathcal{L}} w_{b,2}^{\mathcal{L}}](t) dx \label{est 58}
\end{align}
where four terms were cancelled due to $\nabla\cdot \mathcal{L}_{\lambda_{t}} X_{u} = 0$. A naive attempt shows that analogous cancellations cannot be found for $\RomanI_{1,2}$ of \eqref{est 49b} or $\RomanI_{1,4}$ of \eqref{est 49d}  separately; however, together, we can compute in sum 
\begin{align} 
\RomanI_{1,2} + \RomanI_{1,4} 
=&  - \int_{\mathbb{T}^{2}}  [\partial_{1} (\mathcal{L}_{\lambda_{t}} X_{b})_{1} w_{b,1}^{\mathcal{L}} w_{u,1}^{\mathcal{L}} + \partial_{2} (\mathcal{L}_{\lambda_{t}} X_{b})_{1} w_{b,2}^{\mathcal{L}} w_{u,1}^{\mathcal{L}} \nonumber \\
& \hspace{5mm} + \partial_{1} (\mathcal{L}_{\lambda_{t}} X_{b})_{2} w_{b,1}^{\mathcal{L}} w_{u,2}^{\mathcal{L}} + \partial_{2} (\mathcal{L}_{\lambda_{t}} X_{b})_{2} w_{b,2}^{\mathcal{L}} w_{u,2}^{\mathcal{L}} \nonumber \\
& \hspace{5mm} + \partial_{1} (\mathcal{L}_{\lambda_{t}} X_{u})_{1} \lvert w_{b,1}^{\mathcal{L}} \rvert^{2} + \partial_{2} (\mathcal{L}_{\lambda_{t}} X_{u})_{1} w_{b,2}^{\mathcal{L}} w_{b,1}^{\mathcal{L}} \nonumber\\
& \hspace{5mm}  + \partial_{1} (\mathcal{L}_{\lambda_{t}} X_{u})_{2} w_{b,1}^{\mathcal{L}} w_{b,2}^{\mathcal{L}} + \partial_{2} (\mathcal{L}_{\lambda_{t}} X_{u})_{2} \lvert w_{b,2}^{\mathcal{L}} \rvert^{2}](t) dx + \sum_{i=1}^{8} A_{i}, \label{est 50}
\end{align}
where 
\begin{align*}
A_{1} \triangleq& - \int_{\mathbb{T}^{2}}  \partial_{1} w_{b,1}^{\mathcal{L}} ( \mathcal{L}_{\lambda_{t}} X_{b})_{1} w_{u,1}^{\mathcal{L}}(t)dx, \hspace{3mm} A_{2} \triangleq - \int_{\mathbb{T}^{2}}   \partial_{2} w_{b,1}^{\mathcal{L}} (\mathcal{L}_{\lambda_{t}} X_{b})_{2} w_{u,1}^{\mathcal{L}}(t)dx,  \\
A_{3} \triangleq& -\int_{\mathbb{T}^{2}}  \partial_{1} w_{b,2}^{\mathcal{L}} (\mathcal{L}_{\lambda_{t}} X_{b})_{1} w_{u,2}^{\mathcal{L}}(t)dx,  \hspace{3mm} A_{4} \triangleq - \int_{\mathbb{T}^{2}}   \partial_{2} w_{b,2}^{\mathcal{L}} (\mathcal{L}_{\lambda_{t}} X_{b})_{2} w_{u,2}^{\mathcal{L}}(t)dx, \\
A_{5} \triangleq& -\int_{\mathbb{T}^{2}}   \partial_{1} w_{u,1}^{\mathcal{L}} (\mathcal{L}_{\lambda_{t}} X_{b})_{1} w_{b,1}^{\mathcal{L}}(t)dx, \hspace{3mm} A_{6} \triangleq -\int_{\mathbb{T}^{2}}  \partial_{2} w_{u,1}^{\mathcal{L}} (\mathcal{L}_{\lambda_{t}} X_{b})_{2} w_{b,1}^{\mathcal{L}}(t)dx,  \\
A_{7} \triangleq& -\int_{\mathbb{T}^{2}}  \partial_{1} w_{u,2}^{\mathcal{L}} (\mathcal{L}_{\lambda_{t}} X_{b})_{1} w_{b,2}^{\mathcal{L}}(t)dx, \hspace{3mm} A_{8} \triangleq - \int_{\mathbb{T}^{2}}   \partial_{2} w_{u,2}^{\mathcal{L}} ( \mathcal{L}_{\lambda_{t}} X_{b})_{2} w_{b,2}^{\mathcal{L}}(t)dx. 
\end{align*} 
Let us sum $A_{1}$ with $A_{5}$, $A_{2}$ with $A_{6}$, $A_{3}$ with $A_{7}$, and $A_{4}$ with $A_{8}$ to deduce
\begin{subequations}\label{est 51}
\begin{align}
& A_{1} + A_{5} = - \int_{\mathbb{T}^{2}}  (\mathcal{L}_{\lambda_{t}} X_{b})_{1} \partial_{1} (w_{b,1}^{\mathcal{L}} w_{u,1}^{\mathcal{L}})(t) dx = \int_{\mathbb{T}^{2}}  \partial_{1} (\mathcal{L}_{\lambda_{t}} X_{b})_{1} w_{b,1}^{\mathcal{L}} w_{u,1}^{\mathcal{L}}(t) dx, \label{est 51a}\\ 
& A_{2} + A_{6} = -\int_{\mathbb{T}^{2}}  (\mathcal{L}_{\lambda_{t}} X_{b})_{2} \partial_{2} (w_{b,1}^{\mathcal{L}} w_{u,1}^{\mathcal{L}}) (t) dx= \int_{\mathbb{T}^{2}}  \partial_{2} (\mathcal{L}_{\lambda_{t}} X_{b})_{2} w_{b,1}^{\mathcal{L}} w_{u,1}^{\mathcal{L}}(t) dx,   \label{est 51b}\\
&A_{3} + A_{7} = - \int_{\mathbb{T}^{2}}  (\mathcal{L}_{\lambda_{t}} X_{b})_{1} \partial_{1} (w_{b,2}^{\mathcal{L}} w_{u,2}^{\mathcal{L}}) (t) dx= \int_{\mathbb{T}^{2}}  \partial_{1} (\mathcal{L}_{\lambda_{t}} X_{b})_{1} w_{b,2}^{\mathcal{L}} w_{u,2}^{\mathcal{L}}(t) dx, \label{est 51c}\\ 
&A_{4} + A_{8} = - \int_{\mathbb{T}^{2}}  (\mathcal{L}_{\lambda_{t}} X_{b})_{2} \partial_{2} (w_{b,2}^{\mathcal{L}} w_{u,2}^{\mathcal{L}})(t) dx = \int_{\mathbb{T}^{2}}  \partial_{2} (\mathcal{L}_{\lambda_{t}} X_{b})_{2} w_{b,2}^{\mathcal{L}} w_{u,2}^{\mathcal{L}}(t) dx.  \label{est 51d}
\end{align}
\end{subequations} 
Then the sum of \eqref{est 51a} and \eqref{est 51b} gives us 
\begin{equation}\label{est 53}
A_{1} + A_{5} + A_{2} + A_{6} = \int_{\mathbb{T}^{2}} [\partial_{1} (\mathcal{L}_{\lambda_{t}} X_{b})_{1} + \partial_{2} (\mathcal{L}_{\lambda_{t}} X_{b})_{2} ] w_{b,1}^{\mathcal{L}} w_{u,1}^{\mathcal{L}}  (t) dx= 0
\end{equation}
due to $\nabla \cdot \mathcal{L}_{\lambda_{t}} X_{b} = 0$ while the sum of \eqref{est 51c} and \eqref{est 51d} gives us 
\begin{equation}\label{est 54}
A_{3} + A_{7} +A _{4} +A_{8} = \int_{\mathbb{T}^{2}} [\partial_{1} (\mathcal{L}_{\lambda_{t}} X_{b})_{1} + \partial_{2} (\mathcal{L}_{\lambda_{t}} X_{b})_{2} ] w_{b,2}^{\mathcal{L}} w_{u,2}^{\mathcal{L}} (t) dx = 0 
\end{equation}
due to $\nabla\cdot \mathcal{L}_{\lambda_{t}} X_{b} = 0$ again. Applying \eqref{est 53}-\eqref{est 54} to \eqref{est 50} gives us 
\begin{align*}
&\RomanI_{1,2} + \RomanI_{1,4} \\ 
=&  - \int_{\mathbb{T}^{2}} [\partial_{1} (\mathcal{L}_{\lambda_{t}} X_{b})_{1} w_{b,1}^{\mathcal{L}} w_{u,1}^{\mathcal{L}} + \partial_{2} (\mathcal{L}_{\lambda_{t}} X_{b})_{1} w_{b,2}^{\mathcal{L}} w_{u,1}^{\mathcal{L}} + \partial_{1} (\mathcal{L}_{\lambda_{t}} X_{b})_{2} w_{b,1}^{\mathcal{L}} w_{u,2}^{\mathcal{L}} + \partial_{2} (\mathcal{L}_{\lambda_{t}} X_{b})_{2} w_{b,2}^{\mathcal{L}} w_{u,2}^{\mathcal{L}} \nonumber \\
& \hspace{3mm} + \partial_{1} (\mathcal{L}_{\lambda_{t}} X_{u})_{1} \lvert w_{b,1}^{\mathcal{L}} \rvert^{2} + \partial_{2} (\mathcal{L}_{\lambda_{t}} X_{u})_{1} w_{b,2}^{\mathcal{L}} w_{b,1}^{\mathcal{L}}  + \partial_{1} (\mathcal{L}_{\lambda_{t}} X_{u})_{2} w_{b,1}^{\mathcal{L}} w_{b,2}^{\mathcal{L}} + \partial_{2} (\mathcal{L}_{\lambda_{t}} X_{u})_{2} \lvert w_{b,2}^{\mathcal{L}} \rvert^{2} ](t) dx\nonumber 
\end{align*}
and applying this to \eqref{est 55}, together with \eqref{est 57} and \eqref{est 58} gives us 
\begin{align}
& \langle w_{u}^{\mathcal{L}}, \divergence \left(2( \mathcal{L}_{\lambda_{t}} X_{u}) \otimes_{s} w_{u}^{\mathcal{L}} - 2( \mathcal{L}_{\lambda_{t}} X_{b}) \otimes_{s} w_{b}^{\mathcal{L}} \right) \rangle \nonumber \\
&+ \langle w_{b}^{\mathcal{L}}, \divergence \left(2 w_{b}^{\mathcal{L}} \otimes_{a} (\mathcal{L}_{\lambda_{t}}X_{u})-2w_{u}^{\mathcal{L}} \otimes_{a} (\mathcal{L}_{\lambda_{t}}X_{b})  \right) \rangle   \label{est 59} \\
=&  \int_{\mathbb{T}^{2}} \partial_{1} (\mathcal{L}_{\lambda_{t}} X_{u})_{1} \lvert w_{u,1}^{\mathcal{L}} \rvert^{2} + \partial_{2} (\mathcal{L}_{\lambda_{t}} X_{u})_{1} w_{u,2}^{\mathcal{L}} w_{u,1}^{\mathcal{L}}   \nonumber \\
& \hspace{4mm}  + \partial_{1} (\mathcal{L}_{\lambda_{t}} X_{u})_{2} w_{u,1}^{\mathcal{L}} w_{u,2}^{\mathcal{L}} + \partial_{2} (\mathcal{L}_{\lambda_{t}} X_{u})_{2} \lvert w_{u,2}^{\mathcal{L}} \rvert^{2}   \nonumber \\
&\hspace{3mm} +  \partial_{1} (\mathcal{L}_{\lambda_{t}} X_{b})_{1} w_{u,1}^{\mathcal{L}} w_{b,1}^{\mathcal{L}} + \partial_{2} (\mathcal{L}_{\lambda_{t}} X_{b})_{1} w_{u,2}^{\mathcal{L}} w_{b,1}^{\mathcal{L}} \nonumber \\
& \hspace{4mm} + \partial_{1} (\mathcal{L}_{\lambda_{t}} X_{b})_{2} w_{u,1}^{\mathcal{L}} w_{b,2}^{\mathcal{L}} + \partial_{2} (\mathcal{L}_{\lambda_{t}} X_{b})_{2} w_{u,2}^{\mathcal{L}} w_{b,2}^{\mathcal{L}} dx \nonumber \\
-& \int_{\mathbb{T}^{2}}   \partial_{1} (\mathcal{L}_{\lambda_{t}} X_{b})_{1} w_{b,1}^{\mathcal{L}} w_{u,1}^{\mathcal{L}} + \partial_{2} (\mathcal{L}_{\lambda_{t}} X_{b})_{1} w_{b,2}^{\mathcal{L}} w_{u,1}^{\mathcal{L}} \nonumber \\
& \hspace{4mm} + \partial_{1} (\mathcal{L}_{\lambda_{t}} X_{b})_{2} w_{b,1}^{\mathcal{L}} w_{u,2}^{\mathcal{L}} + \partial_{2} (\mathcal{L}_{\lambda_{t}} X_{b})_{2} w_{b,2}^{\mathcal{L}} w_{u,2}^{\mathcal{L}} \nonumber \\
& \hspace{3mm} + \partial_{1} (\mathcal{L}_{\lambda_{t}} X_{u})_{1} \lvert w_{b,1}^{\mathcal{L}} \rvert^{2} + \partial_{2} (\mathcal{L}_{\lambda_{t}} X_{u})_{1} w_{b,2}^{\mathcal{L}} w_{b,1}^{\mathcal{L}}  \nonumber\\
& \hspace{4mm} + \partial_{1} (\mathcal{L}_{\lambda_{t}} X_{u})_{2} w_{b,1}^{\mathcal{L}} w_{b,2}^{\mathcal{L}} + \partial_{2} (\mathcal{L}_{\lambda_{t}} X_{u})_{2} \lvert w_{b,2}^{\mathcal{L}} \rvert^{2} dx. \nonumber 
\end{align}
Using \eqref{est 20}, the first four terms in the first integral of \eqref{est 59} can be written as 
\begin{align}
&  \int_{\mathbb{T}^{2}} \partial_{1} (\mathcal{L}_{\lambda_{t}} X_{u})_{1} \lvert w_{u,1}^{\mathcal{L}} \rvert^{2} + \partial_{2} (\mathcal{L}_{\lambda_{t}} X_{u})_{1} w_{u,2}^{\mathcal{L}} w_{u,1}^{\mathcal{L}}   \nonumber \\
& \hspace{10mm}  + \partial_{1} (\mathcal{L}_{\lambda_{t}} X_{u})_{2} w_{u,1}^{\mathcal{L}} w_{u,2}^{\mathcal{L}} + \partial_{2} (\mathcal{L}_{\lambda_{t}} X_{u})_{2} \lvert w_{u,2}^{\mathcal{L}} \rvert^{2} dx  = \langle [\nabla_{\text{symm}} (\mathcal{L}_{\lambda_{t}} X_{u} )] w_{u}^{\mathcal{L}}, w_{u}^{\mathcal{L}} \rangle \label{est 60} 
\end{align}
while the last four terms in the second integral can be written as 
\begin{align}
&- \int_{\mathbb{T}^{2}} \partial_{1} (\mathcal{L}_{\lambda_{t}} X_{u})_{1} \lvert w_{b,1}^{\mathcal{L}} \rvert^{2} + \partial_{2} (\mathcal{L}_{\lambda_{t}} X_{u})_{1} w_{b,2}^{\mathcal{L}} w_{b,1}^{\mathcal{L}}  \nonumber \\
& \hspace{10mm}  + \partial_{1} (\mathcal{L}_{\lambda_{t}} X_{u})_{2} w_{b,1}^{\mathcal{L}} w_{b,2}^{\mathcal{L}} + \partial_{2} (\mathcal{L}_{\lambda_{t}} X_{u})_{2} \lvert w_{b,2}^{\mathcal{L}} \rvert^{2}  dx= - \langle [ \nabla_{\text{symm}} (\mathcal{L}_{\lambda_{t}} X_{u}) ] w_{b}^{\mathcal{L}}, w_{b}^{\mathcal{L}} \rangle. \label{est 61} 
\end{align}
Concerning the last four terms in the first integral and the first four terms in the second integral of \eqref{est 59}, it's difficult to write them in such a compact form; however, it turns out that when combined together, four terms cancel out and they can be written via \eqref{est 21} as 
\begin{align}
& \int_{\mathbb{T}^{2}}   \partial_{1} (\mathcal{L}_{\lambda_{t}} X_{b})_{1} w_{u,1}^{\mathcal{L}} w_{b,1}^{\mathcal{L}} + \partial_{2} (\mathcal{L}_{\lambda_{t}} X_{b})_{1} w_{u,2}^{\mathcal{L}} w_{b,1}^{\mathcal{L}} + \partial_{1} (\mathcal{L}_{\lambda_{t}} X_{b})_{2} w_{u,1}^{\mathcal{L}} w_{b,2}^{\mathcal{L}} + \partial_{2} (\mathcal{L}_{\lambda_{t}} X_{b})_{2} w_{u,2}^{\mathcal{L}} w_{b,2}^{\mathcal{L}}  \nonumber \\
& - \partial_{1} (\mathcal{L}_{\lambda_{t}} X_{b})_{1} w_{b,1}^{\mathcal{L}} w_{u,1}^{\mathcal{L}} - \partial_{2} (\mathcal{L}_{\lambda_{t}} X_{b})_{1} w_{b,2}^{\mathcal{L}} w_{u,1}^{\mathcal{L}} - \partial_{1} (\mathcal{L}_{\lambda_{t}} X_{b})_{2} w_{b,1}^{\mathcal{L}} w_{u,2}^{\mathcal{L}} - \partial_{2} (\mathcal{L}_{\lambda_{t}} X_{b})_{2} w_{b,2}^{\mathcal{L}} w_{u,2}^{\mathcal{L}} dx  \nonumber \\
=&  \int_{\mathbb{T}^{2}}   \partial_{2} (\mathcal{L}_{\lambda_{t}} X_{b})_{1} w_{u,2}^{\mathcal{L}} w_{b,1}^{\mathcal{L}} + \partial_{1} (\mathcal{L}_{\lambda_{t}} X_{b})_{2} w_{u,1}^{\mathcal{L}} w_{b,2}^{\mathcal{L}}   - \partial_{2} (\mathcal{L}_{\lambda_{t}} X_{b})_{1} w_{b,2}^{\mathcal{L}} w_{u,1}^{\mathcal{L}} - \partial_{1} (\mathcal{L}_{\lambda_{t}} X_{b})_{2} w_{b,1}^{\mathcal{L}} w_{u,2}^{\mathcal{L}}dx  \nonumber \\
& \hspace{20mm} = - \langle [\nabla_{\text{anti}} (\mathcal{L}_{\lambda_{t}} X_{b}) ] w_{u}^{\mathcal{L}}, w_{b}^{\mathcal{L}} \rangle + \langle [ \nabla_{\text{anti}} (\mathcal{L}_{\lambda_{t}} X_{b})] w_{b}^{\mathcal{L}}, w_{u}^{\mathcal{L}} \rangle. \label{est 62}
\end{align}
Therefore, \eqref{est 60}, \eqref{est 61}, and \eqref{est 62} applied to \eqref{est 59} gives us 
\begin{align}
& \langle w_{u}^{\mathcal{L}}, \divergence \left(2( \mathcal{L}_{\lambda_{t}} X_{u}) \otimes_{s} w_{u}^{\mathcal{L}} - 2( \mathcal{L}_{\lambda_{t}} X_{b}) \otimes_{s} w_{b}^{\mathcal{L}} \right) \rangle \nonumber \\
&+ \langle w_{b}^{\mathcal{L}}, \divergence \left(2 w_{b}^{\mathcal{L}} \otimes_{a} (\mathcal{L}_{\lambda_{t}}X_{u})-2w_{u}^{\mathcal{L}} \otimes_{a} (\mathcal{L}_{\lambda_{t}}X_{b})  \right) \rangle   \nonumber \\
=& \langle [\nabla_{\text{symm}} (\mathcal{L}_{\lambda_{t}} X_{u} )] w_{u}^{\mathcal{L}}, w_{u}^{\mathcal{L}} \rangle + \langle [ \nabla_{\text{anti}} (\mathcal{L}_{\lambda_{t}} X_{b})] w_{b}^{\mathcal{L}}, w_{u}^{\mathcal{L}} \rangle \nonumber  \\
&- \langle [\nabla_{\text{anti}} (\mathcal{L}_{\lambda_{t}} X_{b}) ] w_{u}^{\mathcal{L}}, w_{b}^{\mathcal{L}} \rangle - \langle [ \nabla_{\text{symm}} (\mathcal{L}_{\lambda_{t}} X_{u}) ] w_{b}^{\mathcal{L}}, w_{b}^{\mathcal{L}} \rangle.  \label{est 63}
\end{align}
We apply \eqref{est 63} to \eqref{est 44} to conclude with the definition of $\nabla_{\text{spec}}$ from \eqref{est 64}, 
\begin{equation}\label{est 236}
\RomanI_{1} = - \nu \lVert (w_{u}^{\mathcal{L}}, w_{b}^{\mathcal{L}})(t)  \rVert_{\dot{H}^{1}}^{2} + 2 \int_{\mathbb{T}^{2}} 
\begin{pmatrix}
w_{u}^{\mathcal{L}} \\
w_{b}^{\mathcal{L}} 
\end{pmatrix}
\cdot 
\left[\frac{\nu }{2} \Delta \Id - \nabla_{\text{spec}} \mathcal{L}_{\lambda_{t}} (X_{u},  X_{b}) \right] 
\begin{pmatrix}
w_{u}^{\mathcal{L}} \\
w_{b}^{\mathcal{L}}
\end{pmatrix} (t)  dx
\end{equation}
where $\Id$ is an $\mathbb{M}^{4}$-valued identity matrix. Therefore, following \cite{HR23} we are able to define a time-dependent family of operators 
\begin{equation}\label{est 195} 
\mathcal{A}_{t} \triangleq \frac{\nu }{2} \Delta \Id - \nabla_{\text{spec}} (X_{u}, X_{b})(t) - \infty \hspace{3mm} \forall \hspace{1mm} t \geq 0 
\end{equation} 
as the limit $\lambda \nearrow + \infty$ of 
\begin{equation}\label{est 97}  
\mathcal{A}_{t}^{\lambda} \triangleq \frac{\nu }{2} \Delta \Id - \nabla_{\text{spec}} \mathcal{L}_{\lambda} (X_{u}, X_{b}) - r_{\lambda}(t) \Id \hspace{3mm} \forall \hspace{1mm} t \geq 0 
\end{equation} 
for $r_{\lambda}$ from \eqref{est 140}, allowing us to rewrite \eqref{est 236} as  
\begin{equation}\label{est 65}
\RomanI_{1} = - \nu \lVert (w_{u}^{\mathcal{L}}, w_{b}^{\mathcal{L}})(t)  \rVert_{\dot{H}^{1}}^{2}  + 2 \left\langle 
\begin{pmatrix} 
w_{u}^{\mathcal{L}} \\
w_{b}^{\mathcal{L}}
\end{pmatrix}, 
\mathcal{A}_{t}^{\lambda_{t}} 
\begin{pmatrix} w_{u}^{\mathcal{L}} \\
w_{b}^{\mathcal{L}} 
\end{pmatrix} 
\right\rangle(t)  + r_{\lambda_{t}}(t) \lVert (w_{u}^{\mathcal{L}}, w_{b}^{\mathcal{L}})(t)  \rVert_{L^{2}}^{2}.
\end{equation} 

\begin{proposition}\label{Proposition 4.4} 
\rm{(Cf. \cite[Lemma 4.9]{HR23})}  Let $t \in [T_{i}, T_{i+1})$ and fix Fix $\lambda_{t}$ from \eqref{est 37} with $\mathfrak{a} \in [2,\infty)$.  Then for any $\kappa_{0} \in (0,1)$, all $\eta \in [\frac{1+ \kappa}{2}, 1)$, and all $\kappa \in (0, \kappa_{0}]$, $\RomanI_{2}$ from \eqref{est 43b} satisfies 
\begin{equation}\label{est 95}
\lvert \RomanI_{2} \rvert \lesssim  \lVert (w_{u}^{\mathcal{L}}, w_{b}^{\mathcal{L}}) (t) \rVert_{H^{\eta}}^{2} N_{t}^{\kappa}. 
\end{equation} 
\end{proposition} 

\begin{proof}[Proof of Proposition \ref{Proposition 4.4}]
The proof crucially relies on the fact that within \eqref{est 43b} we can rewrite 
\begin{subequations}\label{est 66} 
\begin{align}
&  ( \mathcal{H}_{\lambda_{t}} X_{u}) \otimes_{s} w_{u}^{\mathcal{L}} -  (\mathcal{H}_{\lambda_{t}} X_{u}) \circlesign{\succ}_{s} w_{u}^{\mathcal{L}} = (\mathcal{H}_{\lambda_{t}} X_{u}) \circlesign{\prec}_{s} w_{u}^{\mathcal{L}} + (\mathcal{H}_{\lambda_{t}}X_{u}) \circlesign{\circ}_{s} w_{u}^{\mathcal{L}}, \\
& -  ( \mathcal{H}_{\lambda_{t}} X_{b}) \otimes_{s} w_{b}^{\mathcal{L}} + (\mathcal{H}_{\lambda_{t}} X_{b}) \circlesign{\succ}_{s} w_{b}^{\mathcal{L}} = -(\mathcal{H}_{\lambda_{t}} X_{b}) \circlesign{\prec}_{s} w_{b}^{\mathcal{L}} - (\mathcal{H}_{\lambda_{t}}X_{b}) \circlesign{\circ}_{s} w_{b}^{\mathcal{L}},  \\
&w_{b}^{\mathcal{L}} \otimes_{a} (\mathcal{H}_{\lambda_{t}} X_{u}) -  w_{b}^{\mathcal{L}} \circlesign{\prec}_{a} (\mathcal{H}_{\lambda_{t}} X_{u}) =  w_{b}^{\mathcal{L}} \circlesign{\succ}_{a} (\mathcal{H}_{\lambda_{t}} X_{u}) + w_{b}^{\mathcal{L}} \circlesign{\circ}_{a} (\mathcal{H}_{\lambda_{t}} X_{u}),  \\
& - w_{u}^{\mathcal{L}} \otimes_{a} (\mathcal{H}_{\lambda_{t}} X_{b}) +  w_{u}^{\mathcal{L}} \circlesign{\prec}_{a} (\mathcal{H}_{\lambda_{t}} X_{b}) =  - w_{u}^{\mathcal{L}} \circlesign{\succ}_{a} (\mathcal{H}_{\lambda_{t}} X_{b}) - w_{u}^{\mathcal{L}} \circlesign{\circ}_{a} (\mathcal{H}_{\lambda_{t}} X_{b}).
\end{align}
\end{subequations}
The rewriting of \eqref{est 66} allows us to estimate from \eqref{est 43b} 
\begin{align}
\lvert \RomanI_{2} \rvert  
\lesssim& \lVert w_{u}^{\mathcal{L}}(t)  \rVert_{H^{\eta}} [ \lVert (\mathcal{H}_{\lambda_{t}} X_{u} ) \circlesign{\prec}_{s} w_{u}^{\mathcal{L}} \rVert_{H^{1-\eta}} + \lVert (\mathcal{H}_{\lambda_{t}} X_{u}) \circlesign{\circ}_{s} w_{u}^{\mathcal{L}} \rVert_{H^{1-\eta}} \nonumber \\
& \hspace{10mm} + \lVert (\mathcal{H}_{\lambda_{t}} X_{b} ) \circlesign{\prec}_{s} w_{b}^{\mathcal{L}} \rVert_{H^{1-\eta}} + \lVert (\mathcal{H}_{\lambda_{t}} X_{b}) \circlesign{\circ}_{s} w_{b}^{\mathcal{L}} \rVert_{H^{1-\eta}}  ](t)  \nonumber \\
&+ \lVert w_{b}^{\mathcal{L}}(t)  \rVert_{H^{\eta}} [ \lVert w_{b}^{\mathcal{L}} \circlesign{\succ}_{a} (\mathcal{H}_{\lambda_{t}} X_{u}) \rVert_{H^{1-\eta}} + \lVert w_{b}^{\mathcal{L}} \circlesign{\circ}_{a} (\mathcal{H}_{\lambda_{t}} X_{u}) \rVert_{H^{1-\eta}}  \nonumber \\
& \hspace{10mm} +  \lVert w_{u}^{\mathcal{L}} \circlesign{\succ}_{a} (\mathcal{H}_{\lambda_{t}} X_{b}) \rVert_{H^{1-\eta}} + \lVert w_{u}^{\mathcal{L}} \circlesign{\circ}_{a} (\mathcal{H}_{\lambda_{t}} X_{b}) \rVert_{H^{1-\eta}}](t) .  \label{est 239}
\end{align}
We estimate e.g. 
\begin{subequations}\label{est 235}
\begin{align}
& \lVert w_{b}^{\mathcal{L}} \circlesign{\succ}_{a} (\mathcal{H}_{\lambda_{t}} X_{u} ) (t) \rVert_{H^{1-\eta}} \overset{\eqref{est 40d}}{\lesssim} \lVert w_{b}^{\mathcal{L}}  (t) \rVert_{H^{1-\eta + \kappa}} \lVert \mathcal{H}_{\lambda_{t}} X_{u} (t) \rVert_{\mathcal{C}^{-\kappa}} \overset{ \eqref{est 39}}{\lesssim} \lVert w_{b}^{\mathcal{L}} (t) \rVert_{H^{\eta}} N_{t}^{\kappa}, \\
& \lVert w_{b}^{\mathcal{L}} \circlesign{\circ}_{a} ( \mathcal{H}_{\lambda_{t}} X_{u}) (t) \rVert_{H^{1-\eta}} \overset{\eqref{est 40e}}{\lesssim} \lVert w_{b}^{\mathcal{L}} (t) \rVert_{H^{1-\eta+ \kappa}} \lVert  \mathcal{H}_{\lambda_{t}} X_{u} (t) \rVert_{\mathcal{C}^{-\kappa}} \overset{ \eqref{est 39}}{\lesssim} \lVert w_{b}^{\mathcal{L}} (t) \rVert_{H^{\eta}} N_{t}^{\kappa}, 
\end{align}
\end{subequations} 
where the both inequalities used the hypothesis that $\eta \geq \frac{1+\kappa_{0}}{2}$ and the second inequality additionally used the hypothesis $\eta < 1$.  Other terms in \eqref{est 239} can be estimated similarly, leading to \eqref{est 95}, concluding the proof of Proposition \ref{Proposition 4.4}.  
\end{proof} 

\begin{remark}\label{Remark 4.3}
Initially defining $w_{b}^{\sharp}$ by 
\begin{align*}
w_{b} = - \mathbb{P}_{L} \divergence (w_{b} \circlesign{\prec}_{s} Q_{u} - w_{u} \circlesign{\prec}_{s} Q_{b}) + w_{b}^{\sharp}
\end{align*}
instead of $ -\mathbb{P}_{L} \divergence (w_{b} \circlesign{\prec}_{a} Q_{u} - w_{u} \circlesign{\prec}_{a} Q_{b}) + w_{b}^{\sharp}$ from \eqref{est 31b}, most of the computations up to \eqref{est 65} actually went through analogously. However, we had trouble with the proof of Proposition \ref{Proposition 4.4} because this wrong choice of $w_{b}^{\sharp}$ leads to 
\begin{align} 
-2 \langle w_{b}^{\mathcal{L}}, \divergence &( w_{b}^{\mathcal{L}} \otimes_{s} (\mathcal{H}_{\lambda_{t}} X_{u})  + (\mathcal{H}_{\lambda_{t}} X_{b}) \otimes_{s} w_{u}^{\mathcal{L}} - 2 (\mathcal{H}_{\lambda_{t}} X_{u}) \circlesign{\succ}_{s} w_{b}^{\mathcal{L}}  \nonumber \\
& - w_{u}^{\mathcal{L}}\otimes_{s} (\mathcal{H}_{\lambda_{t}} X_{b}) - (\mathcal{H}_{\lambda_{t}} X_{u}) \otimes_{s} w_{b}^{\mathcal{L}} + 2 (\mathcal{H}_{\lambda_{t}} X_{b}) \circlesign{\succ}_{s} w_{u}^{\mathcal{L}})  \rangle \label{est 1}  
\end{align}
instead of 
\begin{align*}
- 2 \langle w_{b}^{\mathcal{L}}, \divergence (2w_{b}^{\mathcal{L}} \otimes_{a} (\mathcal{H}_{\lambda_{t}} X_{u}) - 2 w_{b}^{\mathcal{L}} \circlesign{\prec}_{a} (\mathcal{H}_{\lambda_{t}} X_{u})  -2 w_{u}^{\mathcal{L}} \otimes_{a} (\mathcal{H}_{\lambda_{t}} X_{b}) + 2 w_{u}^{\mathcal{L}} \circlesign{\prec}_{a} (\mathcal{H}_{\lambda_{t}} X_{b}) \rangle
\end{align*} 
from the second $L^{2}(\mathbb{T}^{2})$-inner product in $\RomanI_{2}$ of \eqref{est 43b}. We can still write 
\begin{align}
&w_{b}^{\mathcal{L}} \otimes_{s} (\mathcal{H}_{\lambda_{t}} X_{u} ) - (\mathcal{H}_{\lambda_{t}} X_{u}) \circlesign{\succ}_{s} w_{b}^{\mathcal{L}} = w_{b}^{\mathcal{L}} \circlesign{\succ}_{s} (\mathcal{H}_{\lambda_{t}} X_{u}) + w_{b}^{\mathcal{L}} \circlesign{\circ} (\mathcal{H}_{\lambda_{t}} X_{u}), \nonumber\\
& - w_{u}^{\mathcal{L}}\otimes_{s} (\mathcal{H}_{\lambda_{t}} X_{b}) +  \mathcal{H}_{\lambda_{t}} X_{b} \circlesign{\succ}_{s} w_{u}^{\mathcal{L}}  = - w_{u}^{\mathcal{L}}  \circlesign{\succ}_{s} (\mathcal{H}_{\lambda_{t}} X_{b}) - w_{u}^{\mathcal{L}} \circlesign{\circ}_{s} (\mathcal{H}_{\lambda_{t}} X_{b}) \label{est 2} 
\end{align}
which can be bounded similarly to \eqref{est 235}.  However, we have a problem from the remaining terms in \eqref{est 1}, namely 
\begin{equation*}
(\mathcal{H}_{\lambda_{t}} X_{b}) \otimes_{s} w_{u}^{\mathcal{L}} -  (\mathcal{H}_{\lambda_{t}} X_{u}) \circlesign{\succ}_{s} w_{b}^{\mathcal{L}}  - (\mathcal{H}_{\lambda_{t}} X_{u}) \otimes_{s} w_{b}^{\mathcal{L}} +  (\mathcal{H}_{\lambda_{t}} X_{b}) \circlesign{\succ}_{s} w_{u}^{\mathcal{L}} 
\end{equation*} 
do not allow us to repeat the cancellations in \eqref{est 2} due to opposite signs. Introduction of $\otimes_{a}$ and $\circlesign{\prec}_{a}$, and $\circlesign{\circ}_{a}$ and rewriting them in such formulations made it easy to discover the appropriate paracontrolled ansatz in \eqref{est 31b} so that the necessary cancellations occurred smoothly in \eqref{est 66}.  
\end{remark}

\begin{proposition}\label{Proposition 4.5} 
\rm{(Cf. \cite[Lemma 4.10]{HR23})} Let $t \in [T_{i}, T_{i+1})$. 
\begin{enumerate}
\item Fix $\lambda_{t}$ from \eqref{est 37} with $\mathfrak{a} \in [2,\infty)$. Then for any $\kappa_{0} \in (0,\frac{1}{12}]$, all $\eta \in (\frac{2}{3} + \kappa, 1)$, and all $\kappa \in (0, \kappa_{0}]$, $\RomanI_{3}$ from \eqref{est 43c} satisfies 
\begin{equation}\label{est 70}
\lvert \RomanI_{3} \rvert \lesssim \lambda_{t}^{\frac{1}{3}} ( \lVert w_{u}^{\mathcal{L}} \rVert_{H^{\eta}} + \lVert w_{b}^{\mathcal{L}} \rVert_{H^{\eta}})(t)  (N_{t}^{\kappa})^{2}. 
\end{equation} 
\item Fix $\lambda_{t}$ from \eqref{est 37} with $\mathfrak{a} \in [\frac{11}{4}, \infty)$. Then for any $\kappa_{0} \in (0, \frac{1}{44}]$, all $\eta \in [\frac{3}{4} + 3\kappa, 1)$, and all $\kappa \in (0, \kappa_{0}]$, $-2\langle w_{u}^{\mathcal{L}}, \divergence (w_{u}^{\otimes 2} - w_{b}^{\otimes 2}) \rangle - 2\langle w_{b}^{\mathcal{L}}, \divergence (w_{b} \otimes w_{u} - w_{u} \otimes w_{b} ) \rangle$ within $\RomanI_{4}$ from \eqref{est 43d} satisfies 
\begin{align}
&-  2\langle w_{u}^{\mathcal{L}}, \divergence (w_{u}^{\otimes 2} - w_{b}^{\otimes 2}) \rangle(t)  -  2\langle w_{b}^{\mathcal{L}}, \divergence (w_{b} \otimes w_{u} - w_{u} \otimes w_{b} ) \rangle(t)  \nonumber \\
\lesssim& (\lVert w_{u}^{\mathcal{L}} \rVert_{H^{\eta}} + \lVert w_{b}^{\mathcal{L}} \rVert_{H^{\eta}})(t)( \lVert w_{u}^{\mathcal{L}}(t)  \rVert_{H^{\eta}} + \lVert w_{b}^{\mathcal{L}}(t)  \rVert_{H^{\eta}} + N_{t}^{\kappa}) N_{t}^{\kappa}.  \label{est 78}
\end{align}
\end{enumerate} 
\end{proposition} 

\begin{proof}[Proof of Proposition \ref{Proposition 4.5}]
Using \eqref{est 68}, we can bound 
\begin{equation}\label{est 73}
\lvert \RomanI_{3} \rvert  \leq \RomanI_{3,1} + \RomanI_{3,2}, 
\end{equation} 
where 
\begin{subequations}\label{est 69}
\begin{align}
\RomanI_{3,1} \triangleq& 4 [ \lvert \langle w_{u}^{\mathcal{L}}, \divergence (\mathcal{L}_{\lambda_{t}} X_{u} \otimes_{s} w_{u}^{\mathcal{H}} - \mathcal{L}_{\lambda_{t}} X_{b} \otimes_{s} w_{b}^{\mathcal{H}}) \rangle \rvert \nonumber \\
&+ \lvert \langle w_{b}^{\mathcal{L}}, \divergence (w_{b}^{\mathcal{H}} \otimes_{a} \mathcal{L}_{\lambda_{t}} X_{u} - w_{u}^{\mathcal{H}} \otimes_{a} \mathcal{L}_{\lambda_{t}} X_{b} ) \rangle \rvert ](t), \label{est 69a}\\ 
\RomanI_{3,2} \triangleq& 4 [ \lvert \langle w_{u}^{\mathcal{L}}, \divergence (\mathcal{H}_{\lambda_{t}} X_{b} \circlesign{\prec}_{s} w_{u}^{\mathcal{H}} + \mathcal{H}_{\lambda_{t}} X_{u} \circlesign{\circ}_{s} w_{u}^{\mathcal{H}} \nonumber \\
& - \mathcal{H}_{\lambda_{t}} X_{b} \circlesign{\prec}_{s} w_{b}^{\mathcal{H}} - \mathcal{H}_{\lambda_{t}} X_{b} \circlesign{\circ}_{s} w_{b}^{\mathcal{H}} ) \rangle  \rvert \nonumber \\
&+ \lvert \langle w_{b}^{\mathcal{L}}, \divergence ( w_{b}^{\mathcal{H}} \circlesign{\succ}_{a} \mathcal{H}_{\lambda_{t}} X_{u} + w_{b}^{\mathcal{H}} \circlesign{\circ}_{a} \mathcal{H}_{\lambda_{t}} X_{u} \nonumber \\
& \hspace{15mm} - w_{u}^{\mathcal{H}} \circlesign{\succ}_{a} \mathcal{H}_{\lambda_{t}} X_{b} - w_{u}^{\mathcal{H}} \circlesign{\circ}_{a} \mathcal{H}_{\lambda_{t}} X_{b} ) \rangle \rvert ](t). \label{est 69b}
\end{align}
\end{subequations} 
To estimate $\RomanI_{3,1}$, e.g. we work on 
\begin{align}
\lvert \langle w_{u}^{\mathcal{L}}, \divergence (\mathcal{L}_{\lambda_{t}} X_{b} \otimes_{s} w_{b}^{\mathcal{H}} ) \rangle \rvert 
\lesssim \lVert w_{u}^{\mathcal{L}} \rVert_{H^{\eta}} \lVert \mathcal{L}_{\lambda_{t}} X_{b} \otimes_{s} w_{b}^{\mathcal{H}} \rVert_{H^{1-\eta}},  \label{est 241}
\end{align} 
where we estimate 
\begin{align}
& \lVert \mathcal{L}_{\lambda_{t}} X_{b} \circlesign{\succ}_{s} w_{b}^{\mathcal{H}} (t) \rVert_{H^{1-\eta}} 
\overset{\eqref{est 40c}}{\lesssim} \lVert \mathcal{L}_{\lambda_{t}} X_{b} (t) \rVert_{\mathcal{C}^{\frac{1}{3} - \kappa}} \lVert w_{b}^{\mathcal{H}} (t) \rVert_{H^{\frac{2}{3} - \eta + \kappa}} \nonumber\\
\overset{\eqref{est 33} \eqref{est 67}}{\lesssim}& \lambda_{t}^{\frac{1}{3}} \lVert X_{b} (t) \rVert_{\mathcal{C}^{-\kappa}} (1+ \lVert w_{u} \rVert_{L^{2}} + \lVert w_{b} \rVert_{L^{2}})^{1- \mathfrak{a}(\frac{1}{2})}(t) N_{t}^{\kappa} \overset{\eqref{est 39} }{\lesssim} \lambda_{t}^{\frac{1}{3}} (N_{t}^{\kappa})^{2},   \label{est 240} 
\end{align}
where the first inequality used that $\frac{2}{3} + \kappa < \eta$, the second inequality that $\kappa \leq \frac{1}{4}$, and the third that $\mathfrak{a} \geq 2$.   Similarly, 
\begin{align}
\lVert \mathcal{L}_{\lambda_{t}} X_{b} \circlesign{\prec}_{s} w_{b}^{\mathcal{H}} (t)  \rVert_{H^{1-\eta}} 
\overset{\eqref{est 40d}}{\lesssim}& \lVert \mathcal{L}_{\lambda_{t}} X_{b} (t) \rVert_{\mathcal{C}^{\frac{1}{2} - \eta + 2 \kappa}} \lVert w_{b}^{\mathcal{H}}(t)  \rVert_{H^{\frac{1}{2} - 2 \kappa}} \overset{\eqref{est 33}\eqref{est 67}\eqref{est 39} }{\lesssim}  \lambda_{t}^{\frac{1}{3}}  (N_{t}^{\kappa})^{2},  \label{est 109} 
\end{align}
that is justified by $\frac{2}{3} + \kappa < \eta$ and $\kappa_{0} \leq \frac{1}{12}$, and 
\begin{equation}\label{est 110}
 \lVert \mathcal{L}_{\lambda_{t}} X_{b} \circlesign{\circ}_{s} w_{b}^{\mathcal{H}} (t) \rVert_{H^{1-\eta}} 
\overset{\eqref{est 40e}}{\lesssim} \lVert \mathcal{L}_{\lambda_{t}} X_{b} (t) \rVert_{\mathcal{C}^{\frac{1}{3} - \kappa}} \lVert w_{b}^{\mathcal{H}} (t) \rVert_{H^{\frac{2}{3} + \kappa - \eta}}  
\overset{\eqref{est 33} \eqref{est 67} \eqref{est 39} }{\lesssim} \lambda_{t}^{\frac{1}{3}} (N_{t}^{\kappa})^{2}, 
\end{equation} 
where the first inequality used the hypothesis that $\eta < 1$ and second inequality is justified by $\frac{2}{3} + \kappa < \eta$ and $\kappa_{0} \leq \frac{1}{12}$.  Therefore, applying \eqref{est 240}, \eqref{est 109}, and \eqref{est 110} to \eqref{est 241} gives us 
\begin{equation}\label{est 108} 
\lvert \langle w_{u}^{\mathcal{L}}, \divergence (\mathcal{L}_{\lambda_{t}} X_{b} \otimes_{s} w_{b}^{\mathcal{H}} ) \rangle \rvert(t)    \lesssim  \lVert w_{u}^{\mathcal{L}}(t) \rVert_{H^{\eta}}  \lambda_{t}^{\frac{1}{3}} (N_{t}^{\kappa})^{2}. 
\end{equation} 
Similar computations on other three terms in $\RomanI_{3,1}$ in \eqref{est 69a}  lead us to 
\begin{equation}\label{est 72} 
\lvert \RomanI_{3,1} \rvert \lesssim ( \lVert w_{u}^{\mathcal{L}} \rVert_{H^{\eta}} + \lVert w_{b}^{\mathcal{L}} \rVert_{H^{\eta}})(t)  \lambda_{t}^{\frac{1}{3}} (N_{t}^{\kappa})^{2}. 
\end{equation} 
Next, within $\RomanI_{3,2}$ of \eqref{est 69b}, we estimate as an example 
\begin{align}
& \lvert \langle w_{b}^{\mathcal{L}}, \divergence ( w_{b}^{\mathcal{H}} \circlesign{\succ}_{a} \mathcal{H}_{\lambda_{t}} X_{u} + w_{b}^{\mathcal{H}} \circlesign{\circ}_{a} \mathcal{H}_{\lambda_{t}} X_{u}) \rangle(t)  \rvert \nonumber \\
\lesssim& \lVert w_{b}^{\mathcal{L}}(t)  \rVert_{H^{\eta}} ( \lVert w_{b}^{\mathcal{H}} \circlesign{\succ}_{a} \mathcal{H}_{\lambda_{t}} X_{u} \rVert_{H^{1-\eta}} + \lVert  w_{b}^{\mathcal{H}} \circlesign{\circ}_{a} \mathcal{H}_{\lambda_{t}} X_{u} \rVert_{H^{1-\eta}} )(t)  \nonumber \\
\overset{\eqref{est 40d} \eqref{est 40e}}{\lesssim}& \lVert w_{b}^{\mathcal{L}} (t) \rVert_{H^{\eta}}  \lVert w_{b}^{\mathcal{H}} (t) \rVert_{H^{1-\eta + \kappa}} \lVert \mathcal{H}_{\lambda_{t}} X_{u} (t) \rVert_{\mathcal{C}^{-\kappa}} \overset{\eqref{est 67} \eqref{est 39} }{\lesssim} \lVert w_{b}^{\mathcal{H}} (t) \rVert_{H^{\eta}} (N_{t}^{\kappa})^{2}, \label{est 246} 
\end{align}
where the second inequality used the hypothesis that $\eta < 1$ and the third inequality is justified by $\frac{2}{3} + \kappa < \eta$ and $\kappa_{0} \leq \frac{1}{12}$. Similar computations on the other three terms in $\RomanI_{3,2}$ of \eqref{est 69b} lead to 
\begin{align}\label{est 71} 
\lvert \RomanI_{3,2} \rvert \lesssim ( \lVert w_{u}^{\mathcal{L}} \rVert_{H^{\eta}} + \lVert w_{b}^{\mathcal{L}} \rVert_{H^{\eta}})(t)  (N_{t}^{\kappa})^{2};  
\end{align}
thus, we conclude the estimate \eqref{est 70} by applying \eqref{est 72} and \eqref{est 71} to \eqref{est 73}. 

Next, concerning $ 2\langle w_{u}^{\mathcal{L}}, \divergence (w_{u}^{\otimes 2} - w_{b}^{\otimes 2}) \rangle + 2\langle w_{b}^{\mathcal{L}}, \divergence (w_{b} \otimes w_{u} - w_{u} \otimes w_{b} ) \rangle$ within $\RomanI_{4}$ from \eqref{est 43d}, we can make use of cancellations of $\langle w_{u}^{\mathcal{L}}, \divergence (w_{u}^{\mathcal{L}} \otimes w_{u}^{\mathcal{L}}) \rangle = 0$, $ \langle w_{b}^{\mathcal{L}}, \divergence (w_{b}^{\mathcal{L}} \otimes w_{u}^{\mathcal{L}}) \rangle= 0$ and
\begin{equation}\label{est 237}
\langle w_{u}^{\mathcal{L}}, \divergence (w_{b}^{\mathcal{L}} \otimes w_{b}^{\mathcal{L}}) \rangle + \langle w_{b}^{\mathcal{L}}, \divergence (w_{u}^{\mathcal{L}} \otimes w_{b}^{\mathcal{L}}) \rangle = 0
\end{equation} 
to obtain  
\begin{align}
& \langle w_{u}^{\mathcal{L}}, \divergence (w_{u}^{\otimes 2} - w_{b}^{\otimes 2}) \rangle + \langle w_{b}^{\mathcal{L}}, \divergence (w_{b} \otimes w_{u} - w_{u} \otimes w_{b} ) \rangle \nonumber \\
=&\langle w_{u}^{\mathcal{L}}, \divergence ( 2w_{u}^{\mathcal{L}} \otimes_{s} w_{u}^{\mathcal{H}} + (w_{u}^{\mathcal{H}})^{\otimes 2} - 2 w_{b}^{\mathcal{L}} \otimes_{s} w_{b}^{\mathcal{H}} - (w_{b}^{\mathcal{H}})^{\otimes 2} \rangle \nonumber \\
&+ \langle w_{b}^{\mathcal{L}}, \divergence (2w_{b}^{\mathcal{L}} \otimes_{a} w_{u}^{\mathcal{H}} + 2w_{b}^{\mathcal{H}} \otimes_{a} w_{u}^{\mathcal{L}} + 2 w_{b}^{\mathcal{H}} \otimes_{a} w_{u}^{\mathcal{H}}) \rangle. \label{est 75} 
\end{align} 
As an example, let us estimate 
\begin{align*}
 \lVert w_{b}^{\mathcal{H}} \otimes_{a} w_{u}^{\mathcal{H}} (t) \rVert_{H^{\frac{1}{4} - 3 \kappa}} 
\overset{\eqref{est 74}}{\lesssim}& \lVert w_{b}^{\mathcal{H}}  (t) \rVert_{\dot{H}^{\frac{5}{8} - \frac{3\kappa}{2}}} \lVert w_{u}^{\mathcal{H}} (t) \rVert_{\dot{H}^{\frac{5}{8} - \frac{3\kappa}{2}}}  \nonumber\\
\overset{\eqref{est 67}}{\lesssim}& (1+ \lVert w_{u} \rVert_{L^{2}} + \lVert w_{b} \rVert_{L^{2}})^{2( 1 - \mathfrak{a} [ \frac{3}{8} - \frac{\kappa}{2} ])}(t) (N_{t}^{\kappa})^{2} \lesssim (N_{t}^{\kappa})^{2}, 
\end{align*} 
where the last inequality used that $\kappa \leq \frac{1}{44}$ and $\mathfrak{a} > \frac{11}{4}$. Analogous computations on similar terms in \eqref{est 75} show altogether   
\begin{equation}\label{est 113} 
\lVert (w_{u}^{\mathcal{H}})^{\otimes 2} (t) \rVert_{H^{\frac{1}{4} - 3 \kappa}} + \lVert (w_{b}^{\mathcal{H}})^{\otimes 2} (t)\rVert_{H^{\frac{1}{4} - 3 \kappa}} + \lVert w_{b}^{\mathcal{H}} \otimes_{a} w_{u}^{\mathcal{H}} (t)\rVert_{H^{\frac{1}{4} - 3 \kappa}}  \lesssim (N_{t}^{\kappa})^{2} 
\end{equation}  
and allow us to conclude 
\begin{align}
& \langle w_{u}^{\mathcal{L}}, \divergence ( (w_{u}^{\mathcal{H}})^{\otimes 2} - (w_{b}^{\mathcal{H}})^{\otimes 2} \rangle(t) + \langle w_{b}^{\mathcal{L}}, \divergence (2w_{b}^{\mathcal{H}} \otimes_{a} w_{u}^{\mathcal{H}} ) \rangle(t)  \nonumber \\
\lesssim& ( \lVert w_{u}^{\mathcal{L}} \rVert_{H^{\frac{3}{4} + 3 \kappa}} + \lVert w_{b}^{\mathcal{L}} \rVert_{H^{\frac{3}{4} + 3 \kappa}} )(t) (N_{t}^{\kappa})^{2}.  \label{est 76}
\end{align}

Finally, among the rest of the terms in \eqref{est 75}, we estimate as an example 
\begin{align*}
& \langle w_{b}^{\mathcal{L}}, \divergence (w_{b}^{\mathcal{L}} \otimes_{a} w_{u}^{\mathcal{H}}) \rangle(t) \nonumber \\
 \overset{\eqref{est 74}}{\lesssim}&  \lVert w_{b}^{\mathcal{L}} (t)\rVert_{H^{\eta}}  \lVert w_{b}^{\mathcal{L}} (t)\rVert_{\dot{H}^{\frac{3}{2}  - \eta + 2 \kappa} } \lVert w_{u}^{\mathcal{H}} (t)\rVert_{\dot{H}^{\frac{1}{2} - 2 \kappa}}  \overset{\eqref{est 67}}{\lesssim}  \lVert w_{b}^{\mathcal{L}} (t)\rVert_{H^{\eta}}^{2} N_{t}^{\kappa}. 
\end{align*}
Similar computations lead us to altogether 
\begin{align}
&\langle w_{u}^{\mathcal{L}}, \divergence (w_{u}^{\mathcal{L}} \otimes_{s} w_{u}^{\mathcal{H}} - w_{b}^{\mathcal{L}} \otimes_{s} w_{b}^{\mathcal{H}} ) \rangle(t) \nonumber \\
& \hspace{10mm} + \langle w_{b}^{\mathcal{L}}, \divergence (w_{b}^{\mathcal{L}} \otimes_{a} w_{u}^{\mathcal{H}} + w_{b}^{\mathcal{H}} \otimes_{a} w_{u}^{\mathcal{L}}) \rangle  \lesssim \lVert (w_{u}^{\mathcal{L}}, w_{b}^{\mathcal{L}})(t) \rVert_{H^{\eta}}^{2}  N_{t}^{\kappa}.  \label{est 77}
\end{align}
Applying \eqref{est 76} and  \eqref{est 77} to \eqref{est 75} gives us the desired result  \eqref{est 78}. 
\end{proof}  

\begin{proposition}\label{Proposition 4.6}  
\rm{(Cf. \cite[Lemma 4.12]{HR23})} Let $t \in [T_{i}, T_{i+1})$ and fix $\lambda_{t}$ from \eqref{est 37} with $\mathfrak{a} \in [\frac{11}{4}, 3]$ and $\kappa_{0} \in (0, \frac{3}{44}]$. Then there exists $\eta \in [\frac{3}{4} + 3\kappa, 1)$  for all $\kappa \in (0, \kappa_{0}]$ such that $\RomanI_{4}$ from \eqref{est 43d} satisfies 
\begin{align}
& \RomanI_{4} +  2\langle w_{u}^{\mathcal{L}}, \divergence (w_{u}^{\otimes 2} - w_{b}^{\otimes 2}) \rangle(t) + 2\langle w_{b}^{\mathcal{L}}, \divergence (w_{b} \otimes w_{u} - w_{u} \otimes w_{b} ) \rangle(t)  \nonumber \\
\lesssim& (  \lVert w_{u}^{\mathcal{L}} \rVert_{H^{1- \frac{3\kappa}{2}}} + \lVert w_{b}^{\mathcal{L}} \rVert_{H^{1- \frac{3\kappa}{2}}})(t) ( \lVert w_{u}^{\mathcal{L}} (t)\rVert_{H^{2\kappa}} + \lVert w_{b}^{\mathcal{L}} (t)\rVert_{H^{2\kappa}} + N_{t}^{\kappa}) N_{t}^{\kappa} \nonumber \\
&+ ( \lVert w_{u}^{\mathcal{L}} \rVert_{H^{\eta}} + \lVert w_{b}^{\mathcal{L}} \rVert_{H^{\eta}})(t) ( \lVert w_{u}^{\mathcal{L}} (t)\rVert_{H^{\eta}} + \lVert w_{b}^{\mathcal{L}} (t)\rVert_{H^{\eta}} + N_{t}^{\kappa}) (N_{t}^{\kappa})^{2}.  \label{est 93}
\end{align}
\end{proposition} 

\begin{proof}[Proof of Proposition \ref{Proposition 4.6}]
We have from \eqref{est 43d}
\begin{align}
&\RomanI_{4} +  2\langle w_{u}^{\mathcal{L}}, \divergence (w_{u}^{\otimes 2} - w_{b}^{\otimes 2}) \rangle(t) + 2\langle w_{b}^{\mathcal{L}}, \divergence (w_{b} \otimes w_{u} - w_{u} \otimes w_{b} ) \rangle(t)  \label{est 89}\\ 
=& -2 \langle w_{u}^{\mathcal{L}}, \nonumber \\
& \hspace{5mm} \divergence (2 Y_{u} \otimes_{s} w_{u} - 2Y_{b} \otimes_{s} w_{b} - C^{\circlesign{\prec}_{s}} (w_{u}, Q_{u}^{\mathcal{H}}) + Y_{u}^{\otimes 2} + C^{\circlesign{\prec}_{s}} (w_{b}, Q_{b}^{\mathcal{H}}) - Y_{b}^{\otimes 2} ) \rangle(t) \nonumber \\
& - 2 \langle w_{b}^{\mathcal{L}}, \nonumber \\
& \hspace{5mm} \divergence (2w_{b} \otimes_{a} Y_{u} - 2w_{u} \otimes_{a} Y_{b} - C^{\circlesign{\prec}_{a}} (w_{b}, Q_{u}^{\mathcal{H}}) + Y_{b} \otimes Y_{u} + C^{\circlesign{\prec}_{a}} (w_{u}, Q_{b}^{\mathcal{H}}) - Y_{u} \otimes Y_{b}) \rangle(t). \nonumber 
\end{align}
First, e.g. we estimate 
\begin{align*}
& \lVert w_{b} \otimes_{a} Y_{u} (t)\rVert_{H^{\frac{3\kappa}{2}}} \leq [\lVert w_{b} \circlesign{\prec}_{a} Y_{u} \rVert_{H^{\frac{3\kappa}{2}}} + \lVert w_{b} \circlesign{\succ}_{a} Y_{u} \rVert_{H^{\frac{3\kappa}{2}}} + \lVert w_{b} \circlesign{\circ}_{a} Y_{u} \rVert_{H^{\frac{3\kappa}{2}}}](t) \nonumber \\
&\overset{\eqref{est 40c} \eqref{est 40d} \eqref{est 40e}}{\lesssim} [\lVert w_{b} \rVert_{H^{-\frac{\kappa}{2}}} \lVert Y_{u} \rVert_{\mathcal{C}^{2\kappa}} + \lVert w_{b} \rVert_{H^{2\kappa}} \lVert Y_{u} \rVert_{\mathcal{C}^{-\frac{\kappa}{2}}} ](t) \nonumber \\
& \hspace{20mm} \lesssim \lVert w_{b} (t)\rVert_{H^{2 \kappa}} \lVert Y_{u} (t)\rVert_{\mathcal{C}^{2\kappa}} 
\overset{\eqref{est 39} \eqref{est 68} \eqref{est 67}}{\lesssim}  (\lVert w_{b}^{\mathcal{L}} (t)\rVert_{H^{2\kappa}} + N_{t}^{\kappa}) N_{t}^{\kappa}
\end{align*}
so that 
\begin{align*}
&\langle w_{b}^{\mathcal{L}}, \divergence (w_{b} \otimes_{a} Y_{u}) \rangle(t) \lesssim \lVert w_{b}^{\mathcal{L}} (t)\rVert_{H^{1- \frac{3\kappa}{2}}} \lVert w_{b} \otimes_{a} Y_{u} (t)\rVert_{H^{\frac{3\kappa}{2}}} \nonumber \\
\lesssim& \lVert w_{b}^{\mathcal{L}} (t)\rVert_{H^{1- \frac{3\kappa}{2}}} (\lVert w_{b}^{\mathcal{L}} (t)\rVert_{H^{2\kappa}} + N_{t}^{\kappa}) N_{t}^{\kappa}.
\end{align*}
Thus, similar computations on the analogous terms in \eqref{est 89} show altogether 
\begin{align}
& - 2 \langle w_{u}^{\mathcal{L}}, \divergence (2Y_{u} \otimes_{s} w_{u} - 2 Y_{b} \otimes_{s} w_{b} ) \rangle(t)  - 2 \langle w_{b}^{\mathcal{L}}, \divergence (2w_{b} \otimes_{a} Y_{u} - 2 w_{u} \otimes_{a} Y_{b} ) \rangle(t) \nonumber \\
\lesssim& (\lVert w_{u}^{\mathcal{L}} \rVert_{H^{1- \frac{3\kappa}{2}}} + \lVert w_{b}^{\mathcal{L}} \rVert_{H^{1- \frac{3\kappa}{2}}})(t) ( \lVert w_{u}^{\mathcal{L}} (t)\rVert_{H^{2\kappa}} + \lVert w_{b}^{\mathcal{L}} (t)\rVert_{H^{2\kappa}} + N_{t}^{\kappa}) N_{t}^{\kappa}. \label{est 90}
\end{align}
Next, we estimate as an example,
\begin{align*}
&-2 \langle w_{b}^{\mathcal{L}}, \divergence (Y_{b} \otimes Y_{u}) \rangle(t) \nonumber\\
\lesssim&  \lVert w_{u}^{\mathcal{L}} (t)\rVert_{H^{1- \frac{3\kappa}{2}}} (\lVert Y_{b} \circlesign{\succ} Y_{u} \rVert_{H^{\frac{3\kappa}{2}}} + \lVert Y_{b} \circlesign{\prec} Y_{u} \rVert_{H^{ \frac{3\kappa}{2}}} + \lVert Y_{b} \circlesign{\circ}Y_{u} \rVert_{H^{\frac{3\kappa}{2}}})(t)  \\
\overset{\eqref{est 40}}{\lesssim}&  \lVert w_{u}^{\mathcal{L}} (t)\rVert_{H^{1- \frac{3\kappa}{2}}} (\lVert Y_{b} \rVert_{H^{\frac{7\kappa}{4}}} \lVert Y_{u} \rVert_{\mathcal{C}^{-\frac{\kappa}{4}}}   + \lVert Y_{b} \rVert_{\mathcal{C}^{-\frac{\kappa}{4}}} \lVert Y_{u} \rVert_{H^{\frac{7\kappa}{4}}} )(t)   \overset{\eqref{est 39} }{\lesssim}  \lVert w_{u}^{\mathcal{L}} \rVert_{H^{1- \frac{3\kappa}{2}}} (N_{t}^{\kappa})^{2}. \nonumber 
\end{align*}
Working similarly on analogous terms in \eqref{est 89} lead us to altogether 
\begin{align}
& -2 \langle w_{u}^{\mathcal{L}}, \divergence (Y_{u}^{\otimes 2} - Y_{b}^{\otimes 2}) \rangle(t) - 2 \langle w_{b}^{\mathcal{L}}, \divergence (Y_{b} \otimes Y_{u} - Y_{u} \otimes Y_{b}) \rangle(t) \nonumber \\
\lesssim&( \lVert w_{u}^{\mathcal{L}} \rVert_{H^{1- \frac{3\kappa}{2}}} + \lVert w_{b}^{\mathcal{L}} \rVert_{H^{1- \frac{3\kappa}{2}}})(t) (N_{t}^{\kappa})^{2}.  \label{est 91}  
\end{align}
We now work on the remaining commutator terms in \eqref{est 89}, e.g. $2 \langle w_{b}^{\mathcal{L}}, \divergence (C^{\circlesign{\prec}_{a}} (w_{b}, Q_{u}^{\mathcal{H}})) \rangle$, for which which we recall from \eqref{est 30b}
\begin{equation}\label{est 79} 
C^{\circlesign{\prec}_{a}} (w_{b}, Q_{u}^{\mathcal{H}}) = (( \partial_{t} - \nu \Delta) w_{b}) \circlesign{\prec}_{a} Q_{u}^{\mathcal{H}} - 2\nu  \sum_{k=1}^{2} \partial_{k} w_{b} \circlesign{\prec}_{a} \partial_{k} Q_{u}^{\mathcal{H}}. 
\end{equation} 
For the first term in \eqref{est 79} we write using \eqref{est 29b} 
\begin{align} 
&(( \partial_{t} -\nu \Delta) w_{b}) \circlesign{\prec}_{a} Q_{u}^{\mathcal{H}} \label{est 82}\\
=& - [ \mathbb{P}_{L} \divergence ( w_{b} \otimes w_{u} + w_{b} \otimes_{a} D_{u} + Y_{b} \otimes Y_{u} - w_{u} \otimes w_{b} - w_{u} \otimes_{a} D_{b} - Y_{u} \otimes Y_{b}) ] \circlesign{\prec}_{a} Q_{u}^{\mathcal{H}},   \nonumber 
\end{align}
within which we estimate as an example for any $\gamma \in (0, \eta - 2 \kappa]$, 
\begin{align}
&\lVert [ \divergence (w_{b} \otimes w_{u}) ] \circlesign{\prec}_{a} Q_{u}^{\mathcal{H}} (t)\rVert_{H^{1- 2 \kappa - \gamma}} \label{est 120} \\ 
&\overset{\eqref{est 40c}\eqref{est 38a}}{\lesssim} \lVert w_{b} \otimes w_{u} (t)\rVert_{H^{-\frac{\kappa}{2}}} \lVert \mathcal{H}_{\lambda_{t}} Q_{u} (t)\rVert_{\mathcal{C}^{2- \frac{3\kappa}{2} - \gamma}}  \overset{\eqref{est 33} \eqref{est 32} \eqref{est 39} }{\lesssim} \lVert w_{b} (t)\rVert_{L^{4}} \lVert w_{u} (t)\rVert_{L^{4}} \lambda_{t}^{-\gamma} N_{t}^{\kappa}  \nonumber  \\
&\overset{\eqref{est 38} \eqref{est 67}}{\lesssim}  ( \lVert w_{b}^{\mathcal{L}} (t)\rVert_{H^{\frac{1}{2}}}^{2} + \lVert w_{u}^{\mathcal{L}} (t)\rVert_{H^{\frac{1}{2}}}^{2} + (N_{t}^{\kappa})^{2}) (1+ \lVert w_{u} \rVert_{L^{2}} + \lVert w_{b} \rVert_{L^{2}})^{-\mathfrak{a} \gamma}(t) N_{t}^{\kappa}, \nonumber 
\end{align}
where we utilized the embedding $H^{\frac{1}{2}} (\mathbb{T}^{2}) \hookrightarrow L^{4}(\mathbb{T}^{2})$, so that 
\begin{align}
& \langle w_{b}^{\mathcal{L}}, \divergence [\mathbb{P}_{L} \divergence (w_{b} \otimes w_{u}) \circlesign{\prec}_{a} Q_{u}^{\mathcal{H}} ] \rangle(t)  \nonumber \\
&\lesssim \lVert w_{b}^{\mathcal{L}} (t)\rVert_{L^{2}}^{1- \frac{2\kappa + \gamma}{\eta}} \lVert w_{b}^{\mathcal{L}} (t)\rVert_{H^{\eta}}^{\frac{2\kappa + \gamma}{\eta}} [ \lVert w_{b}^{\mathcal{L}} (t)\rVert_{L^{2}}^{2(1- \frac{1}{2\eta})} \lVert w_{b}^{\mathcal{L}} (t)\rVert_{H^{\eta}}^{2( \frac{1}{2\eta})} + \lVert w_{u}^{\mathcal{L}} (t)\rVert_{L^{2}}^{2(1- \frac{1}{2\eta})} \lVert w_{b}^{\mathcal{L}} \rVert_{H^{\eta}}^{2( \frac{1}{2\eta})} + (N_{t}^{\kappa})^{2} ] \nonumber\\
& \hspace{55mm} \times (1+ \lVert w_{u} \rVert_{L^{2}} + \lVert w_{b} \rVert_{L^{2}})^{-\mathfrak{a} \gamma}(t) N_{t}^{\kappa} \label{est 80}
\end{align}
due to Gagliardo-Nirenberg inequalities. For example, a choice of 
\begin{equation}\label{est 83}
\eta = \frac{3}{4} + 3 \kappa \text{ and } \gamma = \frac{1}{2}+ 4 \kappa
\end{equation} 
accomplishes 
\begin{equation}\label{est 81}
\frac{2\kappa + \gamma}{\eta} + 2\left(\frac{1}{2\eta} \right) = 2 \text{ and } 1- \frac{2\kappa + \gamma}{\eta} + 2\left(1- \frac{1}{2\eta} \right) - \mathfrak{a} \gamma \leq 0 
\end{equation} 
due to $\mathfrak{a} \geq 2$. Therefore, we conclude that 
\begin{equation}\label{est 84} 
\langle w_{b}^{\mathcal{L}}, \divergence [\mathbb{P}_{L} \divergence (w_{b} \otimes w_{u}) \circlesign{\prec}_{a} Q_{u}^{\mathcal{H}} ] \rangle(t) \lesssim  ( \lVert w_{b}^{\mathcal{L}} \rVert_{H^{\eta}} + \lVert w_{u}^{\mathcal{L}} \rVert_{H^{\eta}})^{2}(t) N_{t}^{\kappa} + \lVert w_{b}^{\mathcal{L}} (t)\rVert_{H^{\eta}} (N_{t}^{\kappa})^{3}.
\end{equation} 
Similar computations lead to 
\begin{equation}\label{est 85} 
\langle w_{b}^{\mathcal{L}}, \divergence [\mathbb{P}_{L} \divergence (w_{u} \otimes w_{b}) \circlesign{\prec}_{a} Q_{u}^{\mathcal{H}} ] \rangle(t) \lesssim  ( \lVert w_{b}^{\mathcal{L}} \rVert_{H^{\eta}} + \lVert w_{u}^{\mathcal{L}} \rVert_{H^{\eta}})^{2}(t) N_{t}^{\kappa} + \lVert w_{b}^{\mathcal{L}} \rVert_{H^{\eta}} (N_{t}^{\kappa})^{3}.
\end{equation} 
Concerning the rest of the terms in \eqref{est 82}, namely 
\begin{align*}
[ \mathbb{P}_{L} \divergence (w_{b} \otimes_{a} D_{u} + Y_{b} \otimes Y_{u} - w_{u} \otimes_{a} D_{b} - Y_{u} \otimes Y_{b}) ] \circlesign{\prec}_{a} Q_{u}^{\mathcal{H}}, 
\end{align*}
we estimate as an example 
\begin{align}
&   \lVert [ \mathbb{P}_{L} \divergence (w_{b} \otimes_{a} D_{u}) ] \circlesign{\prec}_{a} Q_{u}^{\mathcal{H}} \rVert_{H^{1-\eta}} + \lVert [\mathbb{P}_{L} \divergence (Y_{b} \otimes Y_{u}) ] \circlesign{\prec}_{a} Q_{u}^{\mathcal{H}} \rVert_{H^{1-\eta}} \nonumber \\
\overset{\eqref{est 40c}}{\lesssim}&   [ \lVert w_{b} \otimes_{a} D_{u} \rVert_{H^{-\eta + \frac{3\kappa}{2}}}+ \lVert Y_{b} \otimes Y_{u} \rVert_{H^{-\eta + \frac{3\kappa}{2}}}] \lVert Q_{u} \rVert_{\mathcal{C}^{2- \frac{3\kappa}{2}}},   \label{est 123}
\end{align}
and thus for any $\eta \geq \frac{3}{4} + \frac{3\kappa}{2}$ and hence e.g.  $\eta = \frac{3}{4} + 3 \kappa$ from \eqref{est 83}, via Sobolev embedding $L^{\frac{8}{7}}(\mathbb{T}^{2}) \hookrightarrow H^{-\frac{3}{4}}(\mathbb{T}^{2})$, 
\begin{align}
& \langle w_{b}^{\mathcal{L}}, \divergence ( [ \mathbb{P}_{L} \divergence ( w_{b} \otimes_{a} D_{u} + Y_{b} \otimes Y_{u}) ] \circlesign{\prec}_{a} Q_{u}^{\mathcal{H}} ) \rangle(t) \nonumber \\
&\lesssim \lVert w_{b}^{\mathcal{L}} (t) \rVert_{H^{\eta}} [ \lVert \mathbb{P}_{L} \divergence (w_{b} \otimes_{a} D_{u}) \rVert_{H^{-1-\eta + \frac{3\kappa}{2}}} \lVert Q_{u}^{\mathcal{H}} \rVert_{\mathcal{C}^{2-\frac{3\kappa}{2}}}  + \lVert \mathbb{P}_{L} \divergence (Y_{b} \otimes Y_{u}) \rVert_{H^{-1 - \eta + \frac{3\kappa}{2}}} \lVert Q_{u}^{\mathcal{H}} \rVert_{\mathcal{C}^{2-\frac{3\kappa}{2}}} ](t) \nonumber \\
&\overset{\eqref{est 32} \eqref{est 39} }{\lesssim} \lVert w_{b}^{\mathcal{L}}(t) \rVert_{H^{\eta}} \nonumber \\
& \times [ \lVert w_{b} \circlesign{\prec}_{a} D_{u} \rVert_{H^{-\eta + \frac{3\kappa}{2}}} + \lVert w_{b} \circlesign{\succ}_{a} D_{u} \rVert_{H^{-\eta + \frac{3\kappa}{2}}} + \lVert w_{b} \circlesign{\circ}_{a} D_{u} \rVert_{H^{\frac{1}{2} - 3 \kappa}}    + \lVert Y_{b} \otimes Y_{u} \rVert_{H^{-\frac{3}{4}}} ](t) N_{t}^{\kappa}  \nonumber \\
&\overset{\eqref{est 40d} \eqref{est 40c} \eqref{est 40e}}{\lesssim} N_{t}^{\kappa} \lVert w_{b}^{\mathcal{L}} (t)\rVert_{H^{\eta}} \left( \lVert w_{b} \rVert_{H^{-\eta + \frac{5\kappa}{2}}}  \lVert D_{u} \rVert_{\mathcal{C}^{-\kappa}} + \lVert w_{b} \rVert_{H^{\frac{1}{2} - 2 \kappa}} \lVert D_{u} \rVert_{\mathcal{C}^{-\frac{1}{2} - \eta + \frac{7\kappa}{2}}} + \lVert Y_{b} \otimes Y_{u} \rVert_{L^{\frac{8}{7}}} \right)(t) \nonumber \\
&\overset{\eqref{est 67} \eqref{est 39} }{\lesssim} (N_{t}^{\kappa})^{2} \lVert w_{b}^{\mathcal{L}} (t)\rVert_{H^{\eta}} [  \lVert w_{b}^{\mathcal{L}}(t) \rVert_{H^{\eta}} + N_{t}^{\kappa} ]. \label{est 86}
\end{align} 
Similar computations for the analogous term $[ \mathbb{P}_{L} \divergence (w_{u} \otimes_{a} D_{b} + Y_{u} \otimes Y_{b}) ] \circlesign{\prec}_{a} Q_{u}^{\mathcal{H}}$ in \eqref{est 82} leads to 
\begin{align}
& \langle w_{b}^{\mathcal{L}}, \divergence ( [ \mathbb{P}_{L} \divergence ( w_{u} \otimes_{a} D_{b} + Y_{u} \otimes Y_{b}) ] \circlesign{\prec}_{a} Q_{u}^{\mathcal{H}} ) \rangle(t) \nonumber \\
\lesssim& (N_{t}^{\kappa})^{2} \lVert w_{b}^{\mathcal{L}} (t)\rVert_{H^{\eta}} [  \lVert w_{u}^{\mathcal{L}} (t)\rVert_{H^{\eta}} + N_{t}^{\kappa} ]. \label{est 3}
\end{align}   
Thus, applying \eqref{est 84} \eqref{est 85}, \eqref{est 86}, and \eqref{est 3} to \eqref{est 82} gives us  
\begin{align}
&\langle w_{b}^{\mathcal{L}}, \divergence (( \partial_{t} - \nu \Delta)w_{b} \circlesign{\prec}_{a} Q_{u}^{\mathcal{H}}) \rangle(t)\nonumber\\
\lesssim& ( \lVert w_{u}^{\mathcal{L}} \rVert_{H^{\eta}} + \lVert w_{b}^{\mathcal{L}} \rVert_{H^{\eta}})(t) ( \lVert w_{u}^{\mathcal{L}} (t)\rVert_{H^{\eta}} + \lVert w_{b}^{\mathcal{L}} (t)\rVert_{H^{\eta}} + N_{t}^{\kappa}) (N_{t}^{\kappa})^{2}.  \label{est 88}
\end{align} 
Finally, because $\eta > \frac{3\kappa}{2}$ we can estimate
\begin{align}
 \sum_{k=1}^{2} \langle w_{b}^{\mathcal{L}}, \divergence ( \partial_{k} w_{b} \circlesign{\prec}_{a} \partial_{k} Q_{u}^{\mathcal{H}}) \rangle(t) \overset{\eqref{est 40c}}{\lesssim}& \sum_{k=1}^{2} \lVert w_{b}^{\mathcal{L}} (t)\rVert_{H^{\eta}} \lVert \partial_{k} w_{b} (t)\rVert_{H^{\frac{3\kappa}{2} - \eta}} \lVert \partial_{k} Q_{u}^{\mathcal{H}} (t)\rVert_{\mathcal{C}^{1- \frac{3\kappa}{2}}} \nonumber \\
&\overset{\eqref{est 67}\eqref{est 32}\eqref{est 39} }{\lesssim} \lVert w_{b}^{\mathcal{L}} (t)\rVert_{H^{\eta}} ( \lVert w_{b}^{\mathcal{L}} (t)\rVert_{H^{\eta}} + N_{t}^{\kappa}) N_{t}^{\kappa}.  \label{est 87}
\end{align}  
Applying \eqref{est 88} and \eqref{est 87} to \eqref{est 79} gives us 
\begin{equation*}
\langle w_{b}^{\mathcal{L}}, \divergence  C^{\circlesign{\prec}_{a}} (w_{b}, Q_{u}^{\mathcal{H}}) \rangle(t) \lesssim  ( \lVert w_{u}^{\mathcal{L}} \rVert_{H^{\eta}} + \lVert w_{b}^{\mathcal{L}} \rVert_{H^{\eta}})(t) ( \lVert w_{u}^{\mathcal{L}} (t)\rVert_{H^{\eta}} + \lVert w_{b}^{\mathcal{L}} (t)\rVert_{H^{\eta}} +N_{t}^{\kappa}) (N_{t}^{\kappa})^{2}. 
\end{equation*} 
Similar computations on the three other commutators in \eqref{est 89} give altogether  
\begin{align}
&- 2 \langle w_{u}^{\mathcal{L}}, \divergence (-C^{\circlesign{\prec}_{s}} (w_{u}, Q_{u}^{\mathcal{H}}) + C^{\circlesign{\prec}_{s}} (w_{b}, Q_{b}^{\mathcal{H}})) \rangle(t) \nonumber\\
&- 2 \langle w_{b}^{\mathcal{L}}, \divergence (-C^{\circlesign{\prec}_{a}} (w_{b}, Q_{u}^{\mathcal{H}}) + C^{\circlesign{\prec}_{a}} (w_{u}, Q_{b}^{\mathcal{H}})) \rangle(t) \nonumber\\
&\lesssim ( \lVert w_{u}^{\mathcal{L}} \rVert_{H^{\eta}} + \lVert w_{b}^{\mathcal{L}} \rVert_{H^{\eta}})(t) ( \lVert w_{u}^{\mathcal{L}} (t)\rVert_{H^{\eta}} + \lVert w_{b}^{\mathcal{L}} (t)\rVert_{H^{\eta}} +N_{t}^{\kappa}) (N_{t}^{\kappa})^{2}.   \label{est 92} 
\end{align}
At last, we conclude \eqref{est 93} from applying \eqref{est 90}, \eqref{est 91}, and \eqref{est 92} to \eqref{est 89}. This completes the proof of Proposition \ref{Proposition 4.6}.
\end{proof} 

\begin{corollary}\label{Corollary 4.7} 
\rm{(Cf. \cite[Proposition 4.7]{HR23})} Fix $\lambda_{t}$ from \eqref{est 37} with $\mathfrak{a} = 3$ and $\kappa_{0} \in (0, \frac{1}{44}]$. Then there exists a universal constant $C> 0$ such that for all $\kappa \in (0, \kappa_{0}]$, all $i \in \mathbb{N}_{0}$, and all $t \in [T_{i}, T_{i+1})$, 
\begin{align}
\partial_{t}  \lVert (w_{u}^{\mathcal{L}}, w_{b}^{\mathcal{L}})(t) \rVert_{L^{2}}^{2} &
\leq - \nu \lVert (w_{u}^{\mathcal{L}}, w_{b}^{\mathcal{L}}) (t)\rVert_{\dot{H}^{1}}^{2} + 2 \left\langle 
\begin{pmatrix}
w_{u}^{\mathcal{L}} \\
w_{b}^{\mathcal{L}} 
\end{pmatrix}, 
\mathcal{A}_{t}^{\lambda_{t}} 
\begin{pmatrix}
w_{u}^{\mathcal{L}} \\
w_{b}^{\mathcal{L}} 
\end{pmatrix} 
\right\rangle(t) + r_{\lambda_{t}}(t) \lVert (w_{u}^{\mathcal{L}}, w_{b}^{\mathcal{L}}) (t)\rVert_{L^{2}}^{2}  \nonumber \\
&+ C \lambda_{t}^{\frac{1}{3}} ( \lVert w_{u}^{\mathcal{L}} \rVert_{H^{1- \frac{3\kappa}{2}}} + \lVert w_{b}^{\mathcal{L}} \rVert_{H^{1- \frac{3\kappa}{2}}})(t) (N_{t}^{\kappa})^{2} \nonumber  \\
&+ C ( N_{t}^{\kappa})^{3} ( \lVert w_{u}^{\mathcal{L}} \rVert_{H^{1- \frac{3\kappa}{2}}} + \lVert w_{b}^{\mathcal{L}} \rVert_{H^{1- \frac{3\kappa}{2}}} + \lVert w_{u}^{\mathcal{L}} \rVert_{H^{1- \frac{3\kappa}{2}}}^{2} + \lVert w_{b}^{\mathcal{L}} \rVert_{H^{1- \frac{3\kappa}{2}}}^{2})(t). \label{est 96}
\end{align}
 \end{corollary} 

\begin{proof}[Proof of Corollary \ref{Corollary 4.7}]
The desired result follows from applying \eqref{est 65}, \eqref{est 95}, \eqref{est 70}, \eqref{est 78}, and \eqref{est 93} to \eqref{est 94}. 
\end{proof} 

\begin{proposition}\label{Proposition 4.8} 
\rm{(Cf. \cite[Corollary 5.1]{HR23})} Fix $\lambda_{t}$ from \eqref{est 37} with $\mathfrak{a} = 3$ and $\kappa_{0} \in (0, \frac{1}{44}]$. There exits a constant $C_{1} > 0$ and increasing continuous functions $C_{2}$ and $C_{3}$ from $\mathbb{R}_{\geq 0}$ to $\mathbb{R}_{\geq 0}$ such that for all $\kappa \in (0, \kappa_{0}]$, for all $i \in \mathbb{N}_{0}$, all $i \geq i_{0}$, and $t \in [T_{i}, T_{i+1})$, 
\begin{align}
\partial_{t}  \lVert (w_{u}^{\mathcal{L}}, & w_{b}^{\mathcal{L}}) (t)\rVert_{L^{2}}^{2}  \leq  - \frac{\nu }{2} \lVert (w_{u}^{\mathcal{L}}, w_{b}^{\mathcal{L}}) (t)\rVert_{\dot{H}^{1}}^{2} \nonumber \\
&+ ( C_{1} \ln (\lambda_{t}) + C_{2} (N_{t}^{\kappa}) )[\lVert (w_{u}^{\mathcal{L}}, w_{b}^{\mathcal{L}}) (t)\rVert_{L^{2}}^{2}  + \lVert (w_{u}^{\mathcal{L}}, w_{b}^{\mathcal{L}}  )(T_{i}) \rVert_{L^{2}}^{2}]+ C_{3} (N_{t}^{\kappa})  \label{est 100}
\end{align}  
and consequently 
\begin{align}
& \sup_{t \in [T_{i}, T_{i+1})} \lVert (w_{u}^{\mathcal{L}}, w_{b}^{\mathcal{L}}) (t) \rVert_{L^{2}}^{2} + \frac{\nu }{2} \int_{T_{i}}^{T_{i+1}} \lVert (w_{u}^{\mathcal{L}}, w_{b}^{\mathcal{L}})(s) \rVert_{\dot{H}^{1}}^{2} ds\nonumber  \\
\lesssim& e^{(T_{i+1} - T_{i})( C_{2} (N_{T_{i+1}}^{\kappa}) + C_{1} \ln (\lambda_{T_{i}}) )} \left( \lVert (w_{u}^{\mathcal{L}}, w_{b}^{\mathcal{L}}) (T_{i}) \rVert_{L^{2}}^{2} + C_{3} (N_{T_{i+1}}^{\kappa}) \right).   \label{est 99}
\end{align}  
\end{proposition} 

\begin{proof}[Proof of Proposition \ref{Proposition 4.8}]
We fix an arbitrary $t \in [T_{i}, T_{i+1})$. Within \eqref{est 96}, we can bound 
\begin{subequations}\label{est 98} 
\begin{align}
& \left\langle 2 
\begin{pmatrix}
w_{u}^{\mathcal{L}}, \\
w_{b}^{\mathcal{L}} 
\end{pmatrix},  
\mathcal{A}_{t}^{\lambda_{t}} 
\begin{pmatrix}
w_{u}^{\mathcal{L}} \\
w_{b}^{\mathcal{L}} 
\end{pmatrix}  \right\rangle(t) 
\leq m(N_{t}^{\kappa}) \lVert (w_{u}^{\mathcal{L}}, w_{b}^{\mathcal{L}}) (t)\rVert_{L^{2}}^{2}, \label{est 98a} \\
& \lambda_{t}^{\frac{1}{3}} \overset{\eqref{est 37} \eqref{est 67}}{\lesssim}   1 + \lVert w_{u}^{\mathcal{L}}(T_{i}) \rVert_{L^{2}} + \lVert w_{b}^{\mathcal{L}} (T_{i}) \rVert_{L^{2}} + N_{t}^{\kappa}, \label{est 98b}  
\end{align}
\end{subequations} 
where $m$ is a continuous $\mathbb{R}_{+}$-valued function due to Proposition \ref{Proposition 5.3}. Applying \eqref{est 98} and \eqref{est 232} to \eqref{est 96} gives us 
\begin{align*}
& \partial_{t} (\lVert w_{u}^{\mathcal{L}} \rVert_{L^{2}}^{2} + \lVert w_{b}^{\mathcal{L}} \rVert_{L^{2}}^{2})(t)  \nonumber \\
\leq& - \frac{\nu }{2} \lVert (w_{u}^{\mathcal{L}}, w_{b}^{\mathcal{L}}) (t)\rVert_{\dot{H}^{1}}^{2} + 2m(N_{t}^{\kappa}) \lVert (w_{u}^{\mathcal{L}}, w_{b}^{\mathcal{L}} ) (t)\rVert_{L^{2}}^{2} + C \ln (\lambda_{t}) \lVert (w_{u}^{\mathcal{L}}, w_{b}^{\mathcal{L}}) (t)\rVert_{L^{2}}^{2} \nonumber \\
&+ C( \lVert w_{u}^{\mathcal{L}} (T_{i}) \rVert_{L^{2}} + \lVert w_{b}^{\mathcal{L}}(T_{i}) \rVert_{L^{2}} + N_{t}^{\kappa})^{2} (N_{t}^{\kappa})^{3( \frac{4}{2+ 3 \kappa})} + C (N_{t}^{\kappa})^{\frac{2}{\kappa}} \lVert (w_{u}^{\mathcal{L}}, w_{b}^{\mathcal{L}}) (t)\rVert_{L^{2}}^{2}. 
\end{align*}
Thus, along with $m(N_{t}^{\kappa})$, there exist a constant $C_{1}$ and increasing continuous maps $C_{2}$ and $C_{3}$ from $\mathbb{R}_{\geq 0}$ to $\mathbb{R}_{\geq 0}$ such that \eqref{est 100} holds. 
 
Next, for all $t \in [T_{i}, T_{i+1})$, we have for 
\begin{equation*}
\mu \triangleq C_{1} \ln (\lambda_{T_{i}}) + C_{2} (N_{T_{i+1}}^{\kappa}), 
\end{equation*} 
\begin{align}
& \lVert (w_{u}^{\mathcal{L}}, w_{b}^{\mathcal{L}}) (t) \rVert_{L^{2}}^{2} + \frac{\nu }{2} \int_{T_{i}}^{t} \lVert (w_{u}^{\mathcal{L}}, w_{b}^{\mathcal{L}})(s) \rVert_{\dot{H}^{1}}^{2} ds \nonumber \\
\lesssim& e^{(T_{i+1} - T_{i}) \mu}  \lVert (w_{u}^{\mathcal{L}}, w_{b}^{\mathcal{L}})(T_{i}) \rVert_{L^{2}}^{2} + e^{(T_{i+1} - T_{i}) \mu} C_{3} (N_{T_{i+1}}^{\kappa}). \label{est 102}
\end{align} 
This implies \eqref{est 99} and completes the proof of Proposition \ref{Proposition 4.8}. 
\end{proof} 

\begin{proposition}\label{Proposition 4.9}
\rm{(Cf. \cite[Lemma 5.2]{HR23})} Fix $\lambda_{t}$ from \eqref{est 37} with $\mathfrak{a} = 3$ and $\kappa_{0} \in (0, \frac{1}{44}]$. Consider $i \in \mathbb{N}$ such that $i \geq i_{0} (u^{\text{in}}, b^{\text{in}})$ and $t > 0$. If $T_{i+1} < T^{\max} \wedge t$, then for all $\kappa \in (0, \kappa_{0}]$ there exist constants $C(N_{t}^{\kappa})$ and $\tilde{C}(N_{t}^{\kappa})$ large such that 
\begin{equation}\label{est 105}
T_{i+1} - T_{i} \geq \frac{1}{\tilde{C}(N_{t}^{\kappa}) (1+ \ln (1+ i))} \ln \left( \frac{ i^{2} + 2i - C(N_{t}^{\kappa})}{i^{2} + \tilde{C} (N_{t}^{\kappa})} \right).
\end{equation} 
\end{proposition} 

\begin{proof}[Proof of Proposition \ref{Proposition 4.9}]
We can estimate similarly to \eqref{est 102} 
\begin{equation}\label{est 103} 
 \frac{1}{ C_{1} \ln \lambda_{T_{i}} + C_{2} (N_{t}^{\kappa})}  \ln \left( \frac{ \lVert (w_{u}^{\mathcal{L}}, w_{b}^{\mathcal{L}}) (T_{i+1}) \rVert_{L^{2}}^{2}}{ \lVert (w_{u}^{\mathcal{L}}, w_{b}^{\mathcal{L}})(T_{i}) \rVert_{L^{2}}^{2} + C_{3} (N_{t}^{\kappa})} \right) \leq T_{i+1} - T_{i}.
\end{equation}  
Furthermore, by \eqref{est 38}, \eqref{est 36}, and \eqref{est 67}, 
\begin{subequations}\label{est 104}
\begin{align}
&  \lVert w_{u}^{\mathcal{L}}(T_{i+1} -) \rVert_{L^{2}} + \lVert w_{b}^{\mathcal{L}} (T_{i+1} -) \rVert_{L^{2}} \geq i + 1 - \frac{C N_{t}^{\kappa}}{i + 2}, \\
&  \lVert w_{u}^{\mathcal{L}}(T_{i}) \rVert_{L^{2}} + \lVert w_{b}^{\mathcal{L}} (T_{i}) \rVert_{L^{2}} \leq i + \frac{C N_{t}^{\kappa}}{i}. 
\end{align}
\end{subequations}
Therefore, applying \eqref{est 104} to \eqref{est 103} leads to \eqref{est 105}, completing the proof of Proposition \ref{Proposition 4.9}. 
\end{proof} 

\begin{proposition}\label{Proposition 4.10}
\rm{(Cf. \cite[Lemma 5.3]{HR23})} Fix $\lambda_{t}$ from \eqref{est 37} with $\mathfrak{a} = 3$ and $\kappa_{0} \in (0, \frac{1}{44}]$. Then the following holds for any $\kappa \in (0, \kappa_{0})$ and $\epsilon \in (0, \kappa)$. Suppose that there exist $M > 1$ and $T > 0$ such that 
\begin{equation}\label{est 107} 
\lVert (w_{u}^{\mathcal{L}}, w_{b}^{\mathcal{L}}) (0) \rVert_{\dot{H}^{\epsilon}}^{2} + \sup_{s \in [0, T \wedge T^{\max}]} \lVert (w_{u}^{\mathcal{L}}, w_{b}^{\mathcal{L}})(t) \rVert_{L^{2}}^{2} + \nu \int_{0}^{T \wedge T^{\max}}  \lVert (w_{u}^{\mathcal{L}}, w_{b}^{\mathcal{L}})(s) \rVert_{\dot{H}^{1}}^{2} ds \leq M. 
\end{equation} 
Then there exists $C(T, M, N_{T}^{\kappa}) \in (0,\infty)$ such that 
\begin{equation}\label{est 137}
\sup_{t \in [0, T \wedge T^{\max}]} \lVert (w_{u}^{\mathcal{L}}, w_{b}^{\mathcal{L}})(t) \rVert_{H^{\epsilon}}^{2} \leq C(T, M, N_{T}^{\kappa}). 
\end{equation} 
\end{proposition} 

\begin{proof}[Proof of Proposition \ref{Proposition 4.10}]
First,
\begin{equation}\label{est 132}
\partial_{t} \lVert (w_{u}^{\mathcal{L}}, w_{b}^{\mathcal{L}})(t) \rVert_{\dot{H}^{\epsilon}}^{2}  =  \sum_{k=1}^{4} \RomanII_{k}, 
\end{equation} 
where 
\begin{subequations}\label{est 106} 
\begin{align}
\RomanII_{1} \triangleq& 2 \langle (-\Delta)^{\epsilon} w_{u}^{\mathcal{L}}, \nu \Delta w_{u}^{\mathcal{L}} - \divergence ( 2 (\mathcal{L}_{\lambda_{t}} X_{u}) \otimes_{s} w_{u}^{\mathcal{L}} - 2( \mathcal{L}_{\lambda_{t}} X_{b}) \otimes_{s} w_{b}^{\mathcal{L}} ) \rangle(t) \nonumber  \\
+& 2 \langle (-\Delta)^{\epsilon}  w_{b}^{\mathcal{L}}, \nu \Delta w_{b}^{\mathcal{L}} - \divergence (2w_{b}^{\mathcal{L}} \otimes_{a} ( \mathcal{L}_{\lambda_{t}} X_{u}) - 2w_{u}^{\mathcal{L}} \otimes_{a} (\mathcal{L}_{\lambda_{t}} X_{b} ) \rangle(t),\label{est 106a}\\
\RomanII_{2} \triangleq& -2 \langle (-\Delta)^{\epsilon} w_{u}^{\mathcal{L}}, \divergence ( 2 ( \mathcal{H}_{\lambda_{t}} X_{u}) \otimes_{s} w_{u}^{\mathcal{L}} - 2 (\mathcal{H}_{\lambda_{t}} X_{u}) \circlesign{\succ}_{s} w_{u}^{\mathcal{L}} \nonumber \\
& \hspace{14mm} - 2 ( \mathcal{H}_{\lambda_{t}} X_{b}) \otimes_{s} w_{b}^{\mathcal{L}} + 2(\mathcal{H}_{\lambda_{t}} X_{b}) \circlesign{\succ}_{s} w_{b}^{\mathcal{L}} ) \rangle(t) \nonumber \\
& - 2 \langle (-\Delta)^{\epsilon}  w_{b}^{\mathcal{L}}, \divergence (2w_{b}^{\mathcal{L}} \otimes_{a} (\mathcal{H}_{\lambda_{t}} X_{u}) - 2 w_{b}^{\mathcal{L}} \circlesign{\prec}_{a} (\mathcal{H}_{\lambda_{t}} X_{u}) \nonumber  \\
& \hspace{14mm} -2 w_{u}^{\mathcal{L}} \otimes_{a} (\mathcal{H}_{\lambda_{t}} X_{b}) + 2 w_{u}^{\mathcal{L}} \circlesign{\prec}_{a} (\mathcal{H}_{\lambda_{t}} X_{b}) \rangle(t),  \label{est 106b}\\
\RomanII_{3} \triangleq& - 2 \langle (-\Delta)^{\epsilon} w_{u}^{\mathcal{L}}, \divergence (2X_{u} \otimes_{s} w_{u}^{\mathcal{H}} - 2 (\mathcal{H}_{\lambda_{t}} X_{u}) \circlesign{\succ}_{s} w_{u}^{\mathcal{H}}  \nonumber \\
& \hspace{14mm} - 2 X_{b} \otimes_{s} w_{b}^{\mathcal{H}} + 2(\mathcal{H}_{\lambda_{t}} X_{b}) \circlesign{\succ}_{s} w_{b}^{\mathcal{H}}) \rangle(t) \nonumber \\
&-2 \langle (-\Delta)^{\epsilon} w_{b}^{\mathcal{L}}, \divergence (2w_{b}^{\mathcal{H}} \otimes_{a} X_{u} - 2 w_{b}^{\mathcal{H}} \circlesign{\prec}_{a} (\mathcal{H}_{\lambda_{t}} X_{u})  \nonumber \\
& \hspace{14mm} - 2 w_{u}^{\mathcal{H}} \otimes_{a} X_{b} + 2w_{u}^{\mathcal{H}} \circlesign{\prec}_{a} (\mathcal{H}_{\lambda_{t}} X_{b} ) \rangle(t), \label{est 106c}\\
\RomanII_{4} \triangleq& -2 \langle (-\Delta)^{\epsilon} w_{u}^{\mathcal{L}}, \divergence (w_{u}^{\otimes 2} + 2Y_{u} \otimes_{s} w_{u} - w_{b}^{\otimes 2} - 2Y_{b} \otimes_{s} w_{b} \nonumber  \\
& \hspace{14mm} - C^{\circlesign{\prec}_{s}} (w_{u}, Q_{u}^{\mathcal{H}}) + Y_{u}^{\otimes 2} + C^{\circlesign{\prec}_{s}} (w_{b}, Q_{b}^{\mathcal{H}}) - Y_{b}^{\otimes 2} ) \rangle(t) \nonumber \\
& -2 \langle  (-\Delta)^{\epsilon}  w_{b}^{\mathcal{L}}, \divergence (w_{b} \otimes w_{u} + 2w_{b} \otimes_{a} Y_{u} - w_{u} \otimes w_{b} - 2w_{u} \otimes_{a} Y_{b} \nonumber\\
& \hspace{14mm} - C^{\circlesign{\prec}_{a}} (w_{b}, Q_{u}^{\mathcal{H}}) + Y_{b} \otimes Y_{u} + C^{\circlesign{\prec}_{a}} (w_{u}, Q_{b}^{\mathcal{H}})- Y_{u} \otimes Y_{b}) \rangle(t). \label{est 106d}
\end{align}
\end{subequations} 
First, within $\RomanII_{1}$, e.g. we estimate using \eqref{est 40c}, \eqref{est 40b}, and \eqref{est 40e}
\begin{align*}
&-4 \langle (-\Delta)^{\epsilon} w_{b}^{\mathcal{L}}, \divergence (w_{b}^{\mathcal{L}} \otimes_{a} (\mathcal{L}_{\lambda_{t}} X_{u} ) ) \rangle(t)    \\
\lesssim&  \lVert w_{b}^{\mathcal{L}} (t)\rVert_{H^{1}} ( \lVert w_{b}^{\mathcal{L}} \circlesign{\prec}_{a} ( \mathcal{L}_{\lambda_{t}} X_{u}) \rVert_{H^{2\epsilon}} + \lVert w_{b}^{\mathcal{L}} \circlesign{\succ}_{a} (\mathcal{L}_{\lambda_{t}} X_{u} ) \rVert_{H^{2\epsilon}} + \lVert w_{b}^{\mathcal{L}} \circlesign{\circ}_{a} ( \mathcal{L}_{\lambda_{t}} X_{u}) \rVert_{H^{2\epsilon}})(t) \nonumber \\
\lesssim & \lVert w_{b}^{\mathcal{L}} (t)\rVert_{H^{1}} ( \lVert w_{b}^{\mathcal{L}} \rVert_{H^{-\frac{1}{3} + \kappa + 2 \epsilon}} \lVert \mathcal{L}_{\lambda_{t}} X_{u} \rVert_{\mathcal{C}^{\frac{1}{3} - \kappa}}   +\lVert w_{b}^{\mathcal{L}} \rVert_{H^{2\epsilon}}  \lVert \mathcal{L}_{\lambda_{t}} X_{u} \rVert_{C^{\frac{1}{3} - \kappa}})(t)  \nonumber \\
\leq& C(M, N_{t}^{\kappa})\lVert w_{b}^{\mathcal{L}} (t)\rVert_{H^{1}} \lVert w_{b}^{\mathcal{L}} (t)\rVert_{H^{2\epsilon}} \leq \frac{\nu }{64} \lVert w_{b}^{\mathcal{L}} (t)\rVert_{\dot{H}^{1+ \epsilon}}^{2} + C(M, N_{t}^{\kappa}).  \nonumber 
\end{align*} 
Similar computations on the rest of the terms of $\RomanII_{1}$ in \eqref{est 106a} lead to altogether  
\begin{align*}
& -4 \langle (-\Delta)^{\epsilon} w_{u}^{\mathcal{L}},  \divergence (  (\mathcal{L}_{\lambda_{t}} X_{u}) \otimes_{s} w_{u}^{\mathcal{L}} -  ( \mathcal{L}_{\lambda_{t}} X_{b}) \otimes_{s} w_{b}^{\mathcal{L}} ) \rangle(t)  \\
&- 4 \langle (-\Delta)^{\epsilon}  w_{b}^{\mathcal{L}},   \divergence (w_{b}^{\mathcal{L}} \otimes_{a} ( \mathcal{L}_{\lambda_{t}} X_{u}) - w_{u}^{\mathcal{L}} \otimes_{a} (\mathcal{L}_{\lambda_{t}} X_{b} ) \rangle(t) \leq \frac{\nu }{16}  \lVert (w_{u}^{\mathcal{L}}, w_{b}^{\mathcal{L}})(t) \rVert_{\dot{H}^{1+ \epsilon}}^{2} + C (M, N_{t}^{\kappa}) \nonumber 
\end{align*}
so that applying this to \eqref{est 106a} gives us 
\begin{equation}\label{est 133} 
\RomanII_{1} \leq - \frac{31\nu }{16} \lVert (w_{u}^{\mathcal{L}}, w_{b}^{\mathcal{L}}) (t)\rVert_{\dot{H}^{1+ \epsilon}}^{2} + C(M, N_{t}^{\kappa}). 
\end{equation} 
Next, for $\RomanII_{2}$ in \eqref{est 106b}, for convenience we rewrite it as 
\begin{align}
\RomanII_{2} =& -4 \langle (-\Delta)^{\epsilon} w_{u}^{\mathcal{L}}, \divergence ( ( \mathcal{H}_{\lambda_{t}} X_{u}) \circlesign{\prec}_{s} w_{u}^{\mathcal{L}} + (\mathcal{H}_{\lambda_{t}} X_{u}) \circlesign{\circ}_{s}  w_{u}^{\mathcal{L}}  \nonumber \\
&\hspace{25mm}  - (\mathcal{H}_{\lambda_{t}} X_{b}) \circlesign{\prec}_{s} w_{b}^{\mathcal{L}} - (\mathcal{H}_{\lambda_{t}} X_{b}) \circlesign{\circ}_{s} w_{b}^{\mathcal{L}} \rangle(t) \nonumber  \\
& - 4 \langle (-\Delta)^{\epsilon} w_{b}^{\mathcal{L}}, \divergence ( w_{b}^{\mathcal{L}} \circlesign{\succ}_{a} (\mathcal{H}_{\lambda_{t}} X_{u}) + w_{b}^{\mathcal{L}} \circlesign{\circ}_{a} (\mathcal{H}_{\lambda_{t}} X_{u})  \nonumber \\
& \hspace{25mm} - w_{u}^{\mathcal{L}} \circlesign{\succ}_{a} (\mathcal{H}_{\lambda_{t}} X_{b}) - w_{u}^{\mathcal{L}} \circlesign{\circ}_{a} (\mathcal{H}_{\lambda_{t}} X_{b}) \rangle(t), \label{est 245}
\end{align}
and e.g. estimate 
\begin{align*}
& -4 \langle (-\Delta)^{\epsilon} w_{b}^{\mathcal{L}}, \divergence (w_{b}^{\mathcal{L}} \circlesign{\succ}_{a} (\mathcal{H}_{\lambda_{t}} X_{u}) + w_{b}^{\mathcal{L}} \circlesign{\circ}_{a} (\mathcal{H}_{\lambda_{t}} X_{u}) \rangle(t) \\
\overset{\eqref{est 40}}{\lesssim}& \lVert w_{b}^{\mathcal{L}} (t)\rVert_{H^{\frac{1}{2} + 3 \kappa + \epsilon}}  \lVert w_{b}^{\mathcal{L}}(t) \rVert_{H^{\frac{1}{2} - 2 \kappa + \epsilon}} \lVert \mathcal{H}_{\lambda_{t}} X_{u} (t)\rVert_{\mathcal{C}^{-\kappa}} \nonumber \\
\overset{\eqref{est 39} }{\lesssim}& \lVert w_{b}^{\mathcal{L}} (t)\rVert_{L^{2}}^{\frac{1- 6 \kappa}{2(1+ \epsilon)}}\lVert w_{b}^{\mathcal{L}}  (t)\rVert_{\dot{H}^{1+ \epsilon}}^{\frac{1+ 6 \kappa+ 2 \epsilon}{2(1+ \epsilon)}} \lVert w_{b}^{\mathcal{L}} (t)\rVert_{L^{2}}^{\frac{1+ 4 \kappa}{2(1+ \epsilon)}} \lVert w_{u}^{\mathcal{L}} (t)\rVert_{\dot{H}^{1+ \epsilon}}^{\frac{1- 4 \kappa + 2 \epsilon}{2(1+ \epsilon)}} N_{t}^{\kappa}   \leq \frac{\nu }{64} \lVert w_{b}^{\mathcal{L}} (t)\rVert_{\dot{H}^{1+ \epsilon}}^{2}+  C(M, N_{t}^{\kappa}).  \nonumber 
\end{align*}
Analogous computations on similar terms in \eqref{est 245} lead to altogether 
\begin{equation}\label{est 134} 
\RomanII_{2} \leq \frac{\nu }{16} \lVert (w_{u}^{\mathcal{L}}, w_{b}^{\mathcal{L}}) (t)\rVert_{\dot{H}^{1+ \epsilon}}^{2} + C(M, N_{t}^{\kappa}). 
\end{equation} 
Next, concerning $\RomanII_{3}$ from \eqref{est 106c}, we use \eqref{est 68} to rewrite for convenience 
\begin{align}
\RomanII_{3} =& -4 [ \langle (-\Delta)^{\epsilon} w_{u}^{\mathcal{L}}, \divergence ( \mathcal{L}_{\lambda_{t}} X_{u} \otimes_{s} w_{u}^{\mathcal{H}} - \mathcal{L}_{\lambda_{t}} X_{b} \otimes_{s} w_{b}^{\mathcal{H}} ) \rangle \nonumber \\
& \hspace{5mm}+ \langle (-\Delta)^{\epsilon} w_{b}^{\mathcal{L}}, \divergence (w_{b}^{\mathcal{H}} \otimes_{a} \mathcal{L}_{\lambda_{t}} X_{u} - w_{u}^{\mathcal{H}} \otimes_{a} \mathcal{L}_{\lambda_{t}} X_{b} ) \rangle ](t) \nonumber \\
& -4 [ \langle (-\Delta)^{\epsilon} w_{u}^{\mathcal{L}}, \divergence ( \mathcal{H}_{\lambda_{t}} X_{u} \circlesign{\prec}_{s} w_{u}^{\mathcal{H}} + \mathcal{H}_{\lambda_{t}} X_{u} \circlesign{\circ}_{s} w_{u}^{\mathcal{H}} \nonumber \\
& \hspace{25mm} - \mathcal{H}_{\lambda_{t}} X_{b} \circlesign{\prec}_{s} w_{b}^{\mathcal{H}} - \mathcal{H}_{\lambda_{t}} X_{b} \circlesign{\circ}_{s} w_{b}^{\mathcal{H}} ) \rangle \nonumber \\
& \hspace{5mm} + \langle (-\Delta)^{\epsilon} w_{b}^{\mathcal{L}}, \divergence ( w_{b}^{\mathcal{H}} \circlesign{\succ}_{a} \mathcal{H}_{\lambda_{t}} X_{u} + w_{b}^{\mathcal{H}} \circlesign{\circ}_{a} \mathcal{H}_{\lambda_{t}} X_{u}   \nonumber \\
& \hspace{25mm} - w_{u}^{\mathcal{H}} \circlesign{\succ}_{a} \mathcal{H}_{\lambda_{t}} X_{b} - w_{u}^{\mathcal{H}} \circlesign{\circ}_{a} \mathcal{H}_{\lambda_{t}} X_{b}) \rangle ](t). \label{est 247} 
\end{align} 
As an example, we can estimate for any $\eta \in (\frac{2}{3} + \kappa, 1)$ 
\begin{subequations}\label{est 111} 
\begin{align}
\lVert w_{u}^{\mathcal{H}} \otimes_{a} \mathcal{L}_{\lambda_{t}} X_{b} (t)\rVert_{H^{1-\eta}} &\lesssim \lambda_{t}^{\frac{1}{3}} (N_{t}^{\kappa})^{2}, \\
\lVert w_{b}^{\mathcal{H}} \circlesign{\succ}_{a} \mathcal{H}_{\lambda_{t}} X_{u} (t)\rVert_{H^{1-\eta}} + \lVert w_{b}^{\mathcal{H}} \circlesign{\circ}_{a} \mathcal{H}_{\lambda_{t}} X_{u} (t)\rVert_{H^{1-\eta}} &\lesssim (N_{t}^{\kappa})^{2}, 
\end{align} 
\end{subequations} 
which can be derived similarly to respectively \eqref{est 108} and \eqref{est 246}, and this leads to for any $\eta \in (\frac{2}{3} + \kappa, 1 - \kappa)$  
\begin{align*}
& 4 \langle ( -\Delta)^{\epsilon} w_{u}^{\mathcal{L}}, \divergence ( w_{u}^{\mathcal{H}} \otimes_{a} \mathcal{L}_{\lambda_{t}} X_{b}) \rangle(t) + 4 \langle (-\Delta)^{\epsilon} w_{b}^{\mathcal{L}}, \divergence ( w_{b}^{\mathcal{H}} \circlesign{\succ}_{a} \mathcal{H}_{\lambda_{t}} X_{u} + w_{b}^{\mathcal{H}} \circlesign{\circ}_{a} \mathcal{H}_{\lambda_{t}} X_{u}) \rangle(t) \nonumber \\
& \hspace{3mm} \overset{\eqref{est 111}}{\lesssim} \lVert w_{b}^{\mathcal{L}} (t)\rVert_{H^{\eta + 2 \epsilon}} [ \lambda_{t}^{\frac{1}{3}} (N_{t}^{\kappa})^{2} + (N_{t}^{\kappa})^{2} ] \overset{\eqref{est 36} \eqref{est 67} \eqref{est 107}}{\leq}  \frac{\nu }{64} \lVert w_{b}^{\mathcal{L}} (t)\rVert_{\dot{H}^{1+ \epsilon}}^{2} + C(M, N_{t}^{\kappa}).
\end{align*} 
Analogous computations on similar terms in \eqref{est 247} lead to 
\begin{equation}\label{est 135} 
\RomanII_{3} \leq \frac{\nu }{16} \lVert (w_{u}^{\mathcal{L}}, w_{b}^{\mathcal{L}}) (t)\rVert_{\dot{H}^{1+ \epsilon}}^{2} + C(M, N_{t}^{\kappa}). 
\end{equation} 
Concerning $\RomanII_{4}$ from \eqref{est 106d}, we work on the following four terms which we rewrite using \eqref{est 68}, \eqref{est 18}, and \eqref{est 19} for convenience:
\begin{align}
&-2 \langle (-\Delta)^{\epsilon} w_{u}^{\mathcal{L}}, \divergence (w_{u}^{\otimes 2} - w_{b}^{\otimes 2}) \rangle - 2 \langle ( -\Delta)^{\epsilon} w_{b}^{\mathcal{L}}, \divergence (w_{b} \otimes w_{u} - w_{u} \otimes w_{b}) \rangle \nonumber \\
=&  -2 \langle (-\Delta)^{\epsilon} w_{u}^{\mathcal{L}}, \divergence ( w_{u}^{\mathcal{L}} \otimes w_{u}^{\mathcal{L}}  + 2 w_{u}^{\mathcal{L}} \otimes_{s} w_{u}^{\mathcal{H}} + (w_{u}^{\mathcal{H}})^{\otimes 2}   \nonumber \\
& \hspace{25mm} - w_{b}^{\mathcal{L}} \otimes w_{b}^{\mathcal{L}}  - 2 w_{b}^{\mathcal{L}} \otimes_{s} w_{b}^{\mathcal{H}} - (w_{b}^{\mathcal{H}})^{\otimes 2} \rangle  \nonumber \\
& - 2 \langle (-\Delta)^{\epsilon} w_{b}^{\mathcal{L}}, \divergence ( 2w_{b}^{\mathcal{L}} \otimes_{a} w_{u}^{\mathcal{L}} + 2w_{b}^{\mathcal{L}} \otimes_{a} w_{u}^{\mathcal{H}} - 2 w_{u}^{\mathcal{L}} \otimes_{a} w_{b}^{\mathcal{H}} + 2w_{b}^{\mathcal{H}} \otimes_{a} w_{u}^{\mathcal{H}}).  \label{est 112} 
\end{align}
First, we work on the three products of the lower-order terms within \eqref{est 112}: 
  \begin{align}
& -2 \langle ( -\Delta)^{\epsilon} w_{u}^{\mathcal{L}}, \divergence (w_{u}^{\mathcal{L}} \otimes w_{u}^{\mathcal{L}}  - w_{b}^{\mathcal{L}} \otimes w_{b}^{\mathcal{L}}) \rangle - 2 \langle (-\Delta)^{\epsilon} w_{b}^{\mathcal{L}}, \divergence (2w_{b}^{\mathcal{L}} \otimes_{a} w_{u}^{\mathcal{L}} )\rangle   \nonumber \\
=& -2 [ \int_{\mathbb{T}^{2}} (-\Delta)^{\frac{\epsilon}{2}} [ (w_{u}^{\mathcal{L}} \cdot \nabla) w_{u}^{\mathcal{L}} ] \cdot (-\Delta)^{\frac{\epsilon}{2}} w_{u}^{\mathcal{L}} dx  \nonumber \\
&+ \int_{\mathbb{T}^{2}} (-\Delta)^{\frac{\epsilon}{2}} [ (w_{u}^{\mathcal{L}} \cdot \nabla) w_{b}^{\mathcal{L}} ] \cdot (-\Delta)^{\frac{\epsilon}{2}} w_{b}^{\mathcal{L}} dx \nonumber  \\
& - \int_{\mathbb{T}^{2}} (-\Delta)^{\frac{\epsilon}{2}} [ (w_{b}^{\mathcal{L}} \cdot \nabla) w_{b}^{\mathcal{L}} ] \cdot (-\Delta)^{\frac{\epsilon}{2}} w_{u}^{\mathcal{L}} + (-\Delta)^{\frac{\epsilon}{2}} [ (w_{b}^{\mathcal{L}} \cdot \nabla) w_{u}^{\mathcal{L}}] \cdot (-\Delta)^{\frac{\epsilon}{2}} w_{b}^{\mathcal{L}} dx].   \label{est 115} 
\end{align}
We take advantage of 
\begin{align*}
\int_{\mathbb{T}^{2}} (w_{u}^{\mathcal{L}} \cdot \nabla) (-\Delta)^{\frac{\epsilon}{2}} w_{u}^{\mathcal{L}} \cdot (-\Delta)^{\frac{\epsilon}{2}} w_{u}^{\mathcal{L}} dx = 0,  \hspace{3mm}  \int_{\mathbb{T}^{2}} (w_{u}^{\mathcal{L}} \cdot \nabla) (-\Delta)^{\frac{\epsilon}{2}} w_{b}^{\mathcal{L}} \cdot (-\Delta)^{\frac{\epsilon}{2}} w_{b}^{\mathcal{L}} dx = 0, 
\end{align*}
and while 
\begin{align*}
\int_{\mathbb{T}^{2}} (w_{b}^{\mathcal{L}} \cdot \nabla) (-\Delta)^{\frac{\epsilon}{2}} w_{b}^{\mathcal{L}} \cdot (-\Delta)^{\frac{\epsilon}{2}} w_{u}^{\mathcal{L}} dx \neq 0,  \hspace{3mm}   \int_{\mathbb{T}^{2}} (w_{b}^{\mathcal{L}} \cdot \nabla) (-\Delta)^{\frac{\epsilon}{2}} w_{u}^{\mathcal{L}} \cdot (-\Delta)^{\frac{\epsilon}{2}} w_{b}^{\mathcal{L}} dx \neq 0, 
\end{align*}
it turns out that 
\begin{equation}\label{est 238}
\int_{\mathbb{T}^{2}} (w_{b}^{\mathcal{L}} \cdot \nabla) (-\Delta)^{\frac{\epsilon}{2}} w_{b}^{\mathcal{L}} \cdot (-\Delta)^{\frac{\epsilon}{2}} w_{u}^{\mathcal{L}} dx + \int_{\mathbb{T}^{2}} (w_{b}^{\mathcal{L}} \cdot \nabla) (-\Delta)^{\frac{\epsilon}{2}} w_{u}^{\mathcal{L}} \cdot (-\Delta)^{\frac{\epsilon}{2}} w_{b}^{\mathcal{L}} dx  = 0. 
\end{equation} 
Therefore, we may make use of the classical commutator estimates and embeddings of $H^{\epsilon} (\mathbb{T}^{2}) \hookrightarrow L^{\frac{2}{1-\epsilon}}(\mathbb{T}^{2})$, $H^{\frac{1}{2}}(\mathbb{T}^{2}) \hookrightarrow L^{4}(\mathbb{T}^{2})$, and $H^{\frac{1}{2} - \epsilon}(\mathbb{T}^{2}) \hookrightarrow L^{\frac{4}{1+ 2 \epsilon}}(\mathbb{T}^{2})$  to continue from \eqref{est 115} by 
\begin{align}
& -2 \langle ( -\Delta)^{\epsilon} w_{u}^{\mathcal{L}}, \divergence (w_{u}^{\mathcal{L}} \otimes w_{u}^{\mathcal{L}}  - w_{b}^{\mathcal{L}} \otimes w_{b}^{\mathcal{L}}) \rangle(t) - 2 \langle (-\Delta)^{\epsilon} w_{b}^{\mathcal{L}}, \divergence (2w_{b}^{\mathcal{L}} \otimes_{a} w_{u}^{\mathcal{L}} )\rangle(t)    \nonumber \\
=& -2 [ \int_{\mathbb{T}^{2}} [ (-\Delta)^{\frac{\epsilon}{2}}, (w_{u}^{\mathcal{L}} \cdot \nabla) ] w_{u}^{\mathcal{L}} \cdot (-\Delta)^{\frac{\epsilon}{2}} w_{u}^{\mathcal{L}} dx + \int_{\mathbb{T}^{2}} [(-\Delta)^{\frac{\epsilon}{2}}, (w_{u}^{\mathcal{L}} \cdot \nabla) ] w_{b}^{\mathcal{L}} \cdot (-\Delta)^{\frac{\epsilon}{2}} w_{b}^{\mathcal{L}} dx \nonumber \\
& - \int_{\mathbb{T}^{2}} [(-\Delta)^{\frac{\epsilon}{2}}, (w_{b}^{\mathcal{L}} \cdot \nabla) ] w_{b}^{\mathcal{L}} \cdot (-\Delta)^{\frac{\epsilon}{2}} w_{u}^{\mathcal{L}} dx - \int_{\mathbb{T}^{2}} [(-\Delta)^{\frac{\epsilon}{2}}, (w_{b}^{\mathcal{L}} \cdot \nabla) ] w_{u}^{\mathcal{L}} \cdot (-\Delta)^{\frac{\epsilon}{2}} w_{b}^{\mathcal{L}} dx](t)  \nonumber \\
\lesssim& [\lVert [ (-\Delta)^{\frac{\epsilon}{2}}, w_{u}^{\mathcal{L}} \cdot \nabla ] w_{u}^{\mathcal{L}} \rVert_{L^{\frac{4}{3}}} \lVert (-\Delta)^{\frac{\epsilon}{2}} w_{u}^{\mathcal{L}} \rVert_{L^{4}} + \lVert [ (-\Delta)^{\frac{\epsilon}{2}}, w_{u}^{\mathcal{L}} \cdot \nabla]  w_{b}^{\mathcal{L}} \rVert_{L^{\frac{4}{3}}} \lVert (-\Delta)^{\frac{\epsilon}{2}}w_{b}^{\mathcal{L}} \rVert_{L^{4}} \nonumber \\
&+\lVert [ (-\Delta)^{\frac{\epsilon}{2}}, w_{b}^{\mathcal{L}} \cdot \nabla ] w_{b}^{\mathcal{L}} \rVert_{L^{\frac{4}{3}}} \lVert (-\Delta)^{\frac{\epsilon}{2}} w_{u}^{\mathcal{L}} \rVert_{L^{4}} + \lVert [ (-\Delta)^{\frac{\epsilon}{2}}, w_{b}^{\mathcal{L}} \cdot \nabla] w_{u}^{\mathcal{L}} \rVert_{L^{\frac{4}{3}}} \lVert (-\Delta)^{\frac{\epsilon}{2}} w_{b}^{\mathcal{L}} \rVert_{L^{4}}](t) \nonumber \\
\leq& \frac{\nu }{16} \lVert (w_{u}^{\mathcal{L}}, w_{b}^{\mathcal{L}}) (t)\rVert_{\dot{H}^{1+ \epsilon}}^{2} + C(M) \lVert (w_{u}^{\mathcal{L}}, w_{b}^{\mathcal{L}}) (t)\rVert_{H^{1}}^{2} \lVert (w_{u}^{\mathcal{L}}, w_{b}^{\mathcal{L}}) (t)\rVert_{H^{\epsilon}}^{2} \label{est 114} 
\end{align}
where $[A,B] \triangleq AB - BA$. Next, concerning the products of higher order terms within \eqref{est 112}, e.g. 
\begin{align*}
&-2 \langle (-\Delta)^{\epsilon} w_{b}^{\mathcal{L}}, \divergence (w_{b}^{\mathcal{H}} \otimes_{a} w_{u}^{\mathcal{H}}) \rangle(t) 
\overset{ \eqref{est 113}}{\lesssim} \lVert w_{b}^{\mathcal{L}} (t)\rVert_{H^{\frac{3}{4} + 3 \kappa + 2 \epsilon}} (N_{t}^{\kappa})^{2}   \\
& \hspace{20mm} \lesssim \lVert w_{b}^{\mathcal{L}} (t)\rVert_{L^{2}}^{\frac{1- 12 \kappa - 4 \epsilon}{4(1+ \epsilon)}} \lVert w_{b}^{\mathcal{L}} (t)\rVert_{H^{1+ \epsilon}}^{\frac{3+ 12 \kappa + 8 \epsilon}{4(1+ \epsilon)}} C(N_{t}^{\kappa}) \leq \frac{\nu }{64} \lVert w_{b}^{\mathcal{L}} (t)\rVert_{\dot{H}^{1+ \epsilon}}^{2} + C(M, N_{t}^{\kappa}). \nonumber 
\end{align*} 
Analogous computations on similar terms lead us to altogether 
\begin{align}
& -2 \langle (-\Delta)^{\epsilon} w_{u}^{\mathcal{L}}, \divergence ((w_{u}^{\mathcal{H}})^{\otimes 2} - (w_{b}^{\mathcal{H}})^{\otimes 2} ) \rangle(t) \label{est 117}\\
& -2 \langle (-\Delta)^{\epsilon} w_{b}^{\mathcal{L}}, \divergence (2w_{b}^{\mathcal{H}} \otimes_{a} w_{u}^{\mathcal{H}}) \rangle(t)  \leq  \frac{\nu }{16} \lVert (w_{u}^{\mathcal{L}}, w_{b}^{\mathcal{L}}) (t)\rVert_{\dot{H}^{1+ \epsilon}}^{2} + C(M, N_{t}^{\kappa}). \nonumber  
\end{align}
Finally, concerning products of high and low order terms within \eqref{est 112}, e.g. we estimate as follows 
\begin{align*}
-4\langle ( -\Delta)^{\epsilon} w_{b}^{\mathcal{L}}, \divergence (w_{b}^{\mathcal{L}}& \otimes_{a} w_{u}^{\mathcal{H}}) \rangle(t) \overset{\eqref{est 74}}{\lesssim} \lVert w_{b}^{\mathcal{L}} (t)\rVert_{H^{\frac{5}{6} + 2 \kappa + 2 \epsilon}}   \lVert w_{b}^{\mathcal{L}} (t)\rVert_{H^{\frac{2}{3} - \kappa}} \lVert w_{u}^{\mathcal{H}} (t)\rVert_{H^{\frac{2}{3}- 2 \kappa}}   \\
\overset{\eqref{est 67}}{\lesssim}&  \lVert w_{b}^{\mathcal{L}} (t)\rVert_{L^{2}}^{\frac{1- 2 \kappa}{2(1+ \epsilon)}}  \lVert w_{b}^{\mathcal{L}} (t)\rVert_{\dot{H}^{1+ \epsilon}}^{\frac{3 + 2 \kappa + 4 \epsilon}{2(1+ \epsilon)}} N_{t}^{\kappa} \leq \frac{\nu }{64} \lVert w_{b}^{\mathcal{L}} (t)\rVert_{\dot{H}^{1+ \epsilon}}^{2} + C(M, N_{T}^{\kappa}).  \nonumber 
\end{align*} 
Analogous computations for similar terms lead to altogether 
\begin{align}
& - 4 \langle (-\Delta)^{\epsilon}w_{u}^{\mathcal{L}}, \divergence (w_{u}^{\mathcal{L}} \otimes_{s} w_{u}^{\mathcal{H}} - w_{b}^{\mathcal{L}} \otimes_{s} w_{b}^{\mathcal{H}} ) \rangle (t) \label{est 118} \\
& - 4 \langle (-\Delta)^{\epsilon} w_{b}^{\mathcal{L}}, \divergence (w_{b}^{\mathcal{L}} \otimes_{a} w_{u}^{\mathcal{H}} - w_{u}^{\mathcal{L}} \otimes_{a} w_{b}^{\mathcal{H}}) \rangle(t) \leq \frac{\nu }{16} \lVert (w_{u}^{\mathcal{L}}, w_{b}^{\mathcal{L}}) (t)\rVert_{\dot{H}^{1+ \epsilon}}^{2} + C(M, N_{T}^{\kappa}).  \nonumber 
\end{align}
Applying \eqref{est 114}, \eqref{est 117}, and \eqref{est 118} to \eqref{est 112} leads us to 
\begin{align}
&-2 \langle (-\Delta)^{\epsilon} w_{u}^{\mathcal{L}}, \divergence (w_{u}^{\otimes 2} - w_{b}^{\otimes 2}) \rangle(t) - 2 \langle ( -\Delta)^{\epsilon} w_{b}^{\mathcal{L}}, \divergence (w_{b} \otimes w_{u} - w_{u} \otimes w_{b}) \rangle(t) \nonumber \\
\leq& \frac{3\nu }{16} \lVert (w_{u}^{\mathcal{L}}, w_{b}^{\mathcal{L}}) (t)\rVert_{\dot{H}^{1+ \epsilon}}^{2} + C(M) \lVert (w_{u}^{\mathcal{L}}, w_{b}^{\mathcal{L}}) (t)\rVert_{H^{1}}^{2} \lVert (w_{u}^{\mathcal{L}}, w_{b}^{\mathcal{L}}) (t)\rVert_{H^{\epsilon}}^{2} + C(M, N_{t}^{\kappa}).   \label{est 130} 
\end{align}
We rewrite the rest of the terms in $\RomanII_{4}$ of \eqref{est 106d} as 
\begin{align}
& -2 \langle ( -\Delta)^{\epsilon} w_{u}^{\mathcal{L}},  \divergence (2 Y_{u} \otimes_{s} w_{u} - 2Y_{b} \otimes_{s} w_{b} \nonumber \\
& \hspace{15mm} - C^{\circlesign{\prec}_{s}} (w_{u}, Q_{u}^{\mathcal{H}}) + Y_{u}^{\otimes 2} + C^{\circlesign{\prec}_{s}} (w_{b}, Q_{b}^{\mathcal{H}}) - Y_{b}^{\otimes 2} ) \rangle(t) \nonumber \\
& - 2 \langle (-\Delta)^{\epsilon} w_{b}^{\mathcal{L}}, \divergence (2w_{b} \otimes_{a} Y_{u} - 2 w_{u} \otimes_{a} Y_{b} \nonumber \\
 & \hspace{15mm} - C^{\circlesign{\prec}_{a}} (w_{b}, Q_{u}^{\mathcal{H}}) + Y_{b} \otimes Y_{u} + C^{\circlesign{\prec}_{a}} (w_{u}, Q_{b}^{\mathcal{H}}) - Y_{u} \otimes Y_{b}) \rangle(t) = \sum_{k=1}^{2} \RomanII_{4,k}, \label{est 119} 
\end{align}
where 
\begin{subequations}\label{est 248}
\begin{align}
\RomanII_{4,1} \triangleq&  - 2 \langle (-\Delta)^{\epsilon} w_{u}^{\mathcal{L}}, \divergence (2Y_{u} \otimes_{s} w_{u} + Y_{u}^{\otimes 2} - 2Y_{b} \otimes_{s} w_{b} - Y_{b}^{\otimes 2} ) \rangle(t) \nonumber \\
 & - 2 \langle (-\Delta)^{\epsilon} w_{b}^{\mathcal{L}}, \divergence (2w_{b} \otimes_{a} Y_{u} + Y_{b} \otimes Y_{u} - 2w_{u} \otimes_{a} Y_{b} - Y_{u} \otimes Y_{b}) \rangle(t), \label{est 248a}\\
 \RomanII_{4,2} \triangleq&  2 \langle (-\Delta)^{\epsilon} w_{u}^{\mathcal{L}}, \divergence (C^{\circlesign{\prec}_{s}} (w_{u}, Q_{u}^{\mathcal{H}}) - C^{\circlesign{\prec}_{s}} (w_{b}, Q_{b}^{\mathcal{H}}) \rangle(t) \nonumber \\
 &+ 2 \langle (-\Delta)^{\epsilon} w_{b}^{\mathcal{L}}, \divergence ( C^{\circlesign{\prec}_{a}} (w_{b}, Q_{u}^{\mathcal{H}}) - C^{\circlesign{\prec}_{a}} (w_{u}, Q_{b}^{\mathcal{H}}) \rangle(t). \label{est 248b}
\end{align}
\end{subequations} 
Among the non-commutator terms $\RomanII_{4,1}$ in \eqref{est 248a}, we estimate as an example 
\begin{align*}
& -2 \langle (-\Delta)^{\epsilon} w_{b}^{\mathcal{L}}, \divergence (2w_{b} \otimes_{a} Y_{u} + Y_{b} \otimes Y_{u} ) \rangle(t)  \nonumber \\
\lesssim& \lVert w_{b}^{\mathcal{L}} (t)\rVert_{H^{1- \frac{3\kappa}{2} + 2 \epsilon}}   [ \lVert w_{b} \circlesign{\prec}_{a} Y_{u} \rVert_{H^{\frac{3\kappa}{2}}} + \lVert w_{b} \circlesign{\succ}_{a} Y_{u} \rVert_{H^{\frac{3\kappa}{2}}} + \lVert w_{b} \circlesign{\circ}_{a} Y_{u} \rVert_{H^{\frac{3\kappa}{2}}}  \nonumber\\
& \hspace{15mm} + \lVert Y_{b} \circlesign{\prec}_{a} Y_{u} \rVert_{H^{\frac{3\kappa}{2}}} + \lVert Y_{b} \circlesign{\succ}_{a} Y_{u} \rVert_{H^{\frac{3\kappa}{2}}} + \lVert Y_{b} \circlesign{\circ}_{a} Y_{u} \rVert_{H^{\frac{3\kappa}{2}}} ](t)  \nonumber\\
\overset{\eqref{est 40}}{\lesssim}&  \lVert w_{b}^{\mathcal{L}} (t)\rVert_{H^{1- \frac{3\kappa}{2} + 2 \epsilon}}  [\lVert w_{b}^{\mathcal{L}} \rVert_{H^{2\kappa}} \lVert Y_{u} \rVert_{\mathcal{C}^{2\kappa}} +  \lVert Y_{b} \rVert_{\mathcal{C}^{2\kappa}} \lVert Y_{u} \rVert_{\mathcal{C}^{2\kappa}}](t) \nonumber \\
\overset{\eqref{est 39} \eqref{est 67} \eqref{est 107}}{\leq}& \frac{\nu }{64} \lVert w_{b}^{\mathcal{H}} (t)\rVert_{\dot{H}^{1+ \epsilon}}^{2} + C(M, N_{t}^{\kappa}). 
\end{align*}
Analogous computations on similar non-commutator terms in \eqref{est 248a} lead to   
\begin{equation*}
\RomanII_{4,1} \leq \frac{\nu }{16} \lVert (w_{u}^{\mathcal{L}}, w_{b}^{\mathcal{L}}) (t)\rVert_{\dot{H}^{1+ \epsilon}}^{2} + C(M, N_{t}^{\kappa}).  \label{est 131}
\end{equation*}
Among rest of the terms in \eqref{est 119}, namely \eqref{est 248b}, we can estimate 
\begin{align}
\RomanII_{4,2} \leq \frac{5\nu }{16} \lVert (w_{u}^{\mathcal{L}}, w_{b}^{\mathcal{L}} ) (t)\rVert_{\dot{H}^{1+ \epsilon}}^{2} + C(M, N_{t}^{\kappa}); \label{est 122}
\end{align}
the proof of \eqref{est 122} is similar to previous computations such as \eqref{est 92} and thus we leave this in the Appendix for completeness. Therefore, by applying \eqref{est 131} and \eqref{est 122} to  \eqref{est 119} and then its result and \eqref{est 130} to \eqref{est 106d} gives us 
\begin{equation}\label{est 136}
 \RomanII_{4}  \leq  \frac{9\nu }{16}  \lVert (w_{u}^{\mathcal{L}}, w_{b}^{\mathcal{L}}) (t)\rVert_{\dot{H}^{1+ \epsilon}}^{2}  + C(M, N_{t}^{\kappa}) \left(1+ \lVert (w_{u}^{\mathcal{L}}, w_{b}^{\mathcal{L}}) (t)\rVert_{H^{1}}^{2} \lVert (w_{u}^{\mathcal{L}}, w_{b}^{\mathcal{L}}) (t)\rVert_{H^{\epsilon}}^{2} \right). 
\end{equation} 
At last, we conclude by applying \eqref{est 133}, \eqref{est 134}, \eqref{est 135}, and \eqref{est 136} to \eqref{est 132} to deduce 
\begin{align*}
&\partial_{t} \lVert (w_{u}^{\mathcal{L}}, w_{b}^{\mathcal{L}}) (t)\rVert_{\dot{H}^{\epsilon}}^{2} \nonumber\\
\leq& - \nu \lVert (w_{u}^{\mathcal{L}}, w_{b}^{\mathcal{L}}) (t)\rVert_{\dot{H}^{1+ \epsilon}}^{2} + C(M, N_{t}^{\kappa}) \left( 1 +  \lVert (w_{u}^{\mathcal{L}}, w_{b}^{\mathcal{L}}) \rVert_{H^{1}}^{2} \lVert (w_{u}^{\mathcal{L}}, w_{b}^{\mathcal{L}}) \rVert_{H^{\epsilon}}^{2} \right)(t)
\end{align*} 
which implies \eqref{est 137} and completes the proof of Proposition \ref{Proposition 4.10}. 
\end{proof}  

\begin{proposition}\label{Proposition 4.11}  
\rm{(Cf. \cite[Corollary 5.4]{HR23})} Suppose that $(u^{\text{in}}, b^{\text{in}}) \in (L_{\sigma}^{2} \cap \mathcal{C}^{-1+ \kappa}) \times (L_{\sigma}^{2} \cap \mathcal{C}^{-1+ \kappa})$ for some $\kappa > 0$. If $T^{\max} < \infty$, then $\limsup_{t \nearrow T^{\max}} \lVert (w_{u}, w_{b})(t) \rVert_{L^{2}} = + \infty$.  
\end{proposition} 

\begin{proof}[Proof of Proposition \ref{Proposition 4.11}]
By Proposition \ref{Proposition 4.1}, for any initial data $(u^{\text{in}}, b^{\text{in}}) \in (L_{\sigma}^{2} \cap \mathcal{C}^{-1+ \kappa}) \times (L_{\sigma}^{2} \cap \mathcal{C}^{-1+ \kappa})$ we know that there exists $T^{\max} = T^{\max} (L_{t}^{\kappa}, u^{\text{in}}, b^{\text{in}}) \in (0, \infty]$ and a unique mild solution $(w_{u}, w_{b}) \in \mathcal{M}_{T^{\max}}^{\frac{\gamma}{2}} \mathcal{C}^{\frac{3\kappa}{4}} \times \mathcal{M}_{T^{\max}}^{\frac{\gamma}{2}} \mathcal{C}^{\frac{3\kappa}{4}}$ over $[0, T^{\max})$ with $\gamma = 1- \frac{\kappa}{4}$ so that 
\begin{align*}
\sup_{t \in [0,T^{\max}]} t^{\frac{1}{2} - \frac{\kappa}{8}} \lVert w_{u}^{\mathcal{L}}  (t) \rVert_{\mathcal{C}^{\frac{3\kappa}{4}}} + \sup_{t \in [0, T^{\max}]} t^{\frac{1}{2} - \frac{\kappa}{8}} \lVert w_{b}^{\mathcal{L}} (t) \rVert_{\mathcal{C}^{\frac{3\kappa}{4}}} < \infty. 
\end{align*}
Considering \eqref{est 29}, we see that for any $\zeta < 1 - \kappa$ and $t \in [0, T^{\max})$ we have  
\begin{equation*}
\lVert w_{u}^{\mathcal{L}}(t) \rVert_{H^{\zeta}} + \lVert w_{b}^{\mathcal{L}}(t) \rVert_{H^{\zeta}} < \infty. 
\end{equation*} 
Suppose that there exists some $i_{\max} \in \mathbb{N}_{0}$ such that $T_{i} = T^{\max}$ for all $i \geq i_{\max}$. Then, because for any $M > 1$ and $T > 0$ such that 
\begin{align*}
\left\lVert (w_{u}^{\mathcal{L}}, w_{b}^{\mathcal{L}}) \left( \frac{T^{\max}}{2} \right) \right\rVert_{H^{\epsilon}}^{2} + \sup_{ t \in [ \frac{T^{\max}}{2}, T \wedge T^{\max}]} \lVert (w_{u}^{\mathcal{L}}, w_{b}^{\mathcal{L}}) (t) \rVert_{L^{2}}^{2} + \nu \int_{\frac{T^{\max}}{2}}^{T \wedge T^{\max}} \lVert (w_{u}^{\mathcal{L}}, w_{b}^{\mathcal{L}})(t) \rVert_{H^{1}}^{2} dt \leq M, 
\end{align*}
\eqref{est 137} gives us 
\begin{align*}
\sup_{t \in [\frac{T^{\max}}{2}, T \wedge T^{\max} ]}  \lVert (w_{u}^{\mathcal{L}}, w_{b}^{\mathcal{L}})(t) \rVert_{\mathcal{C}^{-1 + 2 \kappa}}^{2} \leq C(T, M, N_{T}^{\kappa}) < \infty 
\end{align*}
so that we can extend the solution beyond $T^{\max}$ and reach a contradiction. Therefore, we must have $T_{i} < T^{\max}$ for all $i \in \mathbb{N}$. This completes the proof of Proposition \ref{Proposition 4.11}. 
\end{proof} 

With all the results obtained thus far, we are ready prove Theorem \ref{Theorem 2.2}; we do so in the Appendix due to similarity to the proof of \cite[Theorem 2.5]{HR23}. 

\section{Proof of Theorem \ref{Theorem 2.3}}\label{Section 5} 
We start with the long-awaited definition of a HL weak solution of \eqref{est 17}.  
\begin{define}\label{Definition 5.1}  
\rm{(Cf. \cite[Definition 6.1]{HR23})} Given any $(u^{\text{in}}, b^{\text{in}}) \in L_{\sigma}^{2} \times L_{\sigma}^{2}$ and any $\kappa \in (0,1)$, a pair $(v_{u}, v_{b})$ such that each lies in $C([0,\infty); \mathcal{S}' (\mathbb{T}^{2}; \mathbb{R}^{2}))$ is called a global high-low (HL) weak solution to \eqref{est 17} starting from $(u^{\text{in}}, b^{\text{in}})$ if $(w_{u}, w_{b}) = (v_{u} - Y_{u}, v_{b} - Y_{b})$ from \eqref{est 28} satisfies the following, where $(Y_{u}, Y_{b})$ solves \eqref{est 26}. 
\begin{enumerate}
\item For any $T > 0$, there exists a $\lambda_{T} > 0$ such that for any $\lambda \geq \lambda_{T}$, there exists 
\begin{subequations}\label{est 160}   
\begin{align}
&w_{u}^{\mathcal{L}, \lambda}, w_{b}^{\mathcal{L}, \lambda} \in L^{\infty} ([0, T]; L_{\sigma}^{2}) \cap L^{2} ([0,T]; H^{1}), \label{est 160a} \\
&w_{u}^{\mathcal{H}, \lambda}, w_{b}^{\mathcal{H}, \lambda} \in L^{\infty}([0,T]; L_{\sigma}^{2}) \cap L^{2} ([0,T]; B_{4,2}^{1- 2 \kappa})  \label{est 160b} 
\end{align}
\end{subequations} 
that satisfies 
\begin{subequations}\label{est 161} 
\begin{align}
& w_{u}^{\mathcal{H}, \lambda}(t) \triangleq - \mathbb{P}_{L} \divergence (w_{u} \circlesign{\prec}_{s} \mathcal{H}_{\lambda} Q_{u} - w_{b} \circlesign{\prec}_{s} \mathcal{H}_{\lambda} Q_{b} )(t), \hspace{2mm} w_{u}(t) = w_{u}^{\mathcal{L},\lambda}(t) + w_{u}^{\mathcal{H},\lambda}(t) \\
& w_{b}^{\mathcal{H}, \lambda}(t) \triangleq - \mathbb{P}_{L} \divergence (w_{b} \circlesign{\prec}_{a} \mathcal{H}_{\lambda} Q_{u} - w_{u} \circlesign{\prec}_{a} \mathcal{H}_{\lambda} Q_{b} )(t), \hspace{2mm}  w_{b}(t) = w_{b}^{\mathcal{L},\lambda}(t) + w_{b}^{\mathcal{H},\lambda}(t)
\end{align}
\end{subequations}
for all $t \in [0,T]$ and for $Q_{u}$ and $Q_{b}$ defined in \eqref{est 32}. 
\item The pair $(w_{u}, w_{b})$ solves \eqref{est 29} distributionally; i.e., for any $T > 0$ and any $\phi,\psi \in C^{\infty} ([0,T]\times \mathbb{T}^{2})$ such that $\nabla\cdot \phi = \nabla\cdot \psi = 0$, 
\begin{align}
& \langle w_{u}(T), \phi(T) \rangle - \langle w_{u}(0), \phi(0) \rangle = \int_{0}^{T} \langle w_{u}, \partial_{t} \phi + \nu \Delta \phi \rangle \nonumber \\
& \hspace{5mm} + \langle w_{u}, (w_{u} \cdot \nabla) \phi \rangle + \frac{1}{2} \langle D_{u}, (w_{u} \cdot \nabla) \phi \rangle + \frac{1}{2} \langle w_{u}, (D_{u} \cdot \nabla) \phi \rangle  + \langle Y_{u}, (Y_{u} \cdot \nabla) \phi \rangle \nonumber \\
& \hspace{5mm} - \langle w_{b}, (w_{b} \cdot \nabla) \phi \rangle - \frac{1}{2} \langle D_{b}, (w_{b} \cdot \nabla) \phi \rangle - \frac{1}{2} \langle w_{b}, (D_{b} \cdot \nabla) \phi \rangle - \langle Y_{b}, (Y_{b} \cdot \nabla) \phi \rangle dt, \\
& \langle w_{b}(T), \psi(T) \rangle - \langle w_{b}(0), \psi(0) \rangle = \int_{0}^{T} \langle w_{b}, \partial_{t} \psi + \nu \Delta \psi \rangle  \nonumber \\
& \hspace{5mm} + \langle w_{b}, (w_{u} \cdot \nabla) \psi \rangle + \frac{1}{2} \langle w_{b}, (D_{u} \cdot \nabla) \psi \rangle - \frac{1}{2} \langle D_{u}, (w_{b} \cdot \nabla) \psi \rangle + \langle Y_{b}, (Y_{u} \cdot \nabla) \psi \rangle  \nonumber \\
& \hspace{5mm} - \langle w_{u}, (w_{b} \cdot \nabla) \psi \rangle - \frac{1}{2} \langle w_{u}, (D_{b} \cdot \nabla) \psi \rangle + \frac{1}{2} \langle D_{b}, (w_{u} \cdot \nabla) \psi \rangle - \langle Y_{u}, (Y_{b} \cdot \nabla) \psi \rangle dt. 
\end{align}
\end{enumerate}  
\end{define}

The regularity $L^{2}([0,T]; B_{4,2}^{1-2\kappa})$ of $w_{u}^{\mathcal{H},\lambda}$ and $w_{b}^{\mathcal{H},\lambda}$ in \eqref{est 160b} is higher than $L^{2}([0,T]; B_{4,\infty}^{1-2\kappa})$ of $w_{u}^{\mathcal{H},\lambda}$ in \cite[Definition 6.1]{HR23}. 
\begin{proposition}\label{Proposition 5.1}
\rm{(Cf. \cite[Lemma 6.2]{HR23})} Let $\mathcal{N}'' \subset \Omega$  be the null set from Proposition \ref{Proposition 4.2}. Then for any $\omega \in \Omega \setminus \mathcal{N}''$ and $(u^{\text{in}}, b^{\text{in}}) \in L_{\sigma}^{2} \times L_{\sigma}^{2}$, there exists a HL weak solution to \eqref{est 17} starting from $(u^{\text{in}}, b^{\text{in}})$. 
\end{proposition}

\begin{proof}[Proof of Proposition \ref{Proposition 5.1}]
We define for $n \in \mathbb{N}_{0}$
\begin{equation*}
(X_{u}^{n}, X_{b}^{n}) \triangleq (\mathcal{L}_{n} X_{u}, \mathcal{L}_{n} X_{b}) 
\end{equation*} 
where $X_{u}$ and $X_{b}$ solve \eqref{est 14} and \eqref{est 15}, respectively. We define $(Y_{u}^{n}, Y_{b}^{n})$ to be the corresponding solution to \eqref{est 26} with $(X_{u}, X_{b})$ therein replaced by $(X_{u}^{n}, X_{b}^{n})$. Similarly to \eqref{es 139} we define 
\begin{equation}\label{es 140}
D_{u}^{n} \triangleq 2(X_{u}^{n} + Y_{u}^{n}) \text{ and } D_{b}^{n} \triangleq 2(X_{b}^{n} + Y_{b}^{n}), 
\end{equation} 
and that $(w_{u}^{n}, w_{b}^{n})$ to be the solution to 
\begin{subequations}\label{est 141}
\begin{align}
&\partial_{t} w_{u}^{n} + \mathbb{P}_{L} \divergence ((w_{u}^{n})^{\otimes 2} + D_{u}^{n} \otimes_{s} w_{u}^{n} + (Y_{u}^{n})^{\otimes 2} \nonumber \\
& \hspace{16mm} - (w_{b}^{n})^{\otimes 2} - D_{b}^{n} \otimes_{s} w_{b}^{n} - (Y_{b}^{n})^{\otimes 2}) = \nu \Delta w_{u}^{n},  \label{est 141a}\\
&\partial_{t} w_{b}^{n} +  \mathbb{P}_{L} \divergence (w_{b}^{n} \otimes w_{u}^{n} + w_{b}^{n}\otimes_{a} D_{u}^{n} + Y_{b}^{n} \otimes Y_{u}^{n}  \nonumber \\
& \hspace{16mm} - w_{u}^{n} \otimes w_{b}^{n} - w_{u}^{n} \otimes_{a} D_{b}^{n} - Y_{u}^{n} \otimes Y_{b}^{n})  = \nu \Delta w_{b}^{n}, \label{est 141b}\\
& w_{u}^{n}(0,x) = \mathcal{L}_{n} u^{\text{in}}(x), w_{b}^{n}(0,x) = \mathcal{L}_{n} b^{\text{in}}(x), \label{est 141c}
\end{align}
\end{subequations} 
similarly to \eqref{est 29}. Furthermore, we define similarly to \eqref{est 39} and \eqref{est 140},
\begin{subequations}\label{est 148}
\begin{align}
&L_{t}^{n, \kappa} \triangleq 1+ \sum_{j\in \{u,b\}} [ \lVert X_{j}^{n}  \rVert_{ C_{t} \mathcal{C}^{-\kappa}} + \lVert Y_{j}^{n}  \rVert_{C_{t}\mathcal{C}^{2\kappa}}], \label{est 148a}  \\
& N_{t}^{n, \kappa} \triangleq L_{t}^{n, \kappa} +  \sup_{i\in\mathbb{N}} \{ \lVert ( \nabla_{\text{spec}} \mathcal{L}_{\lambda^{i}} (X_{u}^{n}, X_{b}^{n})  \circlesign{\circ} P^{\lambda_{i, n}} - r_{\lambda^{i}}^{n}  \Id \rVert_{C_{t}\mathcal{C}^{-\kappa}} \},  \label{est 148b} \\
& \bar{N}_{t}^{\kappa}(\omega) \triangleq \sup_{n\in\mathbb{N}} N_{t}^{n,\kappa}(\omega),  \label{est 148c} 
\end{align}
\end{subequations} 
with $\{ \lambda^{i} \}_{i\in\mathbb{N}}$ from Definition \ref{Definition 4.2}, where 
\begin{equation*}
P^{\lambda,n}(t,x) \triangleq \left(-\frac{\nu \Delta}{2} + 1 \right)^{-1} \nabla_{\text{spec}} \mathcal{L}_{\lambda}(X_{u}^{n}, X_{b}^{n}) (t,x)
\end{equation*} 
and 
\begin{equation*}
r_{\lambda}^{n} (t) \triangleq  \sum_{k\in\mathbb{Z}^{2} \setminus \{0\}} \frac{1}{4}  l \left( \frac{ \lvert k \rvert}{\lambda} \right) l \left( \frac{\lvert k \rvert}{n} \right) (1- e^{-2\nu  \lvert k \rvert^{2} t})\left( \frac{ \nu \lvert k \rvert^{2}}{2} + 1\right)^{-1}, \hspace{5mm} r_{\lambda}^{n} (t) \leq c \ln (\lambda \wedge n), 
\end{equation*} 
where the inequality took into account of \eqref{est 232}. It follows that $\lim_{n\to\infty} N_{t}^{n,\kappa}(\omega) = N_{t}^{\kappa}(\omega)$ for $N_{t}^{\kappa}$ from \eqref{est 39}, and $\bar{N}_{t}^{\kappa}(\omega) < \infty$ for all $\omega \in \Omega \setminus \mathcal{N}''$ where $\mathcal{N}''$ is the null set from Proposition \ref{Proposition 4.2}. Similarly to Definition \ref{Definition 4.2} we define 
\begin{equation*}
T_{0}^{n} \triangleq 0 \text{ and } T_{i+1}^{n} (\omega, u^{\text{in}}, b^{\text{in}}) \triangleq \inf\{t \geq T_{i}^{n}: \hspace{1mm} \lVert w_{u}^{n}(t) \rVert_{L^{2}} + \lVert w_{b}^{n}(t) \rVert_{L^{2}} \geq i + 1 \},
\end{equation*} 
and 
\begin{align*}
\lambda_{0}^{n} \triangleq& \lambda_{0},   \\
\lambda_{t}^{n} \triangleq& (1+ \lVert w_{u}^{n}(T_{i}^{n}) \rVert_{L^{2}} + \lVert w_{b}^{n}(T_{i}^{n}) \rVert_{L^{2}})^{3} \hspace{3mm} \text{ for } t > 0, t \in [T_{i}^{n}, T_{i+1}^{n}). 
\end{align*} 
Similarly to \eqref{est 32} we consider $Q_{u}^{n}$ and $Q_{b}^{n}$ that solve 
\begin{equation}\label{est 146}
(\partial_{t} - \nu\Delta ) Q_{u}^{n} = 2 X_{u}^{n}, \hspace{1mm} Q_{u}^{n}(0) = 0, \hspace{2mm} \text{ and } \hspace{2mm} (\partial_{t} - \nu\Delta ) Q_{b}^{n} = 2X_{b}^{n}, \hspace{1mm} Q_{b}^{n}(0) = 0 
\end{equation} 
and define similarly to \eqref{est 38} 
\begin{subequations}\label{est 143}
\begin{align}
& Q_{u}^{n,\mathcal{H}}(t) \triangleq \mathcal{H}_{\lambda_{t}} Q_{u}^{n}(t), \hspace{3mm} Q_{b}^{n,\mathcal{H}}(t) \triangleq \mathcal{H}_{\lambda_{t}} Q_{b}^{n}(t), \label{est 143a} \\
& w_{u}^{n,\mathcal{H}} \triangleq - \mathbb{P}_{L} \divergence (w_{u}^{n} \circlesign{\prec}_{s} Q_{u}^{n,\mathcal{H}} - w_{b}^{n} \circlesign{\prec}_{s} Q_{b}^{n,\mathcal{H}}), \hspace{3mm} w_{u}^{n,\mathcal{L}} \triangleq w_{u}^{n} - w_{u}^{n,\mathcal{H}}, \label{est 143b}\\
& w_{b}^{n,\mathcal{H}} \triangleq - \mathbb{P}_{L} \divergence (w_{b}^{n} \circlesign{\prec}_{a} Q_{u}^{n,\mathcal{H}} - w_{u}^{n} \circlesign{\prec}_{a} Q_{b}^{n,\mathcal{H}}), \hspace{3mm} w_{b}^{n,\mathcal{L}} \triangleq w_{b}^{n} - w_{b}^{n,\mathcal{H}}. \label{est 143c}
\end{align}
\end{subequations} 
Under these settings, repeating the proof identically up to Proposition \ref{Proposition 4.8}, we can obtain $\kappa_{0} > 0$ sufficiently small so that there exists a constant $C_{1} > 0$ and increasing continuous maps $C_{2}, C_{3}: \hspace{1mm} \mathbb{R}_{+} \mapsto \mathbb{R}_{+}$ such that 
\begin{align*}
& \sup_{t \in [T_{i}^{n}, T_{i+1}^{n})} \lVert (w_{u}^{n,\mathcal{L}}, w_{b}^{n,\mathcal{L}}) (t) \rVert_{L^{2}}^{2} + \frac{\nu }{2} \int_{T_{i}^{n}}^{T_{i+1}^{n}} \lVert (w_{u}^{n,\mathcal{L}}, w_{b}^{n,\mathcal{L}})(s) \rVert_{H^{1}}^{2} ds\nonumber  \\
\lesssim& e^{(T_{i+1}^{n} - T_{i}^{n})( C_{2} (N_{T_{i+1}^{n}}^{\kappa}) + C_{1} \ln (\lambda_{T_{i}^{n}} \wedge n) )} \left( \lVert (w_{u}^{n,\mathcal{L}}, w_{b}^{n,\mathcal{L}}) (T_{i}^{n}) \rVert_{L^{2}}^{2} + C_{3} (N_{T_{i+1}^{n}}^{n,\kappa}) \right)   
\end{align*} 
for all $\kappa \in (0, \kappa_{0}]$ and $i \in \mathbb{N}$ such that $i \geq i_{0}$. Similarly to Proposition \ref{Proposition 4.9} and the proof of Theorem \ref{Theorem 2.2}, we can also show uniformly over all $n \in \mathbb{N}$ and $i \geq i_{0} (u^{\text{in}}, b^{\text{in}})$, 
\begin{equation*}
T_{i+1}^{n} - T_{i}^{n} \geq \frac{1}{\tilde{C}(\bar{N}_{T_{i+1}^{n}}^{\kappa}) (1+ \ln (1+ i))} \ln \left( \frac{ i^{2} + 2i - C(\bar{N}_{T_{i+1}^{n}}^{\kappa})}{i^{2} + \tilde{C} (\bar{N}_{T_{i+1}^{n}}^{\kappa})} \right) 
\end{equation*} 
for constants $C(\bar{N}_{T_{i+1}^{n}}^{\kappa})$ and $\tilde{C} (\bar{N}_{T_{i+1}^{n}}^{\kappa})$ and thus for every $T > 0$, $i \in \mathbb{N}$ and $i > i_{0} (u^{\text{in}}, b^{\text{in}})$ there exists $\mathfrak{t}(i, \bar{N}_{T}^{\kappa}) \in (0, T]$ such that 
\begin{align*}
\inf_{n\in\mathbb{N}} T_{i}^{n} \geq \mathfrak{t} (i, \bar{N}_{T}^{\kappa}) \hspace{3mm} \mathfrak{t} (i, \bar{N}_{T}^{\kappa}) = T \hspace{3mm} \forall \hspace{1mm} i \text{ sufficiently large}. 
\end{align*}  
Therefore, for all $T > 0$ and $\kappa > 0$ sufficiently small, there exists $C (T, \bar{N}_{T}^{\kappa}) > 0$ such that 
\begin{equation}\label{est 145}
\sup_{n \in \mathbb{N}} \left[  \lVert (w_{u}^{n,\mathcal{L}}, w_{b}^{n,\mathcal{L}}) \rVert_{C_{T}L^{2}}^{2} + \nu \int_{0}^{T} \lVert (w_{u}^{n,\mathcal{L}}, w_{b}^{n,\mathcal{L}}) (t) \rVert_{\dot{H}^{1}}^{2} dt \right] \leq C(T, \bar{N}_{T}^{\kappa}). 
\end{equation} 
Moreover, we can find $\bar{\lambda}_{T} > 0$, in accordance to Definition \ref{Definition 5.1} (1), such that 
\begin{equation}\label{est 144}
\lambda_{t}^{n} \leq \bar{\lambda}_{T} \hspace{3mm} \forall \hspace{1mm} t \in [0,T], n \in \mathbb{N}. 
\end{equation} 
Therefore, extending the definitions \eqref{est 143b}-\eqref{est 143c} to 
\begin{subequations}\label{est 151}
\begin{align}
& w_{u}^{n,\mathcal{H}, \lambda} \triangleq - \mathbb{P}_{L} \divergence (w_{u}^{n} \circlesign{\prec}_{s} \mathcal{H}_{\lambda} Q_{u}^{n} - w_{b}^{n} \circlesign{\prec}_{s} \mathcal{H}_{\lambda} Q_{b}^{n}), \hspace{1mm} w_{u}^{n,\mathcal{L}, \lambda} \triangleq w_{u}^{n} - w_{u}^{n,\mathcal{H}, \lambda}, \\
& w_{b}^{n,\mathcal{H}, \lambda} \triangleq - \mathbb{P}_{L} \divergence (w_{b}^{n} \circlesign{\prec}_{a} \mathcal{H}_{\lambda} Q_{u}^{n} - w_{u}^{n} \circlesign{\prec}_{a} \mathcal{H}_{\lambda} Q_{b}^{n}), \hspace{1mm} w_{b}^{n,\mathcal{L}, \lambda} \triangleq w_{b}^{n} - w_{b}^{n,\mathcal{H}, \lambda} 
\end{align}
\end{subequations} 
for all $\lambda \geq \bar{\lambda}_{T}$, we see that for all $\lambda \geq \bar{\lambda}_{T}$ and hence $\lambda \geq \lambda_{t}^{n}$ for all $t \in [0,T]$  and $n \in \mathbb{N}$ due to \eqref{est 144}, following the previous computations leads to now 
\begin{equation}\label{est 162}
\sup_{n\in\mathbb{N}} \left[ \lVert (w_{u}^{n,\mathcal{L},\lambda}, w_{b}^{n,\mathcal{L},\lambda}) \rVert_{C_{T}L^{2}}^{2} + \nu \int_{0}^{T} \lVert (w_{u}^{n,\mathcal{L}, \lambda}, w_{b}^{n, \mathcal{L}, \lambda})(t) \rVert_{\dot{H}^{1}}^{2} dt \right] \leq C(\lambda, T, \bar{N}_{T}^{\kappa}). 
\end{equation} 
Next, for any $\mathfrak{a} \in [2,\infty), \kappa > 0$ sufficiently small, and $\alpha \leq 1 - 2 \kappa - \frac{1}{\mathfrak{a}}$,  
\begin{equation}\label{est 150}
 \sup_{n\in\mathbb{N}} \left[ \lVert (w_{u}^{n}, w_{b}^{n})(t) \rVert_{C_{T}L^{2}}^{2} + \int_{0}^{T} \lVert (w_{u}^{n}, w_{b}^{n})(t) \rVert_{H^{\alpha}}^{2} dt  \right]  \leq C(T, \bar{N}_{T}^{\kappa}) 
\end{equation} 
due to \eqref{est 67} and \eqref{est 145}. Furthermore, 
\begin{align}
& \left(\lVert w_{u}^{n} \rVert_{H^{1- \frac{3\kappa}{2}}} + \lVert w_{b}^{n} \rVert_{H^{1- \frac{3\kappa}{2}}}\right)(t) \nonumber\\
&\overset{\eqref{est 143b} \eqref{est 143c}}{\lesssim} \lVert \mathbb{P}_{L} \divergence (w_{u}^{n} \circlesign{\prec}_{s} Q_{u}^{n,\mathcal{H}} - w_{b}^{n} \circlesign{\prec}_{s} Q_{b}^{n,\mathcal{H}}) (t) \rVert_{H^{1- \frac{3\kappa}{2}}} \nonumber\\
& \hspace{13mm} + \lVert \mathbb{P}_{L} \divergence ( w_{b}^{n} \circlesign{\prec}_{a} Q_{u}^{n,\mathcal{H}} - w_{u}^{n} \circlesign{\prec}_{a} Q_{b}^{n,\mathcal{H}} ) (t)\rVert_{H^{1- \frac{3\kappa}{2}}} + \lVert w_{u}^{n,\mathcal{L}} (t) \rVert_{H^{1- \frac{3\kappa}{2}}} + \lVert w_{b}^{n,\mathcal{L}}  (t) \rVert_{H^{1- \frac{3\kappa}{2}}} \nonumber\\
& \lesssim ( \lVert w_{u}^{n} \rVert_{H^{-\frac{\kappa}{4}}} + \lVert w_{b}^{n} \rVert_{H^{-\frac{\kappa}{4}}})(t) ( \lVert Q_{u}^{n,\mathcal{H}} \rVert_{\mathcal{C}^{2- \frac{5\kappa}{4}}} + \lVert Q_{b}^{n,\mathcal{H}} \rVert_{\mathcal{C}^{2- \frac{5\kappa}{4}}})(t) + \lVert w_{u}^{n,\mathcal{L}} (t)\rVert_{H^{1- \frac{3\kappa}{2}}} + \lVert w_{b}^{n,\mathcal{L}} (t)\rVert_{H^{1-\frac{3\kappa}{2}}}   \nonumber \\
& \overset{\eqref{est 143a} \eqref{est 146} \eqref{est 148}}{\lesssim} ( \lVert w_{u}^{n} \rVert_{L^{2}} + \lVert w_{b}^{n} \rVert_{L^{2}})(t) \bar{N}_{t}^{\kappa} + \lVert w_{u}^{n,\mathcal{L}} (t)\rVert_{H^{1- \frac{3\kappa}{2}}} + \lVert w_{b}^{n,\mathcal{L}} (t)\rVert_{H^{1- \frac{3\kappa}{2}}}.  \label{est 149} 
\end{align}
It follows from \eqref{est 149}, \eqref{est 145}, and \eqref{est 150} that for all $\kappa > 0$ sufficiently small 
\begin{equation}\label{est 154} 
\sup_{n\in\mathbb{N}} [ \lVert (w_{u}^{n}, w_{b}^{n} ) \rVert_{L^{2} ([0,T]; H^{1- \frac{3\kappa}{2}})}^{2} ]  \lesssim  C(T, \bar{N}_{T}^{\kappa}). 
\end{equation} 
Consequently, for some $N_{1} \in \mathbb{N}$ from \eqref{est 152}  
\begin{align}   
& \lVert (w_{u}^{n, \mathcal{H}, \lambda}, w_{b}^{n, \mathcal{H}, \lambda}) \rVert_{L^{2}([0,T]; B_{4,2}^{1-2\kappa})}^{2}  \nonumber \\
&\overset{\eqref{est 151}\eqref{est 152}}{\lesssim}  \int_{0}^{T} \sum_{m\geq -1} \lvert 2^{m(2-2\kappa)} \sum_{l: l \leq m + N_{1} - 2} [ \lVert  \Delta_{l} w_{u}^{n} \Delta_{m} \mathcal{H}_{\lambda} Q_{u}^{n,\mathcal{H}} \rVert_{L^{4}} + \lVert  \Delta_{l} w_{b}^{n} \Delta_{m} \mathcal{H}_{\lambda} Q_{b}^{n,\mathcal{H}} \rVert_{L^{4}} \nonumber \\
& \hspace{20mm} + \lVert \Delta_{l} w_{b}^{n} \Delta_{m} \mathcal{H}_{\lambda} Q_{u}^{n,\mathcal{H}}  \rVert_{L^{4}} + \lVert \Delta_{l} w_{u}^{n} \Delta_{m} \mathcal{H}_{\lambda} Q_{b}^{n,\mathcal{H}} \rVert_{L^{4}} ] \rvert^{2} dt \nonumber \\
& \hspace{15mm} \lesssim \int_{0}^{T}   \lVert  ( \mathcal{H}_{\lambda} Q_{u}^{n,\kappa}, \mathcal{H}_{\lambda_{t}} Q_{b}^{n,\kappa} ) \rVert_{\mathcal{C}^{2- \frac{3\kappa}{2}}}^{2} \lVert 2^{-m (\frac{\kappa}{2})} \ast_{m} 2^{-m (\frac{\kappa}{2})}  ( \lVert \Delta_{m} w_{u}^{n} \rVert_{L^{4}} + \lVert \Delta_{m} w_{b}^{n} \rVert_{L^{4}}) \rVert_{l^{2}}^{2} dt \nonumber \\
& \hspace{9mm} \overset{\eqref{est 146} \eqref{est 154}}{\lesssim}  \lVert (X_{u}^{n}, X_{b}^{n}) \rVert_{C_{T}\mathcal{C}^{-\kappa}}  C(T, \bar{N}_{T}^{\kappa})
 \overset{\eqref{est 148}}{\lesssim} C(T, \bar{N}_{T}^{\kappa}). \label{est 147} 
\end{align} 
Next, we can estimate e.g. 
\begin{subequations}\label{est 153}
\begin{align}
& \lVert (w_{u}^{n})^{\otimes 2} (t)\rVert_{H^{-2\kappa}}  \lesssim \lVert w_{u}^{n} (t)\rVert_{L^{\frac{4}{1+ 2\kappa}}}^{2} \lesssim \lVert w_{u}^{n} (t)\rVert_{L^{2}} \lVert w_{u}^{n} (t)\rVert_{H^{1- 2 \kappa}}, \label{est 153a}\\
&  \lVert D_{u}^{n} \otimes_{s} w_{u}^{n}(t) \rVert_{H^{-2\kappa}}  \overset{\eqref{est 40c}\eqref{est 40d}\eqref{est 40e}\eqref{est 148} }{\lesssim}  \bar{N}_{T}^{\kappa} \lVert w_{u}^{n}(t) \rVert_{H^{2\kappa}},  \label{est 153b}\\
&  \lVert (Y_{u}^{n})^{\otimes 2} (t)\rVert_{H^{-2\kappa}} \overset{\eqref{est 40c} \eqref{est 40e}}{\lesssim} \lVert Y_{u}^{n} (t)\rVert_{\mathcal{C}^{2\kappa}}^{2}  \overset{\eqref{est 148}}{\lesssim} (\bar{N}_{T}^{\kappa})^{2}, \label{est 153c}
\end{align}
\end{subequations} 
and use these, together with \eqref{est 141}-\eqref{est 148}, to deduce 
\begin{align}
& [\lVert \partial_{t} w_{u}^{n} \rVert_{H^{-1 - 2\kappa}}+ \lVert \partial_{t} w_{b}^{n} \rVert_{H^{-1 - 2\kappa}}](t) \label{est 155}\\
\overset{\eqref{est 141}}{\lesssim}& [\lVert w_{u}^{n} \rVert_{H^{1- 2 \kappa} } + \lVert w_{b}^{n} \rVert_{H^{1- 2 \kappa}} + \lVert (w_{u}^{n})^{\otimes 2} \rVert_{H^{-2\kappa}} + \lVert (w_{b}^{n})^{\otimes 2} \rVert_{H^{-2\kappa}} + \lVert w_{b}^{n} \otimes w_{u}^{n} \rVert_{H^{-2\kappa}} + \lVert w_{u}^{n} \otimes w_{b}^{n} \rVert_{H^{-2\kappa}}  \nonumber \\
&+ \lVert D_{u}^{n} \otimes_{s} w_{u}^{n} - D_{b}^{n} \otimes_{s} w_{b}^{n} \rVert_{H^{-2\kappa}} + \lVert w_{b}^{n} \otimes_{a} D_{u}^{n} - w_{u}^{n} \otimes_{a} D_{b}^{n} \rVert_{H^{-2\kappa}} \nonumber \\
&+ \lVert (Y_{u}^{n})^{\otimes 2} - (Y_{b}^{n})^{\otimes 2} \rVert_{H^{-2\kappa}} + \lVert Y_{b}^{n} \otimes Y_{u}^{n} - Y_{u}^{n} \otimes Y_{b}^{n} \rVert_{H^{-2\kappa}} ](t) \nonumber \\
\overset{\eqref{est 153}}{\lesssim}& [\lVert w_{u}^{n} \rVert_{H^{1- 2 \kappa}} + \lVert w_{b}^{n} \rVert_{H^{1- 2 \kappa}}](t)  \nonumber\\
& \hspace{10mm} + (\lVert w_{u}^{n}(t) \rVert_{H^{1- 2 \kappa}}  +   \lVert w_{b}^{n} (t)\rVert_{H^{1- 2 \kappa}} + \bar{N}_{T}^{\kappa}) (\lVert w_{u}^{n} (t)\rVert_{L^{2}} + \lVert w_{b}^{n} (t)\rVert_{L^{2}} + \bar{N}_{T}^{\kappa}). \nonumber 
\end{align}
Therefore, applying \eqref{est 150} and \eqref{est 154} to  \eqref{est 155} gives us 
\begin{equation}\label{est 156} 
\sup_{n\in\mathbb{N}} \lVert (\partial_{t} w_{u}^{n}, \partial_{t} w_{b}^{n}) \rVert_{L^{2}([0,T]; H^{-1-2\kappa})}^{2} \leq C(T, \bar{N}_{T}^{\kappa}). 
\end{equation} 
Thus, by \eqref{est 150}, \eqref{est 154}, \eqref{est 156}, and Lions-Aubins compactness lemma (e.g. \cite[Lemma 4]{S90} concerning \eqref{est 157c}) there exists a subsequence $\{(w_{u}^{n_{k}}, w_{b}^{n_{k}}) \}$ and $(w_{u}, w_{b})$ such that 
\begin{subequations}\label{est 157} 
\begin{align}
& w_{u}^{n_{k}} \overset{\ast}{\rightharpoonup} w_{u}, w_{b}^{n_{k}} \overset{\ast}{\rightharpoonup} w_{b} \text{ weak-$\ast$ in } L^{\infty} ([0,T]; L^{2}(\mathbb{T}^{2})), \label{est 157a} \\
& w_{u}^{n_{k}} \rightharpoonup w_{u}, w_{b}^{n_{k}} \rightharpoonup w_{b} \text{ weakly in } L^{2} ([0,T]; H^{1- \frac{3\kappa}{2}}(\mathbb{T}^{2})),  \label{est 157b} \\
& w_{u}^{n_{k}} \to w_{u}, w_{b}^{n_{k}} \to w_{b} \text{ strongly in } L^{2}([0,T]; H^{\beta}(\mathbb{T}^{2})) \hspace{1mm} \forall \hspace{1mm} \beta \in \left( -1 - 2 \kappa,  1- \frac{3\kappa}{2}\right).  \label{est 157c} 
\end{align}
\end{subequations} 
With these convergence results, it follows that $(w_{u}, w_{b})$ is a weak solution to \eqref{est 17}. Moreover, it follows from \eqref{est 151}, \eqref{est 161}, \eqref{est 40}, \eqref{est 16}, \eqref{est 148}, and \eqref{est 157} that 
\begin{equation*}
w_{u}^{n,\mathcal{H},\lambda} \to w_{u}^{\mathcal{H},\lambda} \text{ and } w_{b}^{n,\mathcal{H},\lambda} \to w_{b}^{\mathcal{H}, \lambda} \text{ as } n\to\infty \text{ strongly in } L^{2} (0, T; H^{1- 4 \kappa}). 
\end{equation*} 
Finally, from \eqref{est 147} we see that $w_{u}^{\mathcal{H},\lambda}, w_{b}^{\mathcal{H},\lambda} \in L^{2}([0,T]; B_{4,2}^{1-2\kappa})$ as claimed in \eqref{est 160b}. The fact that $w_{u}^{\mathcal{H}, \lambda}, w_{b}^{\mathcal{H}, \lambda} \in L^{\infty} ([0,T]; L_{\sigma}^{2})$ follows from \eqref{est  145} and \eqref{est 150}. Finally, \eqref{est 162} implies the desired result of $w_{u}^{\mathcal{L}, \lambda}, w_{b}^{\mathcal{L}, \lambda} \in L^{\infty} ([0, T]; L_{\sigma}^{2}) \cap L^{2} ([0,T]; H^{1})$ in \eqref{est 160a}.  
\end{proof}

\begin{proposition}\label{Proposition 5.2} 
\rm{(Cf. \cite[Lemma 6.3]{HR23})} Let $\mathcal{N}''$ be the null set from Proposition \ref{Proposition 4.2}. Then, for any $\omega \in \Omega \setminus \mathcal{N}''$ and any $(u^{\text{in}}, b^{\text{in}}) \in L_{\sigma}^{2} \times L_{\sigma}^{2}$, there exists at most one HL weak solution starting from $(u^{\text{in}}, b^{\text{in}})$. 
\end{proposition}

\begin{proof}[Proof of Proposition \ref{Proposition 5.2}]
Le us suppose that 
\begin{equation*}
(v_{u}, v_{b}) \triangleq (w_{u} + Y_{u}, w_{b} + Y_{b}), \hspace{3mm} (\bar{v}_{u}, \bar{v}_{b}) \triangleq (\bar{w}_{u} + Y_{u}, \bar{w}_{b} + Y_{b}), 
\end{equation*} 
are two HL weak solutions and define 
\begin{equation}\label{est 167}
(z_{u}, z_{b}) \triangleq (w_{u} - \bar{w}_{u}, w_{b} - \bar{w}_{b}),
\end{equation} 
and 
\begin{equation}\label{est 172} 
 (z_{u}^{\mathcal{L},\lambda}, z_{b}^{\mathcal{L}, \lambda}) \triangleq (w_{u}^{\mathcal{L},\lambda} - \bar{w}_{u}^{\mathcal{L},\lambda}, w_{b}^{\mathcal{L},\lambda} - \bar{w}_{b}^{\mathcal{L},\lambda}), \hspace{1mm} \text{ and } \hspace{1mm} (z_{u}^{\mathcal{H},\lambda}, z_{b}^{\mathcal{H},\lambda}) \triangleq (z_{u} - z_{u}^{\mathcal{L},\lambda}, z_{b} - z_{b}^{\mathcal{L},\lambda}). 
\end{equation} 
Then 
\begin{subequations}\label{est 163} 
\begin{align}
\partial_{t} z_{u}^{\mathcal{L}, \lambda} -  \nu \Delta z_{u}^{\mathcal{L},\lambda}  &=- \mathbb{P}_{L} \divergence ( 2( \mathcal{L}_{\lambda} X_{u}) \otimes_{s} z_{u}^{\mathcal{L},\lambda} - 2 ( \mathcal{L}_{\lambda} X_{b}) \otimes_{s} z_{b}^{\mathcal{L}, \lambda}) \label{est 163a} \\
& - \mathbb{P}_{L} \divergence ( 2( \mathcal{H}_{\lambda} X_{u} ) \otimes_{s} z_{u}^{\mathcal{L},\lambda} - 2 \mathcal{H}_{\lambda} X_{u} \circlesign{\succ}_{s} z_{u}^{\mathcal{L},\lambda} \nonumber \\
& \hspace{10mm}  -2 ( \mathcal{H}_{\lambda} X_{b}) \otimes_{s} z_{b}^{\mathcal{L},\lambda}+ 2 \mathcal{H}_{\lambda} X_{b} \circlesign{\succ}_{s} z_{b}^{\mathcal{L},\lambda}) \nonumber \\
& - \mathbb{P}_{L} \divergence ( 2X_{u} \otimes_{s} z_{u}^{\mathcal{H},\lambda} - 2 \mathcal{H}_{\lambda} X_{u} \circlesign{\succ}_{s} z_{u}^{\mathcal{H},\lambda} - C^{\circlesign{\prec}_{s}} (z_{u}, \mathcal{H}_{\lambda} Q_{u}) \nonumber \\
& \hspace{10mm} - 2 X_{b} \otimes_{s} z_{b}^{\mathcal{H},\lambda} + 2\mathcal{H}_{\lambda} X_{b} \circlesign{\succ}_{s} z_{b}^{\mathcal{H},\lambda} + C^{\circlesign{\prec}_{s}} (z_{b},\mathcal{H}_{\lambda} Q_{b}) )   \nonumber \\
& - \mathbb{P}_{L} \divergence (w_{u}^{\otimes 2} - \bar{w}_{u}^{\otimes 2} + 2Y_{u} \otimes_{s} z_{u} - w_{b}^{\otimes 2} + \bar{w}_{b}^{\otimes 2} - 2Y_{b} \otimes_{s} z_{b}), \nonumber  \\
\partial_{t} z_{b}^{\mathcal{L},\lambda} - \nu \Delta z_{b}^{\mathcal{L},\lambda} &= - \mathbb{P}_{L} \divergence (2 z_{b}^{\mathcal{L},\lambda} \otimes_{a} \mathcal{L}_{\lambda} X_{u} - 2 z_{u}^{\mathcal{L},\lambda} \otimes_{a} \mathcal{L}_{\lambda} X_{b}) \label{est 163b} \\
 & - \mathbb{P}_{L} \divergence (2z_{b}^{\mathcal{L},\lambda} \otimes_{a} \mathcal{H}_{\lambda} X_{u} - 2 z_{b}^{\mathcal{L},\lambda} \circlesign{\prec}_{a} \mathcal{H}_{\lambda} X_{u} \nonumber \\
 & \hspace{10mm} - 2 z_{u}^{\mathcal{L},\lambda} \otimes_{a} \mathcal{H}_{\lambda} X_{b} + 2 z_{u}^{\mathcal{L},\lambda} \circlesign{\prec}_{a} \mathcal{H}_{\lambda} X_{b}) \nonumber \\
 & - \mathbb{P}_{L} \divergence (2z_{b}^{\mathcal{H},\lambda} \otimes_{a} X_{u} - 2 z_{b}^{\mathcal{H},\lambda} \circlesign{\prec}_{a} \mathcal{H}_{\lambda} X_{u} - C^{\circlesign{\prec}_{a}} (z_{b}, \mathcal{H}_{\lambda} Q_{u}) \nonumber \\
 & \hspace{10mm} - 2z_{u}^{\mathcal{H},\lambda} \otimes_{a} X_{b} + 2z_{u}^{\mathcal{H},\lambda} \circlesign{\prec}_{a} \mathcal{H}_{\lambda} X_{b} + C^{\circlesign{\prec}_{a}} (z_{u}, \mathcal{H}_{\lambda} Q_{b} )) \nonumber \\
& - \mathbb{P}_{L} \divergence (w_{b} \otimes w_{u} - \bar{w}_{b} \otimes \bar{w}_{u} + 2z_{b} \otimes_{a} Y_{u} - w_{u} \otimes w_{b} + \bar{w}_{u} \otimes \bar{w}_{b} - 2z_{u} \otimes_{a} Y_{b}). \nonumber 
\end{align}
\end{subequations} 
Then taking $L^{2}(\mathbb{T}^{2})$-inner products on \eqref{est 163} with $(z_{u}^{\mathcal{L},\lambda}, z_{b}^{\mathcal{L},\lambda})$ gives us 
\begin{equation}\label{est 185}
\frac{1}{2} \partial_{t} \lVert (z_{u}^{\mathcal{L},\lambda}, z_{b}^{\mathcal{L},\lambda})(t) \rVert_{L^{2}}^{2} = \sum_{k=1}^{4} \RomanIII_{k}, 
\end{equation} 
where 
\begin{subequations}\label{est 164} 
\begin{align}
\RomanIII_{1} \triangleq& - \langle z_{u}^{\mathcal{L},\lambda}, - \nu\Delta  z_{u}^{\mathcal{L}, \lambda} + \mathbb{P}_{L} \divergence ( 2( \mathcal{L}_{\lambda} X_{u}) \otimes_{s} z_{u}^{\mathcal{L},\lambda} - 2 (\mathcal{L}_{\lambda} X_{b}) \otimes_{s} z_{b}^{\mathcal{L},\lambda}) \rangle(t)  \label{est 164a} \\
& - \langle z_{b}^{\mathcal{L},\lambda}, - \nu\Delta  z_{b}^{\mathcal{L},\lambda} + \mathbb{P}_{L} \divergence (2z_{b}^{\mathcal{L},\lambda} \otimes_{a} \mathcal{L}_{\lambda} X_{u} - 2 z_{u}^{\mathcal{L},\lambda} \otimes_{a} \mathcal{L}_{\lambda} X_{b}) \rangle(t), \nonumber \\
\RomanIII_{2} \triangleq& - \langle z_{u}^{\mathcal{L},\lambda}, \mathbb{P}_{L} \divergence (2( \mathcal{H}_{\lambda} X_{u}) \otimes_{s} z_{u}^{\mathcal{L},\lambda} - 2 \mathcal{H}_{\lambda} X_{u} \circlesign{\succ}_{s} z_{u}^{\mathcal{L},\lambda}   \label{est 164b} \\
& \hspace{10mm} - 2 (\mathcal{H}_{\lambda} X_{b}) \otimes_{s} z_{b}^{\mathcal{L},\lambda} + 2 \mathcal{H}_{\lambda} X_{b} \circlesign{\succ}_{s} z_{b}^{\mathcal{L},\lambda}) \rangle(t) \nonumber \\
& - \langle z_{b}^{\mathcal{L},\lambda},\mathbb{P}_{L} \divergence (2z_{b}^{\mathcal{L},\lambda} \otimes_{a} \mathcal{H}_{\lambda} X_{u} - 2 z_{b}^{\mathcal{L},\lambda} \circlesign{\prec}_{a} \mathcal{H}_{\lambda} X_{u} \nonumber \\
& \hspace{10mm} -2z_{u}^{\mathcal{L},\lambda} \otimes_{a} \mathcal{H}_{\lambda} X_{b} + 2z_{u}^{\mathcal{L},\lambda} \circlesign{\prec}_{a} \mathcal{H}_{\lambda} X_{b} ) \rangle(t),\nonumber \\
\RomanIII_{3} \triangleq& - \langle z_{u}^{\mathcal{L},\lambda}, \mathbb{P}_{L} \divergence (2X_{u} \otimes_{s} z_{u}^{\mathcal{H},\lambda} - 2 \mathcal{H}_{\lambda} X_{u} \circlesign{\succ}_{s} z_{u}^{\mathcal{H},\lambda} - C^{\circlesign{\prec}_{s}} (z_{u}, \mathcal{H}_{\lambda} Q_{u}) \label{est 164c} \\
& \hspace{10mm} - 2X_{b} \otimes_{s} z_{b}^{\mathcal{H},\lambda} + 2\mathcal{H}_{\lambda} X_{b} \circlesign{\succ}_{s} z_{b}^{\mathcal{H},\lambda} +C^{\circlesign{\prec}_{s}} (z_{b}, \mathcal{H}_{\lambda} Q_{b} )) \rangle(t) \nonumber \\
& - \langle z_{b}^{\mathcal{L},\lambda}, \mathbb{P}_{L} \divergence ( 2z_{b}^{\mathcal{H}, \lambda} \otimes_{a} X_{u} - 2z_{b}^{\mathcal{H},\lambda} \circlesign{\prec}_{a} \mathcal{H}_{\lambda} X_{u} - C^{\circlesign{\prec}_{a}} (z_{b},\mathcal{H}_{\lambda} Q_{u}) \nonumber \\
& \hspace{10mm} - 2z_{u}^{\mathcal{H},\lambda} \otimes_{a} X_{b} + 2z_{u}^{\mathcal{H},\lambda} \circlesign{\prec}_{a} \mathcal{H}_{\lambda} X_{b} + C^{\circlesign{\prec}_{a}} (z_{u}, \mathcal{H}_{\lambda} Q_{b} )) \rangle(t), \nonumber \\
\RomanIII_{4} \triangleq& - \langle z_{u}^{\mathcal{L},\lambda}, \mathbb{P}_{L} \divergence (w_{u}^{\otimes 2} - \bar{w}_{u}^{\otimes 2} + 2 Y_{u} \otimes_{s} z_{u} - w_{b}^{\otimes 2} + \bar{w}_{b}^{\otimes 2} - 2 Y_{b} \otimes_{s} z_{b} ) \rangle(t) \label{est 164d} \\
& - \langle z_{b}^{\mathcal{L},\lambda}, \mathbb{P}_{L} \divergence (w_{b} \otimes w_{u} - \bar{w}_{b} \otimes \bar{w}_{u} + 2 z_{b} \otimes_{a} Y_{u} - w_{u} \otimes w_{b} + \bar{w}_{u} \otimes \bar{w}_{b} - 2 z_{u} \otimes_{a} Y_{b}) \rangle(t). \nonumber 
\end{align}
\end{subequations} 
Within $\RomanIII_{1}$ of \eqref{est 164a} we can estimate 
\begin{align}
- \langle z_{b}^{\mathcal{L}, \lambda}, \mathbb{P}_{L} & \divergence (2z_{b}^{\mathcal{L},\lambda}  \otimes_{a} \mathcal{L}_{\lambda} X_{u}) \rangle(t) 
\lesssim  \lVert z_{b}^{\mathcal{L},\lambda} (t)\rVert_{L^{2}} \lVert \nabla z_{b}^{\mathcal{L},\lambda} (t)\rVert_{L^{2}} \lVert \mathcal{L}_{\lambda} X_{u} (t)\rVert_{\mathcal{C}^{\kappa}} \nonumber\\
\overset{\eqref{est 33} \eqref{est 39}}{\lesssim}&  \lVert z_{b}^{\mathcal{L},\lambda} (t)\rVert_{L^{2}} \lVert \nabla z_{b}^{\mathcal{L},\lambda} (t)\rVert_{L^{2}} \lambda^{2\kappa} N_{T}^{\kappa} \leq \frac{\nu }{64} \lVert z_{b}^{\mathcal{L},\lambda} (t)\rVert_{\dot{H}^{1}}^{2} + C(\lambda, N_{T}^{\kappa}) \lVert  z_{b}^{\mathcal{L},\lambda} (t)\rVert_{L^{2}}^{2}. \label{est 254} 
\end{align} 
Analogous estimates can be achieved the first term $- \langle z_{u}^{\mathcal{L}, \lambda}, \mathbb{P}_{L} \divergence (2 (\mathcal{L}_{\lambda} X_{u}) \otimes_{s} z_{u}^{\mathcal{L},\lambda} \rangle$. For the second and fourth terms, they need to be paired to obtain the necessary cancellations: 
\begin{align}
& \langle z_{u}^{\mathcal{L},\lambda}, \mathbb{P}_{L} \divergence (2 ( \mathcal{L}_{\lambda} X_{b}) \otimes_{s} z_{b}^{\mathcal{L},\lambda} ) \rangle + \langle z_{b}^{\mathcal{L},\lambda}, \mathbb{P}_{L} \divergence ( z_{u}^{\mathcal{L},\lambda} \otimes_{a} \mathcal{L}_{\lambda} X_{b}) \rangle \nonumber \\
\overset{\eqref{est 18} \eqref{est 19}}{=}& \int_{\mathbb{T}^{2}} z_{u}^{\mathcal{L},\lambda} \cdot [(z_{b}^{\mathcal{L},\lambda} \cdot \nabla) \mathcal{L}_{\lambda} X_{b} + (\mathcal{L}_{\lambda} X_{b} \cdot \nabla) z_{b}^{\mathcal{L},\lambda} ] \nonumber \\
& + z_{b}^{\mathcal{L},\lambda} \cdot [(\mathcal{L}_{\lambda} X_{b} \cdot \nabla) z_{u}^{\mathcal{L},\lambda} - (z_{u}^{\mathcal{L},\lambda} \cdot \nabla) \mathcal{L}_{\lambda} X_{b} ] dx  \nonumber  \\
\overset{\eqref{est 252}}{=}& \int_{\mathbb{T}^{2}} z_{u}^{\mathcal{L},\lambda} \cdot (z_{b}^{\mathcal{L},\lambda} \cdot \nabla) \mathcal{L}_{\lambda} X_{b} - z_{b}^{\mathcal{L},\lambda} \cdot (z_{u}^{\mathcal{L},\lambda} \cdot \nabla) \mathcal{L}_{\lambda} X_{b}   dx, \label{est 253} 
\end{align}
where we used 
\begin{equation}\label{est 252} 
\int_{\mathbb{T}^{2}} z_{u}^{\mathcal{L},\lambda} \cdot (\mathcal{L}_{\lambda} X_{b} \cdot \nabla) z_{b}^{\mathcal{L},\lambda} + z_{b}^{\mathcal{L},\lambda} \cdot (\mathcal{L}_{\lambda} X_{b} \cdot \nabla) z_{u}^{\mathcal{L},\lambda} dx = 0.
\end{equation} 
Having obtained this necessary cancellation, \eqref{est 253} can be estimated similarly to \eqref{est 254} now. Thus, we conclude 
\begin{equation}\label{est 186}
\RomanIII_{1} \leq -\frac{15\nu }{16} \lVert (z_{u}^{\mathcal{L},\lambda}, z_{b}^{\mathcal{L},\lambda} )  (t)\rVert_{\dot{H}^{1}}^{2} + C( \lambda, N_{T}^{\kappa}) \lVert (z_{u}^{\mathcal{L},\lambda}, z_{b}^{\mathcal{L},\lambda}) (t)\rVert_{L^{2}}^{2}. 
\end{equation} 
Next, within $\RomanIII_{2}$ of \eqref{est 164b} we can estimate using $\phi \otimes_{a} \psi = \phi \circlesign{\prec}_{a} \psi + \phi \circlesign{\succ}_{a} \psi + \phi \circlesign{\circ}_{a} \psi$, 
\begin{align*}
& - \langle z_{b}^{\mathcal{L},\lambda},\mathbb{P}_{L} \divergence (2z_{b}^{\mathcal{L},\lambda} \otimes_{a} \mathcal{H}_{\lambda} X_{u} - 2 z_{b}^{\mathcal{L},\lambda} \circlesign{\prec}_{a} \mathcal{H}_{\lambda} X_{u}) \rangle(t) \\
\lesssim& \lVert z_{b}^{\mathcal{L}, \lambda} (t)\rVert_{\dot{H}^{1}} ( \lVert z_{b}^{\mathcal{L},\lambda} \circlesign{\succ}_{a} \mathcal{H}_{\lambda} X_{u} \rVert_{L^{2}} + \lVert z_{b}^{\mathcal{L},\lambda} \circlesign{\circ}_{a} \mathcal{H}_{\lambda} X_{u} \rVert_{H^{\kappa}})(t) \nonumber \\
\overset{\eqref{est 40d} \eqref{est 40e}}{\lesssim}& \lVert z_{b}^{\mathcal{L},\lambda} (t)\rVert_{\dot{H}^{1}} \lVert z_{b}^{\mathcal{L},\lambda} (t)\rVert_{L^{2}}^{1-2\kappa} \lVert z_{b}^{\mathcal{L},\lambda} (t)\rVert_{\dot{H}^{1}}^{2\kappa} \lVert X_{u} (t)\rVert_{\mathcal{C}^{-\kappa}} \leq \frac{\nu }{32} \lVert z_{b}^{\mathcal{L},\lambda} (t)\rVert_{\dot{H}^{1}}^{2} + C(N_{T}^{\kappa}) \lVert z_{b}^{\mathcal{L},\lambda} (t)\rVert_{L^{2}}^{2}.  \nonumber 
\end{align*}
Analogous estimates on similar terms lead to 
\begin{equation}\label{est 187}
\RomanIII_{2} \leq \frac{\nu}{16} \lVert (z_{u}^{\mathcal{L}, \lambda}, z_{b}^{\mathcal{L},\lambda}) (t)\rVert_{\dot{H}^{1}}^{2} + C(N_{T}^{\kappa}) \lVert (z_{u}^{\mathcal{L},\lambda}, z_{b}^{\mathcal{L},\lambda}) (t)\rVert_{L^{2}}^{2}. 
\end{equation} 
Next, within $\RomanIII_{3}$ of \eqref{est 164c} we work on 
\begin{align}
& \langle z_{b}^{\mathcal{L}, \lambda}, \mathbb{P}_{L} \divergence ( 2z_{b}^{\mathcal{H},\lambda} \otimes_{a} X_{u} - 2z_{b}^{\mathcal{H},\lambda} \circlesign{\prec}_{a} \mathcal{H}_{\lambda} X_{u} - C^{\circlesign{\prec}_{a}} (z_{b}, \mathcal{H}_{\lambda} (Q_{u} )) \rangle \nonumber \\
\lesssim& \lVert z_{b}^{\mathcal{L},\lambda} \rVert_{H^{1}} [ \lVert z_{b}^{\mathcal{H},\lambda} \otimes_{a} X_{u} - z_{b}^{\mathcal{H},\lambda} \circlesign{\prec}_{a} \mathcal{H}_{\lambda} X_{u} \rVert_{L^{2}}  \nonumber \\
& \hspace{10mm} + \lVert ((\partial_{t} - \nu \Delta) z_{b}) \circlesign{\prec}_{a} \mathcal{H}_{\lambda} Q_{u} \rVert_{L^{2}} + \sum_{k=1}^{2}  \lVert  \partial_{k} z_{b} \circlesign{\prec}_{a} \partial_{k} \mathcal{H}_{\lambda} Q_{u}  \rVert_{L^{2}}] \label{est 166}
\end{align}
where we used \eqref{est 30b}. First, we rewrite using \eqref{est 68} and then $\phi \otimes_{a} \psi = \phi \circlesign{\prec}_{a} \psi + \phi \circlesign{\succ}_{a} \psi + \phi \circlesign{\circ}_{a} \psi$ to compute 
\begin{align}
&  \lVert( z_{b}^{\mathcal{H},\lambda} \otimes_{a} X_{u} - z_{b}^{\mathcal{H},\lambda} \circlesign{\prec}_{a} \mathcal{H}_{\lambda} X_{u} )(t)\rVert_{L^{2}}   \nonumber \\
=& \lVert (z_{b}^{\mathcal{H},\lambda} \otimes_{a} \mathcal{L}_{\lambda} X_{u} + z_{b}^{\mathcal{H},\lambda} \circlesign{\succ}_{a} \mathcal{H}_{\lambda} X_{u} + z_{b}^{\mathcal{H},\lambda} \circlesign{\circ}_{a} \mathcal{H}_{\lambda} X_{u} )(t)\rVert_{L^{2}}  \nonumber \\
\lesssim& \lVert z_{b}^{\mathcal{H}, \lambda} \otimes_{a} \mathcal{L}_{\lambda} X_{u} (t)\rVert_{L^{2}} + \lVert z_{b}^{\mathcal{H}, \lambda} \circlesign{\succ}_{a} \mathcal{H}_{\lambda} X_{u}  (t)\rVert_{L^{2}} + \lVert z_{b}^{\mathcal{H},\lambda} \circlesign{\circ}_{a} \mathcal{H}_{\lambda} X_{u} (t)\rVert_{H^{\kappa}} \nonumber \\
\overset{\eqref{est 40d}}{\lesssim}& [\lVert z_{b}^{\mathcal{H}, \lambda} \rVert_{L^{2}} \lVert \mathcal{L}_{\lambda} X_{u} \rVert_{L^{\infty}} + \lVert z_{b}^{\mathcal{H},\lambda} \rVert_{H^{\kappa}} \lVert \mathcal{H}_{\lambda} X_{u} \rVert_{\mathcal{C}^{-\kappa}} + \lVert z_{b}^{\mathcal{H}, \lambda} \rVert_{H^{2\kappa}} \lVert \mathcal{H}_{\lambda} X_{u} \rVert_{\mathcal{C}^{-\kappa}} ](t) \nonumber \\
\overset{\eqref{est 39}}{\lesssim}& (\lVert z_{b}^{\mathcal{H}, \lambda} \rVert_{H^{\kappa}} \lambda^{2\kappa} + \lVert z_{b}^{\mathcal{H},\lambda} \rVert_{H^{2\kappa}})(t) N_{T}^{\kappa}. \label{est 169} 
\end{align}
Second, we first rewrite using \eqref{est 167} and \eqref{est 29b}, 
\begin{equation*}
(\partial_{t} -\nu \Delta) z_{b} = - \mathbb{P}_{L} \divergence (z_{b} \otimes w_{u} + \bar{w}_{b} \otimes z_{u} + z_{b} \otimes_{a} D_{u} - z_{u} \otimes w_{b} - \bar{w}_{u} \otimes z_{b} - z_{u} \otimes_{a} D_{b}).
\end{equation*} 
Thus, by H$\ddot{\mathrm{o}}$lder's inequality and Sobolev embeddings of $L^{\frac{4}{4-3\kappa}}(\mathbb{T}^{2}) \hookrightarrow H^{-1+ \frac{3\kappa}{2}}(\mathbb{T}^{2})$ and $H^{\frac{3\kappa}{2}} (\mathbb{T}^{2}) \hookrightarrow L^{\frac{4}{2-3\kappa}}(\mathbb{T}^{2})$, 
\begin{align}
& \lVert (( \partial_{t} - \nu\Delta ) z_{b}) \circlesign{\prec}_{a} \mathcal{H}_{\lambda} Q_{u} (t)\rVert_{L^{2}}  \nonumber  \\
&\overset{\eqref{est 40c}}{\lesssim}  \lVert (z_{b} \otimes w_{u} + \bar{w}_{b} \otimes z_{u} + z_{b} \otimes_{a} D_{u} - z_{u} \otimes w_{b} - \bar{w}_{u} \otimes z_{b} - z_{u} \otimes_{a} D_{b}) (t)\rVert_{H^{-1+ \frac{3\kappa}{2}}} \lVert Q_{u} (t)\rVert_{\mathcal{C}^{2- \frac{3\kappa}{2}}} \nonumber\\
&\overset{\eqref{est 32}\eqref{est 39}}{\lesssim} ( \lVert z_{u} \rVert_{L^{2}} + \lVert z_{b} \rVert_{L^{2}})(t) ( \lVert w_{u} (t)\rVert_{H^{2\kappa}} + \lVert \bar{w}_{u} (t)\rVert_{H^{2\kappa}}+ \lVert w_{b} (t)\rVert_{H^{2\kappa}} + \lVert \bar{w}_{b} (t)\rVert_{H^{2\kappa}} + N_{T}^{\kappa} ) N_{T}^{\kappa} \nonumber \\
& \hspace{50mm} + (\lVert z_{u} \rVert_{H^{2\kappa}} + \lVert z_{b} \rVert_{H^{2\kappa}})(t) (N_{T}^{\kappa})^{2}.   \label{est 170}
\end{align}
Third, we estimate 
\begin{equation}\label{est 171} 
\sum_{k=1}^{2} \lVert \partial_{k} z_{b} \circlesign{\prec}_{a} \partial_{k} \mathcal{H}_{\lambda} Q_{u} (t)\rVert_{L^{2}} \overset{\eqref{est 40c}}{\lesssim}   \lVert z_{b} (t)\rVert_{H^{2\kappa}} \lVert Q_{u} (t)\rVert_{\mathcal{C}^{2- \frac{3\kappa}{2}}} \overset{\eqref{est 32}\eqref{est 39}}{\lesssim} \lVert z_{b} (t)\rVert_{H^{2\kappa}} N_{T}^{\kappa}.
\end{equation} 
At last, we apply \eqref{est 169}, \eqref{est 170}, and \eqref{est 171} to \eqref{est 166} to deduce 
\begin{align*}
& \langle z_{b}^{\mathcal{L}, \lambda}, \mathbb{P}_{L} \divergence ( 2z_{b}^{\mathcal{H},\lambda} \otimes_{a} X_{u} - 2z_{b}^{\mathcal{H},\lambda} \circlesign{\prec}_{a} \mathcal{H}_{\lambda} X_{u} - C^{\circlesign{\prec}_{a}} (z_{b}, \mathcal{H}_{\lambda} Q_{u}) \rangle(t) \nonumber \\
&\leq \frac{\nu}{64} \lVert z_{b}^{\mathcal{L},\lambda} (t)\rVert_{\dot{H}^{1}}^{2} + C(N_{T}^{\kappa},\lambda) [ (\lVert z_{b}^{\mathcal{H},\lambda} \rVert_{H^{2\kappa}} + \lVert z_{u} \rVert_{H^{2\kappa}} + \lVert z_{b} \rVert_{H^{2\kappa}} )(t) \nonumber \\
& \hspace{30mm} \times ( \lVert w_{u} \rVert_{H^{2\kappa}} + \lVert \bar{w}_{u} \rVert_{H^{2\kappa}} + \lVert w_{b} \rVert_{H^{2\kappa}} + \lVert \bar{w}_{b} \rVert_{H^{2\kappa}} + 1)(t) ]^{2}.  
\end{align*}
Similar computations on analogous terms of $\RomanIII_{3}$ in \eqref{est 164c} lead to altogether 
\begin{align}
&\RomanIII_{3} \leq \frac{\nu }{16} \lVert (z_{u}^{\mathcal{L},\lambda}, z_{b}^{\mathcal{L},\lambda}) (t)\rVert_{\dot{H}^{1}}^{2}  \label{est 188} \\
&+ C(N_{T}^{\kappa},\lambda) [ ( \lVert (z_{u}^{\mathcal{H},\lambda}, z_{b}^{\mathcal{H},\lambda}) (t)\rVert_{H^{2\kappa}} + \lVert (z_{u},z_{b}) (t)\rVert_{H^{2\kappa}}) (\lVert (w_{u},w_{b}) (t)\rVert_{H^{2\kappa}} + \lVert (\bar{w}_{u},\bar{w}_{b}) (t)\rVert_{H^{2\kappa}} + 1)]^{2}. \nonumber 
\end{align} 

Finally, within $\RomanIII_{4}$ of \eqref{est 164d}, we estimate as an example
\begin{align}
& - \langle z_{b}^{\mathcal{L},\lambda}, \mathbb{P}_{L} \divergence (w_{b} \otimes w_{u} - \bar{w}_{b} \otimes \bar{w}_{u} + 2 z_{b} \otimes_{a} Y_{u}) \rangle(t) \nonumber \\
\overset{\eqref{est 167}}{=}& - \langle z_{b}^{\mathcal{L},\lambda}, \mathbb{P}_{L} \divergence (z_{b} \otimes w_{u} + \bar{w}_{b} \otimes z_{u} + 2 z_{b} \otimes_{a} Y_{u}) \rangle(t) \nonumber \\
\overset{\eqref{est 39}}{\lesssim}& \lVert z_{b}^{\mathcal{L}, \lambda} (t)\rVert_{H^{1}} (\lVert z_{b} \otimes w_{u} (t)\rVert_{L^{2}} + \lVert \bar{w}_{b} \otimes z_{u} (t)\rVert_{L^{2}} + \lVert z_{b} (t)\rVert_{L^{2}} N_{T}^{\kappa}).  \label{est 180} 
\end{align}
By \eqref{est 172} and \eqref{est 161}, we expand 
\begin{subequations}\label{est 174}
\begin{align}
& z_{b} \otimes w_{u} = z_{b}^{\mathcal{L},\lambda} \otimes w_{u}^{\mathcal{L},\lambda}+ z_{b}^{\mathcal{L},\lambda} \otimes w_{u}^{\mathcal{H},\lambda} + z_{b}^{\mathcal{H},\lambda} \otimes w_{u}^{\mathcal{L},\lambda} + z_{b}^{\mathcal{H},\lambda} \otimes w_{u}^{\mathcal{H},\lambda},  \\
&\bar{w}_{b} \otimes z_{u}  = \bar{w}_{b}^{\mathcal{L},\lambda} \otimes z_{u}^{\mathcal{L},\lambda} + \bar{w}_{b}^{\mathcal{L},\lambda} \otimes z_{u}^{\mathcal{H},\lambda} + \bar{w}_{b}^{\mathcal{H},\lambda} \otimes z_{u}^{\mathcal{L},\lambda} + \bar{w}_{b}^{\mathcal{H},\lambda} \otimes z_{u}^{\mathcal{H},\lambda}. 
\end{align}
\end{subequations}
First, we estimate the products of lower-order terms in a straight-forward manner via the Gagliardo-Nirenberg inequality of $\lVert f \rVert_{L^{4}} \lesssim \lVert f \rVert_{L^{2}}^{\frac{1}{2}} \lVert f \rVert_{H^{1}}^{\frac{1}{2}}$, 
\begin{align}
& \lVert z_{b}^{\mathcal{L},\lambda} \otimes w_{u}^{\mathcal{L},\lambda} \rVert_{L^{2}} + \lVert \bar{w}_{b}^{\mathcal{L},\lambda} \otimes z_{u}^{\mathcal{L},\lambda} \rVert_{L^{2}}   \nonumber \\
\lesssim& \lVert z_{b}^{\mathcal{L}, \lambda} \rVert_{L^{2}}^{\frac{1}{2}} \lVert z_{b}^{\mathcal{L},\lambda} \rVert_{H^{1}}^{\frac{1}{2}} \lVert w_{u}^{\mathcal{L},\lambda} \rVert_{L^{2}}^{\frac{1}{2}} \lVert w_{u}^{\mathcal{L},\lambda} \rVert_{H^{1}}^{\frac{1}{2}}+ \lVert \bar{w}_{b}^{\mathcal{L},\lambda} \rVert_{L^{2}}^{\frac{1}{2}} \lVert \bar{w}_{b}^{\mathcal{L},\lambda} \rVert_{H^{1}}^{\frac{1}{2}} \lVert z_{u}^{\mathcal{L},\lambda} \rVert_{L^{2}}^{\frac{1}{2}} \lVert z_{u}^{\mathcal{L},\lambda} \rVert_{H^{1}}^{\frac{1}{2}}. \label{est 181}
\end{align}
Next, let us rely on the following Besov space interpolation inequalities:
\begin{equation}\label{est 173} 
\lVert f \rVert_{L^{4}} \lesssim \lVert f \rVert_{L^{2}}^{\frac{1}{2}} \lVert f \rVert_{B_{4,2}^{\frac{1}{2}}}^{\frac{1}{2}}, \hspace{5mm} \lVert f \rVert_{L^{4}} \lesssim \lVert f \rVert_{L^{2}}^{\frac{1}{2}} \lVert f \rVert_{B_{\infty, 2}^{0}}^{\frac{1}{2}}; 
\end{equation} 
both of them follow from definitions of Besov spaces and we leave proofs in the Appendix for completeness.  Now, among the six other terms $ z_{b}^{\mathcal{L},\lambda} \otimes w_{u}^{\mathcal{H},\lambda}$, $z_{b}^{\mathcal{H},\lambda} \otimes w_{u}^{\mathcal{L},\lambda}$,  $z_{b}^{\mathcal{H},\lambda} \otimes w_{u}^{\mathcal{H},\lambda}$, $\bar{w}_{b}^{\mathcal{L},\lambda} \otimes z_{u}^{\mathcal{H},\lambda}$,  $\bar{w}_{b}^{\mathcal{H},\lambda} \otimes z_{u}^{\mathcal{L},\lambda}$, and $\bar{w}_{b}^{\mathcal{H},\lambda} \otimes z_{u}^{\mathcal{H},\lambda}$ in \eqref{est 174}, we work on the terms that do not involve $z_{u}^{\mathcal{H},\lambda}$ or $z_{b}^{\mathcal{H},\lambda}$ as follows:
\begin{align}
& \lVert z_{b}^{\mathcal{L},\lambda} \otimes w_{u}^{\mathcal{H},\lambda} \rVert_{L^{2}} + \lVert \bar{w}_{b}^{\mathcal{H},\lambda} \otimes z_{u}^{\mathcal{L},\lambda} \rVert_{L^{2}} \lesssim \lVert z_{b}^{\mathcal{L},\lambda} \rVert_{L^{4}} \lVert w_{u}^{\mathcal{H},\lambda} \rVert_{L^{4}} + \lVert \bar{w}_{b}^{\mathcal{H},\lambda} \rVert_{L^{4}} \lVert z_{u}^{\mathcal{L},\lambda} \rVert_{L^{4}} \nonumber \\
\overset{\eqref{est 173}}{\lesssim}& \lVert  z_{b}^{\mathcal{L}, \lambda} \rVert_{L^{2}}^{\frac{1}{2}} \lVert z_{b}^{\mathcal{L},\lambda} \rVert_{H^{1}}^{\frac{1}{2}} \lVert w_{u}^{\mathcal{H},\lambda} \rVert_{L^{2}}^{\frac{1}{2}} \lVert w_{u}^{\mathcal{H},\lambda} \rVert_{B_{4,2}^{\frac{1}{2}}}^{\frac{1}{2}} + \lVert \bar{w}_{b}^{\mathcal{H},\lambda} \rVert_{L^{2}}^{\frac{1}{2}} \lVert \bar{w}_{b}^{\mathcal{H},\lambda} \rVert_{B_{4,2}^{\frac{1}{2}}}^{\frac{1}{2}} \lVert z_{u}^{\mathcal{L},\lambda} \rVert_{L^{2}}^{\frac{1}{2}} \lVert z_{u}^{\mathcal{L},\lambda} \rVert_{H^{1}}^{\frac{1}{2}}.  \label{est 182}
\end{align}
For the products involving higher-order terms in \eqref{est 17}, we estimate 
\begin{align}
& \lVert z_{b}^{\mathcal{H},\lambda} \otimes w_{u}^{\mathcal{L},\lambda} \rVert_{L^{2}} + \lVert z_{b}^{\mathcal{H},\lambda} \otimes w_{u}^{\mathcal{H},\lambda} \rVert_{L^{2}} + \lVert \bar{w}_{b}^{\mathcal{L},\lambda} \otimes z_{u}^{\mathcal{H},\lambda} \rVert_{L^{2}} + \lVert \bar{w}_{b}^{\mathcal{H},\lambda} \otimes z_{u}^{\mathcal{H},\lambda} \rVert_{L^{2}} \label{est 178}\\
\overset{\eqref{est 173}}{\lesssim}& \lVert (z_{u}^{\mathcal{H},\lambda}, z_{b}^{\mathcal{H},\lambda}) \rVert_{L^{2}}^{\frac{1}{2}} \lVert (z_{u}^{\mathcal{H},\lambda}, z_{b}^{\mathcal{H},\lambda}) \rVert_{B_{\infty,2}^{0}}^{\frac{1}{2}} \nonumber\\
& \times [ \lVert w_{u}^{\mathcal{L},\lambda} \rVert_{L^{2}}^{\frac{1}{2}} \lVert w_{u}^{\mathcal{L},\lambda} \rVert_{H^{1}}^{\frac{1}{2}} + \lVert w_{u}^{\mathcal{H},\lambda} \rVert_{L^{2}}^{\frac{1}{2}} \lVert w_{u}^{\mathcal{H},\lambda} \rVert_{B_{4,2}^{\frac{1}{2}}}^{\frac{1}{2}} + \lVert w_{b}^{\mathcal{L},\lambda} \rVert_{L^{2}}^{\frac{1}{2}} \lVert w_{b}^{\mathcal{L},\lambda} \rVert_{H^{1}}^{\frac{1}{2}} + \lVert \bar{w}_{b}^{\mathcal{H},\lambda} \rVert_{L^{2}}^{\frac{1}{2}} \lVert \bar{w}_{b}^{\mathcal{H},\lambda} \rVert_{B_{4,2}^{\frac{1}{2}}}^{\frac{1}{2}}]. \nonumber 
\end{align}
Now, as an example we estimate by Bernstein's and H$\ddot{\mathrm{o}}$lder's inequalities, for some $N_{1}$ from \eqref{est 152},  
\begin{align}
&\lVert z_{u} \rVert_{B_{\infty,2}^{0}} \overset{\eqref{est 172} \eqref{est 167}\eqref{est 161}}{\leq}  \lVert z_{u}^{\mathcal{L},\lambda} \rVert_{B_{\infty,2}^{0}} + \lVert - \mathbb{P}_{L} \divergence ( z_{u} \circlesign{\prec}_{s} \mathcal{H}_{\lambda} Q_{u} - z_{b} \circlesign{\prec}_{s} \mathcal{H}_{\lambda} Q_{b}) \rVert_{B_{\infty,2}^{0}} \nonumber \\
&\overset{\eqref{est 152}}{\leq} C_{1} \lVert z_{u}^{\mathcal{L},\lambda} \rVert_{H^{1}} + C_{2} [ \left( \sum_{m\geq -1} \lvert 2^{m} \sum_{l: l \leq m + N_{1} - 2} \lVert \Delta_{l} z_{u} \otimes_{s} \Delta_{m} \mathcal{H}_{\lambda} Q_{u} \rVert_{L^{\infty}} \rvert^{2} \right)^{\frac{1}{2}} \nonumber\\
& \hspace{35mm} + \left( \sum_{m\geq -1} \lvert 2^{m} \sum_{l: l \leq m + N_{1} - 2} \lVert \Delta_{l} z_{b} \otimes_{s} \Delta_{m} \mathcal{H}_{\lambda} Q_{b} \rVert_{L^{\infty}} \rvert^{2} \right)^{\frac{1}{2}} ] \nonumber \\
&\overset{\eqref{est 33}}{\leq} C_{1} \lVert z_{u}^{\mathcal{L},\lambda} \rVert_{H^{1}} + C_{2} [ \lambda^{-\frac{\kappa}{2}} \lVert Q_{u} \rVert_{\mathcal{C}^{2- \frac{3\kappa}{2}}} \lVert z_{u} \rVert_{B_{\infty,2}^{-1+ 2 \kappa}} + \lambda^{-\frac{\kappa}{2}} \lVert Q_{b} \rVert_{\mathcal{C}^{2-\frac{3\kappa}{2}}} \lVert z_{b} \rVert_{B_{\infty,2}^{-1 + 2 \kappa}}] \nonumber \\
&\overset{\eqref{est 39}}{\leq}  C_{1} \lVert z_{u}^{\mathcal{L},\lambda} \rVert_{H^{1}}  + C_{2} \lambda^{-\frac{\kappa}{2}} N_{T}^{\kappa} ( \lVert z_{u} \rVert_{B_{\infty,2}^{0}} + \lVert z_{b} \rVert_{B_{\infty,2}^{0}}) \leq C_{1} \lVert z_{u}^{\mathcal{L},\lambda} \rVert_{H^{1}}  + \frac{\lVert z_{u} \rVert_{B_{\infty,2}^{0}} + \lVert z_{b} \rVert_{B_{\infty,2}^{0}}}{4} \label{est 175} 
\end{align} 
for $\lambda \geq \bar{\lambda} \vee \lambda_{T}$ where $\lambda_{T}$ is from Definition \ref{Definition 5.1} (1) and $\bar{\lambda} \gg  1$ is taken sufficiently large. Repeating identical computations to \eqref{est 175} for $\lVert z_{b} \rVert_{B_{\infty,2}^{0}}$ gives in sum 
\begin{equation*}
\lVert z_{u} \rVert_{B_{\infty,2}^{0}} + \lVert z_{b} \rVert_{B_{\infty,2}^{0}} \leq C_{1} \lVert (z_{u}^{\mathcal{L},\lambda}, z_{b}^{\mathcal{L},\lambda}) \rVert_{H^{1}} + \frac{ \lVert z_{u} \rVert_{B_{\infty,2}^{0}} + \lVert z_{b} \rVert_{B_{\infty,2}^{0}}}{2} 
\end{equation*}  
and therefore 
\begin{equation}\label{est 176} 
\lVert z_{u} \rVert_{B_{\infty,2}^{0}} + \lVert z_{b} \rVert_{B_{\infty,2}^{0}} \leq 2 C_{1} \lVert (z_{u}^{\mathcal{L},\lambda}, z_{b}^{\mathcal{L},\lambda}) \rVert_{H^{1}}. 
\end{equation} 
Together with Bernstein's inequality, \eqref{est 176} implies 
\begin{equation}\label{est 177} 
\lVert z_{u}^{\mathcal{H}, \lambda} \rVert_{B_{\infty, 2}^{0}} + \lVert z_{b}^{\mathcal{H},\lambda} \rVert_{B_{\infty,2}^{0}} \lesssim \lVert (z_{u}^{\mathcal{L},\lambda}, z_{b}^{\mathcal{L},\lambda}) \rVert_{H^{1}}.
\end{equation} 
Therefore, if we define 
\begin{equation}\label{est 179} 
\lvert\lVert \phi \vert\rVert_{\lambda}  \triangleq \lVert \phi^{\mathcal{L},\lambda}  \rVert_{H^{1}} + \lVert \phi^{\mathcal{H},\lambda} \rVert_{B_{4,2}^{\frac{1}{2} + \kappa}}, 
\end{equation} 
then applying \eqref{est 177} and \eqref{est 179} to \eqref{est 178} gives us 
\begin{align}
& \lVert z_{b}^{\mathcal{H},\lambda} \otimes w_{u}^{\mathcal{L},\lambda} \rVert_{L^{2}} + \lVert z_{b}^{\mathcal{H},\lambda} \otimes w_{u}^{\mathcal{H},\lambda} \rVert_{L^{2}} + \lVert \bar{w}_{b}^{\mathcal{L},\lambda} \otimes z_{u}^{\mathcal{H},\lambda} \rVert_{L^{2}} + \lVert \bar{w}_{b}^{\mathcal{H},\lambda} \otimes z_{u}^{\mathcal{H},\lambda} \rVert_{L^{2}} \nonumber\\
\lesssim&  \lVert (z_{u}^{\mathcal{H},\lambda}, z_{b}^{\mathcal{H},\lambda}) \rVert_{L^{2}}^{\frac{1}{2}} \lVert (z_{u}^{\mathcal{L},\lambda}, z_{b}^{\mathcal{L},\lambda}) \rVert_{H^{1}}^{\frac{1}{2}} \lVert (w_{u}^{\mathcal{L},\lambda}, w_{u}^{\mathcal{H},\lambda}, \bar{w}_{b}^{\mathcal{L},\lambda} \bar{w}_{b}^{\mathcal{H},\lambda}) \rVert_{L^{2}}^{\frac{1}{2}}  \lVert \lvert (w_{u}, \bar{w}_{b}) \rvert \rVert_{\lambda}^{\frac{1}{2}}.  \label{est 183} 
\end{align}
Hence, we can now deduce 
\begin{align}
& - \langle z_{b}^{\mathcal{L},\lambda}, \mathbb{P}_{L} \divergence (w_{b} \otimes w_{u} - \bar{w}_{b} \otimes \bar{w}_{u} - 2 z_{b} \otimes_{a} Y_{u}) \rangle(t) \label{est 184} \\
& \hspace{1mm} \overset{\eqref{est 180}\eqref{est 174}}{\lesssim} \lVert z_{b}^{\mathcal{L}, \lambda} (t)\rVert_{H^{1}} (\lVert[ z_{b}^{\mathcal{L},\lambda} \otimes w_{u}^{\mathcal{L},\lambda}+ z_{b}^{\mathcal{L},\lambda} \otimes w_{u}^{\mathcal{H},\lambda} + z_{b}^{\mathcal{H},\lambda} \otimes w_{u}^{\mathcal{L},\lambda} + z_{b}^{\mathcal{H},\lambda} \otimes w_{u}^{\mathcal{H},\lambda}](t) \rVert_{L^{2}}  \nonumber \\
& \hspace{6mm}  + \lVert [\bar{w}_{b}^{\mathcal{L},\lambda} \otimes z_{u}^{\mathcal{L},\lambda} + \bar{w}_{b}^{\mathcal{L},\lambda} \otimes z_{u}^{\mathcal{H},\lambda} + \bar{w}_{b}^{\mathcal{H},\lambda} \otimes z_{u}^{\mathcal{L},\lambda} + \bar{w}_{b}^{\mathcal{H},\lambda} \otimes z_{u}^{\mathcal{H},\lambda} ](t)\rVert_{L^{2}} + \lVert z_{b}(t) \rVert_{L^{2}} N_{T}^{\kappa})   \nonumber \\
& \hspace{1mm} \overset{\eqref{est 181} \eqref{est 182} \eqref{est 183}}{\leq} \frac{\nu }{64} \lVert (z_{u}^{\mathcal{L},\lambda}, z_{b}^{\mathcal{L},\lambda}) (t)\rVert_{\dot{H}^{1}}^{2} \nonumber \\
& \hspace{6mm} + C(N_{T}^{\kappa}) \lVert (z_{u}^{\mathcal{L}, \lambda},z_{b}^{\mathcal{L},\lambda}, z_{u}^{\mathcal{H},\lambda}, z_{b}^{\mathcal{H},\lambda}) (t)\rVert_{L^{2}}^{2} \left(\lVert (w_{u}^{\mathcal{L},\lambda}, w_{u}^{\mathcal{H},\lambda} \bar{w}_{b}^{\mathcal{L},\lambda}, \bar{w}_{b}^{\mathcal{H},\lambda}) (t)\rVert_{L^{2}}^{2} \lvert \lVert (w_{u},\bar{w}_{b}) (t)\rvert\rVert_{\lambda}^{2} + 1 \right). \nonumber 
\end{align}
Analogous computations to \eqref{est 184} on similar terms in $\RomanIII_{4}$ of \eqref{est 164d} give us 
\begin{align} 
\RomanIII_{4} \leq& \frac{\nu }{16} \lVert (z_{u}^{\mathcal{L},\lambda}, z_{b}^{\mathcal{L},\lambda}) (t)\rVert_{\dot{H}^{1}}^{2} + C(N_{T}^{\kappa}) \lVert (z_{u}^{\mathcal{L},\lambda}, z_{u}^{\mathcal{H},\lambda},z_{b}^{\mathcal{L},\lambda}, z_{b}^{\mathcal{H},\lambda} ) (t)\rVert_{L^{2}}^{2} \label{est 189}\\
& \times \left(\lVert (w_{u}^{\mathcal{L},\lambda}, w_{u}^{\mathcal{H},\lambda},w_{b}^{\mathcal{L},\lambda}, w_{b}^{\mathcal{H},\lambda} ,\bar{w}_{u}^{\mathcal{L},\lambda}, \bar{w}_{u}^{\mathcal{H},\lambda}, \bar{w}_{b}^{\mathcal{L},\lambda}, \bar{w}_{b}^{\mathcal{H},\lambda}) (t)\rVert_{L^{2}}^{2} \lVert \lvert (w_{u},w_{b}, \bar{w}_{u}, \bar{w}_{b}) (t)\rvert\rVert_{\lambda}^{2} + 1 \right). \nonumber 
\end{align} 
At last, applying  \eqref{est 186}, \eqref{est 187}, \eqref{est 188}, and \eqref{est 189} to \eqref{est 185} results in 
\begin{align}
& \frac{1}{2} \partial_{t} \lVert (z_{u}^{\mathcal{L}, \lambda}, z_{b}^{\mathcal{L},\lambda}) (t)\rVert_{L^{2}}^{2}  \leq  - \frac{3\nu}{4} \lVert ( z_{u}^{\mathcal{L},\lambda}, z_{b}^{\mathcal{L}, \lambda} ) (t)\rVert_{\dot{H}^{1}}^{2} + C(\lambda, N_{T}^{\kappa}) \lVert (z_{u}^{\mathcal{L}, \lambda}, z_{b}^{\mathcal{L},\lambda}) (t)\rVert_{L^{2}}^{2}  \label{est 194} \\
&\hspace{4mm} + C(N_{T}^{\kappa},\lambda) [ ( \lVert (z_{u}^{\mathcal{H},\lambda}, z_{b}^{\mathcal{H},\lambda}) \rVert_{H^{2\kappa}} + \lVert (z_{u},z_{b}) \rVert_{H^{2\kappa}})(t) (\lVert (w_{u},w_{b}) \rVert_{H^{2\kappa}} + \lVert (\bar{w}_{u},\bar{w}_{b}) \rVert_{H^{2\kappa}} + 1)(t)]^{2} \nonumber \\
& \hspace{4mm} + C(N_{T}^{\kappa}) \lVert (z_{u}^{\mathcal{L},\lambda}, z_{u}^{\mathcal{H},\lambda},z_{b}^{\mathcal{L},\lambda}, z_{b}^{\mathcal{H},\lambda} ) (t)\rVert_{L^{2}}^{2} \nonumber \\
& \hspace{8mm} \times \left( \lVert (w_{u}^{\mathcal{L},\lambda}, w_{u}^{\mathcal{H},\lambda}, w_{b}^{\mathcal{L},\lambda}, w_{b}^{\mathcal{H},\lambda}, \bar{w}_{u}^{\mathcal{L},\lambda}, \bar{w}_{u}^{\mathcal{H},\lambda}, \bar{w}_{b}^{\mathcal{L},\lambda}, \bar{w}_{b}^{\mathcal{H},\lambda}) (t)\rVert_{L^{2}}^{2} \lVert \lvert (w_{u},w_{b}, \bar{w}_{u}, \bar{w}_{b}) (t)\rvert\rVert_{\lambda}^{2} + 1 \right). \nonumber 
\end{align}
For any $s \in [0, 1- 2 \kappa)$ where $\kappa \in (0, \frac{1}{2})$, e.g. 
\begin{align*}
 \lVert z_{b} (t)\rVert_{H^{s}} &\overset{\eqref{est 172} \eqref{est 167} \eqref{est 161}}{\leq}  \lVert z_{b}^{\mathcal{L},\lambda} (t)\rVert_{H^{s}} + C \lVert (z_{b} \circlesign{\prec}_{a} \mathcal{H}_{\lambda} Q_{u} - z_{u} \circlesign{\prec}_{a} \mathcal{H}_{\lambda} Q_{b})(t) \rVert_{H^{s+1}}   \nonumber \\
& \hspace{10mm} \overset{\eqref{est 40c} \eqref{est 68}}{\leq} \lVert z_{b}^{\mathcal{L},\lambda} (t)\rVert_{H^{s}} + C \lambda^{-\frac{\kappa}{2}}  [ \lVert X_{u}  \rVert_{C_{t}\mathcal{C}^{-\kappa}} +  \lVert X_{b}  \rVert_{C_{t}\mathcal{C}^{-\kappa}}] ( \lVert z_{u} \rVert_{H^{s}} + \lVert z_{b} \rVert_{H^{s}})(t)  \nonumber \\
& \hspace{10mm} \leq \lVert z_{b}^{\mathcal{L},\lambda} (t)\rVert_{H^{s}} + \frac{ \lVert z_{u} (t)\rVert_{H^{s}}+  \lVert z_{b} (t)\rVert_{H^{s}}}{4}
\end{align*}
for $\lambda \geq \bar{\lambda} (\alpha, \kappa, T) \vee \lambda_{T}$ where $\bar{\lambda} (\alpha,\kappa,T) \geq 1$ is taken to be sufficiently large.  Similarly, we can compute $\lVert z_{u} \rVert_{H^{s}} \leq \lVert z_{u}^{\mathcal{L},\lambda} \rVert_{H^{s}} + \frac{ \lVert z_{u} \rVert_{H^{s}} + \lVert z_{b} \rVert_{H^{s}}}{4}$ so that for any $s \in [0, 1 - 2 \kappa)$ where $\kappa \in (0, \frac{1}{2})$ 
\begin{equation}\label{est 190} 
\lVert z_{u} \rVert_{H^{s}} + \lVert z_{b} \rVert_{H^{s}}  \leq 2(\lVert z_{u}^{\mathcal{L},\lambda} \rVert_{H^{s}} + \lVert z_{b}^{\mathcal{L},\lambda} \rVert_{H^{s}}).
\end{equation} 
On the other hand, e.g. if we let 
\begin{equation}\label{est 191} 
M_{T} \triangleq \lVert (w_{u}^{\mathcal{L},\lambda}, w_{u}^{\mathcal{H},\lambda}, w_{b}^{\mathcal{L},\lambda}, w_{b}^{\mathcal{H},\lambda}) \rVert_{L_{T}^{\infty} L_{x}^{2}} + \lVert ( \bar{w}_{u}^{\mathcal{L},\lambda}, \bar{w}_{u}^{\mathcal{H},\lambda}, \bar{w}_{b}^{\mathcal{L},\lambda}, \bar{w}_{b}^{\mathcal{H},\lambda}) \rVert_{L_{T}^{\infty} L_{x}^{2}} < \infty, 
\end{equation} 
then we can estimate by \eqref{est 161} and the fact that $B_{\infty,2}^{s} \subset W^{s,p}$ for all $p\in [2,\infty)$ and $s \in\mathbb{R}$ (see e.g. \cite[p. 152]{BL76}), 
\begin{align}
&\lVert (w_{u}, w_{b})(t) \rVert_{H^{2\kappa}} \label{est 192} \\
\lesssim& [\lVert (w_{u}^{\mathcal{L},\lambda}, w_{b}^{\mathcal{L},\lambda}) \rVert_{L^{2}}^{1-2\kappa} \lVert (w_{u}^{\mathcal{L},\lambda}, w_{b}^{\mathcal{L},\lambda}) \rVert_{H^{1}}^{2\kappa} + \lVert (w_{u}^{\mathcal{H},\lambda}, w_{b}^{\mathcal{H},\lambda}) \rVert_{L^{2}}^{\frac{1-\kappa}{1+ \kappa}} \lVert (w_{u}^{\mathcal{H},\lambda}, w_{b}^{\mathcal{H},\lambda}) \rVert_{W^{4, \frac{1}{2} + \kappa}}^{\frac{2\kappa}{1+ \kappa}}] (t) \nonumber \\
& \hspace{13mm} \overset{\eqref{est 191} \eqref{est 179}}{\lesssim}  M_{T}^{1- 2\kappa} \lVert \lvert (w_{u}, w_{b})(t) \rvert \rVert_{\lambda}^{2\kappa} + M_{T}^{\frac{1- \kappa}{1+ \kappa}} \lVert \lvert (w_{u}, w_{b})(t) \rvert \rVert_{\lambda}^{\frac{2\kappa}{1+ \kappa}}. \nonumber 
\end{align}   
Identically to \eqref{est 192} we can estimate 
\begin{equation*}
\lVert (\bar{w}_{u}, \bar{w}_{b}) (t)\rVert_{H^{2\kappa}} \lesssim M_{T}^{1- 2\kappa} \lVert \lvert (\bar{w_{u}}, \bar{w_{b}}) (t)\rvert \rVert_{\lambda}^{2\kappa} + M_{T}^{\frac{1- \kappa}{1+ \kappa}} \lVert \lvert (\bar{w}_{u}, \bar{w}_{b}) (t)\rvert \rVert_{\lambda}^{\frac{2\kappa}{1+ \kappa}}. 
\end{equation*} 
At last, the claimed uniqueness follows from the following differential inequality: for $\kappa \in (0, \frac{1}{6})$ sufficiently small 
\begin{align*}
& \frac{1}{2} \partial_{t} \lVert (z_{u}^{\mathcal{L}, \lambda}, z_{b}^{\mathcal{L},\lambda}) (t)\rVert_{L^{2}}^{2} \\
&\overset{\eqref{est 194}\eqref{est 172} \eqref{est 191} \eqref{est 190}}{\leq}  - \frac{3\nu }{4} \lVert ( z_{u}^{\mathcal{L},\lambda}, z_{b}^{\mathcal{L}, \lambda} ) (t)\rVert_{\dot{H}^{1}}^{2} + C(\lambda, N_{T}^{\kappa}) \lVert (z_{u}^{\mathcal{L}, \lambda}, z_{b}^{\mathcal{L},\lambda}) (t)\rVert_{L^{2}}^{2}   \nonumber \\
&\hspace{10mm} + C(N_{T}^{\kappa},\lambda) [  \lVert (z_{u}^{\mathcal{L},\lambda},z_{b}^{\mathcal{L},\lambda}) \rVert_{H^{2\kappa}}(t)  (\lVert (w_{u},w_{b}) \rVert_{H^{2\kappa}} + \lVert (\bar{w}_{u},\bar{w}_{b}) \rVert_{H^{2\kappa}} + 1)(t)]^{2} \nonumber \\
& \hspace{10mm} + C(M_{T}, \lambda, N_{T}^{\kappa}) \lVert (z_{u}^{\mathcal{L},\lambda},z_{b}^{\mathcal{L},\lambda})(t) \rVert_{L^{2}}^{2}  \left(\lVert \lvert (w_{u},w_{b}, \bar{w}_{u}, \bar{w}_{b})(t) \rvert\rVert_{\lambda}^{2} + 1 \right) \nonumber \\
&\overset{\eqref{est 192} }{\leq} - \frac{\nu }{2} \lVert  (z_{u}^{\mathcal{L},\lambda}, z_{b}^{\mathcal{L},\lambda}) (t)\rVert_{\dot{H}^{1}}^{2} + C(M_{T}, N_{T}^{\kappa}, \lambda)  \lVert (z_{u}^{\mathcal{L},\lambda}, z_{b}^{\mathcal{L},\lambda}) (t)\rVert_{L^{2}}^{2} ( \lVert \lvert (w_{u}, w_{b}, \bar{w}_{u}, \bar{w}_{b})(t) \rvert \rVert_{\lambda} + 1)^{2}. \nonumber 
\end{align*} 
\end{proof} 

We return to $\mathcal{A}_{t}$, $\mathcal{A}_{t}^{\lambda}$, and the enhanced noise defined in \eqref{est 195}, \eqref{est 97}, and \eqref{est 165}, respectively. We now define the space of enhanced noise $\Theta_{\kappa} \subset  \mathcal{C}^{-1 - \kappa} (\mathbb{T}^{2}; \mathbb{M}^{4}) \times \mathcal{C}^{-2\kappa} (\mathbb{T}^{2}; \mathbb{M}^{4})$ by 
\begin{equation}\label{est 196}
\Theta_{\kappa} \triangleq \overline{ \{ (X_{1}, X_{1} \circlesign{\circ} \left( - \frac{\nu \Delta}{2} + 1\right)^{-1} X_{1} - c): \hspace{1mm} X_{1} \in \mathcal{S}(\mathbb{T}^{2}; \mathbb{M}^{4}), c \in \mathbb{R} \}}, 
\end{equation} 
where the closure is taken w.r.t. the product norm of $\mathcal{C}^{-1-\kappa} (\mathbb{T}^{2}; \mathbb{M}^{4}) \times \mathcal{C}^{-2\kappa} (\mathbb{T}^{2}; \mathbb{M}^{4})$. In order to define 
\begin{equation*}
\mathcal{U}(X) \triangleq \frac{\nu \Delta}{2} \Id + X_{1} 
\end{equation*} 
for any $X = (X_{1}, X_{2}) \in \Theta_{\kappa}$ for some $\kappa > 0$, we define the space of strongly paracontrolled distributions (\cite[Definitions 4.1 and 4.17, Section 4.2]{AC15}) 
\begin{subequations}\label{est 197}
\begin{align}
P \triangleq& \left(-\frac{\nu \Delta}{2} + 1 \right)^{-1} X_{1}, \label{est 197a}\\
\mathcal{X}_{\kappa}(X) \triangleq& \{ \phi \in H^{1-\kappa}: \hspace{1mm} \phi = \phi \circlesign{\prec} P + \phi^{\sharp}, \phi^{\sharp} \in H^{2-2\kappa} \},  \label{est 197b}\\
 \lVert \phi \rVert_{\mathcal{X}_{\kappa}} \triangleq& \lVert \phi \rVert_{H^{1-\kappa}} + \lVert \phi - \phi \circlesign{\prec} P \rVert_{H^{2-2\kappa}}. \label{est 197c}
\end{align}
\end{subequations} 
The following proposition can be obtained from \cite{AC15} (e.g. see \cite[Propositions 4.13 and 4.23, Lemma 4.15]{AC15}); indeed, \cite[Proposition 7.1]{HR23} was a time-dependent higher-dimensional version of such a result from \cite{AC15} and our case is simply the same except $\mathbb{M}^{4}$ rather than $\mathbb{M}^{2}$. 
 
\begin{proposition}\label{Proposition 5.3} 
\rm{(Cf. \cite[Proposition 7.1]{HR23})} Define $\Theta \triangleq \cup_{0 < \kappa < \kappa_{0}}$ with $\Theta_{\kappa}$ defined in \eqref{est 196} and $C_{\text{op}}$ to be a space of closed self-adjoint operators with the graph distance where the convergence in this distance is implied by the convergence in the resolvent sense. Then there exists a $\kappa_{0} > 0$ and a unique map $\mathcal{U}: \hspace{1mm} \Theta \mapsto C_{\text{op}}$ such that the following hold. 
\begin{enumerate}
\item For any smooth $X = (X_{1}, X_{2}) \in [\mathcal{S}(\mathbb{T}^{2}; \mathbb{M}^{4}) \times \mathcal{S}(\mathbb{T}^{2}; \mathbb{M}^{4})] \cap \Theta$ and $\phi \in H^{2}$, $\mathcal{U}$ satisfies 
\begin{equation*}
\mathcal{U} (X) \phi = \frac{\nu\Delta}{2} \phi + X_{1} \circlesign{\prec} \phi + X_{1} \circlesign{\succ} \phi + X_{1} \circlesign{\circ} \phi^{\sharp} + \phi \circlesign{\prec} X_{2} + C^{\circlesign{\circ}} (\phi, P, X_{1}) 
\end{equation*} 
where $P$ is defined in \eqref{est 197a} and 
\begin{equation*}
C^{\circlesign{\circ}} (\phi, P, X_{1}) \triangleq X_{1} \circlesign{\circ} (\phi \circlesign{\prec} P) - \phi \circlesign{\prec} (P \circlesign{\circ} X_{1}). 
\end{equation*} 
If $X_{2} = P \circlesign{\circ} X_{1}$, it follows that $\mathcal{U}(X) \phi = \frac{\nu\Delta }{2} \phi + X_{1} \phi$. 
\item For any $\{X^{n}\}_{n \in\mathbb{N}} \subset \mathcal{S} ( \mathbb{T}^{2}; \mathbb{M}^{4}) \times \mathcal{S}( \mathbb{T}^{2}; \mathbb{M}^{4})$ such that $X^{n} \to X$ in $\Theta_{\kappa}$ as $n\to\infty$ for some $\kappa \in (0, \kappa_{0})$ and $X \in \Theta_{\kappa}$, $\mathcal{U} (X^{n})$ converges to $\mathcal{U}(X)$ in resolvent sense. Moreover, for any $\kappa \in (0, \kappa_{0})$, there exist two continuous maps $m, \textbf{c}: \hspace{1mm} \Theta_{\kappa} \mapsto \mathbb{R}_{+}$ such that 
\begin{equation*}
[m (X), \infty) \subset \rho( \mathcal{U}(X)) \hspace{3mm} \forall \hspace{1mm} X \in \Theta_{\kappa}
\end{equation*} 
where $\rho( \mathcal{U}(X))$ is the resolvent set of an operator $\mathcal{U}(X)$ with an upper bound 
\begin{equation*}
\lVert ( - \mathcal{U} (X) + m )^{-1} \phi \rVert_{\mathcal{X}_{\kappa}} \leq \textbf{c}(X) \lVert \phi \rVert_{L^{2}} \hspace{3mm} \forall \hspace{1mm} m \geq \textbf{m}(X). 
\end{equation*} 
\end{enumerate} 
\end{proposition} 

Now we write for $\xi_{l}$ where $l \in  \{u,b\}$, 
\begin{equation*}
\mathbb{P}_{L} \mathbb{P}_{\neq 0} \xi_{l}(t,x) \overset{\eqref{est 200}}{=} \sum_{m\in\mathbb{Z}^{2} \setminus \{0\}} e^{i2\pi  m \cdot x} \left( \frac{ \partial_{t} \beta_{l,1} (t,m) m_{2} - \partial_{t} \beta_{l,2} (t,m) m_{1}}{\lvert m \rvert^{2}} \right) m^{\bot} 
\end{equation*} 
where $\{\beta_{l,i}(m)\}_{i\in \{1,2\}, m \in \mathbb{Z}_{+}^{2}}$ is a family of $\mathbb{C}$-valued two-sided Brownian motions such that 
\begin{align*}
& \mathbb{E} [\partial_{t} \beta_{u,i} (t,m) \partial_{t} \beta_{u,j} (s, m') ] = \delta(t-s) 1_{ \{i=j\}} 1_{\{m = -m'\}}, \\
& \mathbb{E} [ \partial_{t}\beta_{b,i} (t,m) \partial_{t} \beta_{b,j} (s,m') ] = \delta(t-s) 1_{\{i=j\}} 1_{\{ m = -m'\}}, \\
& \mathbb{E} [ \partial_{t} \beta_{u,i} (t,m) \partial_{t} \beta_{b,j} (s, m') ] = 0, 
\end{align*}
(recall \eqref{STWN u} and \eqref{STWN b}). We define for all $t \in (0,\infty)$ and $m \in \mathbb{Z}^{2} \setminus \{0\}$ 
\begin{equation*}
\zeta_{u}(t,m) \triangleq \frac{ \beta_{u,1} (t,m) m_{2} - \beta_{u,2} (t,m)m_{1}}{\lvert m \rvert} \hspace{1mm} \text{ and } \hspace{1mm} \zeta_{b}(t,m) \triangleq \frac{ \beta_{b,1} (t,m) m_{2} - \beta_{b,2} (t,m) m_{1}}{\lvert m \rvert} 
\end{equation*} 
so that 
\begin{subequations}\label{est 211}
\begin{align}
& \mathbb{E} [ \partial_{t} \zeta_{u} (t,m) \partial_{t} \zeta_{u} (s, m') ] = -  \delta (t-s) 1_{\{ m = -m' \}}, \\
&  \mathbb{E} [ \partial_{t} \zeta_{b} (t,m) \partial_{t} \zeta_{b} (s,m') ] = - \delta(t-s) 1_{\{ m =- m' \}}, \hspace{3mm}  \mathbb{E} [ \partial_{t} \zeta_{u} (t,m) \partial_{t} \zeta_{b} (s,m') ] = 0. 
\end{align}
\end{subequations} 
Denote 
\begin{equation*} 
e_{m}(x) \triangleq e^{i2\pi  m \cdot x} \frac{m^{\bot}}{\lvert m \rvert} 
\end{equation*} 
so that 
\begin{equation*}
\mathbb{P}_{L} \mathbb{P}_{\neq 0} \xi_{u}(t,x)  =  \sum_{m\in\mathbb{Z}^{2}\setminus \{0\}} \partial_{t} \zeta_{u} (t,m) e_{m}(x)  \text{ and }\mathbb{P}_{L} \mathbb{P}_{\neq 0} \xi_{b}(t,x) = \sum_{m\in\mathbb{Z}^{2} \setminus \{0\}} \partial_{t} \zeta_{b}(t,m) e_{m}(x).
\end{equation*} 
Now, by defining for $l \in \{u,b\}$, 
\begin{equation}\label{est 206} 
F_{l}(t,m) \triangleq \int_{0}^{t} e^{-\nu \lvert m \rvert^{2}(t-s)} d \zeta_{l}(s,m), \hspace{3mm} F_{l}^{\lambda} (t,m) \triangleq \int_{0}^{t} e^{-\nu \lvert m \rvert^{2} (t-s)} \mathfrak{l} \left( \frac{ \lvert m \rvert}{\lambda} \right) d \zeta_{l}(s,m) 
\end{equation} 
where $\mathfrak{l}$ is the projection onto lower frequencies from Definition \ref{Definition 3.1}, we can solve from \eqref{est 14} and \eqref{est 15} for $l \in \{u,b\}$, 
\begin{subequations}\label{est 209} 
\begin{align}
&X_{l}(t,x) = \sum_{m\in\mathbb{Z}^{2} \setminus \{0\}} F_{u}(t,m) e_{m}(x), \hspace{3mm}\mathcal{L}_{\lambda} X_{l} (t,x) \overset{\eqref{est 206}}{=} \sum_{m\in\mathbb{Z}^{2} \setminus \{0\}} e_{m}(x) F_{l}^{\lambda}(t,m), \label{est 209a} \\
& \left( - \frac{\nu \Delta}{2} + 1 \right)^{-1} \mathcal{L}_{\lambda} X_{l} = \sum_{m\in\mathbb{Z}^{2} \setminus \{0\}} e_{m}(x) F_{l}^{\lambda} (t,m) \left( \frac{\nu  \lvert m \rvert^{2}}{2} + 1 \right)^{-1}. \label{est 209b} 
\end{align}
\end{subequations} 

\begin{proposition}\label{Proposition 5.4} 
\rm{(Cf. \cite[Lemma 7.2]{HR23})} For any $\kappa > 0$, define $\Theta_{\kappa}$ by \eqref{est 196}, $P^{\lambda}$ and $r_{\lambda}$ by \eqref{est 140}. Then, for any $t \geq 0$, there exists a distribution $\nabla_{\text{spec}} (X_{u}, X_{b})(t) \diamond P_{t} \in \mathcal{C}^{-\kappa} (\mathbb{T}^{2}; \mathbb{M}^{4})$ such that   
\begin{align}
&\left( \nabla_{\text{spec}} (\mathcal{L}_{\lambda^{n}} X_{u}, \mathcal{L}_{\lambda^{n}} X_{b}), ( \nabla_{\text{spec}} (\mathcal{L}_{\lambda^{n}} X_{u}, \mathcal{L}_{\lambda^{n}} X_{b} )) \circlesign{\circ} P^{\lambda^{n}} - r_{\lambda^{n}} \Id\right) \nonumber \\
& \hspace{50mm}  \to ( \nabla_{\text{spec}}(X_{u}, X_{b}), \nabla_{\text{spec}} (X_{u}, X_{b}) \diamond P)  \label{est 228} 
\end{align} 
as $n\to\infty$ both in $L^{p}(\Omega; C_{\text{loc}} (\mathbb{R}_{+}; \Theta_{\kappa}))$ for any $p \in [1,\infty)$ and $\mathbb{P}$-a.s. Finally, there exists a constant $c > 0$ such for all $\lambda \geq 1$, 
\begin{equation}\label{est 232}
r_{\lambda}(t) \leq c \ln (\lambda) 
\end{equation} 
uniformly over all $t \geq 0$. 
\end{proposition} 

\begin{proof}[Proof of Proposition \ref{Proposition 5.4}]
We focus on the more difficult task of proving the convergence of $( \nabla_{\text{spec}} (\mathcal{L}_{\lambda^{n}} X_{u}, \mathcal{L}_{\lambda^{n}} X_{b} )) \circlesign{\circ} P^{\lambda^{n}} - r_{\lambda^{n}} \Id \to \nabla_{\text{spec}} (X_{u}, X_{b}) \diamond P$ as $n\to\infty$. For brevity, we denote $X_{\alpha,\lambda} \triangleq \mathcal{L}_{\lambda} X_{\alpha}$ for $\alpha \in \{u,b\}$. Considering the 16 entries within 
\begin{align}
&4 ( \nabla_{\text{spec}} (\mathcal{L}_{\lambda} X_{u}, \mathcal{L}_{\lambda} X_{b}) \circlesign{\circ} P^{\lambda}) \label{est 212} \\ 
&\overset{\eqref{est 140}}{=} \left(2 \nabla_{\text{spec}} (\mathcal{L}_{\lambda} X_{u}, \mathcal{L}_{\lambda} X_{b}) \circlesign{\circ} \left( - \frac{\nu \Delta}{2} + 1 \right)^{-1} 2\nabla_{\text{spec}} (\mathcal{L}_{\lambda} X_{u}, \mathcal{L}_{\lambda} X_{b}) \right) \nonumber \\
&\overset{\eqref{est 64}}{=}  
\begin{pmatrix}
\begin{pmatrix}
\partial_{1} X_{u,\lambda}^{1} + \partial_{1} X_{u,\lambda}^{1} & \partial_{1} X_{u,\lambda}^{2} + \partial_{2} X_{u,\lambda}^{1} \\
\partial_{2} X_{u,\lambda}^{1} + \partial_{1} X_{u,\lambda}^{2} & \partial_{2} X_{u,\lambda}^{2} + \partial_{2} X_{u,\lambda}^{2} 
\end{pmatrix}  & 
\begin{pmatrix}
0 & \partial_{1} X_{b,\lambda}^{2} - \partial_{2} X_{b,\lambda}^{1} \\
\partial_{2} X_{b,\lambda}^{1} - \partial_{1} X_{b,\lambda}^{2} & 0 
\end{pmatrix} \\
- \begin{pmatrix}
0 & \partial_{1} X_{b,\lambda}^{2} - \partial_{2} X_{b,\lambda}^{1} \\
\partial_{2} X_{b,\lambda}^{1} - \partial_{1} X_{b,\lambda}^{2} & 0 
\end{pmatrix} & - \begin{pmatrix}
\partial_{1} X_{u,\lambda}^{1} + \partial_{1} X_{u,\lambda}^{1} & \partial_{1} X_{u,\lambda}^{2} + \partial_{2} X_{u,\lambda}^{1} \\
\partial_{2} X_{u,\lambda}^{1} + \partial_{1} X_{u,\lambda}^{2} & \partial_{2} X_{u,\lambda}^{2} + \partial_{2} X_{u,\lambda}^{2} 
\end{pmatrix} 
\end{pmatrix} \nonumber \\
&\circlesign{\circ} \left(-  \frac{\nu \Delta}{2} + 1 \right)^{-1} 
\begin{pmatrix}
\begin{pmatrix}
\partial_{1} X_{u,\lambda}^{1} + \partial_{1} X_{u,\lambda}^{1} & \partial_{1} X_{u,\lambda}^{2} + \partial_{2} X_{u,\lambda}^{1} \\
\partial_{2} X_{u,\lambda}^{1} + \partial_{1} X_{u,\lambda}^{2} & \partial_{2} X_{u,\lambda}^{2} + \partial_{2} X_{u,\lambda}^{2} 
\end{pmatrix}  & 
\begin{pmatrix}
0 & \partial_{1} X_{b,\lambda}^{2} - \partial_{2} X_{b,\lambda}^{1} \\
\partial_{2} X_{b,\lambda}^{1} - \partial_{1} X_{b,\lambda}^{2} & 0 
\end{pmatrix} \\
- \begin{pmatrix}
0 & \partial_{1} X_{b,\lambda}^{2} - \partial_{2} X_{b,\lambda}^{1} \\
\partial_{2} X_{b,\lambda}^{1} - \partial_{1} X_{b,\lambda}^{2} & 0 
\end{pmatrix} & - \begin{pmatrix}
\partial_{1} X_{u,\lambda}^{1} + \partial_{1} X_{u,\lambda}^{1} & \partial_{1} X_{u,\lambda}^{2} + \partial_{2} X_{u,\lambda}^{1} \\
\partial_{2} X_{u,\lambda}^{1} + \partial_{1} X_{u,\lambda}^{2} & \partial_{2} X_{u,\lambda}^{2} + \partial_{2} X_{u,\lambda}^{2} 
\end{pmatrix} 
\end{pmatrix}, \nonumber 
\end{align} 
we see a common form of $\partial_{i} X_{\alpha, \lambda}^{j} \circ \left(-\frac{\nu \Delta}{2} + 1 \right)^{-1} \partial_{l} X_{\gamma, \lambda}^{m}$ for $i, j, l, m \in \{1,2\}$ and $\alpha, \gamma \in \{u,b\}$. Let us list several of them here and leave the rest in the Appendix for completeness.  
\begin{subequations}\label{est 213} 
\begin{align}  
& (1,2): \hspace{1mm}  ( \partial_{1} X_{u,\lambda}^{1} + \partial_{1} X_{u,\lambda}^{1}) \circ \left( - \frac{\nu \Delta}{2} + 1 \right)^{-1} ( \partial_{1} X_{u,\lambda}^{2} + \partial_{2} X_{u,\lambda}^{1})  \nonumber \\
& \hspace{30mm} + ( \partial_{1} X_{u,\lambda}^{2} + \partial_{2} X_{u,\lambda}^{1}) \circ \left(-\frac{\nu \Delta}{2} + 1 \right)^{-1}  (\partial_{2} X_{u,\lambda}^{2} + \partial_{2} X_{u,\lambda}^{2}), \label{est 213b}\\
& (4,4): \hspace{1mm}  -(\partial_{2} X_{b,\lambda}^{1} - \partial_{1} X_{b,\lambda}^{2}) \circ \left( -\frac{\nu \Delta}{2}+ 1 \right)^{-1} (\partial_{1} X_{b,\lambda}^{2} - \partial_{2} X_{b,\lambda}^{1}) \nonumber \\
& \hspace{30mm} + ( \partial_{2} X_{u,\lambda}^{1} + \partial_{1} X_{u,\lambda}^{2}) \circ \left( -\frac{\nu \Delta}{2} + 1\right)^{-1} (\partial_{1} X_{u,\lambda}^{2} + \partial_{2} X_{u,\lambda}^{1}) \nonumber \\
& \hspace{30mm } + (\partial_{2} X_{u,\lambda}^{2} + \partial_{2} X_{u,\lambda}^{2}) \circ \left( -\frac{\nu \Delta}{2}+1 \right)^{-1} (\partial_{2} X_{u,\lambda}^{2} + \partial_{2} X_{u,\lambda}^{2}). \label{est 213d} 
\end{align}
\end{subequations} 
Thus, let us define 
\begin{equation}\label{est 210} 
c_{j,m}^{i,l} (k,k') \triangleq -k_{i} k_{j}^{\bot} k_{l}' (k')_{m}^{\bot} \lvert k \rvert^{-1} \lvert k' \rvert^{-1}
\end{equation} 
so that we can write for all $i, j, l, m \in \{1,2\}$ and $\alpha, \gamma \in \{u,b\}$, using \eqref{est 208}, \eqref{est 209}, and \eqref{est 210}, 
\begin{align*}
& \partial_{i} X_{\alpha,\lambda}^{j} \circ \left( - \frac{\nu \Delta}{2} + 1 \right)^{-1} \partial_{l} X_{\gamma,\lambda}^{m} = \sum_{k,k' \in \mathbb{Z}^{2}: k' \neq 0, k \neq k', \lvert c-d \rvert \leq 1} e^{i2\pi k \cdot x} \rho_{c}(k-k') \rho_{d}(k') \mathfrak{l} \left( \frac{ \lvert k-k' \rvert}{\lambda} \right) \mathfrak{l} \left( \frac{ \lvert k ' \rvert}{\lambda} \right)  \nonumber \\
& \hspace{40mm} \times F_{\alpha} (t, k-k') F_{\gamma} (t, k') \left( \frac{ \nu \lvert k ' \rvert^{2}}{2} + 1 \right)^{-1} c_{j,m}^{i,l} (k-k', k').  
\end{align*}
Now we are ready to compute the zeroth Wiener chaos (cf. \cite[Section 1.1]{N95}) which are the renormalization constants (e.g. \cite[Equations (136) and (193)]{Y21a}). Because 
\begin{align*} 
\mathbb{E} [ F_{\alpha}(t, k-k') F_{\gamma} (t, k') ] = -\frac{1 - e^{-2\nu  \lvert k' \rvert^{2} t}}{2 \nu \lvert k' \rvert^{2}} 1_{\{ \alpha = \gamma \}} 1_{\{ k- k' = -k' \}}  
\end{align*}
due to \eqref{est 206} and \eqref{est 211}, we deduce 
\begin{align}
& \mathbb{E}\left[ \partial_{i} X_{\alpha,\lambda}^{j} \circ \left( - \frac{\nu \Delta}{2} + 1 \right)^{-1} \partial_{l} X_{\gamma, \lambda}^{m} \right](x) \nonumber\\
=&  \sum_{k \in \mathbb{Z}^{2} \setminus \{0\}}  \mathfrak{l} \left( \frac{ \lvert k \rvert}{\lambda} \right)^{2} \frac{1- e^{-2\nu  \lvert k \rvert^{2} t}}{2 \nu \lvert k \rvert^{4}} \left( \frac{ \nu \lvert k \rvert^{2}}{2} + 1 \right)^{-1} k_{i} k_{j}^{\bot} k_{l} k_{m}^{\bot} 1_{ \{ \alpha = \gamma \}}. \label{est 214} 
\end{align}  
It follows that all except the diagonal entries of \eqref{est 212} vanish. We show examples of the two cases from \eqref{est 213} and leave the rest in the Appendix for completeness. First, the (1,2)-entry from \eqref{est 213b} vanishes as follows: 
\begin{align}
&\mathbb{E} [ 2 \partial_{1} X_{u,\lambda}^{1} \circ \left( - \frac{\nu \Delta}{2} + 1 \right)^{-1} \partial_{1} X_{u,\lambda}^{2} + 2 \partial_{1} X_{u,\lambda}^{1} \circ \left(-\frac{\nu \Delta}{2} + 1 \right)^{-1} \partial_{2} X_{u,\lambda}^{1} \nonumber \\
& + 2 \partial_{1} X_{u,\lambda}^{2} \circ \left(-\frac{\nu \Delta}{2} + 1 \right)^{-1} \partial_{2} X_{u,\lambda}^{2} + 2 \partial_{2} X_{u,\lambda}^{1} \circ \left(-\frac{\nu \Delta}{2} + 1 \right)^{-1} \partial_{2} X_{u,\lambda}^{2} ](t) \nonumber \\
\overset{\eqref{est 214}}{=}& 2 \sum_{k\in\mathbb{Z}^{2} \setminus \{0\}} \frac{ \mathfrak{l} \left( \frac{\lvert k \rvert}{\lambda} \right)^{2} (1- e^{-2\nu  \lvert k \rvert^{2} t})}{2 \nu \lvert k \rvert^{4}} \left( \frac{\nu \lvert k \rvert^{2}}{2} + 1 \right)^{-1} \nonumber \\
& \times [ k_{1} k_{1}^{\bot} k_{1} k_{2}^{\bot} + k_{1} k_{1}^{\bot} k_{2} k_{1}^{\bot} + k_{1} k_{2}^{\bot} k_{2} k_{2}^{\bot} + k_{2} k_{1}^{\bot} k_{2} k_{2}^{\bot} ] = 0. 
\end{align}
On the other hand, due to \eqref{est 213d} and \eqref{est 214}, the (4,4)-entry can be computed as 
\begin{align*}
& \mathbb{E}[ -\partial_{2} X_{b,\lambda}^{1} \circ \left( - \frac{\nu\Delta }{2} + 1 \right)^{-1} \partial_{1} X_{b,\lambda}^{2} + \partial_{2} X_{b,\lambda}^{1} \circ \left(-\frac{\nu\Delta }{2} + 1 \right)^{-1} \partial_{2} X_{b,\lambda}^{1} \nonumber \\
& + \partial_{1} X_{b,\lambda}^{2} \circ \left( - \frac{\nu\Delta }{2} + 1 \right)^{-1} \partial_{1} X_{b,\lambda}^{2} - \partial_{1} X_{b,\lambda}^{2} \circ \left( - \frac{\nu\Delta }{2} + 1 \right)^{-1} \partial_{2} X_{b,\lambda}^{1} \nonumber \\
& + \partial_{2} X_{u,\lambda}^{1} \circ \left( - \frac{\nu\Delta }{2} + 1 \right)^{-1} \partial_{1} X_{u,\lambda}^{2} + \partial_{2} X_{u,\lambda}^{1} \circ \left( - \frac{\nu\Delta }{2} + 1 \right)^{-1} \partial_{2} X_{u,\lambda}^{1} \nonumber \\
& + \partial_{1} X_{u,\lambda}^{2} \circ \left(-\frac{\nu\Delta }{2} + 1 \right)^{-1} \partial_{1} X_{u,\lambda}^{2} +  \partial_{1} X_{u,\lambda}^{2} \circ \left(-\frac{\nu\Delta }{2} + 1 \right)^{-1} \partial_{2} X_{u,\lambda}^{1} \nonumber \\
& + 4 \partial_{2} X_{u,\lambda}^{2} \circ \left( - \frac{\nu\Delta }{2} + 1 \right)^{-1} \partial_{2} X_{u,\lambda}^{2}](t)  \nonumber  \\
=& - \sum_{k \in \mathbb{Z}^{2} \setminus \{0\}} \frac{ \mathfrak{l} \left( \frac{ \lvert k \rvert}{\lambda} \right)^{2} (1- e^{-2 \nu \lvert k \rvert^{2} t})}{2 \lvert k \rvert^{4}} \left( \frac{\nu  \lvert k \rvert^{2}}{2} + 1 \right)^{-1}  \nonumber \\
& \times [ k_{2} k_{1}^{\bot} k_{1} k_{2}^{\bot} - k_{2} k_{1}^{\bot} k_{2} k_{1}^{\bot} - k_{1} k_{2}^{\bot} k_{1} k_{2}^{\bot} + k_{1} k_{2}^{\bot} k_{2} k_{1}^{\bot} \nonumber  \\
& \hspace{5mm} - k_{2} k_{1}^{\bot} k_{1} k_{2}^{\bot} - k_{2} k_{1}^{\bot} k_{2} k_{1}^{\bot} - k_{1} k_{2}^{\bot} k_{1} k_{2}^{\bot} - k_{1} k_{2}^{\bot} k_{2} k_{1}^{\bot} -4 k_{2} k_{2}^{\bot} k_{2} k_{2}^{\bot} ]  \overset{\eqref{est 140b}}{=} 4 r_{\lambda}(t). 
\end{align*}
In the following series of inequalities, some of which are very similar to computations in past works (e.g. \cite{ZZ15}), it suffices to prove the estimate for the $(4,4)$-entry as an example among the four diagonal entries as the other terms can be handled similarly. By defining 
\begin{equation}\label{est 218} 
\psi_{0} (k,k') \triangleq \sum_{\lvert c-d \rvert \leq 1} \rho_{c} (k) \rho_{d}(k'), 
\end{equation} 
we can compute  using properties of Wick products (e.g. \cite{J97})  
\begin{align}
& \mathbb{E} [ \lvert \Delta_{m} ( \nabla_{\text{spec}} ( \mathcal{L}_{\lambda} X_{u}, \mathcal{L}_{\lambda} X_{b}) \circlesign{\circ} P^{\lambda})_{4,4}(t) - r_{\lambda}(t) \rvert^{2} ] \nonumber \\
\lesssim& \sum_{k, k' \in \mathbb{Z}^{2} \setminus \{0\}} \rho_{m}^{2} (k+k') \lvert \psi_{0}(k, k') \rvert^{2} \mathfrak{l} \left( \frac{ \lvert k \rvert}{\lambda}\right)^{2} \mathfrak{l} \left( \frac{ \lvert k' \rvert}{\lambda} \right)^{2} \left( \frac{1}{\lvert k' \rvert^{2} + 2} \right)^{2}  \nonumber \\
\overset{\eqref{est 218}}{\lesssim}&  \sum_{k, k' \in \mathbb{Z}^{2} \setminus \{0\}: \lvert k \rvert \approx 2^{m}, \lvert k' \rvert \gtrsim 2^{m}} \lvert k' \rvert \left( \sum_{c: m \lesssim c} \frac{1}{2^{c}} \right) \left( \frac{1}{\lvert k' \rvert^{2} + 2} \right)^{2} 
\lesssim 2^{m} \sum_{k' \in \mathbb{Z}^{2} \setminus \{0\}: \lvert k' \rvert  \gtrsim 2^{m}} \frac{1}{\lvert k' \rvert^{3}} \approx 1.  \label{est 227} 
\end{align} 
where we used that $\rho_{m}(k), \rho_{c}(k-k')$, and $\rho_{d}(k')$ imply $m\lesssim c$. Thus, we now conclude for any $p \in [2,\infty)$, via Gaussian hypercontractivity theorem (e.g. \cite[Theorem 3.50]{J97}) 
\begin{align*}
& \sup_{\lambda \geq 1} \mathbb{E} [ \lVert ( \nabla_{\text{spec}} ( \mathcal{L}_{\lambda}, X_{u}, \mathcal{L}_{\lambda} X_{b}) \circlesign{\circ} P^{\lambda} )_{4,4} (t) - r_{\lambda}(t) \rVert_{B_{p,p}^{-\kappa}}^{p} ] \nonumber \\
\lesssim& \sup_{\lambda \geq 1}  \sum_{m \geq -1} 2^{-\kappa m p}  \int_{\mathbb{T}^{2}}   \lVert \Delta_{m} ( \nabla_{\text{spec}} ( \mathcal{L}_{\lambda} X_{u}, \mathcal{L}_{\lambda} X_{b}) \circ P^{\lambda} )_{4,4}(t) - r_{\lambda}(t) \rVert_{L_{\omega}^{2}}^{p}   dx \overset{\eqref{est 227}}{\lesssim} 1. 
\end{align*}
This leads to the convergence of \eqref{est 228} in $L^{p}$ for all $p \in [1,\infty)$. Concerning the convergence of \eqref{est 228} $\mathbb{P}$-a.s., we can compute for $\{ \lambda^{l}\}_{l \in \mathbb{N}}$, similarly to \eqref{est 227} 
\begin{align}
&\mathbb{E} [ \lvert \Delta_{m} [  ( \nabla_{\text{spec}} (\mathcal{L}_{\lambda^{l}} X_{u}, \mathcal{L}_{\lambda^{l}} X_{b}) \circlesign{\circ} P^{\lambda^{l}} )_{4,4}(t) - r_{\lambda^{l}}(t) \label{est 231}\\
& \hspace{5mm} - (\nabla_{\text{spec}} ( \mathcal{L}_{\lambda^{l + 1}} X_{u}, \mathcal{L}_{\lambda^{l + 1}} X_{b}) \circlesign{\circ} P^{\lambda^{l+1}} )_{4,4}(t) + r_{\lambda^{l + 1}} (t) ] (x) \rvert^{2} ] \nonumber \\
\lesssim &  \sum_{k, ' \in\mathbb{Z}^{2} \setminus \{0\}} \rho_{m}^{2} (k+k') \lvert \psi_{0} (k, k') \rvert^{2} \left( \frac{1}{\lvert k' \rvert^{2} + 2} \right)^{2} [ 1_{[\lambda^{l}, \lambda^{l + 1} ]} (\lvert k \rvert) + 1_{[ \lambda^{l}, \lambda^{l + 1} ]} (\lvert k' \rvert) ] \nonumber \\
\lesssim& (\lambda^{l})^{-\frac{\kappa}{4}} \sum_{k, k' \in \mathbb{Z}^{2} \setminus \{0\}: \lvert k \rvert \approx 2^{m}, \lvert k' \rvert \gtrsim 2^{m}} \frac{ \lvert k'\rvert}{2^{m}} \left( \frac{1}{\lvert k' \rvert^{2} + 2} \right)^{2} [ \lvert k-k' \rvert^{\frac{\gamma}{4}} + \lvert k' \rvert^{\frac{\kappa}{4}} ] \lesssim (\lambda^{l})^{-\frac{\kappa}{4}} 2^{\frac{m\kappa}{4}}. \nonumber 
\end{align}
 We conclude via Gaussian hypercontractivity theorem that 
\begin{align*}
& \mathbb{E} [ \lVert ( \nabla_{\text{spec}} (\mathcal{L}_{\lambda^{l}} X_{u}, \mathcal{L}_{\lambda^{l}} X_{b} ) \circlesign{\circ} P^{\lambda^{l}})_{4,4} (t) - r_{\lambda^{l}}(t)  \nonumber \\
& \hspace{20mm} - (\nabla_{\text{spec}} (\mathcal{L}_{\lambda^{l + 1}} X_{b}, \mathcal{L}_{\lambda^{l + 1}} X_{b}) \circlesign{\circ} P^{\lambda^{l + 1}})_{4,4}(t) + r_{\lambda^{l + 1}}(t) \rVert_{B_{p,p}^{-\kappa}}^{p} ] \nonumber \\ 
\lesssim& \sum_{m=-1}^{\infty} 2^{-\kappa p m} \int_{\mathbb{T}^{2}} \lVert \Delta_{m} [ ( \nabla_{\text{spec}} (\mathcal{L}_{\lambda^{l}} X_{u}, \mathcal{L}_{\lambda^{l}} X_{b}) \circlesign{\circ} P^{\lambda^{l}})_{4,4} - r_{\lambda^{l}}(t) \nonumber \\
& \hspace{20mm} - (\nabla_{\text{spec}} (\mathcal{L}_{\lambda^{l + 1}} X_{u}, \mathcal{L}_{\lambda^{l + 1}} X_{b} )\circlesign{\circ} P^{\lambda^{l + 1}})_{4,4} + r_{\lambda^{l + 1}}(t) ] \rVert_{L_{\omega}^{2}}^{p} dx \overset{\eqref{est 231}}{\lesssim} (\lambda^{l})^{- \frac{\kappa p}{8}}.  
\end{align*}
\end{proof} 

\section{Appendix}\label{Appendix}
\subsection{Proof of \eqref{est 122}}
In this subsection we prove \eqref{est 122}. First, we work on one of the terms in $\RomanII_{4,2}$, e.g. $2 \langle ( -\Delta)^{\epsilon} w_{b}^{\mathcal{L}}, \divergence C^{\circlesign{\prec}_{a}} (w_{b}, Q_{u}^{\mathcal{H}}) \rangle$ that we rewrite for convenience using \eqref{est 30b} and \eqref{est 29b}, 
\begin{align}
& 2 \langle ( -\Delta)^{\epsilon} w_{b}^{\mathcal{L}}, \divergence C^{\circlesign{\prec}_{a}} (w_{b}, Q_{u}^{\mathcal{H}}) \rangle \label{est 121} \\
=& - 2 \langle (-\Delta)^{\epsilon} w_{b}^{\mathcal{L}}, \divergence [ [ \mathbb{P}_{L} \divergence ( w_{b} \otimes w_{u} + w_{b} \otimes_{a} D_{u} + Y_{b} \otimes Y_{u}  \nonumber \\
& \hspace{25mm} - w_{u} \otimes w_{b} - w_{u} \otimes_{a} D_{b} - Y_{u} \otimes Y_{b} ) ] \circlesign{\prec}_{a} Q_{u}^{\mathcal{H}} - 2\nu  \sum_{k=1}^{2} \partial_{k} w_{b} \circlesign{\prec}_{a} \partial_{k} Q_{u}^{\mathcal{H}} ] \rangle.  \nonumber 
\end{align}
Making use of \eqref{est 120} with $\mathfrak{a}$ = $3$ by hypothesis, we compute 
\begin{align}
& -2\langle (-\Delta)^{\epsilon} w_{b}^{\mathcal{L}}, \divergence [[ \mathbb{P}_{L} \divergence ( w_{b} \otimes w_{u}) ] \circlesign{\prec}_{a} Q_{u}^{\mathcal{H}} ] \rangle(t)  \nonumber \\
\overset{\eqref{est 120}}{\lesssim}&  [ ( \lVert w_{u}^{\mathcal{L}} \rVert_{L^{2}} + \lVert w_{b}^{\mathcal{L}} \rVert_{L^{2}})^{\frac{ 3 \bar{\eta} - 2 \kappa - \bar{\gamma} - 2 \epsilon -1}{\bar{\eta}}}(t) ( \lVert w_{u}^{\mathcal{L}} \rVert_{H^{\bar{\eta}}} + \lVert  w_{b}^{\mathcal{L}} \rVert_{H^{\bar{\eta}}})^{\frac{1+ 2 \kappa + \bar{\gamma} + 2 \epsilon}{\bar{\eta}}}(t) \nonumber \\
&+ ( \lVert w_{u}^{\mathcal{L}} \rVert_{L^{2}} + \lVert w_{b}^{\mathcal{L}} \rVert_{L^{2}})^{\frac{ \bar{\eta} - 2 \kappa - \bar{\gamma} - 2 \epsilon}{\bar{\eta}}}(t) ( \lVert w_{u}^{\mathcal{L}} \rVert_{H^{\bar{\eta}}} + \lVert w_{b}^{\mathcal{L}} \rVert_{H^{\bar{\eta}}})^{\frac{2\kappa + \bar{\gamma} + 2 \epsilon}{\bar{\eta}}}(t) (N_{t}^{\kappa})^{2} ] \nonumber \\
& \hspace{40mm} \times  (1+ \lVert w_{u}(t) \rVert_{L^{2}} + \lVert w_{b} (t) \rVert_{L^{2}})^{-3 \bar{\gamma}} N_{t}^{\kappa}\nonumber \\ 
\leq& \frac{\nu }{64} \lVert (w_{u}^{\mathcal{L}}, w_{b}^{\mathcal{L}})(t) \rVert_{\dot{H}^{1+ \epsilon}}^{2} + C(M, N_{t}^{\kappa})  \label{est 125} 
\end{align}
via appropriate choices of $\bar{\eta}$ and $\bar{\gamma}$ similarly to \eqref{est 83}, e.g. 
\begin{equation}\label{est 124} 
\bar{\eta} = \frac{3}{4} + 3 \kappa + \epsilon \text{ and } \bar{\gamma} = \frac{1}{2}. 
\end{equation} 
Identical estimates show 
\begin{equation}\label{est 126} 
2\langle (-\Delta)^{\epsilon} w_{b}^{\mathcal{L}}, \divergence [[ \mathbb{P}_{L} \divergence ( w_{u} \otimes w_{b}) ] \circlesign{\prec}_{a} Q_{u}^{\mathcal{H}} ] \rangle(t) \leq \frac{\nu }{64} \lVert (w_{u}^{\mathcal{L}}, w_{b}^{\mathcal{L}})(t) \rVert_{\dot{H}^{1+ \epsilon}}^{2} + C(M, N_{t}^{\kappa}). 
\end{equation} 
Second, within \eqref{est 121}, with the same $\bar{\eta}$ from \eqref{est 124}, we can estimate 
\begin{align}
& -2 \langle (-\Delta)^{\epsilon} w_{b}^{\mathcal{L}}, \divergence [ [ \mathbb{P}_{L} \divergence (w_{b} \otimes_{a} D_{u} + Y_{b} \otimes Y_{u} ) ] \circlesign{\prec}_{a} Q_{u}^{\mathcal{H}} ] \rangle(t)\nonumber \\
&\overset{\eqref{est 123}}{\lesssim} \lVert w_{b}^{\mathcal{L}} (t)\rVert_{H^{\bar{\eta} + 2 \epsilon}} [ \lVert w_{b} \otimes_{a} D_{u} \rVert_{H^{-\bar{\eta} + \frac{3\kappa}{2}}} + \lVert Y_{b} \otimes Y_{u} \rVert_{H^{-\bar{\eta} + \frac{3\kappa}{2}}}](t) \lVert Q_{u} (t)\rVert_{\mathcal{C}^{2- \frac{3\kappa}{2}}} \nonumber \\
& \hspace{40mm} \leq \frac{\nu }{64} \lVert w_{b}^{\mathcal{L}} (t)\rVert_{\dot{H}^{1+ \epsilon}}^{2} + C(M, N_{t}^{\kappa}). \label{est 127} 
\end{align}
Analogous computations lead to 
\begin{equation}\label{est 128} 
2 \langle (-\Delta)^{\epsilon} w_{b}^{\mathcal{L}}, \divergence [ [ \mathbb{P}_{L} \divergence (w_{u} \otimes_{a} D_{b} + Y_{u} \otimes Y_{b} ) ] \circlesign{\prec}_{a} Q_{u}^{\mathcal{H}} ] \rangle(t) \leq \frac{\nu }{64} \lVert w_{b}^{\mathcal{L}} (t)\rVert_{\dot{H}^{1+ \epsilon}}^{2} + C(M, N_{t}^{\kappa}).
\end{equation} 
Lastly, we estimate 
\begin{align}
&4\nu  \sum_{k=1}^{2} \langle (-\Delta)^{\epsilon} w_{b}^{\mathcal{L}}, \divergence ( \partial_{k} w_{b} \circlesign{\prec}_{a} \partial_{k} Q_{u}^{\mathcal{H}} ) \rangle(t) \nonumber \\
\lesssim& \lVert w_{b}^{\mathcal{L}} (t)\rVert_{L^{2}}^{\frac{1- 12 \kappa - 8 \epsilon}{4 (1+ \epsilon)}}  \lVert w_{b}^{\mathcal{L}} (t)\rVert_{H^{1+ \epsilon}}^{\frac{3}{4} (\frac{1+ 4 \kappa + 4 \epsilon}{1+ \epsilon} ) } \lVert \partial_{k} w_{b} \circlesign{\prec}_{a} \partial_{k} Q_{u}^{\mathcal{H}} (t)\rVert_{H^{\frac{1}{4} - 3 \kappa - \epsilon}}  \nonumber \\
\overset{\eqref{est 107} \eqref{est 40c} }{\lesssim}&  C(M) \lVert w_{b}^{\mathcal{L}} (t)\rVert_{H^{1+ \epsilon}}^{\frac{3}{4} (\frac{1+ 4 \kappa + 4 \epsilon}{1+ \epsilon})} ( \lVert w_{b}^{\mathcal{L}} \rVert_{H^{\frac{1}{4} - \frac{3\kappa}{2} + \epsilon}} + \lVert w_{b}^{\mathcal{H}} \rVert_{H^{\frac{1}{4} - \frac{3\kappa}{2} + \epsilon}})(t) \lVert \mathcal{H}_{\lambda_{t}} Q_{u} (t)\rVert_{\mathcal{C}^{2- \frac{3\kappa}{2}}} \nonumber \\
& \hspace{40mm} \overset{\eqref{est 39} }{\leq} \frac{\nu }{64} \lVert w_{b}^{\mathcal{L}} (t)\rVert_{\dot{H}^{1+ \epsilon}}^{2} + C(M, N_{t}^{\kappa}). \label{est 129} 
\end{align}  
Thus, by applying \eqref{est 125}, \eqref{est 126}, \eqref{est 127}, \eqref{est 128}, and \eqref{est 129} to \eqref{est 121} gives 
\begin{equation*}
2 \langle ( -\Delta)^{\epsilon} w_{b}^{\mathcal{L}}, \divergence C^{\circlesign{\prec}_{a}} (w_{b}, Q_{u}^{\mathcal{H}}) \rangle(t)   \leq \frac{5\nu }{64} \lVert (w_{u}^{\mathcal{L}}, w_{b}^{\mathcal{L}} ) (t)\rVert_{\dot{H}^{1+\epsilon}}^{2} + C(M, N_{t}^{\kappa})  
\end{equation*} 
and analogous computations on similar terms lead to \eqref{est 122}.  
  
\subsection{Conclusion of the proof of Theorem \ref{Theorem 2.2}} 
Suppose that $T^{\max} < \infty. $ By Proposition \ref{Proposition 4.11} this implies $\limsup_{t\nearrow T^{\max}} \lVert (w_{u}, w_{b})(t) \rVert_{L^{2}} = + \infty$.  By \eqref{est 36} this implies $T_{i} < T^{\max}$ for all $i \in \mathbb{N}$. Because $T^{\max} < + \infty$, \eqref{est 105} gives us 
\begin{equation*}
T_{i+1} - T_{i} \geq \frac{1}{\tilde{C}(N_{T^{\max}}^{\kappa}) (1+ \ln (1+ i))} \ln \left( \frac{ i^{2} + 2i - C(N_{T^{\max}}^{\kappa})}{i^{2} + \tilde{C} (N_{T^{\max}}^{\kappa})} \right)
\end{equation*} 
where $\sum_{i=1}^{\infty} T_{i+1} - T_{i} < \infty$. On the other hand, the sum over the right hand side over $i \in \mathbb{N}$ blows up to $+ \infty$ and thus a contradiction. 

\subsection{Proof of \eqref{est 173}} 
For both inequalities in \eqref{est 173}, we first rely on the fact that $B_{p ,2}^{s} \subset W^{s,p}$ for all $p\in [2,\infty)$ and $s \in\mathbb{R}$ (see e.g. \cite[p. 152]{BL76}) so that 
\begin{equation*}
\lVert f \rVert_{L^{4}} \lesssim \lVert f \rVert_{B_{4,2}^{0}}  \lesssim \left( \sum_{m\geq -1} 2^{2m (\frac{1}{2} - \frac{1}{4})} \lVert \Delta_{m} f \rVert_{L^{2}} \lVert \Delta_{m} f \rVert_{L^{4}} \right)^{\frac{1}{2}} 
\lesssim  \lVert f \rVert_{L^{2}}^{\frac{1}{2}} \lVert f \rVert_{B_{4,2}^{\frac{1}{2}}}^{\frac{1}{2}}
\end{equation*} 
by Bernstein's inequality and H$\ddot{\mathrm{o}}$lder's inequality. Additionally, 
\begin{equation*}
\lVert f \rVert_{L^{4}} \lesssim \lVert f \rVert_{B_{4,2}^{0}} \lesssim \left( \sum_{m\geq -1} \lVert \Delta_{m} f \rVert_{L^{2}} \lVert \Delta_{m} f \rVert_{L^{\infty}} \right)^{\frac{1}{2}} \lesssim  \lVert f \rVert_{L^{2}}^{\frac{1}{2}} \lVert f \rVert_{B_{\infty,2}^{0}}^{\frac{1}{2}}
\end{equation*} 
by interpolation inequality of $L^{p}$ spaces and H$\ddot{\mathrm{o}}$lder's inequality. 

\subsection{Details of \eqref{est 212}}
In \eqref{est 213} we described the $(1,1)$ and $(4,4)$ entries of \eqref{est 212}; we leave the rest here for completeness. 
\begin{align*}
& (1,1): \hspace{1mm} ( \partial_{1} X_{u,\lambda}^{1} + \partial_{1} X_{u,\lambda}^{1}) \circ \left( -\frac{\nu \Delta}{2}+ 1 \right)^{-1} (\partial_{1} X_{u,\lambda}^{1} + \partial_{1} X_{u,\lambda}^{1}) \nonumber \\
& \hspace{30mm} + ( \partial_{1} X_{u,\lambda}^{2} + \partial_{2} X_{u,\lambda}^{1}) \circ \left( -\frac{\nu \Delta}{2} + 1\right)^{-1} (\partial_{2} X_{u,\lambda}^{1} + \partial_{1} X_{u,\lambda}^{2}) \nonumber \\
& \hspace{30mm } - (\partial_{1} X_{b,\lambda}^{2} - \partial_{2} X_{b,\lambda}^{1}) \circ \left( -\frac{\nu \Delta}{2}+1\right)^{-1} (\partial_{2} X_{b,\lambda}^{1} - \partial_{1} X_{b,\lambda}^{2}),   \\
& (1,3): \hspace{1mm} ( \partial_{1} X_{u,\lambda}^{2} + \partial_{2} X_{u,\lambda}^{1}) \circ \left(-\frac{\nu \Delta}{2} + 1 \right)^{-1} (\partial_{2} X_{b,\lambda}^{1} - \partial_{1} X_{b,\lambda}^{2}) \nonumber \\
& \hspace{30mm} - (\partial_{1} X_{b,\lambda}^{2} - \partial_{2} X_{b,\lambda}^{1} ) \circ \left(-\frac{\nu \Delta}{2} + 1 \right)^{-1} (\partial_{2} X_{u,\lambda}^{1} + \partial_{1} X_{u,\lambda}^{2} ),  \\
& (1,4):  \hspace{1mm}  ( \partial_{1} X_{u,\lambda}^{1} + \partial_{1} X_{u,\lambda}^{1}) \circ \left(-\frac{\nu \Delta}{2} + 1 \right)^{-1} (\partial_{1} X_{b,\lambda}^{2} - \partial_{2} X_{b,\lambda}^{1}) \nonumber \\
& \hspace{30mm} - (\partial_{1} X_{b,\lambda}^{2} - \partial_{2} X_{b,\lambda}^{1} ) \circ \left(-\frac{\nu \Delta}{2} + 1 \right)^{-1} (\partial_{2} X_{u,\lambda}^{2} + \partial_{2} X_{u,\lambda}^{2} ), \\
& (2,1):  \hspace{1mm}  ( \partial_{2} X_{u,\lambda}^{1} + \partial_{1} X_{u,\lambda}^{2}) \circ \left(-\frac{\nu \Delta}{2} + 1 \right)^{-1} (\partial_{1} X_{u,\lambda}^{1} + \partial_{1} X_{u,\lambda}^{1}) \nonumber \\
& \hspace{30mm} + (\partial_{2} X_{u,\lambda}^{2} + \partial_{2} X_{u,\lambda}^{2} ) \circ \left(-\frac{\nu \Delta}{2} + 1 \right)^{-1} (\partial_{2} X_{u,\lambda}^{1} + \partial_{1} X_{u,\lambda}^{2} ), \\
& (2,2): \hspace{1mm}  ( \partial_{2} X_{u,\lambda}^{1} + \partial_{1} X_{u,\lambda}^{2}) \circ \left( -\frac{\nu \Delta}{2}+ 1 \right)^{-1} (\partial_{1} X_{u,\lambda}^{2} + \partial_{2} X_{u,\lambda}^{1}) \nonumber \\
& \hspace{30mm} + ( \partial_{2} X_{u,\lambda}^{2} + \partial_{2} X_{u,\lambda}^{2}) \circ \left( -\frac{\nu \Delta}{2} + 1\right)^{-1} (\partial_{2} X_{u,\lambda}^{2} + \partial_{2} X_{u,\lambda}^{2}) \nonumber \\
& \hspace{30mm } - (\partial_{2} X_{b,\lambda}^{1} - \partial_{1} X_{b,\lambda}^{2}) \circ \left( -\frac{\nu \Delta}{2}+1 \right)^{-1} (\partial_{1} X_{b,\lambda}^{2} - \partial_{2} X_{b,\lambda}^{1}),  \\
& (2,3):  \hspace{1mm}  ( \partial_{2} X_{u,\lambda}^{2} + \partial_{2} X_{u,\lambda}^{2}) \circ \left(-\frac{\nu \Delta}{2} + 1 \right)^{-1} (\partial_{2} X_{b,\lambda}^{1} - \partial_{1} X_{b,\lambda}^{2}) \nonumber \\
& \hspace{30mm} - (\partial_{2} X_{b,\lambda}^{1} - \partial_{1} X_{b,\lambda}^{2} ) \circ \left(-\frac{\nu \Delta}{2} + 1 \right)^{-1} (\partial_{1} X_{u,\lambda}^{1} + \partial_{1} X_{u,\lambda}^{1} ), \\
& (2,4):  \hspace{1mm}  ( \partial_{2} X_{u,\lambda}^{1} + \partial_{1} X_{u,\lambda}^{2}) \circ \left(-\frac{\nu \Delta}{2} + 1 \right)^{-1} (\partial_{1} X_{b,\lambda}^{2} - \partial_{2} X_{b,\lambda}^{1}) \nonumber \\
& \hspace{30mm} - (\partial_{2} X_{b,\lambda}^{1} - \partial_{1} X_{b,\lambda}^{2} ) \circ \left(-\frac{\nu \Delta}{2} + 1 \right)^{-1} (\partial_{1} X_{u,\lambda}^{2} + \partial_{2} X_{u,\lambda}^{1} ), \\
& (3,1):  -( \partial_{1} X_{b,\lambda}^{2} - \partial_{2} X_{b,\lambda}^{1}) \circ \left(-\frac{\nu \Delta}{2} + 1 \right)^{-1} (\partial_{2} X_{u,\lambda}^{1} + \partial_{1} X_{u,\lambda}^{2}) \nonumber \\
& \hspace{30mm} + (\partial_{1} X_{u,\lambda}^{2} + \partial_{2} X_{u,\lambda}^{1} ) \circ \left(-\frac{\nu \Delta}{2} + 1 \right)^{-1} (\partial_{2} X_{b,\lambda}^{1} - \partial_{1} X_{b,\lambda}^{2} ),\\
& (3,2): -( \partial_{1} X_{b,\lambda}^{2} - \partial_{2} X_{b,\lambda}^{1}) \circ \left(-\frac{\nu \Delta}{2} + 1 \right)^{-1} (\partial_{2} X_{u,\lambda}^{2} + \partial_{2} X_{u,\lambda}^{2}) \nonumber \\
& \hspace{30mm} + (\partial_{1} X_{u,\lambda}^{1} + \partial_{1} X_{u,\lambda}^{1} ) \circ \left(-\frac{\nu \Delta}{2} + 1 \right)^{-1} (\partial_{1} X_{b,\lambda}^{2} - \partial_{2} X_{b,\lambda}^{1} ),\\
& (3,3): - (\partial_{1} X_{b,\lambda}^{2} - \partial_{2} X_{b,\lambda}^{1}) \circ \left( -\frac{\nu \Delta}{2}+ 1 \right)^{-1} (\partial_{2} X_{b,\lambda}^{1} - \partial_{1} X_{b,\lambda}^{2}) \nonumber \\
& \hspace{30mm} + ( \partial_{1} X_{u,\lambda}^{1} + \partial_{1} X_{u,\lambda}^{1}) \circ \left( -\frac{\nu \Delta}{2} + 1\right)^{-1} (\partial_{1} X_{u,\lambda}^{1} + \partial_{1} X_{u,\lambda}^{1}) \nonumber \\
& \hspace{30mm } + (\partial_{1} X_{u,\lambda}^{2} + \partial_{2} X_{u,\lambda}^{1}) \circ \left( -\frac{\nu \Delta}{2}+1\right)^{-1} (\partial_{2} X_{u,\lambda}^{1} + \partial_{1} X_{u,\lambda}^{2}), \\
& (3,4): \hspace{1mm}  (\partial_{1} X_{u,\lambda}^{1} + \partial_{1} X_{u,\lambda}^{1}) \circ \left(-\frac{\nu \Delta}{2} + 1 \right)^{-1} (\partial_{1} X_{u,\lambda}^{2} + \partial_{2} X_{u,\lambda}^{1}) \nonumber \\
& \hspace{30mm} + (\partial_{1} X_{u,\lambda}^{2} + \partial_{2} X_{u,\lambda}^{1} ) \circ \left(-\frac{\nu \Delta}{2} + 1 \right)^{-1} (\partial_{2} X_{u,\lambda}^{2} + \partial_{2} X_{u,\lambda}^{2} ),\\
& (4,1): - (\partial_{2} X_{b,\lambda}^{1} - \partial_{1} X_{b,\lambda}^{2}) \circ \left(-\frac{\nu \Delta}{2} + 1 \right)^{-1} (\partial_{1} X_{u,\lambda}^{1} + \partial_{1} X_{u,\lambda}^{1}) \nonumber \\
& \hspace{30mm} + (\partial_{2} X_{u,\lambda}^{2} + \partial_{2} X_{u,\lambda}^{2} ) \circ \left(-\frac{\nu \Delta}{2} + 1 \right)^{-1} (\partial_{2} X_{b,\lambda}^{1} - \partial_{1} X_{b,\lambda}^{2}),\\
& (4,2): - ( \partial_{2} X_{b,\lambda}^{1} - \partial_{1} X_{b,\lambda}^{2}) \circ \left(-\frac{\nu \Delta}{2} + 1 \right)^{-1} (\partial_{1} X_{u,\lambda}^{2} + \partial_{2} X_{u,\lambda}^{1}) \nonumber \\
& \hspace{30mm} + (\partial_{2} X_{u,\lambda}^{1} + \partial_{1} X_{u,\lambda}^{2} ) \circ \left(-\frac{\nu \Delta}{2} + 1 \right)^{-1} (\partial_{1} X_{b,\lambda}^{2} - \partial_{2} X_{b,\lambda}^{1} ),\\
& (4,3): \hspace{1mm}  ( \partial_{2} X_{u,\lambda}^{1} +  \partial_{1} X_{u,\lambda}^{2}) \circ \left(-\frac{\nu \Delta}{2} + 1 \right)^{-1} (\partial_{1} X_{u,\lambda}^{1} + \partial_{1} X_{u,\lambda}^{1}) \nonumber \\
& \hspace{30mm} + (\partial_{2} X_{u,\lambda}^{2} + \partial_{2} X_{u,\lambda}^{2} ) \circ \left(-\frac{\nu \Delta}{2} + 1 \right)^{-1} (\partial_{2} X_{u,\lambda}^{1} + \partial_{1} X_{u,\lambda}^{2} ). 
\end{align*}
All the mathematical expectations of $(1,3)$, $(1,4)$, $(2,3)$, $(2,4)$, $(3,1)$, $(3,2)$, $(4,1)$, and $(4,2)$ entries immediately vanish due to $1_{\{ \alpha = \gamma \}}$ from \eqref{est 214}. Additionally, the entries of $(2,1)$, $(3,4)$, and $(4,3)$ also vanish just like the $(1,2)$-entry. Finally, all of $(2,2)$, $(3,3)$, and $(4,4)$-entries equal $4r_{\lambda}(t)$ just like the $(1,1)$-entry. 

\section*{Acknowledgments}
This work is supported by the Simons Foundation (962572, KY). The author thanks Prof. Zdzis$\l$aw Brze$\acute{\mathrm{z}}$niak, Prof. Paul Razafimandimby, Prof. Jiahong Wu, Prof. Dimitri Volchenkov, and Prof. Vincent Martinez for valuable discussions.


\begin{thebibliography}{100}  
\addtolength{\leftmargin}{0.2in} 
\setlength{\itemindent}{-0.2in} 

\bibitem{ACHS81} G. Ahlers, M. C. Cross, P. C. Hohenberg, and S. Safran, \emph{The amplitude equation near the convective threshold: application to time-dependent heating experiments}, J. Fluid Mech., \textbf{110} (1981), pp. 297--334. 

\bibitem{A42a} H. Alfv$\acute{\mathrm{e}}$n, \emph{On the existence of electromagnetic-hydrodynamic waves}, Nature, \textbf{150} (1942), pp. 405--406.

\bibitem{AC15} R. Allez and K. Chouk, \emph{The continuous Anderson hamiltonian in dimension two}, arXiv:1511.02718 [math.PR], 2015. 

\bibitem{BCD11} H. Bahouri, J.-Y. Chemin and R. Danchin, \emph{Fourier Analysis and Nonlinear Partial Differential Equations}, Springer-Verlag, Berlin Heidelberg, 2011. 

\bibitem{BMR14} D. Barbato, F. Morandin, and M. Romito, \emph{Global regularity for a slightly supercritical hyperdissipative Navier-Stokes system}, Anal. PDE, \textbf{7} (2014), pp. 2009--2027. 

\bibitem{BD07} V. Barbu and G. Da Prato, \emph{Existence and ergodicity for the two-dimensional stochastic magneto-hydrodynamics equations}, Appl. Math. Optim., \textbf{56} (2007), pp. 145--168. 

\bibitem{B50} G. K. Batchelor, \emph{On the spontaneous magnetic field in a conducting liquid in turbulent motion}, Proc. R. Soc. Lond. Ser. A, \textbf{201} (1950), pp. 405--416. 

\bibitem{BBV21} R. Beekie, T. Buckmaster, and V. Vicol, \emph{Weak solutions of ideal MHD which do not conserve magnetic helicity}, Annals of PDE, \textbf{6} (2020), pp. 1--40. 

\bibitem{BL76} J. Bergh and J. L$\ddot{\mathrm{o}}$str$\ddot{\mathrm{o}}$m, \emph{Interpolation Spaces, An Introduction}, Springer-Verlag, Berlin, Heidelberg, New York, 1976. 

\bibitem{BCJ94} L. Bertini, N. Cancrini, and G. Jona-Lasinio, \emph{The stochastic Burgers equation}, Comm. Math. Phys., \textbf{165} (1994), pp. 211--232. 

\bibitem{BG97} L. Bertini and G. Giacomin, \emph{Stochastic Burgers and KPZ equations from particle systems}, Comm. Math. Phys., \textbf{183} (1997), pp.  571--607.  

\bibitem{BV19b} T. Buckmaster and V. Vicol, \emph{Convex integration and phenomenologies in turbulence}, EMS Surveys in Mathematical Sciences, \textbf{6} (2019), pp. 173--263. 

\bibitem{CT92} S. J. Camargo and H. Tasso, \emph{Renormalization group in magnetohydrodynamic turbulence}, Phys. Fluids B, \textbf{4} (1992), pp. 1199--1212.

\bibitem{CWY14} C. Cao, J. Wu, and B. Yuan, \emph{The 2D incompressible magnetohydrodynamics equations with only magnetic diffusion}, SIAM J. Math. Anal., \textbf{46} (2014), pp. 588--602. 

\bibitem{CC18} R. Catellier and K. Chouk, \emph{Paracontrolled distributions and the 3 dimensional stochastic quantization equation}, Ann. Probab., \textbf{46} (2018), pp. 2621--2679. 

\bibitem{C51} S. Chandrasekhar, \emph{The invariant theory of isotropic turbulence in magneto-hydrodynamics}, Proc. R. Soc. Lond. Ser. A, \textbf{204} (1951), pp.  435--449. 

\bibitem{CM10} I. Chueshov and A. Millet, \emph{Stochastic 2D hydrodynamical type systems: well posedness and large deviations}, Appl. Math. Optim., \textbf{61} (2010), pp. 379--420. 

\bibitem{DD02} G. Da Prato and A. Debussche, \emph{Two-dimensional Navier-Stokes equations driven by a space-time white noise}, J. Funct. Anal., \textbf{196} (2002), pp. 180--210. 

\bibitem{DD03b} G. Da Prato and A. Debussche, \emph{Strong solutions to the stochastic quantization equations}, Ann. Probab., \textbf{31} (2003), pp. 1900--1916. 

\bibitem{DL72} G. Duvaut and J. L. Lions, \emph{In$\acute{e}$quations en thermo$\acute{e}$lasticit$\acute{e}$ et magn$\acute{e}$tohydrodynamique}, Arch. Ration. Mech. Anal., \textbf{46} (1972), pp. 241--279.

\bibitem{FMMNZ14} J. Fan, H. Malaikah, S. Monaquel, G. Nakamura, and Y. Zhou, \emph{Global Cauchy problem of 2D generalized MHD equations}, Monatsch. Math., \textbf{175} (2014), pp. 127--131. 

\bibitem{FNS77} D. Forster, D. R. Nelson, and M. J. Stephen, \emph{Large-distance and long-time properties of a randomly stirred fluid}, Phys. Rev. A, \textbf{16} (1977), pp. 732--749. 

\bibitem{FSP82} J.-D. Fournier, P.-L. Sulem, and A. Pouquet, \emph{Infrared properties of forced magnetohydrodynamic turbulence}, J. Phys. A: Math. Gen. \textbf{15} (1982), pp. 1393--1420. 

\bibitem{GP75} R. Graham and H. Pleiner, \emph{Mode-mode coupling theory of the heat convection threshold}, The Physics of Fluids, \textbf{18} (1975), pp. 130--140. 

\bibitem{GIP15} M. Gubinelli, P. Imkeller, and N. Perkowski, \emph{Paracontrolled distributions and singular PDEs}, Forum Math., \textbf{3} (2015), pp. 1--75. 

\bibitem{H11} M. Hairer, \emph{Rough stochastic PDEs}, Comm. Pure Appl. Math., \textbf{LXIV} (2011), pp. 1547--1585.

\bibitem{H13} M. Hairer, \emph{Solving the KPZ equation}, Ann. of Math., \textbf{178} (2013), pp. 559--664. 

\bibitem{H14} M. Hairer, \emph{A theory of regularity structures}, Invent. Math., \textbf{198} (2014), pp. 269--504. 

\bibitem{HM18a} M. Hairer and J. C. Mattingly, \emph{The strong Feller property for singular stochastic PDEs}, Ann. Inst. H. Poincar$\acute{\mathrm{e}}$ Probab. Stat., \textbf{54} (2018), pp. 1314--1340. 

\bibitem{HR23} M. Hairer and T. Rosati, \emph{Global existence for perturbations of the 2D stochastic Navier-Stokes equations with space-time white noise}, arXiv:2301.11059 [math.PR], 2023. 

\bibitem{HZZ19} M. Hofmanov$\acute{\mathrm{a}}$, R. Zhu, and X. Zhu, \emph{Non-uniqueness in law of stochastic 3D Navier-Stokes equations}, J. Eur. Math. Soc., DOI 10.4171/JEMS/1360, 2023. 

\bibitem{HZZ21b} M. Hofmanov$\acute{\mathrm{a}}$, R. Zhu, and X. Zhu, \emph{Global existence and non-uniqueness for 3D Navier-Stokes equations with space-time white noise}, Arch. Ration. Mech. Anal. \textbf{247}, (2023) \url{https://doi.org/10.1007/s00205-023-01872-x}  

\bibitem{HZZ22} M. Hofmanov$\acute{\mathrm{a}}$, R. Zhu, and X. Zhu, \emph{A class of supercritical/critical singular stochastic PDEs: existence, non-uniqueness, non-Gaussianity, non-unique ergodicity},  J. Funct. Anal. \textbf{285}, (2023) \url{https://doi.org/10.1016/j.jfa.2023.110011} 

\bibitem{HS92} P. C. Hohenberg and J. B. Swift, \emph{Effects of additive noise at the onset of Rayleigh-B$\acute{e}$nard convection}, Physical Review A, \textbf{46} (1992), pp. 4773--4785. 

\bibitem{J97} S. Janson, \emph{Gaussian Hilbert Spaces}, Cambridge University Press, United Kingdom, 1997.   

\bibitem{JZ14a} Q. Jiu and J. Zhao, \emph{Global regularity of 2D generalized MHD equations with magnetic diffusion}, Z. Angew. Math. Phys., \textbf{66} (2015), pp. 677--687. 

\bibitem{KPZ86} M. Kardar, G. Parisi, and Y.-C. Zhang, \emph{Dynamic scaling of growing interfaces}, Phys. Rev. Lett., \textbf{56} (1986), pp. 889--892.  

\bibitem{K67} R. H. Kraichnan, \emph{Inertial ranges in two-dimensional turbulence}, The Physics of Fluids, \textbf{10} (1967), pp. 1417--1423. 

\bibitem{LL57} L. D. Landau and E. M. Lifshitz, \emph{Hydrodynamic fluctuations}, J. Exptl. Theoret. Phys. (U.S.S.R.) \textbf{32} (1957), pp. 618--619.

\bibitem{LZ23} H. L$\ddot{\mathrm{u}}$ and X. Zhu, \emph{Sharp non-uniqueness of solutions to 2D Navier-Stokes equations with space-time white noise}, arXiv:2304.06526 [math.PR], 2023. 

\bibitem{L98} T. Lyons, \emph{Differential equations driven by rough signals}, Rev. Mat. Iberoam., \textbf{14} (1998), pp. 215--310. 

\bibitem{MM75} S.-k. Ma and G. F. Mazenko, \emph{Critical dynamics of ferromagnets in $6-\epsilon$ dimensions: General discussion and detailed calculation}, Phys. Rev. B, \textbf{11} (1975), pp. 4077--4100.  

\bibitem{M22} E. Motyl, \emph{Stochastic magneto-hydrodynamic equations (MHD): invariant measures in 2D Poincar$\acute{\mathrm{e}}$ domains}, J. Math. Anal. Appl., \textbf{514} (2022), 126317. 

\bibitem{MW17} J.-C. Mourrat and H. Weber, \emph{Global well-posedness of the dynamic $\Phi^{4}$ model in the plane}, Ann. Probab., \textbf{45} (2017), pp. 2398--2476. 

\bibitem{N65} E. A, Novikov, \emph{Functionals and the random-force method in turbulence theory}, Soviet Phys. JETP, \textbf{20} (1965), pp. 1290--1294. 

\bibitem{N95} D. Nualart, \emph{The Malliavin Calculus and Related Topics}, Springer-Verlag, New York, Berlin, Heidelberg, 1995. 

\bibitem{PW81} G. Parisi and Y. Wu, \emph{Perturbation theory without gauge fixing}, Scientia Sinica, \textbf{XXIV} (1981), pp. 483--496. 

\bibitem{S10} M. Sango, \emph{Magnetohydrodynamic turbulent flows: existence results}, Phys. D., \textbf{239} (2010), pp. 912--923. 

\bibitem{S21} A. Schenke, \emph{The stochastic tamed MHD equations: existence, uniqueness and invariant measures}, Stoch. PDE: Anal. Comp. (2021), https://doi.org/10.1007/s40072-021-00205-x. 

\bibitem{ST83} M. Sermange and R. Temam, \emph{Some mathematical questions related to the MHD equations}, Comm. Pure Appl. Math., \textbf{36} (1983), pp. 635--664. 

\bibitem{S90} J. Simon, \emph{Nonhomogeneous viscous incompressible fluids: existence of velocity, density, and pressure}, SIAM J. Math. Anal., \textbf{21} (1990), pp. 1093--1117. 

\bibitem{SS99} S. S. Sritharan and P. Sundar, \emph{The stochastic magneto-hydrodynamic system}, Infin. Dimens. Anal. Quantum Probab. Relat. Top., \textbf{2} (1999), pp. 241--265.  

\bibitem{SH77} J. Swift and P. C. Hohenberg, \emph{Hydrodynamic fluctuations at the convective stability}, Phys. Rev. A, \textbf{15} (1977), pp. 319--328. 

\bibitem{T07} T. Tao, Global regularity for a logarithmically supercritical defocusing nonlinear wave equation for spherically symmetric data, {\it J. Hyperbolic Differ. Equ.}, \textbf{4} (2007), pp. 259--266. 

\bibitem{T09} T. Tao, \emph{Global regularity for a logarithmically supercritical hyperdissipative Navier-Stokes equation}, Anal. PDE, \textbf{2} (2009), pp. 361--366.

\bibitem{V09} D. Volchenkov, \emph{Renormalization group and instantons in stochastic nonlinear dynamics}, Eur. Phys. J. Special Topics, \textbf{170} (2009), pp. 1--142. 

\bibitem{W11} J. Wu, \emph{Global regularity for a class of generalized magnetohydrodynamic equations}, J. Math. Fluid Mech., \textbf{13} (2011), pp. 295--305. 

\bibitem{YO86} V. Yakhot and S. A. Orszag, \emph{Renormalization group analysis of turbulence. I. basic theory}, J. Sci. Comput., \textbf{1} (1986), pp. 3--51. 


\bibitem{Y14a} K. Yamazaki, \emph{Global regularity of logarithmically supercritical MHD system with zero diffusivity}, Appl. Math. Lett., \textbf{29} (2014), pp. 46--51. 

\bibitem{Y14b} K. Yamazaki, \emph{Remarks on the global regularity of two-dimensional magnetohydrodynamics system with zero dissipation}, Nonlinear Anal., \textbf{94} (2014), pp. 194--205. 

\bibitem{Y18} K. Yamazaki, \emph{Global regularity of logarithmically supercritical MHD system with improved logarithmic powers}, Dyn. Partial Differ. Equ., \textbf{15} (2018), pp. 147--173. 

\bibitem{Y18a} K. Yamazaki, \emph{Second proof of the global regularity of the two-dimensional MHD system with full diffusion and arbitrary weak dissipation}, Methods Appl. Anal., International Press of Boston, \textbf{25} (2018), pp. 73--96. 

\bibitem{Y19a} K. Yamazaki, \emph{Markov selections for the magnetohydrodynamics and the Hall-magnetohydrodynamics systems}, J. Nonlinear Sci., \textbf{29} (2019), pp. 1761--1812. 

\bibitem{Y20a} K. Yamazaki, \emph{Approximating three-dimensional magnetohydrodynamics system forced by space-time white noise}, arXiv:2002.12732 [math.AP], 2020. 

\bibitem{Y20b} K. Yamazaki, \emph{Irreducibility of the three, and two and a half dimensional Hall-magnetohydrodynamics system}, Phys. D, \textbf{401} (2020), 13299 https://doi.org/10.1016/j.physd.2019.132199.

\bibitem{Y21a} K. Yamazaki, \emph{Strong Feller property of the magnetohydrodynamics system forced by space-time white noise}, Nonlinearity, \textbf{34} (2021), https://doi.org/10.1088/1361-6544/abfae7.

\bibitem{Y21d} K. Yamazaki, \emph{Non-uniqueness in law of three-dimensional magnetohydrodynamics system forced by random noise},   arXiv:2109.07015 [math.AP], 2021. 

\bibitem{Y23a} K. Yamazaki, \emph{Three-dimensional magnetohydrodynamics system forced by space-time white noise}, Electron. J. Probab., \textbf{28} (2023), pp. 1--66. 

\bibitem{Y63} V. Yudovich, \emph{Non stationary flows of an ideal incompressible fluid}, Zhurnal Vych Matematika, \textbf{3} (1963), pp. 1032--1066. 

\bibitem{ZS71} V. M. Za$\check{\mathrm{i}}$tsev and M. I. Shliomis, \emph{Hydrodynamic fluctuations near the convection threshold}, Soviet Physics JETP, \textbf{32} (1971), pp. 866--870. 

\bibitem{ZZ15} R. Zhu and X. Zhu, \emph{Three-dimensional Navier-Stokes equations driven by space-time white noise}, J. Differential Equations, \textbf{259} (2015), pp. 4443--4508. 

\bibitem{ZZ17a} R. Zhu and X. Zhu, \emph{Approximating 3D Navier-Stokes equations driven by space-time white noise}, Infin. Dimens. Anal. Quantum Probab. Relat. Top., \textbf{20} (2017), 1750020. 

\end{thebibliography}
\end{document}